\documentclass[a4paper,11pt]{article}

\usepackage[a4paper,marginparwidth=3cm]{geometry}
\usepackage{aeguill}                 
\usepackage{graphicx}
\usepackage{bmpsize}
\usepackage{amsmath,amsfonts,amssymb,amsthm}
\usepackage{xcolor} 
\usepackage{srcltx}
\usepackage[utf8]{inputenc} 
\usepackage[T1]{fontenc} 
\usepackage[english]{babel} 
\usepackage[all]{xy}
\usepackage{bbm}
\usepackage{hyperref}

\title{Algebraic laminations for free products and arational trees}
\author{Vincent Guirardel and Camille Horbez}

\newtheorem{de}{Definition} [section]
\newtheorem{thmbis}{Theorem} 
\newtheorem{propbis}[thmbis]{Proposition} 
\newtheorem{theo}[de]{Theorem} 
\newtheorem*{theo*}{Theorem} 
\newtheorem{prop}[de]{Proposition}
\newtheorem{lemma}[de]{Lemma}
\newtheorem{cor}[de]{Corollary}

\theoremstyle{remark}
\newtheorem{rk}[de]{Remark}
\newtheorem{ex}[de]{Example}

\normalsize

\newcommand{\Out}{\mathrm{Out}} 
\newcommand{\ie}{i.~e. }
\newcommand{\diam}{\mathrm{diam}}
\newcommand{\imp}{\Rightarrow}
\newcommand{\ra}{\rightarrow}
\newcommand{\m}{^{-1}}
\newcommand{\dunion}{\sqcup}
\newcommand{\eps}{\varepsilon}
\renewcommand{\epsilon}{\varepsilon}
\newcommand{\calf}{\mathcal{F}}
\newcommand{\calg}{\mathcal{G}}
\newcommand{\cald}{\mathcal{D}}
\newcommand{\calc}{\mathcal{C}}
\newcommand{\cali}{\mathcal{I}}
\newcommand{\caly}{\mathcal{Y}}

\newcommand{\calz}{\mathcal{Z}}
\newcommand{\calo}{\mathcal{O}}
\newcommand{\calb}{\mathcal{B}}
\newcommand{\calp}{\mathcal{P}}
\newcommand{\calk}{\mathcal{K}}
\newcommand{\calq}{\mathcal{Q}}

\newcommand{\call}{\mathcal{L}}
\newcommand{\calh}{\mathcal{H}}

\newcommand{\Axis}{\mathrm{Axis}}
\newcommand{\Lip}{\mathrm{Lip}}
\newcommand{\vol}{\mathrm{vol}}
\newcommand{\Term}{\mathrm{Term}}

\newcommand{\Z}{\mathcal{Z}}
\newcommand{\bbR}{\mathbb{R}}
\newcommand{\bbH}{\mathbb{H}}

\newcommand{\bbN}{\mathbb{N}}
\newcommand{\bbZ}{\mathbb{Z}}
\newcommand{\actson}{\curvearrowright}
\newcommand{\es}{\emptyset}
\newcommand{\grp}[1]{\langle #1 \rangle}
\newcommand{\dg}{\dagger}
\newcommand{\ol}{\overline}
\newcommand{\ul}{\underline}

\makeatletter
\edef\@tempa#1#2{\def#1{\mathaccent\string"\noexpand\accentclass@#2 }}
\@tempa\rond{017}
\makeatother

\begin{document}
\maketitle

\begin{abstract}
This work is the first step towards a description of the Gromov boundary of the free factor graph of a free product, with applications to subgroup classification for outer automorphisms.

We extend the theory of algebraic laminations dual to trees,
as developed by Coulbois, Hilion, Lustig and Reynolds, to the context of free products; this also gives us an opportunity to give a unified account of this theory. We first show that any $\bbR$-tree with dense orbits in the boundary of the corresponding outer space can be reconstructed as a quotient of the boundary of the group by its dual lamination. We then describe the dual lamination in terms of a band complex on compact $\bbR$-trees
(generalizing Coulbois--Hilion--Lustig's compact heart),  and we analyze this band complex using versions of the Rips machine and of the Rauzy--Veech induction.
An important output of the theory is that the above map from the boundary of the group to the $\bbR$-tree is 2-to-1 almost everywhere.

A key point for our intended application is a unique duality result for arational trees.
It says that if two trees have a leaf in common in their dual laminations, and if one of the trees is arational and relatively free, then they are equivariantly homeomorphic.

This statement is an analogue of a result in the free group
saying that if two trees are dual to a common current
and one of the trees is free arational, then the two trees are equivariantly homeomorphic.
However, we notice that in the setting of free products,
 the continuity of the pairing between trees and currents fails.
For this reason, in all this paper, we work with laminations rather than with currents.


\end{abstract}

\setcounter{tocdepth}{1}
\tableofcontents

\section*{Introduction}

In analogy to curve complexes used to study mapping class groups of surfaces, the free factor graph of a free group $F_N$ has recently turned to be fruitful in the study of $\text{Out}(F_N)$. It is Gromov hyperbolic, as was proved by Bestvina--Feighn \cite{BF14} and the action of an automorphism of $F_N$ is
loxodromic if and only if it is fully irreducible. Its Gromov boundary was described by Bestvina--Reynolds \cite{BR13} and Hamenstädt \cite{Ham13} as the set of equivalence classes of \emph{arational} trees.

Our goal in the present paper and its sequel \cite{GH15} is to extend this description of the boundary 
of the free factor graph to the context of free products, with a view towards obtaining classification results for subgroups of their automorphism groups.

 In the case of free groups, this description of the Gromov boundary relies on the theory of dual laminations developed by Coulbois, Hilion, Lustig and Reynolds, which is used to establish a crucial unique duality result between arational trees and geodesic currents. The main goal of the present paper is to extend the theory of laminations to the context of free products, and state a unique duality result for arational trees in this context. However, geodesic currents are not well-adapted to free products, because there  cannot be any natural continuous intersection pairing between $\bbR$-trees and currents in this broader context. To bypass this difficulty, the duality statement will be phrased purely in terms of dual laminations.

\paragraph*{Free products.}
Consider a free product $G=G_1*\dots * G_k* F_N$, where all groups $G_i$ are countable, and $F_N$ is a free group. A subgroup of $G$ is called \emph{peripheral} if it is conjugate into some $G_i$; we denote by $\calf$ the finite collection of all conjugacy classes of the groups $G_i$. 
In this context, one considers actions on trees relative to $\calf$, or $(G,\calf)$-trees: 
these are trees with an action of $G$ in which every $G_i$ fixes a point.
The Bass--Serre tree $R_0$ of the graph of groups decomposition 
$$\includegraphics{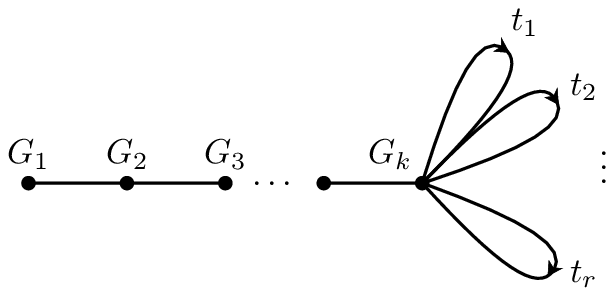}$$
(where one replaces the free product with $F_r$ above by  $N$ successive HNN extensions over the trivial group) 
is part of an outer space $\calo$ generalizing Culler--Vogtmann's outer space \cite{CuVo_moduli,GL07}.
This outer space $\calo$ is the space of all $G$-equivariant isometry classes of
\emph{Grushko trees} of $G$ relative to $\calf$, i.e.\  minimal, isometric $G$-actions on
simplicial metric trees with trivial edge stabilizers and whose non-trivial vertex stabilizers are precisely the groups whose conjugacy class is in $\calf$.

For example, the free product above can be an absolute Grushko decomposition of a finitely generated group $G$ (in which case $G_i$ are freely indecomposable and not cyclic), but other cases are important: for example, in the case where $G=F_N$, and $\calf$ is a system of free factors, 
this relative outer space $\calo$ is often more suitable to study subgroups of $\Out(F_N)$ stabilizing $\calf$.

\paragraph*{Boundaries and laminations.}
 Laminations for free groups were first introduced as an analogue of geodesic laminations on hyperbolic surfaces by Bestvina--Feighn--Handel, who constructed in \cite{BFH97} the attractive and repulsive laminations of a fully irreducible automorphism with a view towards proving the Tits alternative for $\Out(F_N)$ in \cite{BFH00,BFH05}. A comprehensive study of laminations in free groups was made by Coulbois, Hilion and Lustig in \cite{CHL08-1,CHL08-2,CHL08-3}.

A geodesic lamination on a hyperbolic surface can be lifted to its universal cover $\bbH^2$, and each leaf can be identified with its pair of endpoints in the circle at infinity. By analogy, an \emph{algebraic leaf} in the free group $F_N$ is defined as a pair of distinct points in its Gromov boundary, i.e.\ an element of 
$\partial^2 F_N:=(\partial F_N \times \partial F_N)\setminus \Delta$ where $\Delta$ is the diagonal.
If one identifies $F_N$ with the fundamental group of a graph $\Gamma$, such a leaf can be represented by a bi-infinite line in the universal cover $\Tilde \Gamma$.
A \emph{lamination} in $F_N$ is a closed $F_N$-invariant subset of $\partial^2 F_N$ which is flip invariant (i.e.\ invariant under the involution $(\alpha,\omega)\mapsto (\omega,\alpha)$). Laminations are related to geodesic currents on $F_N$ (introduced in \cite{Bon91}, see also \cite{Kap05,Kap06}): a current is a Radon measure on $\partial^2 F_N$, and its support is a lamination.

In the case of a free product $(G,\calf)$, the Bass--Serre tree $R_0$ of the splitting depicted above plays the role of $\Tilde \Gamma$.
Since $R_0$ is not locally compact (unless all $G_i$ are finite), the Gromov boundary $\partial_\infty R_0$ is not compact,
but can be naturally compactified into $\partial R_0:=\partial_\infty R_0\dunion V_\infty(R_0)$, 
where $V_\infty(R_0)$ is the set of vertices of infinite valence in $R_0$. The topology on $R_0\cup\partial R_0$ is the topology generated by open half-trees (the observers' topology).

Given any other tree $R$ in the outer space $\calo$, there is an equivariant quasi-isometry between $R_0$ and $R$ 
which allows to identify canonically 
$\partial R$ with $\partial R_0$. 
This allows to define $\partial (G,\calf)$ without reference to a particular tree.

In this context, we define an \emph{algebraic leaf} as a pair of distinct points in $\partial (G,\calf)$. Given any $R\in \calo$, it corresponds to a non-degenerate line segment in $R$ (either finite, semi-infinite or bi-infinite) with endpoints in $\partial R$. For example, every non-peripheral element $g\in G$ determines an algebraic leaf $(g^{-\infty},g^{+\infty})$, defined as the endpoints of the axis of $g$ in $R$. One then defines a \emph{lamination} as a closed, $G$-invariant, flip-invariant subset of 
$$\partial^2(G,\calf):=\left (\partial (G,\calf)\times \partial (G,\calf)\right) \setminus \Delta $$
where $\Delta$ is the diagonal.


\paragraph*{Theory of dual laminations.}
An overwhelming theme in the study of $\Out(F_N)$ or more generally outer automorphisms of free products is to analyze these groups through their actions on spaces of $\mathbb{R}$-trees. There is a particularly nice duality between $\mathbb{R}$-trees and laminations. In the context of free groups, the theory of dual laminations was developed by Coulbois, Hilion, Lustig \cite{CHL08-1,CHL08-2,CHL08-3,CHL07,CHL09,CH14} and Reynolds \cite{CHR11}, building on the Levitt--Lustig map introduced in \cite{LL03}. A major part of the work of the present paper, required by our intended applications, is to carry this theory to the context of free products. This gives us the opportunity to make a 
unified account of these results while putting them in a more general context.



Let $T$ be an $\bbR$-tree in the compactified outer space $\ol\calo$, i.e.\ endowed with a very small $(G,\calf)$-action \cite{Hor14-1}. The \emph{dual lamination} $L^2(T)$ is defined
by first looking at the closure of all algebraic leaves of the form $(g^{-\infty},g^{+\infty})$, where $g\in G$ is a non-peripheral element whose translation length in $T$ is at most $\eps$,
and then taking the intersection over all $\eps>0$ (see Definition~\ref{dfn_dual_lamination}).
One easily checks that the dual lamination of $T$ is empty if and only if $T$ is a Grushko tree (i.e.\ lies in $\calo$, as opposed to $\partial\calo:=\overline{\calo}\setminus\calo$).

 From now on, we will make the additional assumption that $T\in \ol\calo$ is an $\bbR$-tree on which the $G$-action has dense orbits. 
In this case, there is a well defined Levitt--Lustig continuous surjective map $\calq:\partial (G,\calf)\ra \hat T$, where $\hat T=\ol T\cup \partial_\infty T$ is the metric completion of $T$ together with its Gromov boundary,
 endowed with the observers' topology that makes it compact (see Proposition~\ref{prop_defQ}, generalizing \cite{LL03}). The map $\calq$ is the unique continuous extension of any $G$-equivariant map $R_0\to T$ that is linear on edges. One then shows that fibers of $\calq$ correspond to the dual lamination of $T$ (Proposition~\ref{equality-Q}, or \cite{CHL08-2,CHL07}):
 an algebraic leaf $(\alpha,\omega)$ lies in $L^2(T)$ if and only if $\calq(\alpha)=\calq(\omega)$.
This implies in particular that $L^2(T)$ completely determines $\hat T$ up to equivariant homeomorphism as the quotient $$\hat T\simeq \partial(G,\calf)/L^2(T)$$
(see Corollary~\ref{quotient-lamination}, or \cite{CHL08-2,CHL07} for $F_N$).

\paragraph*{Band complexes.} The somewhat abstract algebraic lamination $L^2(T)$ can also be realized as a very concrete geometric object.

Given a tree $R\in \calo$ (or a Cayley graph of the free group associated to a chosen basis in \cite{CHL09}),
one constructs a foliated \emph{band complex} $\Sigma(T,R)$, whose bands are of the form $K_e\times e$
where $e$ is an edge of $R$, and $K_e$ is a compact subtree of $\ol T$, and are foliated by $\{*\}\times e$ (see Section~\ref{sec-def-band} for precise definitions). The compact trees $K_e$ coincide essentially with the \emph{compact heart} introduced in \cite{CHL09} in the free group case
(to be more precise, the compact heart is the union of the trees $K_e$ over all edges $e$ in the Cayley graph joining $1_{F_n}$ to the chosen basis).
The projections $K_e\ra e$ extend to a natural map $p_R:\Sigma(T,R)\to R$, and the inclusions $K_e\subseteq \ol T$
extend to a map $p_T:\Sigma(T,R)\ra \ol T$.
Every leaf $\call$ of $\Sigma(T,R)$ has a natural structure of a graph, and this graph naturally embeds in $R$ under the projection $\Sigma(T,R)\to R$ (in particular $\call$ is a tree).
Outside this introduction, leaves of $\Sigma(T,R)$ are called \emph{complete $\Sigma(T,R)$-leaves} to avoid confusions with algebraic leaves.
Then one essentially \footnote{$\ol T\setminus T$ consists only of terminal points of $\ol T$ and the space of leaves of $\Sigma(T,R)$ is equivariantly isometric to a $G$-invariant tree sandwiched between $T$ and $\ol T$.} recovers $T$ 
as the space of leaves of $\Sigma(T,R)$:
the fibers of the map $\Sigma(T,R)\ra \ol T$ are precisely the leaves of $\Sigma(T,R)$.

 The algebraic lamination $L^2(T)$ dual to $T$ is then readable from the leaves of $\Sigma(T,R)$ in the following way (Lemma~\ref{lem_L2}). If $\alpha\neq \omega\in \partial_\infty R$, then $(\alpha,\omega)$ is an algebraic leaf in $L^2(T)$ if and only if there is a geometric leaf $\call$ in $\Sigma(T,R)$ whose projection to $R$ contains the bi-infinite line joining $\alpha$ to $\omega$. 
Assume now that $\alpha\in V_\infty(R)$ is a vertex
 with infinite stabilizer $G_\alpha$, and that $\omega$ lies in the Gromov boundary $\partial_\infty R$. Then $(\alpha,\omega)$ lies in $L^2(T)$ if and only if the geometric leaf $\call$ of $\Sigma(T,R)$ containing the unique point of $\Sigma(T,R)$ fixed by $G_\alpha$ is such that $p_R(\call)$ contains the semi-infinite line of $R$ joining $\alpha$ to $\omega$. If both $\alpha$ and $\omega$ belong to $V_\infty(R)$, then $(\alpha,\omega)$ lies in $L^2(T)$ if and only if the points of $\Sigma(T,R)$ fixed by $G_\alpha$ and $G_\omega$ belong to the same leaf of $\Sigma(T,R)$.  
 


It is important to have in mind that the trees $K_e$ used to construct the band complex $\Sigma(T,R)$ are compact $\bbR$-trees (for the metric topology), but they may have infinitely many branch points (those may be dense in $K_e$).
This is unlike in Rips theory (\cite{BF_stable,GLP94,GL95,Gui_actions}) where one works with band complexes whose bands are of the form $K\times e$ where $K$ is an interval or more generally a finite tree (i.e.\ the convex hull of finitely many points).
Although more complicated, the more general band complexes introduced by Coulbois--Hilion--Lustig and used in the present paper have the advantage that one recovers $T$ (plus some extra terminal points) as the space of leaves, and not a \emph{geometric approximation} of $T$ as in \cite{LP}.

We provide an alternative useful description of $\Sigma(T,R)$: it can be identified with a subset of the product $\ol T\times R$,
and more precisely with the \emph{core} of $\ol T\times R$ as introduced in \cite{Gui_coeur} (see Proposition \ref{prop_coeur}).
With this description, the natural maps to $\ol T$ and to $R$ are just the two projections.

\paragraph*{Preimages of $\calq$.} An important result that we will need is that almost every leaf in $\Sigma(T,R)$ has at most 2 ends, or equivalently, that 
the map $\calq:\partial(G,\calf)\ra \hat T$ is 2-to-1 almost everywhere.
\begin{thmbis}[see Theorem~\ref{Q-preimage} and Proposition~\ref{prop_3ends}, generalizing \cite{CH14}] \label{intro_finiteness}
  Let $T\in \ol\calo$ be a tree with dense orbits. Then the following equivalent statements hold:
  \begin{itemize}
  \item For all but finitely many orbits of points $x\in \hat T$, $\calq\m(\{x\})$ contains at most 2 points.
  \item There are only finitely many orbits of leaves in $\Sigma(T,R)$ with at least 3 ends.
  \end{itemize}
\end{thmbis}
In fact, Coulbois and Hilion prove more in \cite{CH14} in the context of free groups: they define an index (called the \emph{$\calq$-index}) that counts the number of orbits of \emph{extra} ends of leaves, and show that this index
is bounded by $2N-2$ where $N$ is the rank of the free group $F_N$.
We will not need such a refinement for our intended applications.

The main tool to prove Theorem~\ref{intro_finiteness} is the pruning process, an extension of a process of the Rips machine.
It takes as input the band complex $\Sigma(T,R)$, and produces a new band complex $\Sigma'\subseteq\Sigma(T,R)$ by removing all terminal segments in all leaves.
From $\Sigma'$, one can construct a new tree $R'\in \calo$ such that $\Sigma'=\Sigma(T,R')$.
Iterating this process yields a sequence of nested band complexes whose intersection is the union of bi-infinite lines contained in leaves of $\Sigma(T,R)$, 
thus corresponding to algebraic leaves in $L^2(T)$.
If this process stops in finite time, we say that the band complex is of \emph{quadratic type} \footnote{\cite[Section 4.3]{CH14} use the terminology \emph{pseudo-surface} and \cite[Definition~3.2]{CHR11} the terminology \emph{surface type}.} in reference to quadratic systems
of generalized equations in Makanin's algorithm.
Controlling the complexity of the band complexes produced by the pruning process is the main tool to prove Theorem~\ref{intro_finiteness} above.


\paragraph*{Reconstructing the lamination of an arational tree from a single leaf.}
One says that a nonsimplicial $\bbR$-tree $T\in\partial\calo$ is \emph{arational} if every proper relative free factor of $G$ acts freely and discretely on $T$
(a proper relative free factor $A$ is a nonperipheral factor of a free product decomposition $G=A*B$ where $A,B\neq \{1\}$, whose Bass--Serre tree is relative to $\calf$). Every orbit in an arational tree is dense \cite{Hor14-3}.
 We say that $T$ is \emph{relatively free} if every point stabilizer is peripheral.
Arational trees which are not relatively free are very special as they are essentially dual to arational laminations on $2$-orbifolds (\cite{Rey12,Hor14-3}, see Definition \ref{dfn_arat-surf}). 

An important result of this paper says that 
one can reconstruct the full dual lamination $L^2(T)$ of a relatively free arational tree $T$ from a single leaf $l\in L^2(T)$ (see \cite{CHR11} in the free group case).

The dual lamination of  any tree $T\in\overline{\calo}$ with dense $G$-orbits satisfies additional properties. 
First, since $L^2(T)$ coincides with the fibers of the map $\calq$, it is \emph{transitively closed}: 
$$ (\alpha,\beta),(\beta,\gamma)\in L\Rightarrow (\alpha,\gamma)\in L\text{ if $\alpha\neq \gamma$.} $$
Moreover, for any vertex $v\in V_\infty(R)$ with stabilizer $G_v$,
$$ (\alpha,g\alpha)\in L,\ g\in G_v \Rightarrow (\alpha,v)\in L.$$
because $\calq(\alpha)$ is the unique fixed point of $G_v$ in $T$, and so is $v$. 
When $v$ is a vertex with finite but nontrivial stabilizer $G_v$, the above property does not make sense since $v\notin V_\infty(R)$.
In all cases, it can be reformulated as $$(\alpha,g\alpha)\in L,(\beta,h\beta)\in L,\ g,h\in G_v,\ \alpha\neq\beta\Rightarrow (\alpha,\beta)\in L.$$  
We thus say that a lamination $L$ is \emph{peritransitively closed} (see Definition~\ref{dfn_peritransitively_closed}) if it is transitively closed and closed under the latter operation
(recall also that a lamination is required to be $G$-invariant, flip-invariant and topologically closed in $\partial^2(G,\calf)$ by definition).
As  just noted, for any tree $T$ with dense orbits, the dual lamination of $T$ is peritransitively closed (Lemma \ref{lem_peritransitively}).

For simplicity, we state the following theorem in the case where $T$ is relatively free action, see Theorem~\ref{lamination-arational} for a full statement. 

\begin{thmbis}[Theorem~\ref{lamination-arational}, see \cite{CHR11} in $F_N$] \label{intro-peritransitive} Let $T\in\overline{\calo}$ be arational and relatively free and let $l_0\in L^2(T)$ be an algebraic leaf.
Then $L^2(T)$ is the smallest peritransitively closed lamination containing $l_0$.
\end{thmbis}

The proof of this theorem relies on the analysis of the band complex associated to $T$. In particular, in the case where the pruning process described above stops, we use a second process, reminiscent of Rauzy--Veech induction and introduced in \cite{CHR11}, called the \emph{splitting process}. 

\paragraph*{A unique duality statement for relatively free arational trees.}
 In the sequel \cite{GH15} of this paper, we describe the boundary of the free factor graph of $(G,\calf)$ in terms of arational trees, thus extending \cite{BR13,Ham13}. 
An important tool used by Bestvina--Reynolds and Hamenstädt in their proof is a unique duality result,
stated in terms of the pairing between currents and trees introduced in \cite{KL09}.
This pairing is a continuous map that associates a non-negative number to a tree in $\calo$ and a current.
Their unique duality statement is the following, in which $T'\approx T$ means that $\hat T'$ and $\hat T$ are equivariantly homeomorphic (for the observers' topology).
    \begin{thmbis}[Bestvina--Reynolds \cite{BR13}, Hamenstädt \cite{Ham13}]\label{duality-free}
Let $T$ and $T'$ be two very small $F_N$-trees, with $T$ free and arational. Assume that there exists a current $\eta\neq 0$
    such that $\grp{T,\eta}=0=\grp{T',\eta}$.
        \\ Then $T'$ is arational and $T'\approx T$.
    \end{thmbis}

In the context of free products (including $G=F_N$ with $\calf\neq \es$), there is no natural continuous pairing between $\ol \calo$ and
relative currents (see below).
This is why we need the following substitute that avoids mentioning currents.

\addtocounter{thmbis}{-1}
\hypertarget{introbiduallink}{
 \renewcommand{\thethmbis}{\arabic{thmbis}'}
\begin{thmbis}[see Theorem~\ref{bidual-free}] \label{introbidual}
Let $T\in\partial\calo$ be a relatively free arational tree, and $T'\in\overline{\calo}$ be any tree.\\ 
Assume that there exists a leaf $l_0\in L^2(T)$ that is also contained in $L^2(T')$.
\\ Then $T'$ is arational and $T'\approx T$.
\end{thmbis}}
\renewcommand{\thethmbis}{\arabic{thmbis}}
\newcommand{\refintrobidual}{\hyperlink{introbiduallink}{\ref*{introbidual}}}

In fact, a similar result is already a main step of  Bestvina--Reynolds's and Hamenstädt's proofs of Theorem~\ref{duality-free} in the case of free groups.
Indeed, a result by Kapovich--Lustig shows that for free groups, one has $\grp{T,\eta}=0$ if and only if the support of the current $\eta$ is contained in the lamination $L^2(T)$ dual to $T$ \cite{KL09}. Theorem~\ref{duality-free} follows from this result together with Theorem~\refintrobidual.

Let us sketch the proof of Theorem \refintrobidual\ from Theorem \ref{intro-peritransitive} saying that $L^2(T)$ is the peritransitive closure of any of its leaves.
One first shows that $T'$ has dense orbits, so that $L^2(T')$ is peritransitively closed and contains the peritransitive closure of $l_0$. 
By Theorem \ref{intro-peritransitive}, one gets that $L^2(T')\supset L^2(T)$.
Since $\hat T$ is the quotient of $\partial(G,\calf)$ by $L^2(T)$ and similarly for $\hat T'$, this means that there is a continuous equivariant map
$p:\hat T\ra \hat T'$ such that $\calq_{T'}=p\circ \calq_T$ (where $\calq_T,\calq_{T'}$ are the Levitt--Lustig maps from $\partial (G,\calf)$ to $\hat T$ and $\hat T'$ respectively). 
If $p$ is not a homeomorphism, then an argument from \cite{BR13} shows that there are uncountably many points
$x\in T'$ with $\#p\m(\{x\})\geq 3$, so $\#\calq_{T'}\m(\{x\})\geq 3$ since $\calq_T$ is surjective.
This contradicts Theorem \ref{intro_finiteness} saying that $\calq_{T'}$ is $2$-to-$1$ almost everywhere.

\paragraph*{Limits in the free factor graph.}
A convenient definition of the graph of free factors is the following: its vertex set is the set of non-trivial actions of $(G,\calf)$
 on simplicial trees with trivial edge stabilizers,
and two trees are joined by an edge if they are compatible or  
there exists a non-peripheral element of $G$ fixing a point in both trees.\footnote{A discussion of the equivalence, up to quasi-isometry, of various standard definitions of the free factor graph, can be found in \cite{GH15}.}
The following result analyzes a sequence of pairs of trees $T_n,T'_n$ joined by an edge of the second type.

\begin{thmbis}\label{main-intro}
Let $(T_n)_{n\in\mathbb{N}},(T'_n)_{n\in\mathbb{N}}\in\ol{\calo}^{\mathbb{N}}$ be sequences of trees such that for all $n\in\mathbb{N}$, there exists a nonperipheral element $g_n\in G$ that is elliptic in both $T_n$ and $T'_n$. 
\\ If $(T_n)_{n\in\mathbb{N}},(T'_n)_{n\in\mathbb{N}}\in\ol{\calo}^{\mathbb{N}}$ converge respectively to $T$ and $T'$ in $\ol{\calo}$,
and if $T$ is relatively free and arational, then so is $T'$, and $T'\approx T$.
\end{thmbis}

This result is the basis of a Kobayashi-type argument showing that arational trees are indeed at infinity of the free factor graph, in the sense that any sequence of trees in $\calo$ converging to an arational tree in $\ol{\calo}$ has unbounded image in the free factor graph. We keep this argument for \cite{GH15}.

 To prove Theorem~\ref{main-intro} in the context of the free group \cite{BR13,Ham13}, one makes the following argument. The current $\eta_{g_n}$ associated to $g_n$ is dual to $T_n$ and $T'_n$ (i.e.\ the pairings $\langle T_n,\eta_{g_n}\rangle=||g_n||_T$ and $\langle T'_n,\eta_{g_n}\rangle=||g_n||_{T'}$ vanish).
The continuity of the pairing together with the compactness of the space of projective currents
implies that $T$ and $T'$ are dual to a common current. If $T$ is arational, one can then use the unique duality statement to deduce $T\approx T'$.

To make this argument work without reference to currents, 
consider $(\alpha_n,\omega_n)=(g_n^{-\infty},g_n^{+\infty})$, which lies in $L^2(T_n)\cap L^2(T'_n)$.
By cocompactness, one can assume that $(\alpha_n,\omega_n)$ converge to some algebraic leaf $(\alpha,\omega)$, and apply
the following substitute to the continuity of the pairing.
It involves the \emph{one-sided lamination}  $L^1(T)\subseteq \partial_\infty (G,\calf)$ dual to $T$:
this is defined as the set of points $\xi\in \partial_\infty (G,\calf)$ such that for any equivariant map
$f:R_0\ra T$ and any ray $\rho\subseteq T$ converging to $\xi$, $f(\rho)$ is bounded (\cite{CHL08-2} or Definition~\ref{dfn_l1}). 

\begin{propbis}[see Proposition \ref{prop-bounded-closed}]\label{intro-cor-bounded-closed}
Let $(T_n)_{n\in\mathbb{N}}\in\overline{\calo}^{\mathbb{N}}$ be a sequence that converges to $T\in \ol\calo$. For all $n\in\mathbb{N}$, let $(\alpha_n,\omega_n)\in L^2(T_n)$, and assume that $((\alpha_n,\omega_n))_{n\in\mathbb{N}}$ converges to some $(\alpha,\omega)\in\partial^2(G,\calf)$.
\\ If $T$ has dense orbits, then $(\alpha,\omega)\in L^2(T)$.
\\ Without assuming that $T$ has dense orbits we get that $(\alpha,\omega)\in (L^1(T)\cup V_\infty(G,\calf))^2$.
\end{propbis}  

Back to the proof of Theorem~\ref{main-intro}, assuming moreover that $T'$ has dense orbits, we get that $L^2(T)$ and $L^2(T')$ share
a common leaf, and one concludes the proof of Theorem \ref{main-intro} using the unique duality statement (Theorem~\refintrobidual).

When $T'$ does not have dense orbits the argument is more involved. First, arationality of $T$ can be used to show that either $\alpha$ or $\omega$ (say $\alpha$) belong to $\partial_\infty(G,\calf)$, so $\alpha\in L^1(T)\cap L^1(T')$. 
Now, to any element $\xi\in \partial_\infty(G,\calf)$, we associate a \emph{limit set} $\Lambda^2(\xi)\subseteq \partial^2(G,\calf)$, by taking all limits of translates of a ray joining a basepoint to $\xi$ in $R_0$, and one proves that $\xi\in L^1(T')$ implies $\Lambda^2(\xi)\subseteq L^2(T')$
(\cite{CHL08-2} or Proposition~\ref{l1-2}).
This additional argument yields 
a common leaf in $\Lambda^2(\xi)\subseteq L^2(T)\cap L^2(T')$. Theorem~\refintrobidual then enables us to conclude that $T'$ is arational, and $T'\approx T$.

\paragraph*{Simple leaves and unique duality statement for all arational trees.}

To state a version of Theorems~\refintrobidual\ and \ref{main-intro} that work for all arational trees (and not only those with a relatively free action), one can work with \emph{simple} leaves instead of considering all leaves in the dual lamination $L^2(T)$. These are defined in the following way: first, we say that an element $g\in G$ is \emph{simple} (relatively to $\calf$) if $g$ is contained in some proper $(G,\calf)$-free factor. There is a way to characterize simplicity of an element $g\in G$ in terms of a Whitehead graph associated to $g$ and to a Grushko tree $R$, analogous to classical work of Whitehead \cite{Whi36} for free groups (see Proposition~\ref{lem_simple}). We then say that an algebraic leaf is \emph{simple} if it is a limit of leaves of the form $(g_n^{-\infty},g_n^{+\infty})$ with $g_n$ simple. This allows to state the following version of Theorem \refintrobidual\ (see Theorem \ref{feng-luo-2} for the corresponding version of Theorem \ref{main-intro}).

\addtocounter{thmbis}{-3}
{\renewcommand{\thethmbis}{\arabic{thmbis}''}
\begin{thmbis}[see Theorem~\ref{bidual-free}]
Let $T\in\partial\calo$ be an arational tree, and $T'\in\overline{\calo}$ be any tree.\\ 
Assume that there exists a simple leaf $l_0\in L^2(T)$ that is also contained in $L^2(T')$.
\\ Then $T'$ is arational and $T'\approx T$.
\end{thmbis}
} 
\addtocounter{thmbis}{2}

\paragraph*{The trouble with currents.}

We define a  \emph{current} as a Radon measure on $\partial^2 (G,\calf)$ which is invariant under the flip and the action of $G$ (see also \cite{Gup} for relative currents in the free group). 
The space of currents, which we denote by $\text{Curr}$, is still projectively compact, but
the pairing $\grp{T,\eta_g}=||g||_T$ does not extend continously.

\begin{thmbis}[See Proposition \ref{trouble}]
Assume that $\calf$ contains an infinite group, and that $(G,\calf)$ is not of the form $G=G_1*G_2$.
\\ Then there does not exist any continuous pairing 
$$\langle \cdot\, ,\cdot\rangle :\overline{\calo}\times\text{Curr}\to\mathbb{R}$$ such that for all $T\in\calo$ and all nonperipheral $g\in G$, one has $\langle T,\eta_g \rangle=||g||_T$.
\end{thmbis}

\begin{rk}
There is a natural way of defining $\grp{T,\eta}$ for any simplicial tree $T\in \overline{\calo}$ with trivial edge stabilizer and any current $\eta$, in such a way that $\grp{T,\eta_g}$ always coincides with $||g||_T$.
The argument shows that this definition is not continuous, and even that the set of pairs $(T,\eta)$ such that $\grp{T,\eta}=0$ is not closed (see Section \ref{sec_currents}).
\end{rk}

Here is an example showing this phenomenon.

\begin{figure}
\begin{center}
\includegraphics{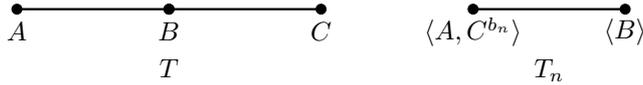}
\caption{An example showing the impossibility to extend continuously the natural intersection pairing in the relative setting.}
\label{fig-ex-not-continuous-intro}
\end{center}
\end{figure}

\begin{ex}\label{ex_non-continu}
Let $G=A*B*C$ with $B$ an infinite group, $\calf=\{A,B,C\}$, and $T$ be the Bass--Serre tree of this splitting (with all edge lengths set to be equal to $1$).
Let $b_n\in B$ be a sequence of distinct elements, and
 $T_n$ be the Bass--Serre tree of the splitting $G=\grp{A,b_n C b_n^{-1}}*B$, and $g_n=a b_n c b_n\m$.
Then $g_n$ is elliptic in $T_n$ so $||g_n||_{T_n}=0$, but $||g_n||_{T_{n+1}}=4$. Moreover, $T_n$ and $\eta_{g_n}$ (the rational current associated to $g_n$, see Section~\ref{sec_currents} for a definition) are convergent sequences in $\ol \calo$ and in the space of currents respectively.
This prevents the existence of a continuous extension (see Section~\ref{sec_currents} for details).
\end{ex}

\paragraph*{Organization of the paper.} In Section \ref{sec-1}, we give background and notations regarding free products, and we define a boundary for $(G,\calf)$ in Section~\ref{sec-bdy}. In Section \ref{sec_currents}, we prove that the natural  
pairing between trees and currents does not admit a continuous extension to the boundary in general. Algebraic laminations are introduced in Section~\ref{sec-4}, where Proposition~\ref{intro-cor-bounded-closed} is proven. In Section~\ref{sec-5}, we characterize simple elements of $G$ in terms of Whitehead graphs, and use this to study algebraic leaves that are obtained as limits of axes of simple elements: this is important when passing from the versions of our results for relatively free arational trees, to the case of arational trees that are not relatively free. In Section~\ref{sec-Q}, we introduce the map $\calq$ for trees with dense orbits, and establish its basic properties. Band complexes are defined in Section~\ref{sec-band-complex}, where it is established that they coincide with the core introduced in \cite{Gui_coeur}. This section also contains a few technical statements specific to the case of free products, which give finiteness results for the band complexes we work with, in spite of the lack of local compactness. 
The pruning process is introduced in Section \ref{sec-pruning}, where we prove that there are only finitely many orbits of points in $\hat{T}$ whose $\calq$-preimage contains at least three points (Theorem~\ref{intro_finiteness}). The splitting process is introduced in Section~\ref{sec-splitting}, and we then prove in Section~\ref{sec-10} that iterating the two processes eventually separate all leaves in the case where $T$ is mixing. Basic facts about arational trees are provided in Section~\ref{sec-arat}. We then analyse the dual lamination of an arational tree in Section~\ref{sec-11}, and establish Theorem~\ref{intro-peritransitive}. The proof of the unique duality theorem, and of Theorem~\ref{main-intro}, is completed in Section~\ref{sec-last}.

\subsection*{Acknowledgments}
The first author thanks Thierry Coulbois for discussions about the relation between the core as in \cite{Gui_coeur} and the heart defined in \cite{CHL09}. We would also like to thank the referee for their careful reading of the paper.

The first author acknowledges support from the Institut Universitaire de France and from the Centre Henri Lebesgue (Labex ANR-11-LABX-0020-01). 
The second author acknowledges support from the Agence Nationale de la Recherche under Grant ANR-16-CE40-0006.

\section{Background and notations}\label{sec-1}

Let $G_1,\dots,G_k$ be a finite collection of countable groups, let $F_N$ be a free group of rank $N$, and let $G:=G_1\ast\dots\ast G_k\ast F_N$. A subgroup of $G$ is \emph{peripheral} if it is conjugate into one of the groups $G_i$, and \emph{nonperipheral} otherwise. We denote by $\calf$ the collection of all conjugacy classes of maximal peripheral subgroups of $G$, i.e.\ $\calf=\{[G_1],\dots,[G_k]\}$, where $[G_i]$ denotes the conjugacy class of $G_i$. The \emph{Kurosh rank} of $(G,\calf)$ is defined as $\text{rk}_K(G,\calf):=k+N$.

\paragraph*{The outer space of a free product and its closure.} Recall that a $G$-action on a tree $T$ (either a simplicial tree or an $\mathbb{R}$-tree) is \emph{minimal} if $T$ does not contain any proper $G$-invariant subtree. It is \emph{relatively free} if all point stabilizers in $T$ are (trivial or) peripheral. A \emph{Grushko tree} is a simplicial metric tree $S$, equipped with a minimal, simplicial, relatively free isometric $G$-action, such that every peripheral subgroup of $G$ fixes a unique point in $S$. The \emph{unprojectivized outer space} $\calo$, introduced in \cite{GL07}, is the space of all $G$-equivariant isometry classes of Grushko trees. The space $\calo$ embeds into the space of all $G$-equivariant isometry classes of minimal isometric $G$-actions on $\mathbb{R}$-trees, equipped with the equivariant Gromov--Hausdorff topology introduced in \cite{Pau88}. The closure $\overline{\calo}$ was identified in \cite{Hor14-1} with the space of \emph{very small} trees, i.e. those trees in which arc stabilizers are either trivial, or else cyclic, root-closed and nonperipheral, and tripod stabilizers are trivial.

By \cite{Lev94} (see \cite[Theorem~4.16]{Hor14-1} for free products), every tree $T\in\overline{\calo}$ decomposes in a unique way as a graph of actions (in the sense of \cite{Lev94}) whose vertices correspond to orbits of connected components of the closure of the set of branch points and inversion points (i.e.\ points with stabilizer of order $2$). Two vertices are joined by an edge if there are connected components in the corresponding orbits which are adjacent in $T$. This decomposition of $T$ as a graph of actions is called the \emph{Levitt decomposition} of $T$. In particular, vertex groups of this decomposition act with dense orbits on the corresponding subtree (which may be reduced to a point). The Bass–-Serre tree of the underlying graph of groups is very small (it can be trivial). 

All maps $f:S\ra T$ with $S$ a simplicial $G$-tree and $T$ either a simplicial $G$-tree or an $\mathbb{R}$-tree will be assumed to be \emph{piecewise linear}: each edge can be subdivided into finitely many
arcs on which the map is linear, i.e.\ sends the arc to a geodesic segment with constant speed.
In the case where $T$ is simplicial, this is equivalent to saying that one can subdivide $S$ and $T$ so that $f$ maps each edge to an edge or to a point, linearly. We make the following easy observation.

\begin{lemma}\label{bdd-distance}
Let $S\in\calo$ and $T\in\overline{\calo}$, and let $f,g:S\to T$ be $G$-equivariant maps.
\\ Then there exists $C>0$ such that for all $x\in S$, one has $d_T(f(x),g(x))\le C$.
\qed
\end{lemma}

\paragraph*{Subgroups of a free product.} A \emph{$(G,\calf)$-free splitting} is a minimal, simplicial $G$-tree with trivial edge stabilizers, in which all peripheral subgroups are elliptic. A \emph{$(G,\calf)$-free factor} is a subgroup of $G$ that occurs as the stabilizer of a point in some $(G,\calf)$-free splitting.

Let $H\subseteq G$ be a subgroup. By considering the minimal subtree for the $H$-action on a Grushko tree,  one gets that $H$ splits as a free product $$H=(\ast_{i\in I}H_i)\ast F,$$ where each $H_i$ is peripheral  and $F$ is a  free group (notice that $I$ can be infinite and that $F$ can be infinitely generated). We say that $H$ \emph{has finite Kurosh rank} if $I$ is finite and $F$ is finitely generated. We then denote by $\calf_{|H}$ the collection of all maximal peripheral subgroups of $H$ (i.e. the $H$-conjugates of the $H_i$).


\section{Boundaries for $(G,\calf)$}\label{sec-bdy}

\paragraph{Boundary of a Grushko tree and boundary of $(G,\calf)$}
Given a Grushko $(G,\mathcal{F})$-tree $R$, we denote by $\partial_{\infty}R$ the Gromov boundary of $R$, by $V_{\infty}(R)$ the collection of vertices of $R$ of infinite valence,
and we let $\partial R:=\partial_{\infty} R\cup V_{\infty}(R)$. We let $\widehat{R}:=R\cup\partial_{\infty}R$. This is a compact space when equipped 
with the \emph{observers' topology}, which is the topology generated by the set of directions in $\widehat{R}$, see \cite[Proposition 1.13]{CHL07}. Recall  that a \emph{direction} based at a point $a$ in a tree $T$ is a connected component of $T\setminus\{a\}$. Thus, by definition, a sequence of points $x_n\in \widehat R$ converges to $x_\infty$ for the observers' topology
if for every $a\in R\setminus \{x_\infty\}$, the point $x_n$ lies in the connected component of $\widehat R\setminus \{a\}$ that contains $x_\infty$ for all $n$ large enough. For example, if $x_\infty\in V_{\infty}(R)$ is a vertex of infinite valence, and $(x_n)_{n\in\mathbb{N}}\in R^\mathbb{N}$ is a sequence of points that belong to pairwise distinct directions based at $x_\infty$, then $(x_n)_{n\in\mathbb{N}}$ converges to $x_\infty$ for the observers' topology.
Note that if $x_\infty\in\partial_\infty R$, then $x_n$ converges to $x_\infty$ for the observers' topology 
if and only if it does for the usual topology on $R\cup \partial_\infty R$ (coming from hyperbolicity of $R$). If $x\in R$ is a point with finite valence, then one can find a finite set of directions in $R$ whose intersection is an open neighbourhood of $x$ in $R$ which only contains points of finite valence. This shows that $R\setminus\partial R$ is open in $\widehat R$. Therefore $\partial R\subseteq \widehat{R}$ is a closed subset of $\widehat{R}$, hence it is compact. We also let $\partial^2R:=\partial R\times \partial R\setminus\Delta$, where $\Delta:=\{(x,x)|x\in\partial R\}$ is the diagonal subset, and we equip $\partial^2R$ with the topology inherited from the product topology on $\partial R\times \partial R$. 

\begin{lemma}\label{preimage-infinite}
Let $R$ and $R'$ be two Grushko $(G,\calf)$-trees. Then any $G$-equivariant piecewise linear map $f:R\to R'$ is continuous for the observers' topology. 
\end{lemma}

Recall that in this paper, all maps between two Grushko $(G,\calf)$-trees are assumed to be piecewise linear on edges.


\begin{proof}
It suffices to prove that the $f$-preimage of any direction $d\subseteq R'$ 
is a finite union of finite intersections of directions in $R$. Denote by $m\in R'$ the basepoint of $d$. Up to subdividing $R$ and $R'$, we can assume that $f$ sends each edge to either an edge or a vertex, and that $m$ is a vertex. Let $e$ be the open edge of $R'$ adjacent to $m$ contained in $d$. Then $f^{-1}(e)$ is a union of open edges in $R$. This union is finite, for otherwise $f\m(e)$ would contain infinitely many edges in the same orbit, and $e$ would have nontrivial (infinite) stabilizer. Each connected component of $f\m(d)$ is a connected component of $R\setminus f\m(\{m\})$ whose terminal edges are in $f\m(e)$. Since $f\m(e)$ is a finite union of edges, $f\m(d)$ has finitely many connected components, and each of these is a finite intersection of directions.
\end{proof}

\begin{lemma}\label{identification}
Let $R$ and $R'$ be two Grushko $(G,\calf)$-trees. 
Then any $G$-equivariant map $f:R\to R'$ has a unique continuous extension $\hat f:\widehat R\ra \widehat R'$. Moreover, 
the map $h:=\hat f_{|\partial R}$ is a homeomorphism $\partial R\ra \partial R'$ that does not depend on $f$, and $h(\partial_{\infty}R)=\partial_{\infty}R'$ and $h(V_{\infty}(R))=V_{\infty}(R')$.
\end{lemma}

\begin{proof}
Let $f:R\ra R'$ be an equivariant map. By Lemma \ref{preimage-infinite}, $f$ is continuous for the observers' topology. 
By \cite[Corollary 1.5]{For_deformation}, the map $f$ induces a quasi-isometry between $R$ and $R'$, hence extends to a continuous map $\hat f:R\cup \partial_\infty R\ra R'\cup \partial_\infty R'$  for the usual topology.
Since the usual topology on $R\cup\partial_\infty R$ agrees with the observers' topology at each point of $\partial_\infty R$,
it follows that $\hat f$ is continuous for the observers' topology. 
Moreover, if $f_1,f_2:R\ra R'$ are two equivariant maps, then they are at bounded distance from one another (Lemma~\ref{bdd-distance}), so they
induce the same map  between the Gromov boundaries of $R$ and $R'$.
Since they also agree on $V_\infty(R)$,  $\hat f_1$ and $\hat f_2$ restrict to the same continuous map 
$h=\hat f_{1|\partial R}:\partial R\ra \partial R'$.
Taking any equivariant map $f':R'\ra R$ yields an inverse for $\hat f$.
\end{proof}

The identifications between the boundaries of all $(G,\calf)$-Grushko trees given by Lemma \ref{identification}, allow us to define the 
space $\partial(G,\calf)$ as $\partial R$, without reference to a particular Grushko tree $R$.
Similarly, if $x\in R$ is any point with trivial stabilizer, the embedding $g\mapsto g.x$ of $G$ into $R$
gives a compact topology on $G\cup \partial(G,\calf)$. This topology does not depend on the choice of $x$ or $R$
because for any $x'\in R'$ with trivial stabilizer, there exists a $G$-equivariant map $f:R\ra R'$ sending $x$ to $x'$.
In plain words, we say that a sequence of elements $g_n\in G$ converges to $\omega\in\partial(G,\calf)$
if for some (equivalently any) $x\in R$ with trivial stabilizer, $g_n.x$ converges to $\omega\in \partial R$ (for the observers' topology on $R\cup\partial R$).

Similarly, there are well defined  subspaces $\partial_{\infty}(G,\calf)$, $V_{\infty}(G,\calf)\subseteq \partial (G,\calf)$,
and since $h:\partial R\to\partial R'$ also gives a natural identification between $\partial^2R$ and $\partial ^2R'$,
there is a well-defined space $\partial^2(G,\calf)$.

\paragraph{Algebraic leaves.}
Elements in $\partial^2(G,\calf)$ are called \emph{algebraic leaves}. Given a Grushko $(G,\calf)$-tree $R$ and an algebraic leaf $(\alpha,\omega)\in \partial^2(G,\calf)$, we denote by $[\alpha,\omega]_R$ the line segment in $R$ joining the points in $\partial R$ corresponding to $\alpha$ and $\omega$: this might be a line, a half-line, or a segment (depending on whether $\alpha$ and $\omega$ belong to $\partial_\infty(G,\calf)$ or to $V_\infty(G,\calf)$), but it is not reduced to a point.
This is the geometric representation of the algebraic leaf in $R$ (note that it completely determines the unordered pair $\{\alpha,\omega\}$).
More generally, given a point $a\in R$ and a point $\omega\in\partial R$, we denote by $[a,\omega]_R$ the line segment in $R$ joining  $a$ to $\omega$ (this can be reduced to a point if $\omega\in V_{\infty}(R)$ and $a=\omega$). The following fact is obvious.

\begin{lemma}\label{lem_approximating}
Let $(\alpha,\omega)\in \partial^2(G,\calf)$, and $R$ a Grushko tree.
Let $a_i\in R\setminus\{ \alpha\}, b_i\in R\setminus\{ \omega\}$ be two sequences of points
converging to $\alpha,\omega$ respectively. Assume that $(\alpha_i,\omega_i)\in \partial^2(G,\calf)$ is such that 
$[a_i,b_i]\subseteq [\alpha_i,\omega_i]_R$.
\\ Then $(\alpha_i,\omega_i)$ converges to $(\alpha,\omega)$.\qed
\end{lemma}


\paragraph{Being carried by a subgroup.}
Let $A\subseteq G$ be a non-peripheral subgroup whose Kurosh rank is finite, and $R$ be a Grushko $(G,\calf)$-tree.
Let $R_A\subseteq R$ be the minimal $A$-invariant subtree, which is in particular a Grushko $(A,\calf_{|A})$-tree.
The inclusion $R_A\subseteq R$ induces a natural continuous embedding of $\partial(A,\calf_{|A})$ into $ \partial (G,\calf)$,
and one checks that it does not depend on the choice of $R$.
In the particular case where $A$ is cyclic, this means that any nonperipheral element $g\in G$ determines a pair of points $(g^{-\infty},g^{+\infty})\in\partial^2 (G,\calf)$.

\begin{de}[\textbf{\emph{Being carried by a subgroup}}]\label{dfn_carried}
 Given $\xi\in\partial_{\infty}(G,\calf)$ and a subgroup $A\subseteq G$ of finite Kurosh rank, we
  say that $\xi$ is \emph{carried} by $A$ if some translate of $\xi$
  belongs to $\partial_{\infty}(A,\calf_{|A})$. 
\\ Similarly, we say that a subset
  $L\subseteq \partial^2 (G,\calf)$ is \emph{carried} by $A$ if every
   algebraic leaf $(\alpha,\omega)\in L$ has a translate contained in
  $\partial^2 (A,\calf_{|A})$.
\end{de}

\begin{lemma}\label{lem_porte_ferme}
  Let $A\subseteq G$ be a subgroup of finite Kurosh rank. 
The set of algebraic leaves in $\partial^2(G,\calf)$ carried by $A$ is closed.
\end{lemma}

\begin{proof}
Fix a Grushko $(G,\calf)$-tree $R$.
Let $R_A\subseteq R$ be the minimal $A$-invariant subtree.
Let $(\alpha,\omega)$ be a limit of leaves $(\alpha_i,\omega_i)$ carried by $A$.
Let $g_i$ be such that $g_i.(\alpha_i,\omega_i)\in \partial^2(A,\calf_{|A})$.
In particular, $g_i.[\alpha_i,\omega_i]_R\subseteq R_A$.

Fix an edge $e\subseteq [\alpha,\omega]_R$. For all $i$ large enough, we have $e\subseteq [\alpha_i,\omega_i]_R$, hence $g_i e\subseteq R_A$.
Since $A$ has finite Kurosh rank, $R_A/A$ is compact, so up to taking a subsequence,
we can assume that 
 $g_i e$ lies in a constant $A$-orbit of edges in $R_A$.  
This means that there exists $a_i\in A$ such that $a_ig_i e$ does not depend on $i$, and therefore neither does $g:=a_ig_i$. 
Since $g_i(\alpha_i,\omega_i)\in \partial^2(A,\calf_{|A})$, we have $a_ig_i(\alpha_i,\omega_i)=g(\alpha_i,\omega_i)\in \partial^2(A,\calf_{|A})$.
Hence $g.(\alpha,\omega)\in \partial^2(A,\calf_{|A})$ is carried by $A$ and so is $(\alpha,\omega)$.
\end{proof}

\section{Currents and  intersection pairing}\label{sec_currents}

We now define a notion of geodesic currents for free products. In the case of free groups, these were defined in \cite{Bon91}, and further studied in \cite{Kap05,Kap06}. We mention that another approach to relative currents in the case where $G$ is a free group has recently been proposed by Gupta in \cite{Gup}. We denote by $i:\partial^2(G,\calf)\to\partial^2(G,\calf)$ the involution defined by $i(\alpha,\omega)=(\omega,\alpha)$. A \emph{geodesic current} on $(G,\calf)$ is a $G$-invariant and $i$-invariant Radon measure $\nu$ on $\partial^2(G,\calf)$ (i.e. $\nu$ gives finite measure to compact subsets of $\partial^2(G,\calf)$). The space $\text{Curr}$ of geodesic currents on $(G,\calf)$ is equipped with the \emph{weak topology}: a sequence $(\nu_n)_{n\in\mathbb{N}}\in\text{Curr}^{\mathbb{N}}$ converges to a current $\nu\in\text{Curr}$ if $\nu_n(S\times S')$ converges to $\nu(S\times S')$ for all disjoint clopen subsets $S,S'\subseteq\partial^2(G,\calf)$.

Every nonperipheral element $g\in G$ 
determines a current 
$$\eta_g=\sum_{h\in G/\grp{g}}\delta_{h(g^{-\infty},g^{+\infty})}+\delta_{h(g^{+\infty},g^{-\infty})}$$
where for $(\alpha,\omega)\in \partial^2(G,\calf)$, $\delta_{(\alpha,\omega)}$ denotes the Dirac mass at $(\alpha,\omega)$.

Alternatively, one can define $\eta_g$ as follows.
Given  a Grushko $(G,\calf)$-tree $R$ and an oriented edge path $J\subseteq R$,
consider the cylinder $Cyl_J\subseteq \partial^2R$ consisting of all $(\alpha,\omega)\in \partial^2R$ such that $J$ is contained in  $[\alpha,\omega]_R$
and oriented towards  $\omega$.
Let $A_g\subseteq R$ be the axis of $g$, and $D\subseteq A_g$ a fundamental domain.
Then $\eta_g(Cyl_J)$ is the number of $h\in G$ such that $h.J\subseteq A_g$ (with no requirement about orientation) 
and the first edge of $h.J$ lies in $D$. Notice that if $\overline{J}$ denotes the segment $J$ with the opposite orientation, then $\eta_g(Cyl_{\overline{J}})=\eta_g(Cyl_J)$. Note also that the values of $\eta_g$ on the cylinders of this form completely determine $\eta_g$.

Similarly, if $\alpha,\omega\in V_\infty(G,\calf)$ are two distinct points of infinite valence, one can define a current 
$$\eta_{\alpha,\omega}=\sum_{h\in G}\delta_{h(\alpha,\omega)}+\delta_{h(\omega,\alpha)}$$
so that $\eta_{\alpha,\omega}(Cyl_J)$ is the number of $h\in G$ such that $h.J\subseteq [\alpha,\omega]_R$ (without requirement about orientation).

 Given $T\in \calo$ and any current $\eta$, there is a naturally defined pairing given by
 \begin{equation*}
\grp{T,\eta}=\sum_{e} l(e) \int_{Cyl_e} \eta \tag{*}
\end{equation*}
where the sum is taken on a set of representatives of orbits of edges of $T$.
One can check that $\grp{T,\eta_g}=||g||_T$ and $\grp{T,\eta_{\alpha,\omega}}=d_T(v_\alpha,v_\omega)$ 
(where $v_\alpha,v_\omega$ are the points in $T$ fixed by the peripheral subgroups corresponding to $\alpha,\omega\in V_\infty(G,\calf)$.

In contrast to \cite{KL09}, where the existence of a continuous pairing on $\overline{cv_N}\times\text{Curr}$ was established in the case where $G$ is a free group, the following proposition suggests that geodesic currents are not so well-adapted to the study of free products.

\begin{figure}
\begin{center}
\includegraphics[width=\linewidth]{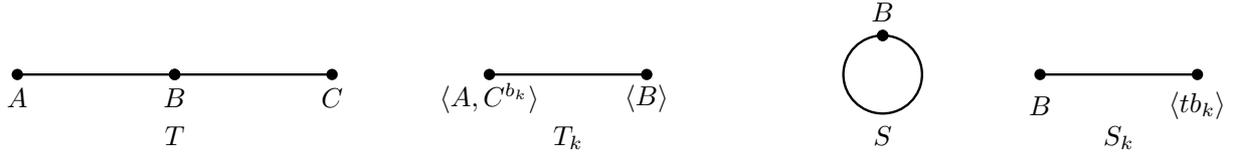}
\caption{Examples showing the impossibility to extend continuously the natural intersection pairing in the relative setting.}
\label{fig-ex-not-continuous}
\end{center}
\end{figure}

\begin{prop}\label{trouble}
Assume that $\calf$ contains an infinite group, and that either $\text{rk}_K(G,\calf)\ge 3$ or else that $\text{rk}_f(G,\calf)\ge 1$. Then there does not exist any continuous map  $$i:\overline{\calo}\times\text{Curr}\to\mathbb{R}$$ such that for all $T\in\calo$ and all nonperipheral $g\in G$, one has $\langle T,\eta_g \rangle=||g||_T$.
\end{prop}

\begin{rk}
The definition of the pairing $(\ast)$ makes sense for any simplicial trees $T$ with trivial edge stabilizers.
Any continuous extension of the pairing must satisfy this formula for such a tree $T$ 
(this amounts to let $l(e)$ go to $0$ for some edges of a tree in $\calo$).
The argument above then shows that with this extended definition of the pairing, 
the set of pairs $(T,\eta)\in \ol\calo\times Curr(G,\calf)$ such that $\grp{T,\eta}=0$ is not even closed. 
\end{rk}

\begin{proof}
Assume by contradiction that $i:\overline{\calo}\times\text{Curr}\to\mathbb{R}$ is such a continuous map.
Since the translation length functions are continuous on $\ol \calo$ and $\calo$ is dense on $\ol \calo$, 
one necessarily has $\langle T,\eta_g \rangle=||g||_T$ for all nonperipheral elements $g\in G$ and all $T\in\ol\calo$.
We will prove the proposition by exhibiting a sequence of trees $T_k\in \ol \calo$ having a well-defined limit $T$ in $\ol \calo$,
and a sequence of nonperipheral elements $g_k\in G$ such that $\eta_{g_k}$ have a well-defined limit $\eta$ in the space of (non-projective) currents,
and such that $||g_k||_{T_k}=0$ while  $||g_{k+1}||_{T_k}$ is a constant $m>0$ independent of $k$. 
Then one should simultaneously have $\langle T,\eta\rangle=0$ and $\langle T,\eta \rangle=m$, a contradiction.

Assume first that $\text{rk}_K(G,\calf)\ge 3$ so that $G$ can be written as $G=A*B*C$ relative to $\calf$, with $B\in \calf$ infinite and $A,C\neq\{1\}$.
Let $(b_k)_{k\in\bbN}$ be a sequence of distinct elements of $B$,  $a\in A\setminus\{1\},c\in C\setminus\{1\}$ and let $g_k:=ab_k c b_k\m$. 
 Take for $T_k$ the Bass--Serre tree of the graph of groups decomposition of $G$ represented on Figure~\ref{fig-ex-not-continuous},
where all edges have length $1$.
Hence  $\langle T_k,\eta_{g_k}\rangle =||g_k||_{T_k}=0$ and $\langle T_k,\eta_{g_{k+1}}\rangle=||g_{k+1}||_{T_k}=4$.
Moreover, $T_k$ converges to $T$ as $k\ra \infty$. 
Let $R$ be a Grushko $(G,\calf)$-tree, and let $\beta\in V_\infty(R)$ be the point fixed by $B$.
Then $\eta_{g_k}$ converges to 
$\eta_{\beta,a\beta}+\eta_{\beta,c\beta}$ (the proof of this fact is left as an exercise for the reader).

If $\text{rk}_f(G,\calf)\geq 1$ and $\text{rk}_K(G,\calf)= 2$, then $G=B*\grp{t}$ for some $B\in \calf$ infinite and $t$ nonperipheral.
Let $S_k,S$ be the trees showed on Figure~\ref{fig-ex-not-continuous}, where the edges of $S_k$ have length $1$, and the edges of $S$ have length $2$.
Then $||tb_k||_{S_k}=0$,  $||tb_{k+1}||_{S_k}=3$, and $S_k$ converges to $S$.
Note that $S$ is a Grushko $(G,\calf)$-tree, and denote by $\beta\in V_\infty(S)$ the vertex fixed by $B$.
Then $\eta_{tb_k}$ converges to $\eta_{\beta,t\beta}$.  
\end{proof}

\begin{rk}
If one works in $F_3=\grp{a,b,c}$ without peripheral structure, and consider usual currents in $\partial^2 F_3$ 
instead of $\partial^2 (F_3,\{\grp{a},\grp{b},\grp{c}\})$, then the sequence of currents associated to 
$g_k=ab^k c b^{-k}$ does not converge in the space of (non-relative) currents in $F_3$. 
On the other hand, rescaling it by $\frac{1}{k}$ yields a convergent sequence, but $\frac{1}{k+1}||g_{k+1}||_{T_k}=\frac{4}{k+1}$
has the same limit as $\frac{1}{k}||g_{k}||_{T_k}=0$, so there is no contradiction in this case.
\end{rk}

\section{Algebraic laminations}\label{sec-4}

In the present section, we introduce algebraic laminations for free products, and we define the algebraic lamination $L^2(T)$ dual to a very small tree $T\in\ol\calo$.
 We also define the one-sided lamination $L^1(T)$ dual to $T$, and relate it to $L^2(T)$ in Section \ref{sec-l12}. 
All these objects were introduced in \cite{CHL08-1,CHL08-2} in the context of free groups.
In Section \ref{sec-scs}, we prove that if $(T_n)_{n\in\mathbb{N}}\in\ol\calo^{\mathbb{N}}$ is a sequence that converges to a tree $T$ with dense orbits, and if $(\alpha_n,\omega_n)\in L^2(T_n)$ converges to $(\alpha,\omega)\in\partial^2(G,\calf)$, then $(\alpha,\omega)\in L^2(T)$. Though this fails to hold in general when $T$ is not assumed to have dense orbits, Proposition \ref{prop-bounded-closed} also provides a weaker (and crucial) statement valid for all trees $T\in\ol\calo$.

\subsection{General definition}

Recall that $i:\partial^2(G,\calf)\to\partial^2(G,\calf)$ is the involution defined by $i(\alpha,\omega)=(\omega,\alpha)$.
 
\begin{de}[\textbf{\emph{Algebraic lamination}}]
  A \emph{$(G,\mathcal{F})$-algebraic lamination} is a  closed $G$-invariant and $i$-invariant subset of
  $\partial^2 (G,\calf)$.
\end{de}

If $L$ is any subset of $\partial^2 (G,\calf)$, there is a  smallest lamination containing $L$, we call it the lamination \emph{generated} by $L$.

 Given a Grushko tree $R$, and a finite set $E$ of edges of $R$, we denote by $C_E\subseteq \partial^2 (G,\calf)$ the subspace made of all  pairs $(\alpha,\omega)\in \partial^2 (G,\calf)$ such that $[\alpha,\omega]_R$ 
 contains one of the edges in $E$: this is a compact set.
 In particular, if $E$ contains a representative of each orbit of edges, then any non-empty lamination $\Lambda$ has non-empty intersection with $C_E$,
and $\Lambda=G.(\Lambda\cap C_E)$. 

\subsection{Dual lamination of an $\mathbb{R}$-tree}

We will be mainly interested in algebraic laminations obtained by the following construction. 
Let $T\in\overline{\mathcal{O}}$. 
For all $\epsilon>0$, we let $L^2_{\epsilon}(T)$ be the closure in $\partial^2 (G,\calf)$ of the set of pairs $(g^{-\infty},g^{+\infty})$ where $g$ runs among the set of all 
non-peripheral elements with $||g||_T\le\epsilon$.
\begin{de}[\textbf{\emph{Algebraic lamination dual to a tree}}]\label{dfn_dual_lamination}
The algebraic lamination  \emph{dual} to the tree $T$ is
 $$L^2(T):=\bigcap_{\epsilon>0}L^2_{\epsilon}(T).$$
\end{de}

\begin{lemma}\label{lamination-nonempty}
A tree $T\in \overline{\mathcal{O}}$ is a Grushko $(G,\calf)$-tree if and only if its dual lamination is empty.
\end{lemma}

\begin{proof}
If $T\in\calo$, then $L^2_{\epsilon}(T)=\emptyset$ for every $\eps>0$ smaller than the length of the shortest edge of $T$, so $L^2(T)=\emptyset$. Conversely, if $T\in\partial\mathcal{O}$, then for all $\eps>0$, the lamination
  $L^2_{\epsilon}(T)$ is non-empty. Indeed, this is clear if $T$ is not relatively free, and if $T$ is not simplicial, the finiteness of orbits of directions \cite[Corollary 4.8]{Hor14-1} implies the existence of elements of arbitrarily small translation length,
so $L^2_{\epsilon}(T)\neq\emptyset$  for all $\epsilon>0$. Let $E$ be a finite collection of edges of $R$ that contains one representative in every $G$-orbit of edges.
Then $L^2_{\epsilon}(T)\cap C_E$ is non-empty for every $\eps>0$.
Since the intersections $L^2_{\epsilon}(T)\cap C_E$ form a decreasing collection of nonempty compact sets as $\epsilon$ goes to $0$, we get that $L^2(T)\neq \es$.
\end{proof}

\begin{lemma}\label{l2-carry}
Let $T\in\overline{\calo}$ is a simplicial $(G,\calf)$-tree. Then all algebraic leaves in $L^2(T)$ are carried by a vertex stabilizer in $T$. 
\end{lemma}

\begin{proof}
Let $(\alpha,\omega)\in L^2(T)$. Let $\eps>0$ be smaller than the length of the shortest edge of $T$. Since $(\alpha,\omega)\in L^2_{\epsilon}(T)$, we can find a sequence $((g_n^{-\infty},g_n^{+\infty}))_{n\in\mathbb{N}}$ that converges to $(\alpha,\omega)$ in $\partial^2(G,\calf)$, with $||g_n||_T\le\epsilon$ for all $n\in\mathbb{N}$. By definition of $\eps$, all elements $g_n$ are contained in a vertex stabilizer of $T$. Let now $\widehat{T}$ be a Grushko tree that collapses onto $T$. Since the minimal subtrees in $\widehat{T}$ of the vertex stabilizers of $T$ are pairwise disjoint, convergence of $(g_n^{-\infty},g_n^{+\infty})$ implies that all $g_n$ are contained in the same vertex stabilizer $G_v$ of $T$. Therefore $(\alpha,\omega)$ is carried by $G_v$.  
\end{proof}

If $T\in\overline{\calo}$ does not have dense orbits, then its Levitt decomposition yields a non-trivial action on a simplicial tree $S$, with a $1$-Lipschitz collapse map $T\ra S$.
 
\begin{cor}\label{carry-gen-l2}
Let $T\in\overline{\calo}$. Then every algebraic leaf in $L^2(T)$ is carried by a vertex group of the Levitt decomposition of $T$.
\end{cor}

\begin{proof}
Let $S$ be the Levitt decomposition of $T$. Since there is a $1$-Lipschitz collapse map $T\ra S$, we have $L^2(T)\subseteq L^2(S)$. It then follows from Lemma \ref{l2-carry} that all leaves in $L^2(T)$ are contained in a vertex group of $S$. 
\end{proof}

\subsection{One-sided dual laminations}

We now introduce the \emph{one-sided lamination} $L^1(T)$ dual to a tree $T\in\ol\calo$, see \cite[Section 5]{CHL08-2} in the context of free groups. The relation between $L^1(T)$ and $L^2(T)$ will be discussed in Section~\ref{sec-l12} below.

\begin{de}[\textbf{\emph{One-sided dual lamination}}]\label{dfn_l1}
Given $T\in\overline{\calo}$, we define  the \emph{one-sided lamination} $L^1(T)\subseteq \partial_\infty (G,\calf)$  as the set of points $\xi\in \partial_\infty (G,\calf)$ 
such that for every Grushko tree $R$, every $x_0\in R$, and every $G$-equivariant map $f:R\to T$, 
$f([x_0,\xi]_R)$ is a bounded subset of $T$.
\end{de}

Note that $\xi$ is assumed  in the above definition to be at infinity, not in $V_\infty(G,\calf)$. Also note that the condition in the above definition depends neither on the choice of the Grushko tree $R$, nor on the choice of the basepoint $x_0$, nor on the choice of the map $f$ (see Lemma~\ref{bdd-distance}). The following lemma is the analogue of Lemma \ref{l2-carry} for the one-sided lamination $L^1(T)$.

\begin{lemma}\label{carry-subgp}
Let $S\in\overline{\calo}$ be a simplicial tree, and let $\xi\in L^1(S)$. Then $\xi$ 
is carried by  a vertex stabilizer of $S$.
\end{lemma}

\begin{proof} 
One can find a Grushko $(G,\calf)$-tree $R$ and an equivariant map $f:R\ra S$ sending edge to edge.
Indeed, starting from any Grushko $(G,\calf)$-tree $R$, and from an equivariant map sending vertex to vertex, one can subdivide $R$ 
so that $f$ sends each edge to an edge or a vertex.
If some edge $e$ is collapsed to a vertex, then subdividing $e$ and sending its midpoint to a distinct vertex allows to assume
that $f$ sends edge to edge.

Since $f([x_0,\xi]_R)$ is bounded, there exists a  vertex $v\in S$ such that $f\m(v)\cap [x_0,\xi]_R$ is infinite. 
Since $f\m(v)$ consists only of vertices, and there are only finitely many orbits of them,
there is sequence of points $x_n\in [x_0,\xi]_R\cap f\m(v)$ converging to $\xi$ lying in the same orbit, say $x_n=h_nx_1$.
In particular, denoting by $G_v$ the stabilizer of $v$, we have $h_n\in G_v$, and  $\xi$ is carried by $G_v$.
\end{proof}
 
As in the previous section, we deduce the following statement. 
 
\begin{cor}\label{carry-gen}
Let $T\in\overline{\calo}$, and let $\xi\in L^1(T)$. Then $\xi$ 
is carried by a vertex group of the Levitt decomposition of $T$.
\qed
\end{cor}
 
\subsection{Limit set of a point in $\partial_\infty (G,\calf)$} 

The following notion will be useful for relating the laminations $L^1(T)$ and $L^2(T)$ associated to a tree $T\in\overline{\calo}$.

\begin{de}[\textbf{\emph{Limit set}}] 
Given $\xi\in\partial_\infty (G,\calf)$, the \emph{limit set} of $\xi$
is the algebraic lamination $\Lambda^2(\xi)\subseteq\partial^2 (G,\calf)$ defined as follows:
\\ we have $(\alpha,\omega)\in\Lambda^2(\xi)$ if up to exchanging the roles of $\alpha$ and $\omega$,
there exists a sequence $(g_n)_{n\in\mathbb{N}}\in G^{\mathbb{N}}$  converging to $\alpha$ 
such that $g_n.\xi$ converges to $\omega$ as $n\ra \infty$.  
\end{de}

\begin{rk}
If one fixes a Grushko tree $R$ and a basepoint $x_0$, we have $(\alpha,\omega)\in\Lambda^2(\xi)$ if and only if $[\alpha,\omega]_R$ can be obtained as a limit of translates of the ray $[x_0,\xi]_R$. An easy example is that if $\xi$ is an endpoint of a nonperipheral element $g\in G$, then $\Lambda^2(\xi)$ is equal to the lamination made of all translates of $(g^{-\infty},g^{+\infty})$.  
\end{rk}

\begin{lemma}\label{limit_nonempty}
For all   $\xi\in\partial_{\infty}(G,\calf)$,
the limit set $\Lambda^2(\xi)$ is non-empty.
\end{lemma}

\begin{proof}
   Let $[x_0,\xi]_R$ be a ray in a Grushko  $(G,\calf)$-tree $R$ based at
  a point $x_0$ with trivial stabilizer, and let $e_n$ be a sequence of edges
  going to infinity in this ray, and oriented towards $\xi$.  Up to
  passing to a subsequence, we can assume that $e_n=g_n e$ for some
  fixed edge $e$ of $R$. Up to passing to a further subsequence,
  $g_n\m x_0$ converges to some $\alpha\in \partial R$, and $g_n\m \xi$
  to some $\omega\in \partial R$, and $\alpha\neq \omega$ because they
  are separated by the edge $e$.
\end{proof}

\begin{lemma}\label{lem_l1generated}
Let $(\omega,\omega')\in\partial^2(G,\calf)$, with $\omega\in \partial_\infty(G,\calf)$.
\\ Then $\Lambda^2(\omega)$ is contained in the lamination generated by $(\omega,\omega')$.
\end{lemma}

\begin{proof}
  Let $(\alpha,\beta)$ be a leaf in $\Lambda^2(\omega)$.
Let $R$ be a  Grushko tree, and choose a base point $x_0$ in $[\omega,\omega']_R$.
By definition of $\Lambda^2(\omega)$ there exists $g_n\in G$ such that
$g_n.x_0\ra \alpha$ and $g_n.\omega \ra \beta$.

We claim that $g_n\omega'\ra \alpha$.
Let $\Tilde I_n=g_n.[x_0,\omega]_R\cap [\alpha,\beta]_R$.
If $\alpha$ (resp. $\beta$) lies in $V_\infty(G,\calf)$, then for $n$ large enough, $g_n.[x_0,\omega]_R$ contains $\alpha$ (resp. $\beta$)
together with an edge $e_n$ (resp. $e'_n$) incident on $\alpha$ (resp. $\beta$), and not contained in $[\alpha,\beta]_R$. If $\alpha$ (resp. $\beta$) lies in $\partial_\infty(G,\calf)$, then we take the convention that $e_n$ (resp. $e'_n$) is empty. We then let  $I_n:=\Tilde I_n\cup e_n\cup e'_n$.
As $g_n.[x_0,\omega]_R$ contains $\tilde{I}_n$ and the edges $e_n$ and $e'_n$, we have $I_n\subseteq g_n.[x_0,\omega]_R$. Since $[x_0,\omega]_R\subseteq [\omega',\omega]_R$, we deduce that $g_n.[\omega,\omega']_R$ contains $I_n$, so $g_n(\omega',\omega)$ converges to $(\alpha,\beta)$.
\end{proof}

\begin{lemma}\label{carry-carry}
Let $\xi\in\partial_{\infty}(G,\calf)$, and let $A\subseteq G$ be a subgroup of finite Kurosh rank. 
\\ If $\xi$ is carried by $A$, then $\Lambda^2(\xi)$ is carried by $A$. 
\end{lemma}

\begin{proof}
Let $\alpha\in \partial_{\infty}(G,\calf)$ be such that $(\alpha,\xi)$ is carried by $A$.
By Lemma \ref{lem_l1generated}, $\Lambda^2(\xi)$ is contained in the lamination $L$ generated by $(\alpha,\xi)$.
By Lemma \ref{lem_porte_ferme}, every leaf of $L$ is carried by $A$, hence so is $\Lambda^2(\xi)$.
\end{proof} 

The following lemma gives a converse to Lemma \ref{carry-carry} in the case where $A\subseteq G$ is a non-peripheral cyclic subgroup.

\begin{lemma}\label{carry}
Let $A\subseteq G$ be a non-peripheral cyclic subgroup, 
and let $\xi\in\partial_\infty (G,\calf)$ be a point that is not carried by $A$. Then $\Lambda^2(\xi)$ contains a leaf that is not carried by $A$. 
\end{lemma}

\begin{proof}
Fix a Grushko $(G,\calf)$-tree $R$ whose edges have length $1$.
Let $a$ be a generator of $A$, let
$L$ be the axis of $a$ in $R$, and let $l$ be its translation length.

We claim that if $|gL\cap L|>l$ for some $g\in G$, then $gL=L$. Indeed, if $|gL\cap L|>l$, then one of the elements $gag^{-1}a$ or $gag^{-1}a^{-1}$ (depending on the relative orientations of the axes of $a$ and $gag^{-1}$) fixes a nondegenerate arc in $gL\cap L$. Since $R$ has trivial arc stabilizers, we deduce that $gag\m=a^{\pm 1}$, so $gL=L$. 

Fix $x_0\in R$, and write $[x_0,\xi]_R$ as a union of intervals $I_i$ of length $l+2$, which are unions of edges, and such that 
$|I_i\cap I_{i+1}|\geq l+1$ for all $i\in\mathbb{N}$.
We claim that for infinitely many $i$,  the interval $I_i$ is not contained in any translate of $L$.
Indeed, otherwise, there exists $i_0$ such that  for all $i\geq i_0$, there exists $g_i\in G$ such that $I_i$ is contained in $g_iL$.
Since $|I_i\cap I_{i+1}|>l$, we get that $g_i L=g_{i+1}L$. Thus,
 a subray of $[x_0,\xi]_{R}$ is contained in $g_{i_0}L$, contradicting that $\xi$ is not carried by $A$. 

So consider $i_k\ra \infty$ such that for all $k$, the interval $I_{i_k}$ is not contained in any translate of $L$. By cocompactness, and up to passing to a subsequence, there exists $h_k\in G$ such that  $h_k I_{i_k}$ 
contains a fixed edge $e$.
Up to passing to a further subsequence, we can  thus assume that $h_k(x_0,\xi)$ converges to $(\alpha,\omega)\in \Lambda^2(\xi)$.
If $\alpha$ or $\omega$ lies in $V_\infty(G,\calf)$, then 
$(\alpha,\omega)$ is clearly not carried by the cyclic group $A$.
Otherwise, up  to passing to a further subsequence, we can assume that the segment $J=h_k I_{i_k}$ does not depend on $k$.
Since $J\subseteq [\alpha,\omega]_R$ is not contained in any translate of $L$, we deduce that $(\alpha,\omega)$ is not carried by $A$.
\end{proof}

\subsection{Consequences of bounded backtracking}

Given two $G$-trees $T,T'\in\overline{\mathcal{O}}$ and a $G$-equivariant map $f:T\to T'$, 
one says that $C\geq 0$ is a \emph{bounded backtracking (BBT) constant} for $f$, or that $f$ is a $C$-BBT map,
if for any segment $[x,y]\subseteq T$ and any $z\in [x,y]$, one has $d_{T'}(f(z),[f(x),f(y)])\le C$.
We define $BBT(f)$ as the minimal BBT constant of $f$.

\begin{lemma}\label{lem_Q}
Let $R$ be a Grushko $(G,\calf)$-tree, let $T\in \ol\calo$, and let $f:R\ra T$ be a $C$-BBT map.
\begin{enumerate}
\item[1.] For every $g\in G$ elliptic in $T$, we have $\text{diam}(Fix_T(g))\le 2C$.
\item[2.] For all $a\in R$ and all $\omega\in \partial R$, if $f([a,\omega]_R)$ is bounded,
then there is a subray $[b,\omega]_R \subseteq [a,\omega]_R$ such that $f([b,\omega]_R)$ is contained in a ball of radius $3C$.
\item[3.] For all $(\alpha,\omega)\in L^2(T)$, $f([\alpha,\omega]_R)$ has diameter at most $20C$.
\end{enumerate}
\end{lemma}

\begin{proof} 1. Let $u,v\in T$ be fixed by $g$. 
We can assume that $g$ is non-peripheral since otherwise it has a unique fixed point in $T$.
Let $\Tilde u,\Tilde v\in R$ be preimages of $u$ and $v$. 
For any $k\in \bbZ$, we have $f(g^k\Tilde u)=f(\Tilde u)$, and the definition of the BBT then implies that $\diam(f([\Tilde u,g^k\Tilde u]))\leq C$.
Now choose $k\in \bbZ$ such that $I=[\Tilde u, g^k \Tilde u]$ intersects  $J=[\Tilde v, g \Tilde v]$.
Then $f(I)\cap f(J)\neq \es$ and since both have diameter at most $C$, it follows that $d_T(u,v)\leq 2C$.

2. 
Let $\eta:=\sup_{x\in [a,\omega]_R}d_T(f(a),f(x))$, and let $b\in [a,\omega]_R$ be a point that satisfies $d_T(f(a),f(b))\ge \eta-C$. 
Fix $y\in [b,\omega]_R$, and denote by $A,B,Y\in T$ the $f$-images of $a,b,y$. 
Let $P$ be the projection of $B$ on $[A,Y]$, so $PB\leq C$ because $f$ is $C$-BBT.
Then $BY=2PB+AY-AB\leq 2C+\eta-(\eta-C)\leq 3C$. 

3.  By Assertion 2, there is a segment $I=[a,b]\subseteq [\alpha,\omega]_R$ such that the sets $X=f([a,\alpha]_R)$ and $Y=f([b,\omega]_R)$ both have diameter at most $6C$. 
By definition of $L^2(T)$, there exists a sequence of nonperipheral elements $g_i\in G$ such that $(g_i^{-\infty},g_i^{+\infty})$ converges to $(\alpha,\omega)$, and $||g_i||_T\ra 0$.  
For $i$ large enough, we have $I\subseteq \Axis_R(g_i)$. Since $BBT(f)\leq C$, the image $f(\Axis_R(g_i))$ is contained in the $C$-neighbourhood of the characteristic set of $g_i$ in $T$. Since $f(\Axis_R(g_i))$ intersects $X$ and $Y$, we are done if $g_i$ is elliptic in $T$ by Assertion 1.
We can therefore assume that $g_i$ is hyperbolic in $T$, 
and since $||g_i||_T$ tends to $0$, there are infinitely many distinct elements $g_i$. Since there are only finitely many elements of $G$ sending some edge of $I$ into $I$, we can assume that $I\subseteq [a,g_i.a]$ for all $i\in\mathbb{N}$. Since $f(\Axis_R(g_i))$ is contained in the $C$-neighbourhood of $\Axis_T(g_i)$, we have $d_T(f(a),g_if(a))\le ||g_i||_T+2C$. Thus, $f(I)\subseteq f([a,g_i.a])$ has diameter at most $||g_i||_T+4C$. Since $I$ intersects both $[a,\alpha]_R$ and $[b,\omega]_R$, the image $f([\alpha,\omega]_R)$ has diameter at most $||g_i||_T+16C$. The claim follows.
\end{proof}

\subsection{A criterion for checking that an algebraic leaf is dual to a tree}

The following lemma gives a criterion for checking that a pair $(\alpha,\omega)\in\partial^2(G,\calf)$ belongs to the dual lamination of a tree $T\in\ol\calo$.

\begin{lemma}\label{criterion-l2}
Let $T\in\ol\calo$, and let $(\alpha,\omega)\in\partial^2(G,\calf)$. \\
Assume that there exists $K>0$ such that for all $\eps>0$, there exists a Grushko $(G,\calf)$-tree $R_{\eps}$ of covolume smaller than $\eps$ with a $1$-Lipschitz $G$-equivariant map $f_{\eps}:R_\eps\ra T$, such that $f_{\eps}([\alpha,\omega]_{R_\eps})$ has diameter at most $K\eps$.
\\ Then $(\alpha,\omega)\in L^2(T)$. 
\end{lemma}

\begin{rk}\label{rk-criterion}
As follows from the proof, the condition that $\text{diam}(f_{\eps}([\alpha,\omega]_{R_\eps}))\le K\eps$ can be replaced by the \emph{a priori} weaker assumption that there exist $a,b\in [\alpha,\omega]_{R_\eps}$, and a subset $Y_\eps\subseteq T$ of diameter at most $K\eps$, such that both $f_\eps([a,\alpha]_{R_\eps})$ and $f_\eps([b,\omega]_{R_\eps})$ are contained in $Y_\eps$.
\end{rk}

\begin{proof}
We will prove that for every $\eps>0$, we have $(\alpha,\omega)\in L^2_{(K+4)\eps}(T)$. Consider nonstationary sequences of vertices $x_i,y_i\in R_\eps$ 
at distance at most $\eps$ from $[\alpha,\omega]_{R_\eps}$ and converging to $\alpha,\omega$ respectively  (notice that one can take $x_i,y_i\in [\alpha,\omega]_{R_\eps}$ if $\alpha$ and $\omega$ belong to $\partial_\infty(G,\calf)$).
Since $R_\eps$ has covolume at most $\eps$, there exists an element $g_i\in G$ such that $[x_i,y_i]\subseteq \Axis_{R_\eps}(g_i)$ and $d_{R_{\eps}}(y_i,g_ix_i)\leq 2\eps$. In particular, the pairs $(g_i^{-\infty},g_i^{+\infty})$ converge to $(\alpha,\omega)$.
Since $f_\eps$ is $1$-Lipschitz, we have $d_{T}(f_\eps(y_i),g_i f_\eps(x_i))\leq 2\eps$. On the other hand, since $x_i$ and $y_i$ lie at distance at most $\eps$ from $[\alpha,\omega]_{R_\eps}$, using the fact that $f_{\epsilon}([\alpha,\omega]_{R_{\epsilon}})$ has diameter at most $K\epsilon$, we get that $d_{T}(f_\eps(x_i),f_\eps(y_i))\leq (K+2)\eps$.
It follows that $d_{T}(f_\eps(x_i),g_if_\eps(x_i))\leq (K+4)\eps$, and hence $||g_i||_T\leq (K+4)\eps$. Therefore $(\alpha,\omega)\in L^2_{(K+4)\eps}(T)$. 
\end{proof}

The existence of Grushko trees admitting $1$-Lipschitz $\epsilon$-BBT maps towards $T$ is guaranteed in the case where $T$ has dense orbits by the following lemma.

\begin{lemma}\label{bbt}
For all $T\in\overline{\mathcal{O}}$ with dense orbits, and all $\epsilon>0$, there exists a Grushko $(G,\mathcal{F})$-tree $R_{\epsilon}$ with covolume at most $\eps$, and a $1$-Lipschitz $G$-equivariant map $f:R_\eps\ra T$ with $BBT(f)\le\epsilon$.  
\end{lemma}

\begin{proof}
Since $T$ has dense orbits, arc stabilizers in $T$ are trivial \cite[Proposition 4.17]{Hor14-1}.
By \cite[Theorem 5.3]{Hor14-1}, we can find a sequence of Grushko $(G,\calf)$-trees $S_n$ converging non-projectively to $T$, admitting $1$-Lipschitz $G$-equivariant maps $f_n:S_n\ra T$, and such that the covolumes of the trees $S_n$ converge to $0$. By \cite[Lemma 4.1]{BFH97} (see \cite[Proposition~3.12]{Hor14-2} for free products),
we have that $BBT(f_n)\leq  \text{vol}(S_n/G)\Lip(f_n)<\eps$ for $n$ large enough, which concludes the proof.
\end{proof}

\subsection{From $L^1(T)$ to $L^2(T)$}\label{sec-l12}

Given $T\in\overline\calo$, the goal of the present section is to relate $L^1(T)$ to $L^2(T)$.

\begin{prop}\label{l1-2}
For all $T\in\overline{\calo}$, and all $\xi\in L^1(T)$, one has $\Lambda^2(\xi)\subseteq L^2(T)$.
\end{prop}

\begin{proof}
By Corollary \ref{carry-gen}, $\xi$ is carried by a  vertex group $A$ of the Levitt decomposition of $T$,
and $A$ carries $\Lambda^2(\xi)$ by Lemma \ref{carry-carry}.
This way, the proof reduces to the case where $T$ has dense orbits.

Let $(\alpha,\omega)\in \Lambda^2(\xi)$.
Fix $\eps>0$, and consider a Grushko $(G,\calf)$-tree $R_{\eps}$ of covolume at most $\eps$, with a $1$-Lipschitz $\eps$-BBT $G$-equivariant map 
$f_\eps:R_\eps\ra T$ (this exists in view of Lemma~\ref{bbt}). 
By Assertion~2 in Lemma~\ref{lem_Q}, there exists a geodesic ray $\gamma=[x_0,\xi]_{R_\eps}$ in $R_{\eps}$, such that $f(\gamma)$ has diameter at most $6\eps$ in $T$. Without loss of generality, we can assume $x_0$ has trivial stabilizer. 

Let $(g_n)_{n\in\mathbb{N}}\in G^{\mathbb{N}}$ be a sequence of elements such that $g_nx_0\ra \alpha$ 
and $g_n\xi\ra \omega$. Let $[x_0,x_n]\subseteq \gamma$ be an initial segment such that $g_n.x_n$ converges to $\omega$.
For each $n$, consider $h_n\in G$ whose axis has a fundamental domain $D_n$ containing $[x_0,x_n]$, and such that 
$|D_n\setminus [x_0,x_n]|\leq 2\vol(R_\eps/G)\leq 2\eps$. 
Then $g_n.(h_n^{-\infty},h_n^{+\infty})$ converges to $(\alpha,\omega)$ in $\partial^2 (G,\calf)$. 
In addition  $||h_n||_T\leq 2\eps+6\eps$, so
 $(\alpha,\omega)\in L^2_{8\eps}(T)$. Since this holds for every $\eps$, we have
$\Lambda^2(\xi)\subseteq L^2(T)$.
\end{proof}

\begin{rk}
In \cite[Section 5]{CHL08-2}, Coulbois--Hilion--Lustig actually prove that in the context of free groups (with $\calf=\emptyset$), if $T$ is a tree with dense $F_N$-orbits, then $L^2(T)=\Lambda^2(L^1(T))$. Proposition~\ref{l1-2} above only shows one inclusion. Though we believe that equality still holds in the context of free products, we will not prove the reverse inclusion, as we will not need it in the sequel of the present paper. 
\end{rk}

\begin{cor}\label{l1l2}
Let $T,T'\in\overline{\calo}$. 

\begin{enumerate}
\item If $L^1(T)\cap L^1(T')\neq\emptyset$, then
  $L^2(T)\cap L^2(T')\neq\emptyset$.

\item 
  If $L^1(T)\cap L^1(T')$ contains a point carried by a
  nonperipheral subgroup $A\subseteq G$ of finite Kurosh rank, then
  $L^2(T)\cap L^2(T')$ contains a leaf that is carried by $A$.

\item 
  If $L^1(T)\cap L^1(T')$ contains a point $\xi$ that is not carried
  by some nonperipheral cyclic subgroup $A\subseteq G$, then
  $L^2(T)\cap L^2(T')$ contains a leaf that is not carried by $A$.
\end{enumerate}
\end{cor}

\begin{proof}
If $\xi\in L^1(T)\cap L^1(T')$, then Proposition~\ref{l1-2} implies that $L^2(T)\cap L^2(T')\supseteq \Lambda^2(\xi)$
which is non-empty by Lemma~\ref{limit_nonempty}.
The second assertion then follows from Lemma~\ref{carry-carry}, and the last follows from Lemma~\ref{carry}.
\end{proof}

\subsection{Varying the tree}\label{sec-scs}

We now study the behaviour of the lamination $L^2(T)$ as $T$ varies in $\ol \calo$.
Our main result is Proposition \ref{prop-bounded-closed} which says in particular that if
$T_n$ converges to a tree $T\in \ol\calo$ with dense orbits, and if $(\alpha_n,\omega_n)\in L^2(T_n)$ converges to $(\alpha,\omega)\in \partial^2 (G,\calf)$ then $(\alpha,\omega)\in L^2(T)$. Here, the condition that $T$ has dense orbits is crucial, as shown by
Remark \ref{rk-bad} below. However, Proposition \ref{prop-bounded-closed} also provides a statement that is valid for all trees in $\ol\calo$: this is our substitute for the lack of continuity of Kapovich--Lustig's intersection form in the context of free products.

Given trees $T_n,T\in\ol\calo$, and equivariant maps $f_n:R\ra T_n$, $f:R\ra T$ that are linear on edges of $R$, we say that $f_n$ \emph{converges} to $f$ if $T_n$ converges to $T$,
and if for every point $v\in R$, $f_n(v)$ converges to $f(v)$. By definition of the equivariant Gromov--Hausdorff topology, this means that for all $g\in G$, 
$d_{T_n}(f_n(v),gf_n(v))$ converges to $d_T(f(v),gf(v))$, and this implies that $d_{T_n}(f_n(x),f_n(y))$ converges to 
 $d_{T}(f(x),f(y))$ for all $x,y\in R$.
Note that this implies that $\Lip(f_n)$ converges to $\Lip(f)$: this follows from the fact that if $\{e_1=[v_1^1,v_1^2],\dots,e_k=[v_k^1,v_k^2]\}$ is a finite set of representatives of the $G$-orbits of edges in $R$, then we have $$\Lip(f)=\max_{i\in\{1,\dots,k\}}\frac{d_T(f(v_i^1),f(v_i^2))}{d_R(v_i^1,v_i^2)}$$ (and we have the similar expressions for $\Lip(f_n)$).
Also note that if $T_n$ converges to $T$, then for every Grushko $(G,\calf)$-tree $R$ and every $f:R\ra T$ linear on edges, 
there exist maps $f_n:R\ra T_n$ converging to $f$: these are defined by mapping a finite set $\{v_1,\dots,v_k\}$ of representatives of the $G$-orbits of vertices in $R$ to approximations of $f(v_i)$ in $T_n$, and extending equivariantly on vertices and linearly on edges.

\begin{lemma}\label{bounded-closed}
Let $R$ be a Grushko $(G,\calf)$-tree, let $T\in\overline{\calo}$, and let $f_n:R\ra T_n$ be a sequence of maps converging to $f:R\ra T$
for some $T_n\in \ol\calo$ converging to $T$.
Consider a sequence of algebraic leaves $(\alpha_n,\omega_n)\in \partial^2(G,\calf)$ converging to $(\alpha,\omega)\in\partial^2(G,\calf)$. 
\\ If $\diam(f_n([\alpha_n,\omega_n]_R))\leq K$ 
for all $n\in\mathbb{N}$, then $\diam(f([\alpha,\omega]_R))\leq K$.
\end{lemma}

\begin{proof}
Let $I\subseteq [\alpha,\omega]_R$ be a compact segment. Then for all $n\in\mathbb{N}$ sufficiently large, the segment $I$ is contained in $[\alpha_n,\omega_n]_R$. 
Therefore, the diameter of $f_n(I)$ is at most $K$. This implies that the diameter of $f(I)$ is at most $K$. As this is true for any subsegment $I\subseteq [\alpha,\omega]_R$, the diameter of $f([\alpha,\omega]_R)$ is at most $K$.
\end{proof}

\begin{prop}\label{prop-bounded-closed}
Let $T\in\overline{\calo}$, and let $(T_n)_{n\in\mathbb{N}}\in\overline{\calo}^{\mathbb{N}}$ be a sequence that converges to $T$. 
Consider a sequence of algebraic leaves $(\alpha_n,\omega_n)\in L^2(T_n)$ converging to $(\alpha,\omega)\in\partial^2(G,\calf)$. \\
Then $(\alpha,\omega)\in (L^1(T)\cup V_{\infty}(G,\calf))^2$.
\\ If in addition $T$ has dense orbits, then $(\alpha,\omega)\in L^2(T)$.
\end{prop}  

\begin{rk}\label{rk-bad}
Notice however that if $T$ does not have dense orbits, then $l$ may fail to belong to $L^2(T)$. It may even happen that a sequence of trees $T_n$ in $\partial\calo$ converges to a Grushko tree $T$ (as in the examples from the proof of Proposition \ref{trouble}), in which case $L^2(T_n)\neq\emptyset$ while $L^2(T)=\emptyset$. 
\end{rk}

\begin{proof}
Let $R$ be a Grushko $(G,\calf)$-tree. Let $f:R\to T$ be a $G$-equivariant map, and let $f_n:R\to T_n$ be $G$-equivariant maps that converge to $f$.
Let $C:=\vol(R/G)\cdot \Lip(f)$ so that $f$ is $C$-BBT. Since $\Lip(f_n)$ converges to $\Lip(f)$,  the map 
$f_n$ is $2C$-BBT for $n$ large enough. 
Since $(\alpha_n,\omega_n)\in L^2(T_n)$, Lemma \ref{lem_Q} implies that $\diam(f_n([\alpha_n,\omega_n]_R))\leq 40C$. 
By Lemma \ref{bounded-closed}, we have $\diam(f([\alpha,\omega]_R))\leq 40C$, so $(\alpha,\omega)\in (L^1(T)\cup V_\infty(G,\calf))^2$. 

In the case where $T$ is further assumed to have dense orbits, then for all $\eps>0$, Lemma~\ref{bbt} provides a Grushko $(G,\calf)$-tree $R_{\eps}$ of covolume smaller than $\eps$ with a $1$-Lipschitz $G$-equivariant map $f_{\eps}:R_\eps\ra T$. The argument from the previous paragraph then shows that 
$f_{\eps}([\alpha,\omega]_{R_\eps})$ has diameter at most $40\eps$. Lemma \ref{criterion-l2} implies that $(\alpha,\omega)\in L^2(T)$.
\end{proof}

 The following corollary will be crucial for proving the version of our unique duality statement for arational trees given in Corollary \ref{feng-luo} below. 

\begin{cor}\label{feuille-commune}
Let $(T_n)_{n\in\mathbb{N}},(T'_n)_{n\in\mathbb{N}}\in\overline{\calo}^{\mathbb{N}}$ be two sequences of trees that converge to $T,T'$ respectively. 
\\ Assume that $L^2(T_n)\cap L^2(T'_n)\neq\emptyset$ for every $n\in \bbN$. 
Assume moreover that $T$ has dense orbits, and that for all $v\neq v'\in V_\infty(G,\calf)$, the group $\langle G_v,G_{v'}\rangle$ is not elliptic in $T$.
\\ Then $L^1(T)\cap L^1(T')\neq\emptyset$, and hence $L^2(T)\cap L^2(T')\neq\emptyset$.
\end{cor}

\begin{proof}
For all $n\in\mathbb{N}$, let $(\alpha_n,\omega_n)\in L^2(T_n)\cap L^2(T'_n)$, and let $(\alpha,\omega)\in\partial^2(G,\calf)$ be a limit of translates of $(\alpha_n,\omega_n)$. Since $T$ has dense orbits, Proposition \ref{prop-bounded-closed} implies that  $(\alpha,\omega)\in L^2(T)$, and that $\alpha,\omega\in L^1(T')\cup V_{\infty}(G,\calf)$. The hypothesis made on $T$ means that $L^2(T)$ does not contain any leaf with both endpoints in  $V_{\infty}(G,\calf)$, so up to exchanging the roles of $\alpha$ and $\omega$, we can assume that  $\omega\in\partial_{\infty}(G,\calf)$. Then $\omega\in L^1(T)\cap L^1(T')$, showing that $L^1(T)\cap L^1(T')\neq\emptyset$. The fact that $L^2(T)\cap L^2(T')\neq\emptyset$ then follows from Corollary \ref{l1l2}. 
\end{proof}

We will also need the following variation on Proposition~\ref{prop-bounded-closed}. 

\begin{prop}\label{prop-bounded-closed-v2}
Let $T\in\ol\calo$ be a tree with dense orbits, and let $(T_n)_{n\in\mathbb{N}}\in\ol\calo^{\mathbb{N}}$ be a sequence that converges to $T$. Let $(g_n)_{n\in\mathbb{N}}\in G^{\mathbb{N}}$ be a sequence of nonperipheral elements such that $||g_n||_{T_n}$ converges to $0$. Let $(\alpha,\omega)\in\partial^2(G,\calf)$ be an accumulation point of $(g_n^{-\infty},g_n^{+\infty})$.
\\ Then $(\alpha,\omega)\in L^2(T)$.
\end{prop}

\begin{proof}
Fix $\eps>0$. By Lemma~\ref{bbt}, there exists a tree $R\in\calo$ of covolume at most $\eps$ which admits a $1$-Lipschitz equivariant map $f_\eps:R\ra T$.
Let $f_n:R\ra T_n$ be a sequence of maps converging to $f$. Since $\Lip(f_n)$ converges to $\Lip(f)$,
the map $f_n$ is $2\eps$-BBT for $n$ large enough.

Assume first that $\alpha,\omega$ both lie in $\partial_\infty(G,\calf)$.
Let $I_n:=\Axis_{R}(g_n)\cap [\alpha,\omega]_{R}$. The endpoints of $I_n$ converge to $\alpha$ and $\omega$. 

If $||g_n||_{R}$ is bounded, then for $n,n'$ large enough, the axes of $g_n$ and $g_{n'}$
have an overlap that is large compared to their translation length, so the commutator $[g_n,g_{n'}]$
fixes an edge in $R$, hence is trivial. In this case, for $n$ large enough $g_n$ is a power of
some fixed element $h$, and $h$ is elliptic in $T$. Then $(\alpha,\omega)$ are the endpoints of the axis of $h$, and the result is clear.

If $||g_n||_R$ is unbounded, one can choose a fundamental domain $J_n$ of the axis of $g_n$ 
whose endpoints converge to $\alpha$ and $\omega$. 
Since $f_n$ is $2\eps$-BBT and $||g_n||_{T_n}$ tends to zero, $f_n(J_n)$ has diameter
at most $5\eps$ for $n$ large enough. 
In particular for any fixed interval $J\subset [\alpha,\omega]_R$, $f_n(J)$ has diameter at most
$5\eps$ for $n$ large enough.
Since $f_n$ converges to $f$, it follows that $f(J)$ has diameter at most $5\eps$. As this is true for every subsegment $J\subseteq [\alpha,\omega]_R$, we deduce that $f([\alpha,\omega])$ has diameter at most $5\eps$. As this is true for every $\eps>0$, Lemma~\ref{criterion-l2} ensures that $(\alpha,\omega)\in L^2(T)$.

If $\alpha$ or $\omega$ lies in $V_\infty(G,\calf)$, the argument is similar. If both $\alpha$ and $\omega$ lie in $V_\infty(G,\calf)$, then either $g_n$ is eventually constant and the result is clear, or $||g_n||_R\geq d(\alpha,\omega)$ for $n$ large enough.
In all cases, one then finds a fundamental domain $J_n$ for the axis of $g_n$ containing larger and larger segments $J_n\subset [\alpha,\omega]_R$ (with $\alpha\in J_n$ or $\omega\in J_n$ if  $\alpha\in V_\infty(G,\calf)$ or $\omega\in V_\infty(G,\calf)$), 
and one concludes as above.
\end{proof}

\section{Simple elements and simple leaves}\label{sec-5}

The goal of the present section is to introduce the notions of simple elements of $(G,\calf)$ and simple leaves of an algebraic lamination, and prove an analogue of Corollary~\ref{feuille-commune} for those. 

\subsection{Simple elements}

An element $g\in G$ is \emph{simple} (relative to $\calf$) 
if it is non-peripheral and  contained in a proper free factor of $(G,\calf)$. The goal of the present section is to give a criterion for characterizing simple elements in terms of so-called \emph{Whitehead graphs}, which were first introduced in \cite{Whi36} in the context of free groups.

Let $S$ be a Grushko tree. Let $v$ be a vertex in the quotient graph $S/G$. Let $\widetilde{v}$ be the lift of $v$ in $S$, we denote its stabilizer by $G_v$. 
Let $E_v$ be the set of oriented edges of $S/G$ with origin $v$.
For each  $e\in E_v$, we fix a lift $\widetilde{e}$ in $S$ originating at $\widetilde{v}$. Let $g\in G$ be a nonperipheral element. The \emph{Whitehead graph} $\text{Wh}_S(g,v)$ is the labeled graph with vertex set $E_v$, two directions corresponding to oriented edges $e_1$ and $e_2$ being joined by an edge labeled by an element $h\in G_v$ if the axis of $g$ in $S$ crosses a turn in the $G$-orbit of $(\widetilde{e_1},h\widetilde{e_2})$. Notice that there may be several edges joining a given pair of vertices if the labels are distinct; the labels on $\text{Wh}_S(g,v)$ depend on the choice of the lifts of the edges in $E_v$.

Given a connected subgraph $A\subseteq\text{Wh}_S(g,v)$, and a point $x\in A$, we define the \emph{monodromy} $\text{Mon}(A,x)$ of $A$ based at $x$ as the subgroup of $G_v$ made of those elements which label closed loops in $A$ based at $x$. Connectedness of $A$ implies that the conjugacy class of $\text{Mon}(A,x)$ is independent of $x$, so we denote it by $\text{Mon}(A)$.

We note that the conjugacy class $\text{Mon}(A)$ does not depend on the choice of the lifts of the edges in $E_v$. Indeed, if one changes the choices of lifts $\Tilde e_i$  of $e_i$, to $\Tilde e'_i=g_i\Tilde e_i$ (with $g_i\in G_v$), then the label $h$ of each edge $\text{Wh}_S(g,v)$ joining $e_i$ to $e_j$ is changed to $g_i h g_j^{-1}$.

Moreover, if $\tau$ is a maximal subtree of $\text{Wh}_S(g,v)$, then one can change the choices of lifts of the edges adjacent to $v$ so that each edge in $\tau$ has trivial label. 
If $A$ is a connected subgraph of $\text{Wh}_S(g,v)$ with trivial monodromy, one can therefore choose the lifts of the edges adjacent to $v$ so that all edges in $A$ have trivial label.

A vertex $x\in\text{Wh}_S(g,v)$ is an \emph{admissible cutpoint} if there exists a decomposition of the form $\text{Wh}_S(g,v)=A\cup B$, where $A$ is a connected subgraph of $\text{Wh}_S(g,v)$ with $\text{Mon}(A)=\{1\}$, and $A\cap B=\{x\}$. We say that $\text{Wh}_S(g,v)$ has an \emph{admissible cut} if it either has an admissible cutpoint, or is a disjoint union $\text{Wh}_S(g,v)=A\sqcup B$, where $A$ is a connected subgraph with $\text{Mon}(A)=\{1\}$.

We aim to prove the following characterization of simple elements.

\begin{prop}\label{lem_simple}
A nonperipheral element $g\in G$ is simple if and only if for every Grushko tree $S$, there exists a vertex $v\in S/G$ such that $\text{Wh}_S(g,v)$ has an admissible cut.
\end{prop}

We will start by proving the following lemma.

\begin{lemma}\label{lemma-cutpoint}
Let $S$ be a Grushko tree, and let $g\in G$ be a nonperipheral element. 
The following conditions are equivalent. 
\begin{itemize}
\item[$(i)$] There exists a vertex $v\in S/G$ such that $\text{Wh}_S(g,v)$ has an admissible cut.
\item[$(ii)$] There exist a vertex $\tilde v\in S$, two sets $\tilde A, \tilde B$ of edges originating at $\tilde v$ such that 
\begin{itemize}
\item $\tilde A\cap h\tilde A=\es$ for all $h\in G_v\setminus \{1\}$,
\item $\tilde B$ is $G_v$-invariant,
\item $(G_v.\tilde A)\cup \tilde B$ contains all edges originating at $\tilde v$,
\item $\tilde A\cap\tilde B$ contains at most one edge, 
\item if $\{e,e'\}$ is a turn in the axis of $g$ in $S$, then there exists $h\in G$ such that $h.\{e,e'\}\subseteq\tilde  A$ or $h.\{e,e'\}\subseteq\tilde  B$.
\end{itemize}
\item[$(iii)$] There exists a Grushko tree $S'$, and a non-injective cellular map (i.e. sending edge to edge or vertex) from $S'$ to $S$, such that $||g||_{S'}=||g||_S$.
\end{itemize}
\end{lemma}

\begin{rk}
In $(iii)$, in the situation where there is a cellular morphism $f:S'\to S$ (which is a map sending edge to edge), the conclusion says that $g$ is legal for the train-track structure defined by $f$. Notice here that the Grushko trees $S$ and $S'$ can have vertices of valence $2$.
When $f$ is a collapse map, the conclusion says that the axis of $g$ in $S'$ does not contain any collapsed edge. In all cases, the tree $S'$ has more orbits of edges than $S$, but $||g||_{S'}=||g||_S$.
\end{rk}

\begin{proof}
We first prove that $(i)\Leftrightarrow(ii)$. Assume that $(i)$ holds, and consider a decomposition $\text{Wh}_S(g,v)=A\cup B$, where $A$ is connected and has trivial monodromy, and $A\cap B$ contains at most one point. Let $\tilde A$ be a set consisting of exactly one lift of each edge corresponding to a vertex of the subgraph $A\subseteq\text{Wh}_S(g,v)$, where these lifts are chosen so that all labels of edges in $A$ are trivial. Let $\Tilde B$ be the full preimage of $B$ in $S$. Then for these choices of lifts, Assertion $(ii)$ holds.

Conversely, let $\tilde A,\tilde B$ as in $(ii)$, and $A,B$ their image in $S/G$. 
The construction of the Whitehead graph requires to choose lifts of the edges in $E_v$, and we make sure that the lifts of the edges in $A$ are chosen in $\Tilde A$.
We view $A$ and $B$ as induced subgraphs of $\text{Wh}_S(g,v)$.
It is clear that $A\cup B$ contains all vertices of  $\text{Wh}_S(g,v)$ and that $A\cap B$ contains at most one vertex.
The last hypothesis from $(ii)$ ensures that every edge in $\text{Wh}_S(g,v)$ is contained either in $A$ or in $B$, and that all edges in $A$ have trivial label. 
The subgraph $A$ might fail to be connected; however, up to replacing $A$ by one of its connected components $A_0$, and replacing $B$ by $B\cup (A\setminus A_0)$, we get a decomposition of $\text{Wh}_S(g,v)$ as in the definition of an admissible cut, showing that Assertion $(i)$ holds.
\\

\begin{figure}[htb]
\begin{center}
\includegraphics[width=.8\textwidth]{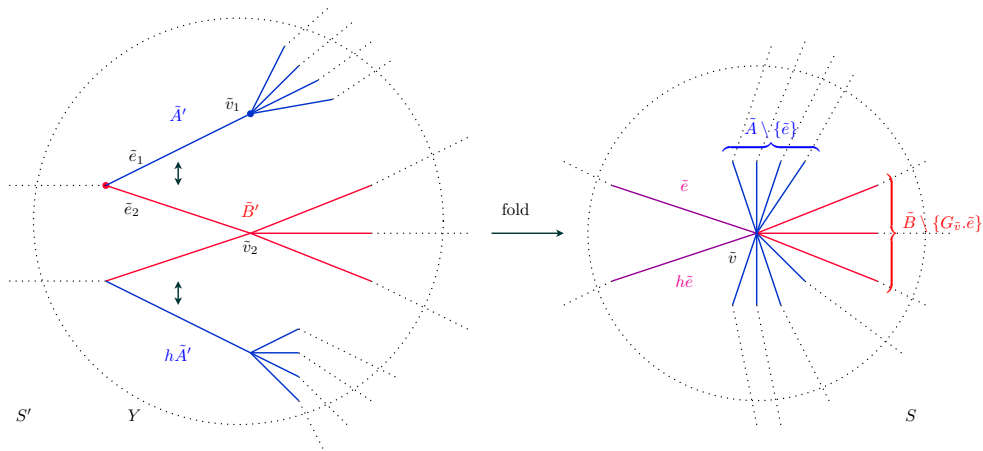}
\caption{$(ii)\Leftrightarrow(iii)$ in the case of a fold}
\label{fig-geom-class}
\end{center}
\end{figure}

We now prove that $(iii)\imp(ii)$. Let $f:S'\ra S$ be a non-injective cellular map.
Since any such map is a composition of collapses and folds, up to changing $S'$ we can assume that $f:S'\ra S$ is a fold or a collapse.
We first consider the case where $f$ is the fold defined by a pair of edges $\tilde e_1,\tilde e_2\subseteq S'$ with the same origin (see Figure \ref{fig-geom-class}). 
Since $S$ has trivial edge stabilizers, the edges $\tilde e_1$ and $\tilde e_2 $ are not in the same orbit. Since $S$ and $S'$ have the same vertex stabilizers, the non-common endpoints $\tilde v_1,\tilde v_2$ of $\tilde e_1,\tilde e_2$ are not in the same $G$-orbit, and either $\tilde v_1$ or $\tilde v_2$ has trivial stabilizer. By symmetry, we assume that $G_{\tilde v_1}=\{1\}$. Let $\tilde{e}:=f(\tilde e_1 )=f(\tilde e_2 )$, and $\tilde{v}:=f(\tilde v_1)=f(\tilde v_2)$. We then have $G_{\tilde v}=G_{\tilde v_2}$.

We define $\Tilde A'$ as the set of edges  of $S'$ incident on $\tilde v_1$, and $\Tilde B'$ as the set of edges incident on $\tilde v_2$. Then $\tilde A'\cap h.\tilde A'=\es$ for every $h\in G_{\tilde v}\setminus\{1\}$, 
 while $\tilde B'$ is $G_{\tilde v}$-invariant. Let $\Tilde A:=f(\Tilde A')$ and $\Tilde B:=f(\Tilde B')$, and note that $\Tilde A\cap \Tilde B=\{\Tilde e\}$. Therefore $\tilde A$ and $\tilde B$ satisfy the first four assertions from $(ii)$.
Since $g$ is legal, any turn in the axis of $g$ in $S$ is the image under $f$ of a turn in the axis of $g$ in $S'$. It follows that the last requirement from $(ii)$ is also satisfied.
\\

\begin{figure}[htb]
\begin{center} 
\includegraphics[width=.6\textwidth]{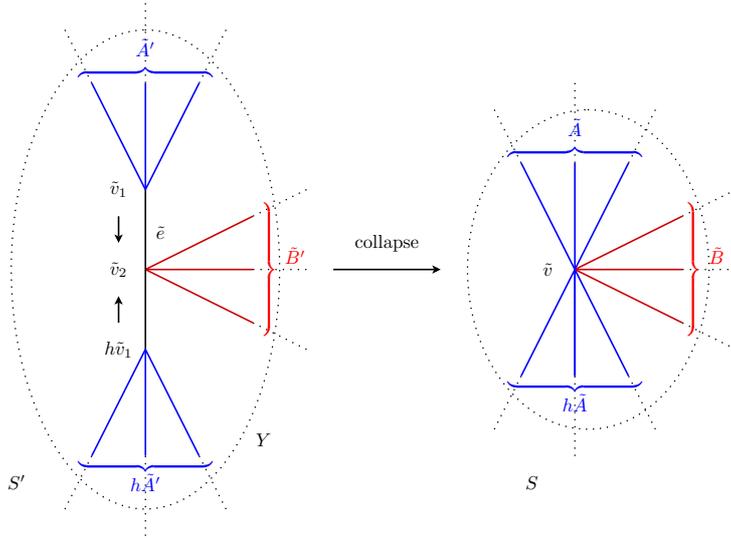}
\caption{$(ii)\Leftrightarrow(iii)$ in the case of a collapse}
\label{fig-geom-class-2}
\end{center}
\end{figure}

\indent We now assume that there exists a Grushko tree $S'$, and a collapse map $f:S'\to S$, with $||g||_{S'}=||g||_S$ (see Figure \ref{fig-geom-class-2}). Let $\tilde e$ be an edge in $S'$ that is collapsed to a point by $f$, and let $\tilde v_1$ and $\tilde v_2$ be the extremities of $\tilde e$. Again, since $S$ and $S'$ have the same vertex stabilizers, the vertices $\tilde v_1$ and $\tilde v_2$ belong to distinct $G$-orbits, and either $\tilde v_1$ or $\tilde v_2$ (say $\tilde v_1$) has trivial stabilizer.  

Let $\Tilde{A}'$ (resp. $\Tilde{B}'$) be the set of edges of $S'$ which are incident on $\tilde v_1$ (resp. $\tilde v_2$), and not in the $G$-orbit of $\tilde e$. Let $\Tilde{A}:=f(\Tilde{A}')$ and let $\Tilde{B}:=f(\Tilde{B}')$. Note that $\tilde A\cap\tilde B=\emptyset$, and that $\tilde B$ is $G_v$-invariant. 
Since $||g||_{S'}=||g||_S$, the axis of $g$ in $S'$ does not cross any edge in the $G$-orbit of $\tilde e$; $(ii)$ follows.  
\\

We prove that $(ii)\imp (iii)$. 
We first assume that $\tilde A\cap \tilde B=\{\tilde e\}$ is non-empty (see Figure \ref{fig-geom-class}).
Consider $$Y=\left(\Big(\coprod_{h\in G_{\tilde v}} h.\Tilde A\Big) \dunion \Tilde B\right)\Bigg/{\sim}$$ where $\sim$ is the equivalence relation
identifying the endpoint of $h.\Tilde e$ in $h.\Tilde A$ with its copy in $\Tilde B$.
There is a natural non-injective cellular map from $Y$ to the star of $\tilde v$.
This is a (bounded) tree with an action of $G_{\tilde v}$, and one defines $S'$ by replacing equivariantly the star of $\tilde v$ in $S$ by $Y$.
Since each vertex at distance $1$ from $\tilde v$ has a unique preimage in $Y$, there is a natural way to attach back the edges of $S$ to $Y$,
thus defining a $G$-tree $S'$, with a non-injective map $S'\ra S$ induced by the map above.
Let $l$ be the axis of $g$ in $S$.
Our assumption on $g$ implies that every turn of $l$ lifts to $S'$.  
The lift $l'$ of $l$ in $S'$ is therefore the axis of $g$ in $S'$, and it embeds isometrically in $S'$ so $||g||_{S'}=||g||_S$.

\indent We now assume that $\tilde A\cap \tilde B=\emptyset$  (see Figure \ref{fig-geom-class-2}). 
We view $\Tilde A$ and $\tilde B$ as subtrees of $S$ and we denote by $\tilde v_1$ (resp. $\tilde v_2$) the copy of $\tilde v$ in $\Tilde A$ (resp. $\Tilde B$). Let $Y$ be the set obtained from 
 $$\left(\Big(\coprod_{h\in G_{\tilde v}} h.\Tilde A\Big) \dunion \Tilde B\right)$$
by attaching an edge $\tilde e$ from $v_B$ to $v_A$ and extending by equivariance. Then there is a natural collapse map from $Y$ to the star of $v$ in $S$. As above, there is a natural way to attach back the edges of $S$ to $Y$, which defines the desired tree $S'$. Since $\tilde A\cap \tilde B=\emptyset$, the axis of $g$ lifts to $S'$ and does not cross any of the newly added edges, so $||g||_{S'}=||g||_S$. 
\end{proof}

\begin{proof}[Proof of Proposition \ref{lem_simple}]
We first assume that $g\in G$ is simple, i.e. there exists a $(G,\calf)$-free splitting $G=A\ast B$, with $g\in A$. Let $S$ be a Grushko tree, and let $S_A$ be the minimal $A$-invariant subtree of $S$, which is a Grushko $(A,\calf_{|A})$-tree. Let $S_B$ be a Grushko $(B,\calf_{|B})$-tree, and let $R$ be a Grushko $(G,\calf)$-tree obtained from the splitting $A\ast B$ by blowing up $A$ (resp. $B$) into $S_A$ (resp. $S_B$). Then up to rescaling and subdividing the edges of $S$, there exists a  cellular map $f:R\to S$, which is the identity when restricted to $S_A$: in particular $||g||_{S}=||g||_{R}$.
If $f$ is an isometry, then the axis of $g$ misses some orbit of edges, hence $\text{Wh}_S(g,v)$ has an isolated vertex.
If $f$ is not injective, the existence of a vertex $v\in S/G$ such that $\text{Wh}_S(g,v)$ has an admissible cut thus follows from Lemma \ref{lemma-cutpoint}.

Conversely, assume that for every Grushko $(G,\calf)$-tree $S$, there exists a vertex $v\in S/G$ such that $\text{Wh}_S(g,v)$ has an admissible cut. 
Let $S$ be a Grushko tree with all edges of length 1. An iterative application of Lemma \ref{lemma-cutpoint} shows that we can find a sequence of Grushko trees $(S_i)_{i\in\mathbb{N}}$ with all edges of length $1$, and with $\text{vol}(S_i/G)=\text{vol}(S/G)+i$, such that $||g||_{S_i}=||g||_S$ for all $i\in\mathbb{N}$. 
Define a \emph{natural edge} of $S_i$ as a connected component of the complement of the set of branch points in $S_i/G$. Since there is a bound on the number of natural edges in $S_i/G$, for all sufficiently large $i\in\mathbb{N}$, one of the natural edges $e$ of $S_i$ has length greater than $||g||_S$. This implies that the axis of $g$ in $S_i$ does not cross any edge in the orbit of $e$, and therefore $g$ is elliptic in the $(G,\calf)$-free splitting obtained from $S_i$ by collapsing all natural edges outside of the orbit of $e$ to points. Therefore $g$ is simple. 
\end{proof}

%


\subsection{Simple algebraic leaves}
\label{sec_simple}

An algebraic leaf is \emph{simple} if it is the limit of a sequence of pairs $(g_n^{-\infty},g_n^{+\infty})$ where for all $n\in\bbN$, $g_n$ is a simple 
element.


\begin{lemma}\label{rational} Let $g\in G$ which is not simple. Then $g^{+\infty}$ is not the endpoint of any simple leaf.
\end{lemma}

\begin{rk}
  This lemma implies that an element $g\in G$ is simple if and only if the leaf $(g^{-\infty},g^{+\infty})$ is simple.
\end{rk}

\begin{proof}
Let $(g_n)_{n\in\mathbb{N}}\in G^{\mathbb{N}}$ be a sequence of elements such that $(g_n^{-\infty},g_n^{+\infty})$ converges to $(\alpha,g^{+\infty})$ for some $\alpha\in \partial (G,\calf)$. We aim to show that $g_n$ is nonsimple for all sufficiently large $n\in\mathbb{N}$.
 
By Lemma \ref{lem_simple}, 
there exists a Grushko tree $S$ so that for all vertices $v\in S/G$, the Whitehead graph $\text{Wh}_S(g,v)$ has no admissible cut. Notice that every compact subset $K$ of the axis of $g$ in $S$ has a translate $g^{i_0}K$ that is contained in $[\alpha,g^{+\infty}]_S$, so for $n$ large enough 
$g^{i_0}K\subseteq [g_n^{-\infty},g_n^{+\infty}]_S$. 
This implies that for $n$ large enough, and for each vertex $v\in S/G$, the Whitehead graph of $g_n$ at $v$ contains the Whitehead graph of $g$ at $v$, with the same labels. In particular, for all vertices $v\in S/G$, the Whitehead graph $\text{Wh}_S(g_n,v)$ has no admissible cut. 
Lemma \ref{lem_simple} 
then implies that $g_n$ is nonsimple.
\end{proof}

\begin{lemma}\label{lem_l1_simple}
  Let $(\omega',\omega)$ be a simple leaf, with $\omega\in \partial_\infty(G,\calf)$.

Then any leaf in $\Lambda^2(\omega)$ is simple.
\end{lemma}

\begin{proof}
Let $L$ be the lamination generated by $(\omega,\omega')$, which consists of simple leaves.
By Lemma \ref{lem_l1generated}, $\Lambda^2(\omega)\subseteq L$.
\end{proof}

Given $T\in\overline{\calo}$, we denote by $L^2_{simple}(T)$ the sublamination made of all simple leaves in $L^2(T)$ (this set is closed). 

\begin{cor}\label{rat}
Let $T,T'\in\overline{\calo}$, and let $(T_n)_{n\in\mathbb{N}},(T'_n)_{n\in\mathbb{N}}\in\overline{\calo}^{\mathbb{N}}$ be two sequences of trees that converge to $T,T'$ respectively. 
 Assume moreover that $T$ has dense orbits, and that for all  $v\neq v'\in V_{\infty}(G,\calf)$, the group $\langle G_v,G_{v'}\rangle$ is not elliptic in $T$.
\\ If $L^2_{simple}(T_n)\cap L^2_{simple}(T'_n)\neq\emptyset$ for every $n\in \bbN$, then $L^2_{simple}(T)\cap L^2_{simple}(T')\neq\emptyset$.
\end{cor}



\begin{proof}
Consider $l_n\in L^2_{simple}(T_n)\cap L^2_{simple}(T'_n)$, and $l=(\alpha,\omega)\in\partial^2(G,\calf)$ be an algebraic leaf obtained as a limit of translates of $l_n$.
Then $l$ is simple and  by Proposition \ref{prop-bounded-closed} we have $l\in L^2(T)$ and $\alpha,\omega\in L^1(T')\cup V_{\infty}(G,\calf)$.
As in the proof of Corollary \ref{feuille-commune}, the hypothesis on $T$ tells us that up to exchanging $\alpha$ and $\omega$, 
 we can assume $\omega\in \partial_\infty(G,\calf)$.
In particular, $\omega\in L^1(T)\cap L^1(T')$. 
By Proposition \ref{l1-2}, $\Lambda^2(\omega)\subseteq L^2(T)\cap L^2(T')$.
By Lemma \ref{lem_l1_simple}, $\Lambda^2(\omega)$ is simple. Since it is non-empty (Lemma \ref{limit_nonempty}) this concludes the proof.
\end{proof}

We finish this section by mentioning the following fact.

\begin{prop}\label{prop:l2simple}
For every tree $T\in\partial\calo$, one has $L^2_{simple}(T)\neq\emptyset$.
\end{prop}

\begin{proof}
We first assume that $T$ has dense orbits. Let $(T_n)_{n\in\mathbb{N}}\in\calo^{\mathbb{N}}$ be a sequence that converges to $T$. Then $\vol(T_n/G)$ converges to $0$. In particular, we can find a sequence of nonperipheral simple elements $g_n$ such that $||g_n||_{T_n}$ converges to $0$. Proposition~\ref{prop-bounded-closed-v2} ensures that every accumulation point of $(g_n^{-\infty},g_n^{+\infty})$ in $\partial^2(G,\calf)$ is a simple leaf in $L^2(T)$.

We now assume that $T$ does not have dense orbits. If $T$ has a simplicial edge with nontrivial stabilizer, then this stabilizer is simple and nonperipheral, so the conclusion follows. If all simplicial edges in $T$ have trivial stabilizer, the conclusion follows from Corollary~\ref{carry-gen-l2}. 
\end{proof}

\section{The Levitt--Lustig map $\calq$ for trees with dense orbits}\label{sec-Q}

In all this section, $T\in \ol\calo$ is a tree with dense orbits.  We denote by $\overline{T}$ the metric completion of $T$, and by $\partial_{\infty}T$ the Gromov boundary of $T$.
In this section, we first extend Levitt--Lustig's map $\mathcal{Q}:\partial(G,\calf)\ra \ol T\cup\partial_\infty T$ (see \cite{LL03}) to the context of free products.
This map identifies the two endpoints of every algebraic leaf in $L^2(T)$  and maps them to $\ol T$ (and not to $\partial_\infty T$), so it induces a map $\calq^2:L^2(T)\ra \ol T$, which  is continuous for the metric topology on $\ol T$ (this extends \cite[Proposition 8.3]{CHL07}).
In Section~\ref{sec_Qcontinu},
we show that $\calq$ is continuous for the observers' topology (see \cite{CHL07} in the free group).
In Section~\ref{sec_QL2}, we prove that the dual lamination $L^2(T)$ coincides with the fibers of $\calq$ (this extends \cite[Proposition 8.5]{CHL08-2}),
which implies that $\ol T\cup \partial_\infty T$ can be identified with the quotient $\partial (G,\calf)/L^2(T)$.


\subsection{Definition of the maps $\calq$ and $\calq^2$}\label{sec_defq}

Lemma \ref{lem_Q} can be used to extend Levitt--Lustig's map $\mathcal{Q}$ (see \cite{LL03}) to the context of free products: this is the content of Proposition~\ref{prop_defQ} below. We will first need the following lemma.

\begin{lemma}\label{quant-f1f2}
Consider $T\in\overline{\calo}$, 
$R_1$ and $R_2$ two Grushko trees, and  $f_1:R_1\to T$ and $f_2:R_2\to T$ two $G$-equivariant maps. 
Then there exists a $G$-equivariant map $g:R_2\to R_1$ such that for all $x\in R_2$, one has $d_T(f_1\circ g(x),f_2(x))\le 2BBT(f_1)$.\\ 
In particular, for all $\omega\in\partial (G,\calf)$, there exist $a_1\in R_1$ and $a_2\in R_2$ such that $f_1([a_1,\omega]_{R_1})$ is contained in the $2BBT(f_1)$-neighbourhood of $f_2([a_2,\omega]_{R_2})$.
\end{lemma}

\begin{proof}
The last assertion in the lemma follows from the first by choosing $a_1$ and $a_2$ such that $[a_1,\omega]_{R_1}\subseteq g([a_2,\omega]_{R_2})$. 
We now prove the first assertion.

Fix $C\leq BBT(f_1)$, and subdivide the edges of $R_2$, so that the  $f_2$-image of each edge of $R_2$ has diameter at most $C$. 
Let  $g:R_2\to R_1$ be any $G$-equivariant map that is linear on edges, that sends any vertex with non-trivial stabilizer
to its fix point in $R_1$, and  sends any other vertex $v_2\in R_2$ to an $f_1$-preimage of $f_2(v_2)$.
Then $f_1\circ g$ and $f_2$ coincide on the vertices of $R_2$.
If $x$ is contained in an edge $[u,v]$ in $R_2$, 
then $g(x)\in [g(u),g(v)]$, so $f_1(g(x))$ lies at distance at most $BBT(f_1)$ from 
$[f_1(g(u)),f_1(g(v))]=[f_2(u),f_2(v)]$. Since the diameter of $f_2([u,v])$ is at most $C$, we get
$d_T(f_2(x),f_1(g(x)))\leq BBT(f_1)+C\leq 2BBT(f_1)$.
\end{proof}

\renewcommand{\theenumi}{\roman{enumi}}
\renewcommand{\labelenumi}{(\theenumi)}

\begin{prop}\label{prop_defQ}
Let $T\in\overline{\mathcal{O}}$ be a tree with dense orbits. There exists a unique $G$-equivariant map $\mathcal{Q}:\partial (G,\calf)\to\overline{T}\cup\partial_{\infty} T$ such that the following holds.
For every $\epsilon>0$, and every Grushko tree $R_\eps$ admitting an $\eps$-BBT map $f_\eps:R_\eps\ra T$
and for all $\omega\in \partial (G,\calf)$, 
\begin{enumerate}
\item \label{it_1} if $\omega\in V_\infty(G,\calf)$, then $\calq(\omega)=f_\eps(\omega)$;
\item \label{it_2} if $\omega\in \partial_\infty(G,\calf)$ and $f_{\eps|[x_0,\omega]_{R_\eps}}$ is unbounded (for some base point $x_0\in R_\eps$), 
then $\calq(\omega)\in \partial_\infty T$ and
 $f_{\eps|[x_0,\omega]_{R_\eps}}$ converges to $\calq(\omega)$;
  \item \label{it_3}
if $\omega\in \partial_\infty(G,\calf)$ and $f_{\eps|[x_0,\omega]_{R_\eps}}$ is bounded (for some base point $x_0\in R_\eps$), 
then $\calq(\omega)\in \overline{T}$, 
and there exists $y_0\in [x_0,\omega]_{R_\eps}$ such that
 $f_\eps([y_0,\omega]_{R_\eps})\subseteq B_{10\eps}(\calq(\omega))$.
\end{enumerate}
\end{prop}

\begin{rk}
  An important point is that the point $\calq(\omega)$ does not depend on the
  map $f_\eps$.
\end{rk}


\begin{proof}
We follow Levitt--Lustig \cite[Proposition 3.1]{LL03}. Uniqueness of the map $\calq$ is obvious (using the fact that there exist Grushko trees with $1$-Lipschitz $\eps$-BBT maps towards $T$ for arbitrary small values of $\eps$, see Lemma \ref{bbt}).
We fix a Grushko tree $R$, an equivariant map $f:R\ra T$, and a basepoint $a\in R$.
Fix $\omega\in \partial (G,\calf)$. 


If $\omega\in V_{\infty}(G,\calf)$, then we define $\calq(\omega)$  as the unique point in $T$ which is fixed by the stabilizer of $\omega$. 
This definition clearly satisfies $(i)$ for every map $f_\eps$.

Assume now that $\omega\in\partial_{\infty} (G,\calf)$, and that $f([a,\omega]_R)$ is unbounded. 
We have $BBT(f)<\infty$, so $f_{|[a,\omega]_R}$ converges to a point in $\partial_\infty T$, which we define
as $\calq(\omega)$. It then easily follows from Lemma \ref{quant-f1f2} that for any $G$-equivariant map $f':R'\ra T$ from another Grushko tree to $T$, and any $a'\in R'$, 
$f'_{|[a',\omega]_{R'}}$  converges to $\calq(\omega)$, showing that $(ii)$ holds. 

We now consider the case where  $f([a,\omega]_R)$ is bounded. Then given any equivariant $C$-BBT map $f':R'\ra T$ from a Grushko tree $R'$ to $T$, 
and any base point $b\in R'$, $f'([b,\omega]_{R'})$ is bounded.
Assertion~2 from Lemma \ref{lem_Q} enables us to choose a point $\calq_{f'}(\omega)\in T$ (depending on $f'$), so that there exists $a'\in R'$ such that $f'([a',\omega]_{R'})\subseteq B_{3C}(\calq_{f'}(\omega))$. 

We now claim that if $f_1:R_1\to T$ and $f_2:R_2\to T$ are $G$-equivariant maps, then $d_{T}(\calq_{f_1}(\omega),\calq_{f_2}(\omega))\le 5BBT(f_1)+3BBT(f_2)$. Property $(iii)$ follows from this claim, by defining $\calq(\omega)$ as the limit of any sequence $\calq_{f_n}(\omega)$ (necessarily Cauchy) with $BBT(f_n)$ converging to $0$. 

We now prove the above claim.
By the last statement in Lemma \ref{quant-f1f2}, there exist $a_1\in R_1$ and $a_2\in R_2$ such that $f_1([a_1,\omega]_{R_1})$ is contained in 
the $2BBT(f_1)$-neighbourhood of $f_2([a_2,\omega]_{R_2})$.
Since by construction, $f_i([a_i,\omega]_{R_i})$ is eventually contained in the ball of radius $3BBT(f_i)$ around $\calq_{f_i}(\omega)$ for all $i\in\{1,2\}$,
we get that $d_{T}(\calq_{f_1}(\omega),\calq_{f_2}(\omega))\leq 3BBT(f_1)+3BBT(f_2)+2BBT(f_1)$. This proves our claim, and finishes the proof of Proposition \ref{prop_defQ}. 
\end{proof}

The following lemma shows that leaves of $L^2(T)$ are contained in fibers of $\calq$. A converse inclusion
will be proved in Proposition \ref{equality-Q}.

\begin{lemma}\label{fibreQ}
  For every algebraic leaf $(\alpha,\omega)\in L^2(T)$, we have $\calq(\alpha)=\calq(\omega)$ and this point lies in $\ol T$ (not in $\partial_\infty T$).
\end{lemma}

\begin{proof}
Let $\eps>0$ and let $f:R_\eps\ra T$ be an $\eps$-BBT map given by Lemma \ref{bbt}. Assertion~3 from Lemma~\ref{lem_Q} shows that $f([\alpha,\omega]_{R_\eps})$ has diameter at most $20\eps$.
Proposition \ref{prop_defQ} thus implies that $\calq(\alpha)$ and $\calq(\omega)$ lie in $\ol T$, and that one extremity of the line $[\alpha,\omega]_R$ is eventually mapped in a ball of radius $10\eps$ centered at $\calq(\alpha)$, and the other extremity is eventually mapped in a ball of radius $10\eps$ centered at $\calq(\omega)$. Hence $d_{\overline{T}}(\calq(\alpha),\calq(\omega))\le 40\eps$. As this is true for all $\eps>0$, we have $\calq(\alpha)=\calq(\omega)$. 
\end{proof}

\begin{de}
We define $\calq^2:L^2(T)\ra \ol T$ by $\calq^2(\alpha,\omega):=\calq(\alpha)=\calq(\omega)$.
\end{de}

\begin{rk}
The map $\calq^2$ is not surjective in general. 
This will be clearer later when $L^2(T)$ will be identified with pairs of endpoints of leaves in a band complex:
points of $T$ whose leaves are one-ended are not in the image of $\calq^2$. Trees that are dual to a Levitt-type system of isometries yield examples
(see \cite{Gaboriau_bouts,Levitt_exotic}).
On the other hand, the image of $\calq^2$ is dense in $T$, because every orbit in $T$ is dense.
\end{rk}

The map $\calq^2$ has a strong continuity property. This is an adaptation of \cite[Proposition~8.3]{CHL08-2}.

\begin{prop}\label{q2}
The map $\mathcal{Q}^2:L^2(T)\to\overline{T}$ is continuous for the metric topology on $\overline{T}$. 
\end{prop}

\begin{proof}
Let $\eps>0$ and let $f:R_\eps\ra T$ be an $\eps$-BBT map given by Lemma \ref{bbt}.
Let $(\alpha,\omega)\in L^2(T)$, and let $((\alpha_n,\omega_n))_{n\in\mathbb{N}}\in L^2(T)^{\mathbb{N}}$ be a sequence converging to $(\alpha,\omega)$. 
By Lemma \ref{lem_Q}, $f([\alpha,\omega]_{R_\eps})$  has diameter at most $20\eps$,
and is therefore contained in the $30\eps$-neighbourhood of $\calq(\omega)=\calq^2(\alpha,\omega)$ by Proposition \ref{prop_defQ}.
Similarly, $f([\alpha_n,\omega_n]_{R_\eps})$ in contained in the $30\eps$-neighbourhood of $\calq^2(\alpha_n,\omega_n)$ for all $n$.
For all sufficiently large $n\in\mathbb{N}$, one has $[\alpha_n,\omega_n]_{R_\eps}\cap [\alpha,\omega]_{R_\eps}\neq\emptyset$ 
so $\calq^2(\alpha_n,\omega_n)$ is at distance at most $60\eps$ from $\calq^2(\alpha,\omega)$.
\end{proof}

\subsection{Continuity of $\calq$ for the observers' topology} \label{sec_Qcontinu}

\begin{lemma}\label{lem_continuite}
Let $T\in\overline{\calo}$ be a tree with dense orbits. 
\\ Then for all $x\in\ol T$, the orbit map $G\to\overline{T},g\mapsto gx$ admits a unique continuous extension $G\cup\partial(G,\calf)\to\overline{T}\cup\partial_{\infty}T$ for the observers' topology on $\overline{T}\cup\partial_{\infty}T$, and this extension coincides with $\calq$ on $\partial(G,\calf)$.
\end{lemma}

\begin{proof} 
Since the observers' topology on $ \overline T\cup \partial_\infty T$ is first countable, it is enough to prove that for all $x\in\overline{T}$, if $(g_n)_{n\in\mathbb{N}}\in G^{\mathbb{N}}$ is a sequence that converges to $\omega\in\partial (G,\calf)$, then $(g_n.x)_{n\in\mathbb{N}}$ converges to $\calq(\omega)$ (for the observers' topology on $ \overline T\cup \partial_\infty T$). 

We can assume without loss of generality that $x\in T$. Indeed, assume that the result has been proved for all $x\in T$, and let $x\in\overline{T}$. Let $(g_n)_{n\in\mathbb{N}}\in G^\mathbb{N}$ be a sequence that converges to $\omega$. If $(g_n.x)_{n\in\mathbb{N}}$ does not converge to $\calq(\omega)$, then there exists a direction $d$ in $T$ containing $\calq(\omega)$, and such that $g_i.x\notin d$ for infinitely many $i$. Let $d'$ be another direction in $T$ that is stritcly contained in $d$ and also contains $\calq(\omega)$. Then for every point $x'\in T$ that is close enough to $x$, we have $g_i.x'\notin d'$ for infinitely many $i$. Therefore $(g_n.x')_{n\in\mathbb{N}}$, does not converge to $\calq(\omega)$, a contradiction.

We therefore let $x\in T$, and assume towards a contradiction that there exists a direction $d\subseteq T$ containing $\calq(\omega)$ and such that $g_i.x\notin d$ for all $i$ (up to passing to a subsequence). We denote by $u\in  T$ be the base point of $d$.

We first distinguish the case where $\calq(\omega)\in \partial_\infty T$.
Let $f:R\ra T$ be a $1$-Lipschitz map, with BBT constant $C$, and $a\in R$ be a preimage of $x$ (this exists by minimality because $x\in T$).
The ray $[a,\omega]_R$ contains a subray $[b,\omega]_R$ such that $f([b,\omega]_R)\subseteq d$ and stays at distance at least $2C$ from $u$.
Since $g_n.a$ converges to $\omega$, there exists $i,j\in \bbN$ such that $[g_ia, g_ja] \cap [b,\omega]_R$ contains a point $c$. Since by assumption,
$g_nx=f(g_na)\notin d$ for all $n\in \bbN$, $f(c)$ is at distance at least $2C$ from $[g_ix,g_jx]$, contradicting that $BBT(f)\leq C$.

We now consider the case  where $\calq(\omega)\in \ol T$.
Let $\eps>0$ be such that the ball  $B_{20\eps}(\calq(\omega))$ is contained in $d$.
Let $R_\eps$ be a Grushko tree and $f:R_\eps\ra T$ a $1$-Lipschitz map with $BBT(f)\leq \eps$.
Let $a\in R_\eps$ be such that $f(a)=x$. 

Assume first that $\omega\in \partial_\infty (G,\calf)$.
By definition of $\calq$, there exists $b\in [a,\omega]_{R_\eps}$ such that $f([b,\omega]_{R_\eps})\subseteq B_{10\eps}(\calq(\omega))\subseteq d$. Since by hypothesis $g_i. a$ converges to $\omega$, there exist $i,j$ so that the segment $[g_i.a,g_j.a]_{R_\eps}$ 
contains a point $c$ in $[b,\omega]_{R_\eps}$. 
Since $BBT(f)\leq \eps$, we have $d_{T}(f(c),[g_i.x,g_j.x])\leq \eps$, but since $g_i.x,g_j.x\notin d$ and $f(c)\in d$, 
we get that $d_T(f(c),u)\leq \eps$, so  $d_{\overline{T}}(u,\calq(\omega))\leq 11\eps$, contradicting our choice of $\eps$.

If $\omega\in V_{\infty}(R_\eps)$,
we can  find $i,j$ 
 so that $\omega \in [g_i.a,g_j.a]_{R_\eps}$. 
Since $BBT(f)\le\epsilon$,  we get $d_{T}(f(\omega),u)\leq \eps$, so $d_{T}(\calq(\omega),u)\le\eps$ 
because  $f(\omega)=\calq(\omega)$.
This is a contradiction. 
\end{proof}

\begin{cor}\label{cor_compact}
  If $K\subseteq \overline T$ is compact for the metric topology, if $(g_i)_{i\in\mathbb{N}}\in G^{\mathbb{N}}$ is a sequence of elements converging to $\omega\in \partial (G,\calf)$, and if $x\in\overline{T}$ is such that $g_i\m.x\in K$ for all $i\in\mathbb{N}$, then $\calq(\omega)=x$.
\end{cor}

\begin{proof}
  Up to a subsequence, we can assume that $g_i\m.x$ converges to $y\in K$.
This means that  $d_{\overline{T}}(x,g_i y)$ tends to $0$, which implies that $g_i y$ converges to $x$ in the observers' topology.
By Lemma \ref{lem_continuite}, we have $\calq(\omega)=x$.
\end{proof}

\subsection{The map $\calq$ and $L^2(T)$}\label{sec_QL2}

We have already seen in Lemma \ref{fibreQ} that the endpoints of every algebraic leaf in $L^2(T)$ have the same image under $\calq$. The following lemma (see \cite[Proposition 8.5]{CHL08-2} for free groups) proves the converse.

\begin{prop}\label{equality-Q}
Let $T\in\overline{\mathcal{O}}$ be a tree with dense orbits.
\\ Then for $\alpha,\omega\in \partial (G,\calf)$, one has $\mathcal{Q}(\alpha)=\mathcal{Q}(\omega)$
if and only if $\alpha=\omega$ or $(\alpha,\omega)\in L^2(T)$.
\end{prop}

\begin{proof}
In view of Lemma \ref{fibreQ}, we only need to prove that if $\alpha\neq\omega$ have the same image under $\calq$, then $(\alpha,\omega)\in L^2(T)$. 

Let us first prove that $\calq(\alpha)=\calq(\omega)$ cannot lie in $\partial_\infty T$.
Assume otherwise and fix a Grushko tree $R$ with a map $f:R\ra T$, 
and let  $a\in R$ be a base point. Since the ray $[f(a),\calq(\omega)]_T$ is contained in $f([a,\alpha]_{R})$ and in $f([a,\omega]_{R})$,
there exist two sequences $x_n\in [a,\alpha]_{R}$ and $y_n\in [a,\omega]_{R}$ converging to $\alpha$ and $\omega$ respectively,
and such that $f(x_n)=f(y_n)$ converges to $\calq(\omega)$. This contradicts that $BBT(f)<\infty$.

Let $R_\eps$ be a Grushko tree of covolume smaller than $\eps$, with a $1$-Lipschitz map $f_\eps:R_\eps\ra T$ with $BBT(f_\eps)\leq \eps$. By Proposition \ref{prop_defQ}, there exist $x,y\in [\alpha,\omega]_{R_\eps}$ such that $f_\eps([x,\alpha]_R)$ and $f_\eps([y,\omega]_R)$ are contained in $B_{10\eps}(\calq(\alpha))$ (this also obviously holds if $\alpha$ or $\omega$ lies in $V_\infty(R)$). 
Lemma \ref{criterion-l2} (in the form provided by Remark \ref{rk-criterion}) then implies that $(\alpha,\omega)\in L^2(T)$.
\end{proof}

Proposition \ref{equality-Q} says that $L^2(T)$ essentially coincides with the equivalence relation on $\partial (G,\calf)$ given by the fibers of $\calq:\partial (G,\calf)\ra \overline{T}\cup \partial_{\infty} T$,
the only difference being the diagonal $\Delta\subseteq \partial (G,\calf)\times\partial (G,\calf)$.
Slightly abusing notations, we denote by $\partial (G,\calf)/L^2(T)$ the quotient of $\partial (G,\calf)$ by this equivalence relation. The map $\calq$ is onto because its image is compact and every orbit of $\overline{T}\cup \partial_{\infty} T$ is dense (this easily follows from the fact that
every orbit is dense in $T$ for the metric topology).
Being a continuous map between compact spaces, the map $\calq$ induces a homeomorphism on the quotient (this is \cite[Corollary 2.6]{CHL07} in the context of free groups).

\begin{cor}\label{quotient-lamination}
The map induced by $\calq$ from $\partial (G,\calf)/L^2(T)$ to $\overline{T}\cup\partial_{\infty}T$ is a homeomorphism for the observers' topology
on $\overline{T}\cup\partial_{\infty}T$.\qed
\end{cor}

\begin{cor}\label{cor_partition}  
There is no partition $\partial (G,\calf)=A\dunion B$ with $A,B$ non-empty closed subsets of $\partial(G,\calf)$ that are saturate under the equivalence relation given by $L^2(T)$.
\end{cor}

\begin{proof}
This follows from Corollary \ref{quotient-lamination}, together with the connectedness of $\overline{T}\cup\partial_{\infty}T$ (for the observers' topology).
\end{proof}

\section{Band complexes}\label{sec-band-complex}
  
In this section, we construct band complexes generalizing those by Coulbois--Hilion--Lustig \cite{CHL09} based on the compact heart they introduced. 
These are not simplicial complexes in general, as bands are homeomorphic to $K\times [0,1]$ for some compact $\bbR$-tree $K$ (whose set of branch points might be dense).

In all this section, we fix $T\in\overline{\mathcal{O}}$ a tree with dense orbits, and $R$ a Grushko tree.

\subsection{Definitions}\label{sec-def-band}

\paragraph{Compact trees.}

Given a point $u\in R$, we define $L^2_u(T)$ as the subspace of $L^2(T)$ made of those pairs $(\alpha,\omega)\in\partial^2(G,\calf)$ such that $u\in [\alpha,\omega]_R$. We then let $\Omega_u:=\calq^2(L^2_u(T))$, and we let $K_u$ be the convex hull in $\overline{T}$ of $\Omega_u$. 
We also notice that when $u\in R$ belongs to the interior of an edge $e\subseteq R$, the sets $L^2_u(T)$, $\Omega_u$ and $K_u$ only depend on $e$: we will also denote them by $L^2_e(T)$, $\Omega_e$ and $K_e$ in this case. 
Notice that when $u\in R$ lies in the interior of an edge or is a vertex of finite valence, the set $L^2_u(T)$ is compact. Continuity of $\calq^2$ (Proposition \ref{q2}) then implies that $\Omega_u$ and $K_u$ are also compact (for the metric topology on $\ol T$) in this case. 
If $u$ has infinite valence, since there are finitely many $G_u$-orbits of edges incident on $u$, it follows that $\Omega_u$ and $K_u$ are $G_u$-cocompact.

\begin{lemma} \label{lem_non-empty}
For every $u\in R$, the set $\Omega_u$ is non-empty.
\end{lemma}

\begin{proof}
If $u$ belongs to the interior of an edge, then the two connected components of $R\setminus \{u\}$ yield a partition of $\partial (G,\calf)$ into two closed subsets. If $\Omega_u=\es$, then these two sets are saturate for $L^2(T)$, contradicting Corollary \ref{cor_partition}. The result also follows when $u$ is a vertex of $R$, because the set $\Omega_u$ is then the union of the sets $\Omega_e$ over all edges $e\subseteq R$ that are incident on $u$. 
\end{proof}

\paragraph{The band complex.}

\begin{de}[\emph{\textbf{The band complex $\Sigma(T,R)$}}]
We let $$\Sigma(T,R):=\{(x,u)\in\overline{T}\times R|x\in K_u\}.$$
\end{de}

When $T$ and $R$ are clear from the context, we will simply write $\Sigma$ instead of $\Sigma(T,R)$. This is a closed $G$-invariant subset of $\overline{T}\times R$ (for the metric topology). The diagonal action of $G$ on $\overline{T}\times R$ induces a $G$-action on $\Sigma$. Denote by $p_T:\Sigma\to\overline{T}$ and $p_R:\Sigma\to R$ the projections, which are $G$-equivariant. We view $\Sigma$ as a foliated complex, where the leaves are the connected components of the fibers of $p_T$ (actually, Proposition \ref{fibers-connected} below will say that these fibers are connected). 
\\

\begin{figure}[htbp]
\begin{center}
\includegraphics{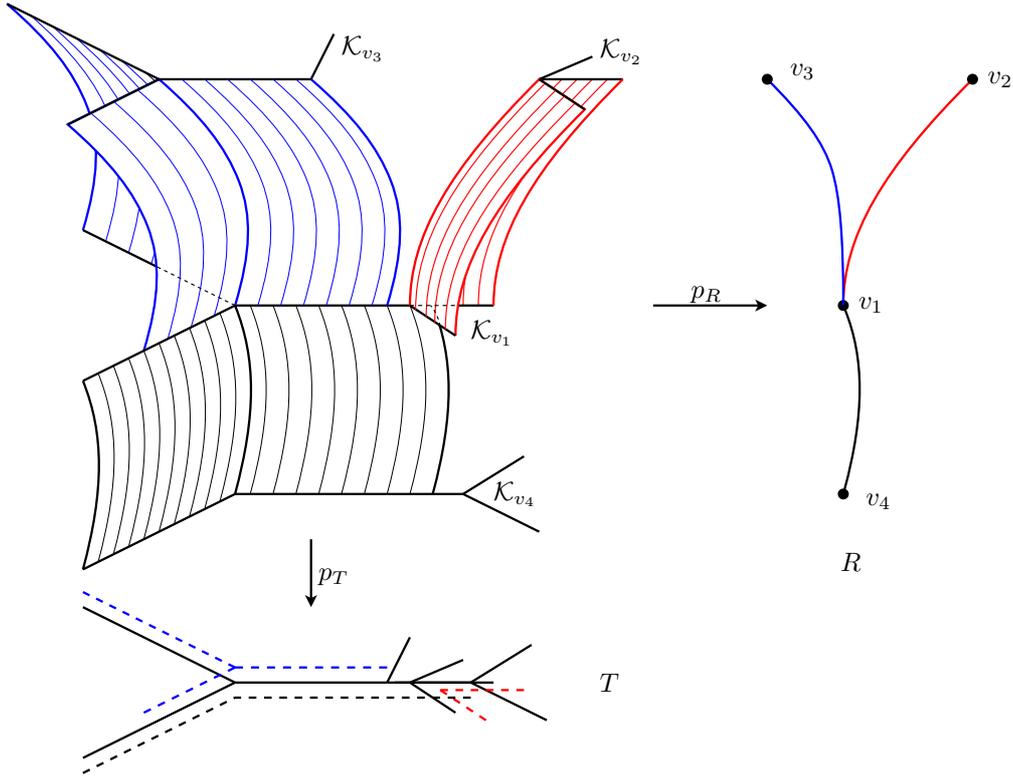}
\caption{A band complex.}
\label{fig-band-complex}
\end{center}
\end{figure}

\indent Here is an alternative description of $\Sigma$, illustrated on Figure~\ref{fig-band-complex}. Given an edge $e\subseteq R$, we let $B_e:=K_e\times e\subseteq \Sigma$, which we call a \emph{band} of $\Sigma$. We also call $K_e\times\mathring{e}$ an \emph{open band}. Given a point $v\in R$, we let $\calk_v:=K_v\times\{v\}\subseteq\Sigma$, which we call a \emph{base tree} when $v$ is a vertex. Then $\Sigma$ is also the band complex obtained from $$\Sigma_V:=\coprod_{v\in V(R)}\calk_v\text{ and }\calb:=\coprod_{e\in E(R)}B_e$$ as $$\Sigma=(\Sigma_V\dunion \calb)/{\sim}$$ where, for each edge $e=[u,v]$ of $R$, we identify $K_e\times \{v\}\subseteq B_e=K_e\times [u,v]$ with the corresponding subset of $\calk_v$ (and similarly for $u$).

The space $\Sigma$ is the generalization of the universal covering of the band complex considered
by \cite{CHL09} in the case of the free group, where $R$ is a rose. Since in this case, $R$ has a unique orbit of vertices,
there is a single base tree $K_v$ up to the group action, and this tree is compact since $R$ is locally finite.
This compact tree $K_v$ is Coulbois--Hilion--Lustig's compact heart \cite{CHL09}.
In the present section (see Proposition \ref{prop_coeur} below), we will see that this set also coincides with the core of $\overline{T}\times R$, as introduced by the first author in \cite{Gui_coeur}.

\paragraph{Special vertices.} We make the following observation.

\begin{lemma}
For every $v\in V_\infty(R)$, there exists a unique point in $\calk_v$ which is fixed by $G_v$.
\end{lemma}

\begin{proof}
Let $x\in\Omega_v$. By $G$-equivariance of $\calq^2$, for all $g\in G_v$, we have $gx\in\Omega_v$. As $K_v$ is convex, this implies that the unique point of $T$ fixed by $G_v$ belongs to $K_v$.   
\end{proof}

When $v\in V_{\infty}(R)$, we denote by $v_{\Sigma}$ the unique point in $\calk_v$ fixed by $G_v$, which we call a \emph{special vertex} of $\calk_v$ (if $G_v$ is finite, then $\calk_v$ has no special vertex). Sometimes, a point in $V_\infty(R)$ is viewed as a point in the boundary $\partial R$ and is denoted by $\omega$ instead of $v$,
we then use the notation $\omega_\Sigma$ for the corresponding special vertex.
The special vertices of $\Sigma$ are exactly the points of $\Sigma$ whose isotropy group is infinite.

\paragraph{Leaves of the band complex and the set $\Omega$.}
To make the distinction between algebraic leaves (i.e. elements of $\partial^2(G,\calf)$) and leaves of $\Sigma$ explicit, the latter will be called \emph{complete $\Sigma$-leaves}.
Given $x\in\Sigma$, we denote by $\call_x$ the complete $\Sigma$-leaf containing $x$.
Note that the restriction of the projection $p_R$ to $\call_x$ is injective so each complete $\Sigma$-leaf is a tree.
Any segment 
contained in a complete $\Sigma$-leaf will be called a \emph{$\Sigma$-leaf segment}.
If $x,y$ are in the same complete $\Sigma$-leaf $\call$, we denote by $[x,y]_\call$ the unique $\Sigma$-leaf segment joining them.
A bi-infinite $\Sigma$-leaf segment $l$ is called a \emph{$\Sigma$-leaf line}.
By extension, we also call a \emph{$\Sigma$-leaf line} any leaf segment $l$
joining two special vertices in $\Sigma$, or a simply-infinite leaf segment whose base point is a special vertex. 
In particular, a $\Sigma$-leaf line has two well-defined endpoints in $\partial (G,\calf)$: these are the endpoints of $p_R(l)$.
Similarly, we define a \emph{$\Sigma$-leaf semi-line} as an oriented, semi-infinite leaf segment, or a finite leaf segment whose extremity is a special vertex. A $\Sigma$-leaf semi-line has a well-defined extremity in $\partial (G,\calf)$, and a base point in $\Sigma$.

The following set will be of particular importance in later sections in the study of the band complex $\Sigma$.

\begin{de}[\emph{\textbf{The set $\Omega$}}]\label{def_omega}
  We denote by $\Omega\subseteq \Sigma$ the union of all $\Sigma$-leaf
  lines (whose endpoints, by definition,  can belong to either
  $V_\infty(G,\calf)$ or to $\partial_\infty (G,\calf)$).
\end{de}

\subsection{Simple connectedness}

\begin{lemma}\label{thm_cell-like}
  Let $f:X\ra Y$ be a continuous surjective map between compact, metrizable, finite-dimensional, locally contractible spaces. 
  Assume that for all $y\in Y$, $f\m(y)$ is contractible.
  \\ Then $f$ is a homotopy equivalence.
\end{lemma}

\begin{proof}
By \cite[Theorem~V.7.1]{Hu}, the spaces $X$ and $Y$ are absolute neighbourhood retracts. Then $f$ is a cell-like map, so by Haver's theorem (see \cite[Theorem~7.1.6]{vM} or \cite[Theorem~p.511]{Lac}), $f$ is a homotopy equivalence.
\end{proof}

\begin{lemma}\label{contractible}
The space $\Sigma$ is simply connected. Moreover, for all compact subtrees $K_R\subseteq R$, there exists a compact subtree $K_T^0\subseteq\ol T$ such that for every compact subtree $K_T$ containing $K_T^0$, the space $\Sigma\cap (K_R\times K_T)$ is simply connected.
\end{lemma}

\begin{proof}
We first explain how to deduce simple connectedness from the second assertion. We first show that $\Sigma$ is connected. If not, let $C_1$ and $C_2$ be two distinct connected components of $\Sigma$. Let $K_R$ be a compact subtree of $R$ that intersects both $p_R(C_1)$ and $p_R(C_2)$. Then for every compact subtree $K_T\subseteq T$, the space $\Sigma\cap (K_T\times K_R)$ is disconnected, a contradiction. This shows that $\Sigma$ is connected. Moreover, if $\gamma:S^1\to\Sigma$ is any loop, then $\gamma(S^1)$ is contained in $\Sigma\cap (K_T\times K_R)$ for some compact subtrees $K_T\subseteq\ol T$ and $K_R\subseteq R$, so the second assertion implies that $\Sigma$ is simply connected. 

We now prove the second assertion. Let $K_R\subseteq R$ be a compact subtree. Using Lemma~\ref{lem_non-empty}, we can find a compact subtree $K_T^0\subseteq \ol T$ such that $p_R(\Sigma\cap(K_T^0\times K_R))=K_R$; notice then that every compact subtree $K_T\subseteq \ol T$ that contains $K_T^0$ also satisfies this property. We will apply Lemma~\ref{thm_cell-like} to the compact set $C:=\Sigma\cap (K_T\times K_R)$ and the continuous surjective map $p_R:C\to K_R$. The space $K_R$ is compact, metrizable, finite-dimensional and locally contractible, being a compact subtree of $R$. The space $C$ is metrizable and finite-dimensional, being a subspace of a product of two trees. It is compact because $K_T$ and $K_R$ are compact, and $\Sigma$ is a closed subspace of $\overline{T}\times R$. Point preimages of $p_R$ are compact subtrees of $\ol T$, hence contractible, so we only need to show that $C$ is locally contractible. 
If $x\in C$ is such that $p_R(x)$ is contained in the interior of an edge $e$ of $R$, then $x$ has a neighbourhood homeomorphic to 
 $K_e\times \rond e$. If $p_R(x)$ is a vertex $v\in R$, then for some neighbourhood $V$ of $v$ in $R$, $p_R\m(V)$ retracts by deformation onto $p_R\m(\{v\})$.
It follows that $C$ is locally contractible, which concludes the proof.
\end{proof}

\begin{rk}
Since $\Sigma$ is connected, it follows that $p_T(\Sigma)$ is a connected subspace of $\overline{T}$. In particular, by minimality of $T$, we have $T\subseteq p_T(\Sigma)$. 
One can show that $p_T$ is never surjective on $\ol T$: indeed, the complete metric space $\ol T$ would be covered by countably many compact trees $K_e$, 
and each of them has empty interior (in the metric topology), contradicting the Baire property.
\end{rk}

%
%

\subsection{Comparing leaves of $\Sigma$ with leaves of $L^2(T)$}

\begin{lemma}\label{lem_pT}
Let $l$ be a $\Sigma$-leaf semi-line with extremity $\omega\in \partial_{\infty} (G,\calf)$. Then $p_T(l)=\{\calq(\omega)\}$. 
\\
Similarly, if $\omega\in V_{\infty}(G,\calf)$ and $\call_{\omega_\Sigma}$ is the complete $\Sigma$-leaf through the special point $\omega_{\Sigma}$,
then $p_T(\call_{\omega_\Sigma})=\{\calq(\omega)\}$. 
\end{lemma}

\begin{proof}
The second assertion is obvious by definition of $\calq$ on $V_\infty(G,\calf)$. To prove the first, let $m$ be the midpoint of an edge $e$ of $R$  
such that there exists a sequence
of elements $g_i\in G$ converging to $\omega$, with $g_i.m\in p_R(l)$ for all $i\in\mathbb{N}$.
Let $x:=p_T(l)$. Then for all $i\in\mathbb{N}$, we have $(x,g_i.m)\in\Sigma$, and hence $(g_i^{-1}.x,m)\in\Sigma$. This implies that $g_i\m.x$ lies in the compact subset  $K_e\subseteq \overline T$ for all $i\in\mathbb{N}$, so $\calq(\omega)=x$ by Corollary \ref{cor_compact}.
\end{proof}

\begin{lemma}\label{lem_L2}
A subset of $\overline{T}\times R$ is a $\Sigma$-leaf line if and only if it is of the form $\{\calq(\alpha)\}\times [\alpha,\omega]_R$ with $(\alpha,\omega)\in L^2(T)$.\\
In particular, for all $(\alpha,\omega)\in\partial^2 (G,\calf)$, one has $(\alpha,\omega)\in L^2(T)$ if and only if there exists a $\Sigma$-leaf line with endpoints $\alpha$ and $\omega$, and in this situation such a $\Sigma$-leaf line is unique. 
\end{lemma}

\begin{proof}
For all $(\alpha,\omega)\in L^2(T)$, we have $\calq(\alpha)=\calq(\omega)$ (Proposition \ref{equality-Q}), and it then follows from the definition of $\Sigma$ that $\{\calq(\alpha)\}\times [\alpha,\omega]_R\subseteq\Sigma$. 
Conversely, if $l\subseteq \Sigma$ is a $\Sigma$-leaf line whose endpoints are $\alpha,\omega\in \partial (G,\calf)$, then Lemma \ref{lem_pT} shows that $\calq(\alpha)=\calq(\omega)$, and hence $(\alpha,\omega)\in L^2(T)$.
In this case, we have $p_T(l)=\{\calq(\alpha)\}$ (Lemma \ref{lem_pT}), so $l=\{\calq(\alpha)\}\times p_R(l)=\{\calq(\alpha)\}\times [\alpha,\omega]_R$.
\end{proof}

\begin{cor}\label{cor_omega}
For all $u\in R$, and all $x\in K_u$, one has $x\in\Omega_u$ if and only if there is a $\Sigma$-leaf line that contains $(x,u)$.
%
\end{cor}

\begin{proof}
If $x\in\Omega_u$, 
then there exists $(\alpha,\omega)\in L^2_u(T)$ such that $x=\calq(\alpha)=\calq(\omega)$.
This means in particular that $u\in [\alpha,\omega]_R$, so $\{x\}\times [\alpha,\omega]_R$ is a $\Sigma$-leaf line that contains $(x,u)$. 
%
Conversely, assume that $l\subseteq \Sigma$ is a $\Sigma$-leaf line that contains $(x,u)$, then $l$ is of the form $\{x\}\times [\alpha,\omega]_R$ for some $(\alpha,\omega)\in L^2(T)$ with $\calq(\alpha)=x$. Since
$u\in p_R(l)$, we have $(\alpha,\omega)\in L^2_u(T)$, so $x\in\Omega_u$. 
\end{proof}


\begin{rk}\label{rk_omega}
As a consequence of Corollary~\ref{cor_omega}, the set $\Omega$ from Definition~\ref{def_omega} is also the union over all edges $e$ of $\Omega_e\times e$. 
In particular, let $B_e=K_e\times e$ be a band of  $\Sigma$. Since $K_e$ is the convex hull of $\Omega_e$,
every terminal point $x\in K_e$ belongs to $\Omega_e$ and therefore 
the  $\Sigma$-leaf segment $\{x\}\times e$ is contained in $\Omega$.
\end{rk}

\subsection{Finiteness properties}\label{sec-finiteness}

When working in the context of free products, one new phenomenon arises that did not occur in the context of free groups: in our case, the trees $\calk_v$ are no longer compact when $v\in V_{\infty}(R)$, and there are infinitely many bands attached to $\calk_v$ (there are finitely many $G$-orbits of them, though, but for example there are infinitely many orbits of pairs of bands attached to $\calk_v$). In the present section, we present some finiteness results that will be key for extending the work in \cite{CHL07,CH14,CHR11} to our context.

\begin{lemma}\label{lem_finitude_0}
For any vertex $v\in R$, and any point $x\in \calk_v$ which is not a special vertex, there is a neighbourhood $V$ of $x$  in $\calk_v$ which intersects only finitely many bands.
\end{lemma}

\begin{proof}
If $\calk_v$ has no special point (\ie if $G_v$ is finite), then there are only finitely many bands incident on $\calk_v$, so the lemma is trivial.
Otherwise, denote by $v_{\Sigma}$ the special point of $\calk_v$, and by $d$ the distance from $x$ to $v_{\Sigma}$.
If infinitely many bands intersect the ball of radius $d/2$ around $x$ in $\calk_v$, 
then since there are only finitely many $G_v$-orbits of bands intersecting $\calk_v$, 
we can assume that these bands are in the $G_v$-orbit of a band $B_e$. Thus, 
there are infinitely many elements $g_i\in G_v$ 
such that 
$d_{\overline T}(x,g_i K_e)\leq d/2$.
Thus $K_e$ contains a point at distance less than $d/2$ from $g_i\m x$, contradicting compactness of $K_e$ for the metric topology.
\end{proof}

From Lemma \ref{lem_finitude_0}, one immediately deduces the following facts.

\begin{cor}\label{cor_loc_fini}
1) If $C\subseteq \calk_v$ is a compact subset avoiding the special vertex of $\calk_v$, then $C$ intersects only finitely many bands.
\\
2) Any complete $\Sigma$-leaf $\call\subseteq \Sigma$ is a locally finite tree, except at points $x\in \call$ with infinite stabilizer.
\qed
\end{cor}

\begin{rk}
Note however that the valence of points with trivial stabilizer in a given leaf might \emph{a priori} still be unbounded.
\end{rk}

\begin{cor}\label{finitude-1}
For any vertex $v\in R$, and any point $x\in \calk_v$, there exists an open neighbourhood $V$ of $x$ in $\calk_v$ such that every band which intersects $V$ contains $x$.
\end{cor}

\begin{proof}
If $x$ is not a special point, then Corollary \ref{finitude-1} follows from Lemma \ref{lem_finitude_0} by possibly passing to a smaller neighbourhood of $x$. If $x$ is a special point, then there are only finitely many orbits of bands attached to $\calk_v$ that do not contain $x$, so there is a lower bound $d$ on the distance from $x$ to a band that does not contain $x$. In this case, we choose for $V$ an open ball centered at $x$ of radius $d/2$.
\end{proof}


\begin{cor}\label{cor_limit_leaves}
  Let $x_i\in \calk_v$ be a sequence converging to $x\in \calk_v$, where $x$ is not the special vertex of $\calk_v$. Let $l_i$ be a $\Sigma$-leaf semi-line starting from $x_i$.
\\ Then there is a $\Sigma$-leaf semi-line $l$ starting from $x$, such that for each finite initial subsegment $I$ of $l$,
there are infinitely many indices $i$ such that $p_R(I)\subseteq p_R(l_i)$. 
\\ In particular,  given any band $B$ in $\Sigma$, if for every $i$, the semi-line $l_i$ starts with an edge in $B$, then so does $l$. If for every $i$, the semi-line $l_i$ starts with an edge that is not in $B$, then so does $l$.
\end{cor}

\begin{rk}\label{rk_limit_leaves}
   The corollary still holds with the same proof if $x$ is a special point, assuming that the initial edges of $l_i$ all lie in a common band $B$.
\end{rk}

\begin{rk}
  It may happen that the extremity of $l$ is a special vertex of $\Sigma$ even if the extremity of every $l_i$  belongs to $\partial_\infty (G,\calf)$.
\end{rk}

\begin{proof}
By   Lemma \ref{lem_finitude_0}, up to extracting a subsequence, we can assume that the first edges $[x_i,y_i]$ of $l_i$
lie in a common band $B_1$. Let $y$ be the limit of $y_i$. If $y$ is a special point, we can take $l=[x,y]$ and we are done.
If not, one can apply the same argument to the sequence $y_i$: up to extracting a further subsequence, we can assume that
the second edges $[y_i,z_i]$ of $l_i$ lie in a common band $B_2$. Iterating this argument proves the corollary.
\end{proof}

\begin{lemma}\label{cor_voisinage_fini}
Let $x\in \calk_v$, and let $c\subseteq \call_x$ be a connected component of $\call_x\setminus \{x\}$. Denote by $B$ the band that contains the initial edge of $c$. If $c$ is finite, then $x$ is not extremal in $B\cap\calk_v$, and there exists an open ball $U$ around $x$ in $\calk_v$ such that for all $x'\in U$, 
the unique connected component $c_{x'}$ of $\call_{x'}\setminus\{x'\}$ that intersects $B$ is finite, and $p_R(c_{x'})\subseteq p_R(c)$.
\end{lemma}

\begin{rk}\label{rk_voisinage_fini} In particular, if $\call_x$ is finite, then there exists a neighbourhood $U$ of $x$ such that $\mathcal{L}_{x'}$ is finite for all $x'\in U$, and $p_R(\call_{x'})\subseteq p_R(\call_x)$. Indeed, consider a neighbourhood $V$ of $x$ in $\calk_v$ such that every band intersecting $V$ contains $x$ (Corollary \ref{finitude-1}). 
Applying Lemma \ref{cor_voisinage_fini} for each of the finitely many bands containing $x$ proves our claim. 
Using this lemma, we will actually prove in Proposition \ref{prop_infini} below that no complete $\Sigma$-leaf is finite. 
\end{rk}

\begin{proof}
Since $c$ is finite, it contains no special point, so $c$ does not contain any $\Sigma$-leaf semi-line. Notice also that $x$ is not extremal in $B\cap\calk_v$: otherwise, denoting by $e$ the $p_R$-image of the initial edge of $c$, we would have $p_T(x)\in\Omega_e$, and $c$ would be infinite.
 
If there is a sequence of points $x_i$ converging to $x$ with $c_{x_i}$ infinite, then one can
find $\Sigma$-leaf semi-lines $l_i$ starting at $x_i$ whose first edges are contained in the band $B$. Corollary \ref{cor_limit_leaves} then implies that $c$ contains a $\Sigma$-leaf semi-line, a contradiction.

If $p_R(c_{x_i})\not \subseteq p_R(c)$ for some sequence $x_i$ converging to $x$, 
then up to extracting a subsequence, we can assume that there exists a $\Sigma$-leaf segment $(x,y]\subseteq c$, and $\Sigma$-leaf segments $(x_i,z_i]\subseteq c_{x_i}$
that can be written $(x_i,z_i]=(x_i,y_i]\cup [y_i,z_i]$ with  $p_R([x_i,y_i])=p_R([x,y])$, with $y_i$ converging to $y$, and where $[y_i,z_i]$ is an edge of $\call_{x_i}$ with $p_R([y_i,z_i])\not\subseteq p_R(\call_x)$. Applying Corollary \ref{finitude-1} at the point $y$, we get a contradiction.
\end{proof}

Given $v\in V_{\infty}(R)$, there are only finitely many orbits of bands incident on  $\calk_v$.
However, the set of $G_v$-orbits of \emph{pairs of bands} incident on  $\calk_v$ is infinite. 
Given a band $B$ incident on a base tree $\calk_v$, we denote by $\ul B=B\cap \calk_v$.
The following lemma records another finiteness result that will be useful later.

\begin{lemma}\label{lem_finitude}
For all $v\in V_{\infty}(R)$, there are finitely many $G_v$-orbits of pairs of bands $(B,B')$  incident on $\calk_v$ such that $v_\Sigma \notin B\cup B'$, and $v_\Sigma$ does not separate $B$ from $B'$.
\\
There are also only finitely many $G_v$-orbits of pairs of bands $(B,B')$ incident on $\calk_v$ with $B\cap B'\neq\es$ and $v_{\Sigma}\notin B\cap B'$.

More generally, for $\eps>0$ small enough, there are only finitely many $G_v$-orbits of pairs of bands $(B,B')$ incident on $\calk_v$ 
with $v_{\Sigma}\notin B\cap B'$ and $d(\ul B,\ul B')\leq \eps$.
\end{lemma}

\begin{proof}
The first assertion follows from finiteness of the number of $G_v$-orbits of bands, together with the observation that for all bands $B,B'$ that do not contain $v_\Sigma$, there is at most one $g\in G_v$ such that $v_\Sigma$ does not separate $B$ from $gB'$. 

The second assertion follows from the third.

To derive the third assertion from the first, let $\eps_0>0$ be such that $d(\ul B,v_\Sigma)\geq \eps_0$ for all bands $B$ not containing $v_\Sigma$.
Such $\eps_0$ exists because there are only finitely many $G_v$-orbits of such bands. If $B,B'$ both avoid $v_\Sigma$, then there is at most one element $g\in G_v$
such that $d(\ul B,g\ul B')<2\eps_0$. 
If $B$ does not contains $v_\Sigma$, consider $p$ the projection of $v_\Sigma$ on $\ul B$ and $\eps_1$ be such that only finitely many bands
contain a point at distance at most $\eps_1$ from $p$ (Lemma \ref{lem_finitude_0}).
Then there are only finitely many bands containing $v_\Sigma$ and at distance at most $\eps_1$ from $B$.
This proves the lemma.
\end{proof}

\subsection{Complete $\Sigma$-leaves are unbounded.}

\begin{prop}\label{prop_infini}
For all $x\in \Sigma$, the complete $\Sigma$-leaf $\call_x$ contains a $\Sigma$-leaf semi-line whose extremity lies in $\partial_\infty (G,\calf)$ (in particular $\call_x$ is unbounded).
\end{prop}

\begin{rk}\label{rk_one_band}
In particular, every point is contained in at least one band.
\end{rk}

\begin{proof} It is easy to see that if $\call_x$ contains two different special points, then it 
contains a $\Sigma$-leaf line with both endpoints in $\partial_\infty (G,\calf)$.
So assume that $\call_x$ contains at most one special point. 
Let $v\in R$ be such that $x\in\calk_v$.
By Corollary \ref{cor_omega}, 
 $\call_x$ contains a $\Sigma$-leaf line if  $x\in \Omega_v$, so we can as well assume that $x\notin \Omega_v$.

\indent We will first assume that $\call_x$ does not contain any special point. 
Assume towards a contradiction that $\call_x$ contains no $\Sigma$-leaf semi-line whose extremity lies in $\partial_{\infty} (G,\calf)$. 
We are going to produce a partition $\partial (G,\calf)=A\dunion B$ contradicting Corollary \ref{cor_partition}.
Since $\call_x$ is locally finite (Corollary \ref{cor_loc_fini}) and does not contain any infinite ray by hypothesis, König's lemma \cite{Kon27} implies that $\call_x$ is finite.
By Remark~\ref{rk_voisinage_fini}, there exists a neighbourhood $U$ of $x$ in $\calk_v$ such that for all $x'\in U$, the complete $\Sigma$-leaf $\call_{x'}$ is finite, and $p_R(\call_{x'})\subseteq p_R(\call_x)$. 

Denote by $\hat U\subseteq \Sigma$ the union of all complete $\Sigma$-leaves through a point in $U$.
Its closure is compact, and up to choosing $U$ small enough, we can assume that the closure of $\hat U$ does not contain any special vertex of $\Sigma$.
We claim that $\Sigma\setminus\hat U$ is disconnected. 

Indeed, since $x\notin \Omega_v$, there exists $a\neq b\in \Omega_v$ (in particular, $a,b\notin \hat U$) such that $x\in [a,b]_{\calk_v}$.
Assume that there is a path $\gamma$ joining $a$ to $b$ avoiding $\hat U$.
Without loss of generality, one can assume that $\gamma$ is a concatenation of segments contained in base trees and of $\Sigma$-leaf segments contained in a band.
Choose $\gamma$ of minimal combinatorial length. If $p_R(\gamma)$ is reduced to a point, $\gamma$ cannot avoid $\hat U$, a contradiction.
Since $p_R(\gamma)$ is a loop, it thus has a backtracking point, and has a subpath of the form $\gamma_1.\gamma_2.\gamma_3$ with 
$\gamma_1,\gamma_3$ leaf segments in a band $B_e$, and $\gamma_2$ in a base tree. This subpath can be replaced by $\gamma'_2$ in the other base tree
of $B_e$, thus shortening $\gamma$ ($\gamma'_2$ still avoids $\hat U$ because one can choose it so that every leaf through a point in $\gamma'_2$
meets $\gamma_2$).
This contradiction proves that  $\Sigma\setminus \hat U$ is disconnected.

Write $\Sigma\setminus \hat U=\Sigma_A\dunion \Sigma_B$ a decomposition into two non-empty clopen sets.
Consider $R_A=p_R(\Sigma_A)$ and $R_B=p_R(\Sigma_B)$. 
Note that $R_A\cup R_B$ contains $R\setminus p_R(\call_x)$ and that $R_A\cap R_B$ is contained in the finite tree $p_R(\call_x)$.
This implies that $R_A\setminus p_R(\call_x)$ and $R_B\setminus p_R(\call_x)$ are unions of connected components of $R\setminus p_R(\call_x)$
and that $\partial_\infty R_A\cap \partial_\infty R_B=\es$.
Define  $A\subseteq\partial R$ as the union of $\partial_\infty R_A$ together with $p_R(\Sigma_A\cap V_\infty(\Sigma))$ (where $V_\infty(\Sigma)$ denotes the set of special points of $\Sigma$),
and $B$ in a similar fashion.

The sets $A$ and $B$ are disjoint and partition $\partial R$.
They are saturate under $L^2(T)$  in view of Lemma \ref{lem_L2}, because there is no leaf of $\Sigma$ joining $\Sigma_A$ to $\Sigma_B$.
To check that $A$ is closed, consider a sequence $a_n\in A$ converging to $c\in \partial R$.
In the case where $c\notin p_R(\call_x)$, then for $n$ large enough, $a_n$ and $c$ are in the closure of the same connected component of $R\setminus p_R(\call_x)$, so $c\in A$.
So assume that $c\in p_R(\call_x)$. Then $c$ is a vertex of $R$. 
For $n$ large enough, the edge $e_n\subseteq R$ with origin $c$ and pointing towards $a_n$ is not contained in the finite tree $p_R(\call_x)=p_R(\hat U)$.
It follows that the bands $B_{e_n}$ are contained in $\Sigma_A$. 
All these bands 
intersect infinitely many connected components of $\calk_c\setminus \{ c_\Sigma\}$. 
Since $\hat U$ intersects only one such component, $c_\Sigma$ lies in $\Sigma_A$, so $c\in A$.
We conclude that $A$ is closed, and likewise so is $B$,
which finishes the proof of Proposition \ref{prop_infini} when $\call_x$ does not contain any special point.

\indent We finally treat the case where $\call_x$ contains exactly one special point $v_\Sigma$. Assume by contradiction that $\call_x$ 
contains no $\Sigma$-leaf semi-line whose extremity lies in $\partial_\infty R$. Being locally finite (Corollary \ref{cor_loc_fini}), each connected component $c$ of $\call_x\setminus\{ v_\Sigma\}$ is finite.
Let $B_c$ be the unique band containing $x$ and with non-empty intersection with $c$. By Lemma~\ref{cor_voisinage_fini}, there is an open ball $U_c$ around $v_\Sigma$ in $\calk_v$ such that for all $x'\in U_c$, the connected component $c'$ of $\call_{x'}$ that intersects $B_c$ is finite and $p_R(c')\subseteq p_R(c)$. Since there are only finitely many $G_{ v_\Sigma}$-orbits of components $c$, one can take $U_c=U$ independent of $c$. By Corollary \ref{finitude-1}, there is a neighbourhood $V$ of $v_\Sigma$ in $\calk_v$ such that for all $x'\in V$,
any band containing $x'$ is one of the bands $B_c$. 
It follows that for all but countably many points $x'\in U\cap V$, the complete $\Sigma$-leaf $\call_{x'}$ is bounded and contains no special point.  This is a contradiction by the argument above.
\end{proof}

\subsection{Fibers are connected}

The following result says that fibers of $p_T:\Sigma\ra \overline T$ are connected.

\begin{prop}\label{fibers-connected} 
 For all $x,x'\in \Sigma$, we have $p_T(x)=p_T(x')$ if and only if $x$ and $x'$ are in the same complete $\Sigma$-leaf.
\end{prop}

\begin{proof}
The ``if'' statement is obvious. We let $x,x'\in\Sigma$ be such that $p_T(x)=p_{T}(x')$, and we aim to show that $x$ and $x'$ are in the same complete $\Sigma$-leaf. Let $l$ (resp. $l'$) be a $\Sigma$-leaf semi-line joining $x$ (resp. $x'$) to a point $\omega\in\partial_\infty R$ (resp. $\omega'\in\partial_{\infty}R$): this exists by Proposition \ref{prop_infini}. Then $p_T(l)=p_T(l')$.

If $\omega=\omega'$, then $p_R(l)$ and $p_R(l')$ eventually coincide. Since $p_T(l)=p_T(l')$, $l$ and $l'$ eventually coincide, 
and $x,x'$ are in the same complete $\Sigma$-leaf.

If $\omega\neq \omega'$, then $\calq(\omega)=\calq(\omega')$ by Lemma \ref{lem_pT}, so $(\omega,\omega')\in L^2(T)$. 
By Lemma \ref{lem_L2}, there is a $\Sigma$-leaf line $l''$ joining $\omega$ to $\omega'$. Lemma \ref{lem_pT} shows that $p_T(l'')=p_T(l)=p_T(l')$, so $l''$ eventually coincides with $l$ 
at one of its ends, 
and with $l'$ on its other end. This implies that $x$ and $x'$ are in the same complete $\Sigma$-leaf.
\end{proof}

\begin{rk}
Given any two points $x,y\in\Sigma$, there exists a path $\gamma$ from $x$ to $y$ which is a finite concatenation of segments contained in base trees and of leaf segments. In fact by Proposition~\ref{fibers-connected}, there always exists such a $\gamma$ whose projection to $\ol T$ has no backtracking. In particular, this enables us to refine Lemma~\ref{contractible} and deduce that for all compact subtrees $K_R\subseteq R$ and $K_T\subseteq \ol T$, the space $\Sigma\cap (K_T\times K_R)$ is simply connected when non-empty. Indeed, in the proof of that lemma, it is enough to show that, letting $C:=\Sigma\cap (K_T\times K_R)$, the set $p_R(C)$ is a connected subset of $K_R$; this follows from the above observation.  
\end{rk}

\subsection{Comparison with the core}

We now deduce from the connectedness of the fibers of $p_T$ that $\Sigma$ is precisely the core introduced in \cite{Gui_coeur} by the first author (this was noticed beforehand by the first author and Thierry Coulbois in the context of free groups in an informal discussion). We briefly recall the definition of the core of the product of two $G$-trees $T$ and $T'$, and refer the reader to \cite{Gui_coeur} for details. A \emph{quadrant} in $T\times T'$ is a product of two directions $\delta\subseteq T$ and $\delta'\subseteq T'$. Fix a basepoint $(x_0,x'_0)\in T\times T'$. A quadrant $Q$ is \emph{heavy} if there exists an infinite sequence $(g_n)_{n\in\mathbb{N}}$ such that $d_T(x_0,g_nx_0)\to +\infty$, and $d_{T'}(x'_0,g_nx'_0)\to +\infty$, and $g_n.x_0\in Q$ for all $n\in\mathbb{N}$. Otherwise it is \emph{light}. The \emph{core} of $T\times T'$ is then defined as the complement of the union of all light quadrants in $T\times T'$. 
When $T\in \ol  \calo$ is a tree with dense orbits, the core is also characterized as the smallest nonempty $G$-invariant closed connected subset of $T\times T'$ with connected fibers (see \cite[Théorème~principal]{Gui_coeur}).  

\begin{prop}\label{prop_coeur}
As a subset of $\overline T\times R$, the set $\Sigma$ is precisely the core of $\overline T\times R$.
\end{prop}

\begin{proof} 
  Since $\Sigma$ is closed and connected, and since fibers of $p_T$ and $p_R$ are connected,
it follows from \cite[Proposition~5.1]{Gui_coeur}  that $\Sigma$ contains the core $\calc$.
Conversely, it is enough to prove that for all $u\in R$, and all $x\in\Omega_u$, we have $(x,u)\in \calc$.
Indeed, since $K_u$ is the convex hull of $\Omega_u$, and since $\calc$ has convex fibers, this will imply that $\Sigma\subseteq \calc$.

Consider $\delta$ a direction in $\ol T$ containing $x$, and $\delta'$ a direction in $R$ containing $u$.
Since $x\in \Omega_u$, there exists $(\alpha,\omega)\in L^2_u(T)$ such that $x=\calq(\alpha)=\calq(\omega)$ and 
$u\in [\alpha,\omega]_R$.
This implies that either $\alpha$ or $\omega$, say $\omega$, lies in the closure of $\delta'$.
Consider $g_i\in G$ a sequence of nonperipheral elements such that $(g_i^{-\infty},g_i^{+\infty})$ converges to $(\alpha,\omega)$. 
One easily checks that we can choose $g_i$  hyperbolic in $T$:
apply \cite[Lemma 1.3]{Gui_coeur} with $I_i=[a_i,b_i]$ as in Lemma \ref{lem_approximating} (i.e.\ $a_i\neq \alpha$ (resp.\ $b_i\neq \omega$) converge to $\alpha$ (resp. $\omega$))
to get a semi-group
$S'$ of elements whose axes contain $I_i$, and such that $\grp{S'}=G$, which implies that some element of $S'$ is hyperbolic in $T$.
For $i$ large enough, $g_i^{+\infty}$ lies in the closure of $\delta'$.
Moreover, by continuity of $\calq$ for the  observers' topology, $\calq(g_i^{+\infty})$ lies in the closure of $\delta$ for some $i_0$ large enough.
Since $\calq(g_{i_0}^{+\infty})$ is the endpoint of the axis of $g_{i_0}$, this is enough to conclude that the powers of $g_{i_0}$ make the quadrant $\delta\times \delta'$ heavy.
Since this applies to any quadrant containing $(x,u)$, it follows by definition of the core that $(x,u)\in\calc$.
\end{proof}


\section{Pruning and preimages of $\calq$}\label{sec-pruning}

The main result of the present section is the following theorem, which is the main theorem of \cite{CH14} in the context of free groups. Its proof uses a pruning process on the band complex $\Sigma$ associated to $T$ and to a Grushko tree $R$, which was introduced in \cite{CH14} for free groups.

\begin{theo}\label{Q-preimage}
Let $T\in\overline{\mathcal{O}}$ be a tree with dense orbits. Then for all but finitely many orbits of points $x\in \ol T\cup \partial_\infty T$, the set $\mathcal{Q}^{-1}(x)$ contains at most two points.
\end{theo}

We will deduce Theorem \ref{Q-preimage} from the following proposition.

\begin{prop}\label{prop_3ends}
Let $T\in\overline{\mathcal{O}}$ be a tree with dense orbits, let $R$ be a Grushko tree,
and $\Sigma=\Sigma(T,R)$ the associated band complex.
\\ Then $\Sigma$ has only finitely many orbits of complete $\Sigma$-leaves with at least 3 ends.
\end{prop}

\begin{proof}[Proof of Theorem \ref{Q-preimage} from Proposition \ref{prop_3ends}]
Let $x\in\overline{T}\cup\partial_\infty T$ be such that $\calq^{-1}(x)$ contains  three distinct points $\omega_1,\omega_2,\omega_3\in \partial (G,\calf)$. 
By Proposition~\ref{equality-Q}, for all $1\leq i < j\leq 3$, we have $(\omega_i,\omega_j)\in L^2(T)$. 
By Lemma~\ref{lem_L2}, there exists a complete $\Sigma$-leaf $\call_{i,j}$ containing a $\Sigma$-leaf line joining $\omega_i$ to $\omega_j$
and $p_T(\call_{i,j})=x$. By Proposition \ref{fibers-connected}, $\call_{i,j}=\call$ does not depend on $i,j$ so
$\call$ has at least 3 ends. Indeed, this is clear if the three points $\omega_i$ belong to $\partial_\infty(G,\calf)$. If some $\omega_i$ lies in $V_\infty(G,\calf)$, then $\call$ contains a special point, and $\call$ has infinitely many ends. 
Proposition~\ref{prop_3ends} concludes the proof.
\end{proof}

\subsection{The pruning process: elementary step}

\begin{figure}[htbp]
\begin{center}
\includegraphics[width=.6\textwidth]{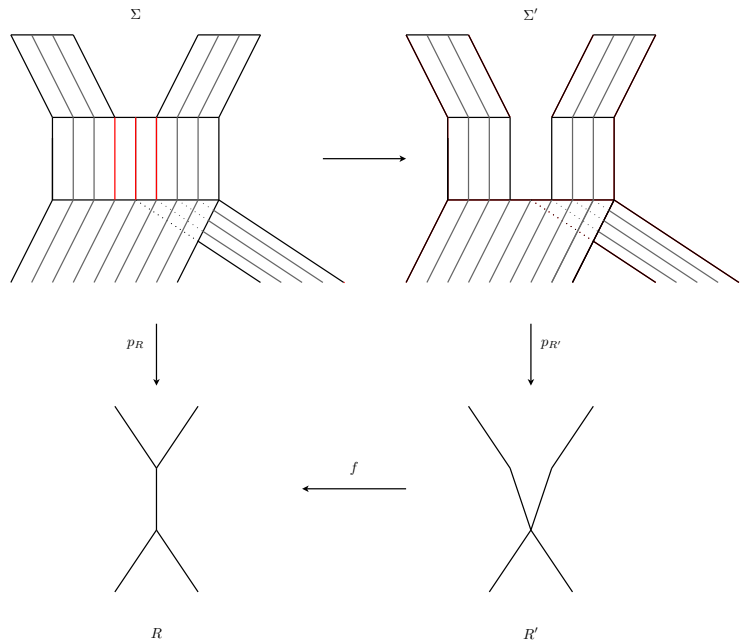}
\caption{Applying one step of the pruning process. All red leaf segments are terminal in their complete $\Sigma$-leaf.}
\label{fig-pruning}
\end{center}
\end{figure}

Let $T\in\overline{\calo}$ be a tree with dense orbits, let $R$ be a Grushko tree, and let $\Sigma:=\Sigma(T,R)$. Starting from $\Sigma$, we define a new band complex $\Sigma'$ as follows, as in \cite{CH14}	in the case of free groups. This construction is illustrated on Figure~\ref{fig-pruning}.
For each complete $\Sigma$-leaf $\call$, let $\Term(\call)$ be the union of all terminal edges $[x,y)_{\call}$ of $\call$,
where $x$ is its terminal vertex. Then $\call\setminus \Term(\call)$ is a subtree of $\call$, and it is unbounded since $\call$ is unbounded.
Let $\caly\subseteq \Sigma$ be the union  over all complete $\Sigma$-leaves $\call$ of $\Term(\call)$, and let $\Sigma'=\Sigma\setminus \caly$. 
Every point in $\caly$ is contained in exactly one band by Remark \ref{rk_one_band}.

Notice that $\caly$ is open in $\Sigma$: this is because if $x\in \calk_v$ (for some vertex $v\in R$) is contained in a single band of $\Sigma$, then there exists a neighbourhood $U$ of $x$ in $\calk_v$ such that every $x'\in U$ is contained in a single band (Corollary \ref{finitude-1}).

\begin{de}
 We say that $\Sigma'=\Sigma\setminus \caly$ is obtained by \emph{applying one step of the pruning process} to $\Sigma$. 
\end{de}

Note that $\caly$ contains no special point because such a point has infinite valence in its  complete $\Sigma$-leaf. 
It follows that for any $\Sigma$-leaf line $l\subseteq \call$,  we have $l\cap \Term(\call)=\es$ 
and $l\subseteq \Sigma'$ (regardless of whether its endpoints are in $V_\infty(G,\calf)$ or in $\partial_\infty (G,\calf)$).
\\
\\
\indent The set $\Sigma'$ is defined as a subspace of $\Sigma$. 
It has a natural structure of a band complex where each $\Sigma'$-band is a maximal subset of a band $K\times [0,1]$ of the form $K'\times [0,1]$ 
with $K'\subseteq K$ closed and connected.

 The goal of the present section is to show that the band complex $\Sigma'$ is of the form $\Sigma(T,R')$ for some Grushko tree $R'$. The motivation behind this alternative description is that it will allow us to apply to $\Sigma'$ all the results concerning band complexes. 

We let $R':=\Sigma'/{\sim}$, where $x\sim y$ whenever $p_R(x)=p_R(y)$, and $x$ and $y$ belong to the same connected component of $\Sigma'\cap p_R^{-1}(p_R(x))$, and we denote by $p_{R'}:\Sigma'\ra R'$ the quotient map.
 
The set $R'$ has a structure of a graph whose vertices are the connected components of $\Sigma'_V:=\Sigma'\cap\Sigma_V$, where $\Sigma_V$ denotes the subset of $\Sigma$ which is the union of all base trees $\calk_v$ with $v\in V(R)$. 
Its edges correspond to $\Sigma'$-bands.

The graph $R'$ also comes with a natural $G$-action. 
There exists a $G$-equivariant map $f:R'\to R$ sending vertex to vertex and edge to edge and making the following diagram commute: 
$$\xymatrix{
  \Sigma' \ar@{^(->}[r] \ar[d]^{p_{R'}} &\Sigma \ar[d]^{p_R} \\
R' \ar[r]^{f} & R
}$$

The goal of the present section is to prove the following proposition.

\begin{prop}\label{Rips-iteration}
The graph $R'$ is a Grushko tree.
 In addition, $\Sigma'$ can be identified with $\Sigma(T,R')$ in the following sense. 
\\ Denote by $q_{R'}:\Sigma(T,R')\to R'$ and $q_T:\Sigma(T,R')\to\overline{T}$ the two projections. 
There exists a homeomorphism $\Phi:\Sigma'\to\Sigma(T,R')$ that sends leaf to leaf and such that $p_{R'}=q_{R'}\circ\Phi$ and $(p_T)_{|\Sigma'}=q_T\circ\Phi$.
\end{prop}

We first record the following observation, which will be crucial in the sequel.

\begin{lemma}\label{prop-rips}
Every $\Sigma$-leaf line $l$ is contained in $\Sigma'$, $p_{R'}$ restricts to an embedding on $l$, 
and $p_{R'}(l)$ isometrically embeds into $R$.
\qed
\end{lemma} 

\begin{lemma}\label{lem_BY}
  Each connected component $Y$ of $\caly\cap \calk_v$ is contained in a single band $B_Y$.
\end{lemma}

\begin{proof}
Denote $\caly_v:=\caly\cap\calk_v$.
  Given a band $B$, the set of points $\caly_v\cap B$ is open in $\calk_v$: indeed,
for each $y\in \caly_v\cap B$, there is a neighbourhood $V_y$ of $y$ such that the only band meeting $V_y$ is $B$ (Corollary \ref{finitude-1}).
In particular, $V_y\subseteq \caly_v$, and since every point is contained in at least one band, $V_y\subseteq B$.

As $B$ varies, the disjoint open sets $\caly_v\cap B$ cover $Y$. Connectedness of $Y$ thus ensures that $Y\subseteq \caly_v\cap B$ for some band $B$.
\end{proof}

We now record some properties of the foliated complex $\Sigma'$. 
We say that a subset of an $\bbR$-tree is a \emph{finite subtree} if its closure is the convex hull of a finite collection of points. 

\begin{lemma}\label{properties}
For each vertex $v$  in $R$, the set $\calk_v \cap \caly$ is a disjoint union of finite open subtrees of $\calk_v$ and there are only finitely many
$G_v$-orbits of them.
\\ The $p_T$-images of the endpoints of these subtrees are in $T$, not in $\overline T\setminus T$.
\end{lemma}

\begin{proof}
We already mentionned that $\caly$ is open in $\Sigma$.
To prove that $\calk_v\cap \caly$ has finitely many connected components up to the action of $G_v$, we will define a $G_v$-equivariant map 
\begin{displaymath}
\begin{array}{cccc}
\left\{\parbox{4cm}{\centering connected components of $\calk_v\cap\caly$}\right\} & \to & \left\{\parbox{4cm}{\centering nonempty finite sets of bands incident on  $\calk_v$}\right\}
\end{array}
\end{displaymath}
\noindent sending $Y$ to a finite set of bands $\calb_Y$, and prove that it is injective and only takes finitely many values up to the $G_v$-action.

Let $v_0\in\calk_v\setminus\caly$, which we choose to be equal to the special vertex $v_\Sigma$ if $G_v$ is infinite. Given a connected component $Y$ of $\calk_v\cap\caly$, we define $\calb_Y$ as the set of bands  incident on $\calk_v$ that do not intersect $Y$, and do not intersect the connected component of $\calk_v\setminus Y$ that contains $v_0$ (in other words $Y$ separates $v_0$ from every band in $\calb_Y$). 
We claim that $\calb_Y$ is non-empty. Indeed, consider $z$ an extremal point of $\calk_v$ that is separated from $v_0$ by $Y$.
Since $z$ is extremal in $\calk_v$, it lies in $\Omega$ 
(see Remark \ref{rk_omega})
and therefore, there are at least two bands of $\Sigma$ containing $z$.
One of these bands misses $Y$, and thus belongs to $\calb_Y$.

Since no two bands in $\calb_Y$ are separated by $v_0$, Lemma \ref{lem_finitude} implies that $\calb_Y$ is finite. 

Lemma \ref{lem_finitude} also implies that given a band $B$, there are only finitely many possibilities for the sets $\calb_Y$ containing $B$. 
As there are finitely many orbits of bands, when $Y$ varies, the set $\calb_Y$ only takes finitely many values up to the action of $G_v$. 

We now check injectivity, i.e. if $Y\neq Y'$, then $\calb_Y\neq \calb_{Y'}$.
Up to exchanging the roles of $Y$ and $Y'$, we can assume that the segment $[v_0,y]$ joining $v_0$ to $\overline Y$ does not intersect $\ol Y'$.
Let $x$ be an endpoint of $\overline Y$ distinct from $y$.
Since $Y$ is open in $\calk_v$ and endpoints of $\calk_v$ are not in $\caly$, this implies that $x\notin \caly$,
 so $\calb_Y$ contains a band $B$ containing $x$. Since $[v_0,y]\cup[y,x]$ does not intersect $Y'$,  we have 
 $B\notin\calb_{Y'}$, and hence $\calb_Y\neq\calb_{Y'}$. This proves injectivity,  and concludes the proof of the fact that $\calk_v\cap\caly$ has only finitely many $G_v$-orbits of connected components.

The above construction also shows that each connected component $Y$ of $\calk_v\cap\caly$ is a finite subtree of $\calk_v$, because all endpoints of $\overline{Y}$ coincide with the projection of $v_0$ to either $\overline{Y}$, or to one of the finitely many bands in $\mathcal{B}_Y$. 
This concludes the proof of the first assertion of the lemma.

There remains to check that for every endpoint $x$ of $Y$, we have $p_T(x)\in T$. We will assume that $x$ is an endpoint of $\calk_v$, otherwise the conclusion is clear. 
Let $B_Y$ be the band containing $Y$ (Lemma \ref{lem_BY}).
Then for any band $B\neq B_Y$ containing $x$, we have $B\cap \calk_v=\{x\}$. This means that $B$ is a \emph{singleton}: $B\simeq \{*\}\times [0,1]$.
If $\call_x$ intersects infinitely many singletons, then since there are only finitely many orbits of them, this implies that the stabilizer of $p_T(x)$
is non-trivial, hence $p_T(x)\in T$. So we assume that $\call_x$ intersects only finitely many singletons.
Since $B$ is contained in $\Omega$, there is a $\Sigma$-leaf line $l$  containing $B$. 
Let $[y_1,y_2]_{\call_x}$ be a maximal segment in $l$, containing $B$ and made of singletons.
Denote by $v_1,v_2$ the vertices of $R$ such that $y_i\in \calk_{v_i}$. 
Then $\calk_{v_1}$ and $\calk_{v_2}$ are not reduced to a point,
and since fibers of $p_T$ are $\Sigma$-leaves (Proposition~\ref{fibers-connected}), we have
$p_T(\calk_{v_1})\cap p_T(\calk_{v_2})=\{p_T(x)\}$. It follows that $p_T(x)$ is not an endpoint of $p_T(\calk_{v_1})\cup p_T(\calk_{v_2})$. 
Since it is not an endpoint in $\overline T$, we have $p_T(x)\in T$.
\end{proof}

\begin{lemma}\label{lem_he}
The inclusion $\Sigma'\subseteq \Sigma$ is a homotopy equivalence. In particular $\Sigma'$ is simply connected.
\end{lemma}

\begin{proof} The second part of the lemma follows from the first by Lemma \ref{contractible}. 
Let $Y$ be a connected component of $\calk_v\cap\mathcal{Y}$. 
By Lemma \ref{properties}, the set $\overline{Y}$ is a closed finite tree. 
Let $B_Y$ be the unique band of $\Sigma$ containing $Y$ (Lemma~\ref{lem_BY}), and identify $Y\times [0,1]$ with the corresponding subset of $B_Y$, 
where $Y$ is identified with $Y\times\{0\}$. 
We have that $Y\times \{1\}$ is contained in $\Sigma'$. This is because
every complete $\Sigma$-leaf is infinite by Proposition \ref{prop_infini}, so for each terminal edge $[x,x')_{\call_x}\in \Term(\call_x)$,
we have $x'\notin\caly$.
Thus, there is a deformation retraction of $B_Y$ to  $B_Y\setminus (Y\times [0,1))$ with support in $Y\times [0,1)$.
Doing this for every $Y$, we get a deformation retraction from $\Sigma$ to $\Sigma'$, which proves the lemma.
\end{proof}

A \emph{boundary segment} of the band $B=K\times [0,1]$ is a leaf segment of the form $\{x\}\times [0,1]$ with $x$ a terminal point in $K$
(note that each band may have infinitely many boundary segments if $K$ is not a finite tree).


\begin{lemma}\label{lem_bords}
Every boundary segment of a band of $\Sigma'$ is contained in $\Omega$.
\end{lemma}

\begin{proof}
Consider a $\Sigma$-leaf segment $[x,y]_{\call_x}$ in a band $B$ of $\Sigma$, which becomes a boundary segment of the band $B'\subseteq B$ of $\Sigma'$.
Let $v\in R$ be such that $x\in \calk_v$.
We can assume that  $[x,y]_{\call_x}$ is not a boundary segment of $B$,  otherwise we are done by  Remark \ref{rk_omega}. 
Then, up to exchanging the roles of $x$ and $y$, there are leaf segments $[x_i,y_i]_{\call_{x_i}}\subseteq B$, with $x_i\in \calk_v\cap \caly$ converging to $x$. 
Consider $l_i$ any $\Sigma$-leaf semi-line starting from $x_i$ (it exists by Proposition \ref{prop_infini}).
Its initial segment  $[x_i,y_i]$ is contained in $B$ since $x_i$ is a terminal point in its leaf. 
By Corollary \ref{cor_limit_leaves} (and Remark \ref{rk_limit_leaves} if $x$ is the special vertex of $\calk_v$), 
there exists a $\Sigma$-leaf semi-line $l$ based at $x$ with initial segment $[x,y]_{\call_x}$.

Since $x\notin\caly$, there exists another band $B_2\neq B$ in $\Sigma$ containing $x$. 
 If $x$ is either a special vertex or an endpoint of  $B_2\cap \calk_v$, then 
$x\in \Omega$, so there is a $\Sigma$-leaf line $l'$ containing $x$. Using the semi-line $l$ constructed above, it easily follows that $[x,y]_{\call_x}\subseteq \Omega$.

 We now assume that $x$ is not a special vertex, and that there are at least two directions based at $x$ in $B_2\cap \calk_v$. 
Since  $[x,y]_{\call_x}$ is a boundary segment of $B'$, then one of  these two directions intersects $B'$ trivially.
Thus, there exists a sequence of points $x'_i\in (B_2\setminus B')\cap \calk_v$ converging to $x$. 
Let $l'_i$ be a $\Sigma$-leaf semi-line starting from $x'_i$. We claim that its initial segment $[x'_i,z_i]_{l'_i}$ is not contained in $B$.
Indeed, assuming otherwise, then $[x'_i,z_i]_{l'_i}$ is contained in $B\setminus B'$, so either $x'_i$ or $z_i$ is a terminal point of its leaf.
But $x'_i$ is not a terminal point of its leaf ($x'_i\in B_2\cap B$), and neither is $z_i$ (it is contained in the semi-line $l'_i$), a contradiction.

Since no $l'_i$ starts with an edge in the band $B$, Corollary \ref{cor_limit_leaves} implies that there exists a $\Sigma$-leaf semi-line starting from $x$, which does not start with an edge in $B$. 
Then $l\cup l'$ is a $\Sigma$-leaf line containing $[x,y]_{\call_x}$, showing that $[x,y]_{\call_x}\subseteq \Omega$.
\end{proof}

The following lemma proves the first assertion of  Proposition \ref{Rips-iteration}.

\begin{lemma}\label{is-Grushko}
The graph $R'$ is a Grushko tree. 
\end{lemma}

\begin{proof}
Since $\Sigma'$ is connected, so is $R'$; we will first check that $R'$ is simply connected. Let $\gamma$ be a cycle in $R'$. Then $\gamma$ can be lift to a loop $\gamma'$ in $\Sigma'$ which is a concatenation of leaf segments and of geodesic segments contained in the trees $\calk_v$. Since $\Sigma'$ is simply connected (Lemma \ref{lem_he}), the loop $\gamma'$ can be filled by a disk. Projecting this disk to $R'$ shows that $\gamma$ can be filled by a disk. 

Recall the existence of a $G$-equivariant map $f:R'\to R$ making the following diagram commute:
$$\xymatrix{
  \Sigma' \ar@{^(->}[r] \ar[d]^{p_{R'}} &\Sigma \ar[d]^{p_R} \\
R' \ar[r]^{f} & R
}$$ 
Since each peripheral subgroup fixes a point in $\Sigma'$, it also does in $R'$.
Since $f$ maps edge to edge and $R$ has trivial edge stabilizers, so does $R'$.
Since point stabilizers in $R$ are peripheral, so are point stabilizers in $R'$.
Thus, there remains to prove that the action of $G$ on $R'$ is minimal.

By Lemma \ref{properties}, the $f$-preimage of each edge $e$ of $R$ is a finite set of edges of $R'$ (corresponding to the connected components of the intersection of $\Sigma'$ with the band $B_e$ of $\Sigma$ corresponding to $e$). Therefore $G$ acts cocompactly on $R'$.
Since every boundary segment of a band of $\Sigma'$ lies in $\Omega$ by Lemma \ref{lem_bords}, and since by definition $\Omega$ is a union of $\Sigma$-leaf lines, no edge of $R'$ is terminal.
This implies that $R'$ is minimal and concludes the proof. 
\end{proof}

\begin{proof}[Proof of Proposition \ref{Rips-iteration}]
The fact that $R'$ is a Grushko tree has already been proved in Lemma \ref{is-Grushko}.

Define $\Phi:\Sigma'\ra \ol T\times R'$ by sending $x\in\Sigma'$ to $(p_T(x),p_{R'}(x))$.
This continuous equivariant map is injective because the fibers of $p_{R'}$ embed into $\ol T$. We aim to show that $\Phi(\Sigma')=\Sigma(T,R')$: it will then be clear that $\Phi$ sends any leaf of $\Sigma'$ to a leaf of $\Sigma(T,R')$, and that
$p_{R'}=q_{R'}\circ\Phi$ and $p_{T}=q_{T}\circ\Phi$.

Let us first check that $\Phi(\Sigma')\subseteq\Sigma(T,R')$.
We first claim that $\Phi(\Omega)\subseteq\Sigma(T,R')$. Indeed, let $x\in\Omega$, 
and $l$ a $\Sigma$-leaf line containing $x$, with endpoints $\alpha,\omega$. Since $(p_{R'})_{|l}$ is an isometric embedding (Lemma \ref{prop-rips}), we have $(\alpha,\omega)\in L^2_{v'}(T)$, where $v':=p_{R'}(x)$. This means by definition that $(p_T(x),v')\in \Sigma(T,R')$, and proves our claim. 
 
In general (i.e. if $x$ is no longer assumed to belong to $\Omega$), there exist $x_1,x_2\in \Omega$ such that $p_{R'}(x_1)=p_{R'}(x_2)=p_{R'}(x)$, and $p_{T}(x)\in [p_{T}(x_1),p_{T}(x_2)]$: indeed, the preimage  $p_{R'}\m(\{v'\})$ is the convex hull of points lying in boundary segments of  $\Sigma'$-bands, and such points belong to $\Omega$ by Lemma \ref{lem_bords}. Since fibers of $q_{R'}:\Sigma(T,R')\ra R'$ are convex, this implies that $\Phi(x)\in \Sigma(T,R')$.

There remains to prove that $\Sigma(T,R')\subseteq \Phi(\Sigma')$.
In view of Proposition \ref{prop_coeur}, it suffices to check that $\Phi(\Sigma')$ is a closed connected subset of $\ol T\times R'$
with connected fibers. 
The map $q_{R'}:\Phi(\Sigma')\ra R'$ has closed fibers because $\caly$ is open, and $q_{R'}\m(e)\subseteq q_{R'}\m(v)$ for every edge $e$ incident on $v$.
It easily follows that $\Phi(\Sigma')$ is closed.
Since $\Phi(\Sigma')$ is connected, we are left checking that it has connected fibers. 

For $y\in \ol T$, the set $p_T\m(\{y\})\cap\Sigma'$ is connected (it is a leaf with its terminal edges removed), hence coincides (if non-empty) with a complete $\Sigma'$-leaf
which we denote by $\call'_y$. Then $\Phi(\Sigma')\cap (\{y\}\times R')=\Phi(\call'_y)$ is connected. 
And any point $u'\in R'$ corresponds by definition to a connected component $\calk_{u'}$ of some $\calk_u\cap \Sigma'$, 
so $\Phi(\Sigma')\cap (\ol T\times \{u'\})=\Phi(\calk_{u'})$ is connected.

This shows that $\Phi$ is a continuous, injective map onto $\Sigma(T,R')$. That it is a homeomorphism is proved by checking that the map $\Psi:\Sigma(T,R')\to\Sigma'$ defined by letting $\Psi((x,v))=(x,f(v))$ (where $f:R'\to R$ is the natural map) is an inverse of $\Phi$.
\end{proof}

\subsection{The pruning process, and analysis of $\Omega$}\label{sec-iteration}

\paragraph{The pruning process.}
Starting from an $\bbR$-tree with dense orbits $T$ and a foliated band complex $\Sigma(T,R)=\Sigma^{(0)}$, we define inductively for all $i\in\mathbb{N}$ a foliated band complex  $\Sigma^{(i)}:=(\Sigma^{(i-1)})'\subseteq \Sigma^{(i-1)}$ 
by applying one step of the pruning process to $\Sigma^{(i-1)}$. Then we define $R^{(i)}$ as the Grushko tree associated to $\Sigma^{(i)}$, so that
$\Sigma^{(i)}=\Sigma(T,R^{(i)})$ by Proposition \ref{Rips-iteration}.
It may happen that no leaf of $\Sigma^{(i)}$ has a terminal edge in which case $\Sigma^{(j)}=\Sigma^{(i)}$ for all $j\geq i$.
In the rest of this section, we also consider this case  (many statements from Sections~\ref{sec-iteration} to~\ref{sec-omega-int} are actually obvious in this case; the analysis of this important case will be made in Section~\ref{sec-surface-type}).
 
Note that an edge $e$ in a $\Sigma$-leaf $\call$ lies in  $\Sigma\setminus\Sigma^{(i)}$ if and only if 
one of the two connected components of $\call\setminus \rond e$ has depth at most $i-1$ as a tree rooted at the endpoint of $e$.
Similarly, a vertex $x$ of $\call$ lies in  $\Sigma\setminus\Sigma^{(i)}$ if and only if there is some edge $e$ incident on $x$
such that the connected component of $\call\setminus \rond e$ containing $x$ has depth at most $i-1$.

\paragraph{The set $\Omega$ is what is left when everything else is forgotten.}

We now relate the pruning process with the subset $\Omega$ of $\Sigma$ introduced in Definition~\ref{def_omega}.

\begin{lemma}\label{omega-rips}
We have $\Omega=\cap_{i\in\mathbb{N}} \Sigma^{(i)}$. Moreover, for any compact subset $K\subseteq \Sigma\setminus\Omega$, 
there exists $i\in\mathbb{N}$ such that $K\subseteq \Sigma\setminus\Sigma^{(i)}$.
\end{lemma}

\begin{proof}
Since $\Sigma^{(i+1)}$ is obtained from $\Sigma^{(i)}$ by removing an open set, the second assertion follows from the first.

Lemma \ref{prop-rips} implies that $\Omega\subseteq\cap_{i\in\mathbb{N}}\Sigma^{(i)}$; we now prove the converse inclusion. 
Consider $x\in \Sigma\setminus \Omega$. Then $x$ is not a special vertex, so $\call_x\setminus\{x\}$ has only finitely many connected components (Corollary \ref{cor_loc_fini}).
Since there is no $\Sigma$-leaf line through $x$,
there is exactly one connected component $c_0$ of $\call_x\setminus\{x\}$ which is not a finite tree. 
Thus $\call_x\setminus c_0$ has finite depth, so $x\in \Sigma\setminus \Sigma^{(i)}$ for some $i$.
\end{proof}

\paragraph{Connectedness of the fibers of $\Omega$.}

\begin{lemma}\label{omega-cf}
The restriction of $p_T$ to $\Omega$ has connected fibers.
\end{lemma}

\begin{proof}
Let $x,y\in \Omega$ be such that $p_T(x)=p_T(y)$. Let $\call$ be the complete $\Sigma$-leaf containing $x$ and $y$.
By definition of $\Omega$, there are $\Sigma$-leaf lines $l_x,l_y\subseteq \call$ containing
$x$ and $y$ respectively. 
The segment $[x,y]_\call$ is then clearly contained in a $\Sigma$-leaf line obtained by concatenating it with half lines of $l_x$ and $l_y$.
This shows that $[x,y]_\call\subseteq \Omega$ and concludes the proof.
\end{proof}

\paragraph{Finiteness properties of the set $\Omega$.}

\newcommand{\pI}{p^{\cali}}
\newcommand{\bpI}{\ol{p}^{\cali}}

We will now use the pruning induction to show finiteness properties of $\Omega$. The set $\Omega$ has a natural structure of a (usually disconnected) band complex where each $\Omega$-band is 
a subset of $B_e=K_e\times e$ of the form $K'\times e$ where $K'$ is a connected component of $\Omega_e$.

\begin{de}[Incidence graph]
 We define the \emph{incidence graph} $\mathcal{I}$ of $\Omega$ as $\mathcal{I}:=\Omega/{\sim}$, where $x\sim y$ if and only if $p_R(x)=p_R(y)$ and $x,y$ belong to the same connected component of $\Omega\cap(p_R^{-1}(p_R(x)))$. 
\end{de}

This has a natural structure of a graph whose vertices are the connected components of $\Omega_V:=\Omega\cap \Sigma_V$, edges corresponding to $\Omega$-bands.

Note that the definition of $\cali$ is similar to the definition of $R^{(i)}$ from $\Sigma^{(i)}$. Since $\bigcap_i \Sigma^{(i)}=\Omega$,
the idea is that $\cali$ will be a limit of the Grushko trees $R^{(i)}$: this is made precise in Lemma~\ref{lem_separation} below. This will allow to deduce finiteness properties of $\Omega$, given in Corollaries~\ref{omega-forest} and~\ref{omega-index}.

In general $\cali$ may be disconnected, and may have uncountably
many vertices and edges.
Since points in $\Omega$ are contained in a $\Sigma$-leaf line, 
all vertices of $\mathcal{I}$ have valence at least $2$. For all $i\in\mathbb{N}$, there is a natural $G$-equivariant map
$\pI_{R^{(i)}}:\cali\to R^{(i)}$, sending vertices to vertices and edges to edges: if $C$ is a connected component of $\Omega_V$,
then the map $p_{R^{(i)}}:\Sigma^{(i)}\ra R^{(i)}$ sends $C$ to a vertex which we define as $\pI_{R^{(i)}}(C)$, and it is defined on edges in a similar fashion.
Notice that the maps $\pI_{R^{(i)}}$ fail to be injective in general. However, we have the following result.

\begin{lemma}\label{lem_separation}
Given any 
finite subgraph $F\subseteq\cali$, there exists $i\in\mathbb{N}$ such that $\pI_{R^{(i)}}$ is injective in restriction to $F$.
\\
Given any finite subgraph $\overline{F}\subseteq\cali/G$, there exists $i\in\mathbb{N}$ such that the quotient map $\bpI_{R^{(i)}}:\cali/G\to R^{(i)}/G$ is injective in restriction to $\overline{F}$. 
\end{lemma}

\begin{proof}
The first part of the lemma is a consequence of Lemma \ref{omega-rips}: if $C,C'$ are two distinct connected components of 
$\Omega\cap \calk_u$ for some $u\in R$, then the segment in $\calk_u$ joining $C$ to $C'$ contains a point in $\Sigma\setminus \Omega$. Therefore this segment contains a point in $\Sigma\setminus \Sigma^{(i)}$ for $i$ large enough, showing that $\pI_{R^{(i)}}(C)\neq \pI_{R^{(i)}}(C')$.

Assume that some connected component $C$ of $\Omega\cap \calk_u$ is not in the same orbit as
a component $C'$ of $\Omega\cap \calk_{u'}$.
If $u,u'$ are not in the same orbit of $R$, then the images of $C$ and $C'$ under $\bpI_{R^{(i)}}$ are not in the same orbit.
So, up to changing $C'$ to a translate, we can assume that $C$ and $C'$ are two connected components of $\Omega\cap \calk_u$.
Since $(G.C')\cap \calk_u=G_u.C'$,
it suffices to prove that there is a finite set $F\subseteq \calk_u\setminus\Omega_u$ such that for all $g\in G_u$, the sets $C$ and $gC'$ are contained in distinct components of $\Omega_u\setminus G_u.F$.
To prove this fact, 
denote by $u_\Sigma$ the special vertex in $\calk_u$. If $u_\Sigma\in C$, then one can take for $F$ any point in the segment joining $C$ to $C'$ that does not belong to $\Omega$. 
The case where $u_\Sigma\in C'$ is symmetric, so assume $u_\Sigma\notin C\cup C'$.
Then there is at most one $g\in G_u$ such that the segment joining $C$ to $g C'$ avoids $u_\Sigma$.
One can then take $F=\{a,b\}$ where $a$ (resp. $b$) is a point outside $\Omega$ on the segment joining $C$ to $g C'$
(resp. $C$ to $u_\Sigma$).
\end{proof}

Since $R^{(i)}$ is a tree for all $i\in\mathbb{N}$ (and therefore $R^{(i)}$ does not contain any cycle), we deduce as a consequence of Lemma \ref{lem_separation} the following fact.

\begin{cor}\label{omega-forest}
The graph $\cali$ is a forest.
\qed
\end{cor}

Since there is a bound on the number of $G$-orbits of branch points and on directions at these points in any Grushko tree, we also deduce from Lemma \ref{lem_separation} the following fact (this is \cite[Corollary 3.7]{CH14} in the context of free groups).

\begin{cor}\label{omega-index}
There is a bound, only depending on $\text{rk}_K(G,\calf)$, on the number of $G$-orbits of branch vertices in $\cali$ and of $G$-orbits of directions at these vertices.
\qed 
\end{cor}

\subsection{Finiteness properties for $\Omega_{int}$}\label{sec-omega-int}

We now want to control the number of orbits of nondegenerate segments contained in $\Omega_V$.
Let $\Omega_{int}\subseteq\Omega$ be the subset made of all points $x\in\Omega$ for which there exists a transverse nondegenerate interval 
$I_x\subseteq \Omega$ that contains $x$, with $p_R(I_x)=\{p_R(x)\}$. 
It has a structure of a band complex where each $\Omega_{int}$-band is an $\Omega$-band $K\times e$ with $K$ not reduced to a point.

In this section, we prove that $\Omega_{int}$ has finitely many orbits of connected components, and that there are at most
finitely many orbits of leaves with at least 3 ends that are not completely contained in such a connected component (Lemma \ref{thin}).

\begin{lemma}\label{lambda-cf}
Let $\Lambda$ be a connected component of $\Omega_{int}$. Then the restriction of $p_T$ to $\Lambda$ has connected fibers.
\end{lemma}

\begin{proof}
Let $x,y\in \Lambda$ be such that $p_T(x)=p_T(y)$. Let $\gamma:[0,L]\ra \Lambda$ be a path joining $x$ to $y$.
We can assume that $\gamma$ is a finite concatenation of leaf segments and of unit speed geodesic segments in $\Omega_V$.
Subdividing $\gamma$, we can assume without loss of generality that $p_T(\gamma(t))\neq p_T(x)$ for all $t\in (0,1)$.
Then there exists $\eps>0$ such that for all $t\leq \eps$, we have $p_T\circ \gamma (t)= p_T\circ \gamma(L-t)$.
Since the restriction of $p_T$ to $\Omega$ has connected fibers (Lemma \ref{omega-cf}), 
this implies that $\gamma(t)$ and $\gamma(L-t)$ are in the same $\Sigma$-leaf, 
and the leaf segment $l_t$ joining them is also in $\Omega$. This implies that for all $t\leq \eps$, we have $l_t\subseteq \Omega_{int}$ thus $l_t\subseteq \Lambda$, so $l_0$ is a path in $p_T\m(\{p_T(x)\})\cap \Lambda$
joining $x$ to $y$. This shows that fibers of $p_T$ are connected. 
\end{proof}

The following observation will turn out to be useful in the upcoming analysis.

\begin{lemma}\label{baire}
Given any nondegenerate interval $I\subseteq\Omega_{int}\cap\Omega_V$, there exists a non-degenerate subsegment $J\subseteq I$ contained in two distinct $\Omega_{int}$-bands.
\end{lemma}

\begin{proof}
Up to restricting to a smaller interval, we can assume that the closure of $I$ does not contain any special point
and that there are only finitely many bands $B_1,\dots, B_k$ of $\Sigma$ that meet $I$ (see Lemma \ref{lem_finitude_0}). 
For all $i\in\{1,\dots,k\}$, we let 
 $F_i$ be the intersection of $I$ with the $\Omega$-bands contained in $B_i$, and for $i\neq j$ we let $F_{i,j}:=F_i\cap F_j$. 
Then the sets $F_{i,j}$ are closed, and since $I\subseteq\Omega$, for every point $x\in I$, there exist two bands $B_i,B_j$ incident on $I$ such that $\call_x\cap\Omega$ has leaf segments contained in $B_i$ and $B_j$, so $x\in F_{i,j}$. This implies that the $F_{i,j}$ cover $I$. Therefore one of the sets $F_{i,j}$ has nonempty interior. Any non-degenerate interval $J\subseteq F_{i,j}$ satisfies the conclusion of the lemma.

\end{proof}

Let $\cali_{int}\subseteq\cali$ be the subgraph whose vertices are the connected components of $\Omega_{int}\cap\Omega_V$, and whose edges correspond to $\Omega_{int}$-bands.

\begin{lemma}\label{omega-int-fini}
The graph $\cali_{int}$ has finitely many orbits of connected components. 
\\ Moreover, each connected component $I$ of $\cali_{int}$ is a tree, and the action on $I$ of its stabilizer $G_I$
is minimal and is a Grushko $(G_I,\calf_{|G_I})$-tree.
Any peripheral group intersecting $G_I$ non-trivially is contained in $G_I$.
\end{lemma}

\begin{proof}
Since $\cali_{int}$ is a subgraph of the forest $\cali$ (Corollary \ref{omega-forest}), each connected component $I$ of $\cali_{int}$ is a tree.

It follows from Lemma \ref{baire} that all vertices in $\cali_{int}$ have valence at least $2$. 
We claim that every connected component of $\cali_{int}$ contains a branch point of $\cali_{int}$, and is the convex hull of its branch points.
Otherwise, since $\cali_{int}$ does not contain any valence $1$ vertex, there is a semi-line $L$ in $\cali_{int}$ containing no branch point of $\cali_{int}$. Let $C$ be a connected component of $\Omega_{int}\cap\Omega_V$ corresponding to an interior vertex $v$ in $L$,
and let $B_1,B_2\subseteq \Omega_{int}$ be the two $\Omega_{int}$-bands incident on $C$. 
Since there is no other $\Omega_{int}$-band incident on $C$, it follows from Lemma \ref{baire} that  $C=B_1\cap B_2$. 
This is true for every component corresponding to a vertex of the semi-line $L$, 
so for all $x\in C$, the complete $\Sigma$-leaf $\call_x$ contains a $\Sigma$-leaf semi-line having the same projection in $R$ as $L$.
As $C$ is not reduced to a point (by definition of $\Omega_{int}$), we obtain a contradiction to Lemma \ref{lem_pT}. This proves our claim.

By Corollary \ref{omega-index}, there are only finitely many orbits of branch points in $\cali_{int}$.
This implies that $\cali_{int}$ has only finitely many orbits of connected components, finitely many orbits of vertices, and also finitely many orbits of edges because there are finitely many orbits of directions at branch points in $\cali_{int}$. 
Thus, for each connected component $I$ of $\cali_{int}$, its stabilizer $G_I$ acts cocompactly on $I$.
Since $I$ has no terminal point, this implies that the action of $G_I$ on $I$ is minimal.

The map $p_R^{\cali}:\cali\ra R$ sends edge to edge and vertex to vertex.
It follows that edge stabilizers of $I$ are trivial, and that vertex stabilizers are peripheral subgroups. 

We now claim that if $G_v$ is a peripheral subgroup of $G$ that intersects $G_I$ nontrivially, then the unique vertex $v_\Sigma\in \Sigma$ fixed by $G_v$ lies in the connected component $\Lambda$ of $\Omega_{int}$ correspoding to $I$. This will imply in particular that all peripheral subgroups of $G_I$ are elliptic in $I$, so $G_I$ is a Grushko $(G_I,\calf_{|G_I})$-tree. To prove the claim, let $h$ be a nontrivial element in $G_I\cap G_v$. Since $p_R^{\cali}(I)$ is a connected $h$-invariant subset of $R$, it contains its fix point $v$, so $\Lambda\cap \calk_v\neq \es$.  
Let $x\in \calk_v\cap \Lambda$. Then $x,hx\in\Lambda$, and $v_\Sigma$ belongs to the unique embedded segment of $\calk_v$ that joins $x$ and $hx$. Since $p_T(\Lambda)$ is connected, there exists $y\in\Lambda$ with $p_T(y)=p_T(v_\Sigma)$.
Then $p_T(hy)=p_T(v_\Sigma)$, and since fibers of $(p_T)_{|\Lambda}$ are connected (Lemma \ref{lambda-cf}), the leaf segment joining $y$ to $hy$ is contained in $\Lambda$. 
Since $v_\Sigma$ is contained in this leaf segment, this proves our claim, and finishes the proof of the lemma. 
\end{proof}

\begin{lemma}\label{no-term}
Every point $x\in\Omega_{int}\cap\Omega_V$ lies in at least two distinct $\Omega_{int}$-bands.
\end{lemma}

\begin{proof}
Let $x\in\Omega_{int}\cap\Omega_V$, and let $C$ be the component of $\Omega_{int}\cap\Omega_V$ that contains $x$. 
We can assume that $x$ is not a special vertex since the result is clear in this case.
Assume towards a contradiction that $x$ lies in at most one $\Omega_{int}$-band.
Denote by $\calk_v$ the base tree of $\Sigma$ containing $C$.
By Lemma \ref{lem_finitude_0} there exists a neighbourhood $V_x$ of $x$ in $\calk_v$ that intersects only finitely many $\Sigma$-bands.
Since two distinct $\Omega_{int}$-bands contained in a common $\Sigma$-band cannot be in the same orbit,
Lemma \ref{omega-int-fini} shows that there are only finitely many $\Omega_{int}$-bands that intersect $V_x$.
Denote these $\Omega_{int}$-bands by $B_0,\dots,B_n$ with $x\notin B_1\cup\dots\cup B_n$.
Then there is a neighbourhood of $x$ that meets at most one $\Omega_{int}$-band, contradicting Lemma \ref{baire}.
\end{proof}

Given a complete $\Sigma$-leaf $\call$, we let $\call_\Omega:=\call\cap\Omega$. From Lemma \ref{no-term}, we will deduce the following fact.

\begin{lemma}\label{thin}
There are only finitely many orbits of complete $\Sigma$-leaves $\call$ with at least $3$ ends such that $\call_\Omega$ is not contained within a single connected component of $\Omega_{int}$.
\\ Moreover, if $\Lambda,\Lambda'$ are two distinct components of $\Omega_{int}$, then there is at most one complete $\Sigma$-leaf that intersects both $\Lambda$ and $\Lambda'$.
\end{lemma}

\begin{proof}
Recall that an edge $e$ of the incidence graph $\cali$ of $\Omega$ corresponds to an $\Omega$-band.
We say that $e$ is \emph{thin} if this $\Omega$-band is reduced to a leaf segment, and \emph{thick} otherwise.

Let $\call$ be a complete $\Sigma$-leaf as in the statement of the proposition.
By definition of $\Omega$, $\call_\Omega$ also has at least 3 ends.
Let $\call_\cali$ be the image of $\call_\Omega$ in $\cali$ under the natural map. Since this map is locally injective, hence injective, $\call_\cali$ has at least $3$ ends, and therefore contains a branch point $b$ of $\cali$. Since $\call_\Omega$ is not contained within a single connected component of $\Omega_{int}$, its image $\call_\cali$ contains a thin edge $e$.

We claim that $\call_\cali$ also contains a thin edge $e'$ adjacent to a branch point of $\cali$. The proposition follows from this claim because there are only finitely many orbits of such edges $e'$, and $e'$ being thin, it corresponds to a unique complete $\Sigma$-leaf. 

To prove the claim, consider the path in $\call_\cali$ joining $e$ to $b$. Then the last thin edge on this path is adjacent to a branch point: this follows from the fact that if some thick edge is incident on some vertex $v\in\cali$, then there must be another thick edge incident of $v$ in view of Lemma \ref{no-term}.  

We now prove the last assertion from the lemma. Assume on the contrary that $p_T(\Lambda)\cap p_T(\Lambda')$ contains more than one point. Then  $p_T(\Lambda)\cap p_T(\Lambda')$ contains a nondegenerate arc, which is in particular infinite modulo $G$. Therefore there are leaves in infinitely many distinct orbits that meet both $\Lambda$ and $\Lambda'$. Since any such leaf has at least $3$ ends, this contradicts the first assertion.
\end{proof}

\paragraph{Leaves within a connected component of $\Omega_{int}$.}

\begin{lemma}
Let $\Lambda$ be a connected component of $\Omega_{int}$.
Let $T_\Lambda\subseteq T$ the $G_\Lambda$-minimal subtree, and let $\ol T_\Lambda$ be the closure of $T_\Lambda$ in $\ol T$ (which is a completion of $T_\Lambda$).
\\ Then $T_\Lambda\subseteq p_T(\Lambda)\subseteq \ol{T}_\Lambda$ and every $G_\Lambda$-orbit is dense in $T_\Lambda$.
\end{lemma}

\begin{proof}
  Since $\Lambda$ is connected, $p_T(\Lambda)$ is a $G_\Lambda$-invariant subtree of $\ol T$, so $T_\Lambda\subseteq p_T(\Lambda)$.
  
If $p_T(\Lambda)$ is not contained in $\ol{T}_\Lambda$, then there is an arc $I\subseteq p_T(\Lambda)\setminus \ol{T}_{\Lambda}$.
Since branch points of $T$ are dense in every segment, and since there are only finitely many $G$-orbits of directions at branch points,
there exist small disjoint intervals $J_1,J_2\subseteq I$ and $g\in G$ such that $g.J_1=J_2$, preserving the orientation induced by $I$. In particular, $g$ is hyperbolic in $T$, and its axis intersects $I$.
It follows that $p_T(g\Lambda)\cap p_T(\Lambda)$ contains the arc $J_2$, so there is more than one leaf that meets both $\Lambda$ and $g\Lambda$, so $g\Lambda=\Lambda$ by Lemma~\ref{thin}. So $g\in G_\Lambda$; this contradicts the fact that $I$ lies outside $\ol{T}_\Lambda$.

We now prove that every $G_\Lambda$-orbit is dense in $T_\Lambda$. If not, then $T_\Lambda$ contains an arc $I$ that does not contain any branch point of $T_\Lambda$. As above, branch points of $T$ are dense in $I$, and the same argument as above provides a hyperbolic element $g\in G_\Lambda$ such that $gI\cap I$ is nondegenerate, a contradiction.    
\end{proof}

\begin{lemma}\label{coeur_elagage}
Let $\Lambda$ be a connected component of $\Omega_{int}$. Let $R_\Lambda$ be the connected component of $\cali_{int}$ corresponding to $\Lambda$,
and $T_\Lambda$ the $G_\Lambda$-minimal subtree.
\\ Then $\Lambda=\Sigma(T_\Lambda,R_\Lambda)$. 
\end{lemma}

\begin{proof} 
Consider the product map $\Phi=(p_T,\pi):\Lambda\ra \ol T_\Lambda\times R_\Lambda$ (where $\pi:\Lambda\to R_\Lambda$ is the natural projection), and let $\Lambda':=\Phi(\Lambda)$. The map $\Phi$ is continuous and injective. 


The fibers of $\pi$ are connected by definition, and the fibers of $(p_{T})_{|\Lambda}$ are connected by Lemma~\ref{lambda-cf}. Fibers of $\pi$ being closed, one easily deduces that $\Lambda'$ is closed. Since $\Lambda'$ is connected, it follows from \cite{Gui_coeur} that $\Lambda'$ contains the core $\calc(\ol T_\Lambda\times R_\Lambda)$. 

Conversely, let $(x,u)\in\Lambda'$, and let $z\in \Lambda$ be a preimage of $(x,u)$. Let $\call_z$ be the complete $\Sigma$-leaf through $z$. Then $\call_z\cap \Lambda$ contains a bi-infinite line through $z$ (Lemma~\ref{no-term}), and this line projects isometrically to $R_\Lambda$ (in particular its endpoints belong to $\partial(G_\Lambda,\calf_{|{G_\Lambda}})$). This shows that $x\in L^2_u(R_\Lambda)$, hence $(x,u)\in\Sigma(T_\Lambda, R_\Lambda)$.

We have thus proved that $\calc(\ol T_\Lambda\times R_\Lambda)\subseteq\Lambda'\subseteq\Sigma(T_\Lambda,R_\Lambda)$. This concludes the proof by Proposition~\ref{prop_coeur} (which uses the fact that $T_\Lambda$ has dense orbits).
\end{proof}

\subsection{Band complexes of quadratic type and analysis of $\calq$-preimages}\label{sec-surface-type}

In this subsection, we study connected components of $\Omega_{int}$.
We focus on such a component $\Lambda$.
We know that leaves of $\Lambda$ have no terminal point (Lemma~\ref{no-term}) and Proposition~\ref{surface_3ends} will show that $\Lambda$ is of quadratic type in the sense below.
In particular, such a component has only finitely many orbits of leaves with at least 3 ends.

By Proposition \ref{coeur_elagage}, $\Lambda=\Sigma(T_\Lambda,R_\Lambda)$ is
the band complex corresponding to the $\bbR$-tree $G_\Lambda\actson T_\Lambda$ with respect to the Grushko tree $R_\Lambda$.
We thus use generic notations and study any band complex $\Sigma(T,R)$ whose leaves have no terminal point.
Note that equivalently, this assumption means that the pruning process does not affect $\Sigma(T,R)$.

\begin{prop}\label{surface_3ends}
  Let $T\in\overline{\calo}$ be a tree, let $R$ be a Grushko tree, and let $\Sigma:=\Sigma(T,R)$. 
  Then the following statements are equivalent. 
\begin{itemize}
\item The pruning process does not affect $\Sigma$. 
\item Leaves of the band complex $\Sigma$ have no terminal point.
\item All but finitely many orbits of complete $\Sigma$-leaves are bi-infinite lines.
\end{itemize}
In this case, no non-degenerate interval in a base tree $\calk_v$ is contained in 3 bands.
\end{prop}

\begin{de}[\textbf{\emph{Band complex of quadratic type}}]\label{dfn_surface}
Given a tree $T\in\overline{\calo}$ and a Grushko tree $R$, we say that the band complex $\Sigma=\Sigma(T,R)$ is \emph{of quadratic type} if it satisfies the equivalent assertions from Proposition~\ref{surface_3ends}. 
\end{de}

\begin{rk} The word \emph{quadratic} is reminiscent of Makanin's algorithm where quadratic generalized equations are those where each segment is contained in two bands. 
In \cite{CH14}, the authors call these band complexes \emph{pseudo-surface}, and \emph{of surface type} in \cite{CHR11}.
We prefer the word quadratic to avoid any confusion, in particular with measured foliations on surfaces. \end{rk}

A tree $T\in\overline{\calo}$ with dense orbits is \emph{of quadratic type} if there exists a Grushko tree $R$ such that $\Sigma(T,R)$ is of quadratic type. We also say in this case that $T$ is of quadratic type with respect to $R$. 

In particular, if for some Grushko tree $R'$, the pruning process applied to the band complex $\Sigma(T,R')$ eventually halts,
then $T$ is of quadratic type (with respect to some $R''$).

\paragraph*{Proof of Proposition \ref{surface_3ends}.}
 The equivalence between the first two assertions is clear from the definition of the pruning process. The third assertion implies the second: indeed, if some leaf of $\Sigma$ has a terminal point, then there is a whole interval of points that are terminal in their leaf, hence uncountably many leaves are not bi-infinite lines. Thus, in order to prove the equivalence between the three assertions, we only need to show that the second assertion of the proposition implies that
there are only finitely many orbits of leaves with at least 3 ends (Corollary \ref{omega-3-fini}).

The proof is similar to \cite[Section~4]{CH14} for free groups; it goes as follows.
The main point will be to prove that there is no non-degenerate interval $I$ of $\Sigma$ contained in 3 bands.
To prove this, we are going to approximate $\Sigma$ by a band complex $\Sigma^\epsilon\subseteq \Sigma$ whose bases are finite trees (defined in the next paragraph). In the terminology of \cite{GLP94}, the band complex $\Sigma^\epsilon$ yields a finite system of isometries on a finite tree.
If $\Sigma$ contains such an interval, then the system of isometries will have an interval of comparable length contained in the domain of at least 3 isometries.
Since every point of $\Sigma$ lies in two bands, in the approximation, the measure of the set of points in the domain of at most one partial isometry
will go to zero (Lemma \ref{measure}). But a result by Gaboriau--Levitt--Paulin \cite{GLP94} shows that in such a system of isometries (with independent generators), the measure of the set of points in the domain of at least 3 partial isometries is bounded from above by the measure of
the set of points in the domain of at most one partial isometry. This will give a contradiction.
\\

Given an $\bbR$-tree $K$, and $\eps>0$, define $K^{\dg \eps}$ as the set of all $x\in K$ such that 
$x$ is the midpoint of a segment of length $2\eps$ contained in $K$.
Note that if $K$ is a compact $\bbR$-tree, then for all $\eps>0$, $K^{\dg \eps}$ is a finite tree. Indeed, otherwise, the tree $K^{\dg \eps}$ would contain 
infinitely many terminal points $x_n\in K^{\dg \eps}$. By definition of $K^{\dg \eps}$, there exists $y_n\in K$ at distance $\eps$ from $x_n$ such that
$[x_n,y_n]\cap K^{\dg \eps}=\{x_n\}$. The points $y_n$ are at mutual distance at least $2\eps$
so the sequence $(y_n)_{n\in\bbN}$ has no converging subsequence in the metric topology, contradicting the compactness of $K$. We define 
$$\Sigma^{\eps}:=\bigcup_{u\in R} \calk_u^{\dg \eps}\subseteq \Sigma.$$
The set $\Sigma^{\eps}$ has the structure of a band complex whose base trees are the sets $\calk_v^{\dg \eps}$ for $v$ vertex in $R$,
and whose bands are of the form $K_e^{\dg\eps}\times e$ for $e$ edge of $R$.
Notice that $\eps>0$ can be chosen small enough so that $\calk_u^{\dg\eps}\neq\emptyset$ for all $u\in R$. We will always assume that this holds in what follows.

Let $v\in R$ be a vertex. Let $K_0\subseteq \calk_v$ be a fundamental domain for the action of $G_v$ on $\calk_v$. 

\begin{lemma}\label{measure}
Assume that leaves of $\Sigma$ have no terminal point.
\\ Let $J_\eps \subseteq \calk_v$ be the set of points $x\in \calk_v^{\dg\eps}$ 
that are contained in at most one $\Sigma^{\eps}$-band. 
\\
Then the Lebesgue measure of $J_\eps\cap K_0$ converges to $0$ as $\eps$ goes to $0$.
\end{lemma}

\begin{proof}
Let $v_\Sigma\in \calk_v$ be the vertex fixed by $G_v$ if $G_v$ is non-trivial (i.e.\ the special vertex if $G_v$ is infinite), 
and $v_\Sigma\in \calk_v$ an arbitrary basepoint otherwise. Without loss of generality, we take for $K_0$ a union of closures of components
of $\calk_v\setminus \{v_\Sigma\}$.

 Let $\calb$ be the collection of
  all bands of $\Sigma$ that meet $\calk_v$; for each band $B\in \calb$, we let $\ul B:=B\cap \calk_v$. 

We claim that if $\eps$ is small enough, the following holds.
  Let $x\in \Sigma^{\eps}\cap \calk_v$ be a point that is a terminal vertex
  in $\call_x\cap\Sigma^{\eps}$.  Then $x$ lies either in the bridge
  $[v_\Sigma,\ul B^{\dg\eps}]$ for some band $B$ such that $v_\Sigma\in \ul B\setminus \ul B^{\dg\eps}$,
  or in the bridge $[\ul B^{\dg\eps},\ul B'^{\dg\eps}]$  for some bands $B,B'\in\calb$ satisfying $\ul B\cap \ul B'\neq \es$ and $v_\Sigma\notin \ul B$.

This claim implies the lemma. Indeed, any such bridge is an interval of length at most $2\eps$.
Moreover, the number of bridges of the first kind contained in $K_0$ is bounded by the number of $G_v$-orbits of bands
incident on $\calk_v$: indeed, given a bridge $[v_\Sigma,\ul B^{\dg\eps}]$ contained in $K_0$, 
one has $v_\Sigma\notin \ul B^{\dg\eps}$ so $\ul B^{\dg\eps}\subseteq K_0$, and any two distinct bands such that $\ul B^{\dg\eps}\subseteq K_0$ are in distinct orbits.

To take care of bridges of the second kind, it suffices to observe that there are only finitely many $G_v$-orbits of pairs of bands 
$B,B'$ with $\ul B\cap \ul B'\neq \es$ and $v_\Sigma\notin B$ (Lemma~\ref{lem_finitude}).

We now prove the claim. By Lemma \ref{lem_finitude}, there exists $\eps>0$ such that 
$d(\ul B,\ul B')>2\eps$ as soon as $\ul B\cap \ul B'=\es$. 
Write $x$ as the midpoint of a segment $[a,b]\subseteq \calk_v$ of length $2\eps$.  Since leaves of $\Sigma$ have no terminal point, 
$a$ lies in two distinct bands, and since $x$ is terminal
in $\call_x\cap \Sigma^\eps$, 
 at least one of these bands $B$ is such that $x\notin \ul B^{\dg\eps}$, and in particular $b\notin \ul B$. 
Arguing symetrically, we deduce that there exist two bands $B,B'\in\calb$ with $a\in \ul B\setminus \ul B'$, 
$b\in \ul B'\setminus \ul B$, and such that $x\notin \ul B^{\dg\eps}\cup \ul B'^{\dg\eps}$. One easily checks that this implies that $\ul B^{\dg\eps}\cap \ul B'^{\dg\eps}=\emptyset$ and $x$ lies in the bridge $[\ul B^{\dg\eps},\ul B'^{\dg\eps}]$.  Notice also that $\ul B\cap \ul B'\neq \emptyset$ in view of our choice of $\eps$. 
If $v_\Sigma\notin \ul B\cap \ul B'$, then up to exchanging the roles of $B$ and $B'$, we see that $x$ lies in a bridge of the second kind.
If $v_\Sigma\in \ul B\cap \ul B'$, the fact that $[v_\Sigma,\ul B^{\dg\eps}]\cup [v_\Sigma,\ul B'^{\dg\eps}]$ contains 
$[\ul B^{\dg\eps},\ul B'^{\dg\eps}]$ shows that $x$ in a bridge of the first kind.
\end{proof}

We denote by $\Sigma_3$ the subset of $\Sigma$ made of points that belong to at least $3$ distinct bands of $\Sigma$. 

\begin{lemma}\label{omega-3}
Assume that no leaf of $\Sigma$ has a terminal point. 
\\ Then for every vertex $v\in R$, the set $\Sigma_3\cap\calk_v$ contains no nondegenerate interval.
\end{lemma}

\begin{proof}
We construct from $\Sigma^{\eps}$ a system of partial isometries on $$K^{\eps}:=\left(\cup_{v\in V(R)}\calk_v^{\dg\eps}\right)\!\Big/G$$
which is a finite union of finite trees. 
We cannot do this directly because a band $B$ of $\Sigma^\eps$ may fail to inject in the quotient by $G$ if $B\cap gB\neq \es$
for some  peripheral $g\in G$. This happens only if $B$ contains a special point of $\Sigma$.
For this reason, we subdivide
each band $B$ of $\Sigma^{\eps}$ containing a special vertex $v$ 
 along the leaf segment $l_{v}$ of $B$ containing $v$. The band $B$ contains at most two special vertices (one at each end),
and $B\setminus l_v$ has finitely many connected components because $B$ being a band of $\Sigma^\eps$, it is a product of a finite tree by $[0,1]$.
Denote by $\calb'$ the set of subdivided bands in $\Sigma^\eps$,
and by $B'_1,\dots,B'_n$ some representatives of the $G$-orbits of these bands, with a chosen orientation (i.e.\ an identification with $K\times [0,1]$ for some finite tree $K$).
Now each $B'_i$ embeds under the quotient map to $\Sigma^\eps/G$,
and determines a partial isometry $\phi_{i}$ of $K^\eps$. 

We claim that this finite system of isometries on $K^\eps$ has independent generators,
in the sense of \cite[Section 5]{GLP94}, i.e.\ there exists no nontrivial reduced word  $\phi_{i_r}^{\eps_r}\dots \phi_{i_1}^{\eps_1}$ in the partial isometries and their inverses 
that is defined and restricts to the identity on a non-degenerate interval.
Indeed, consider a word $w=\phi_{i_r}^{\eps_r}\dots \phi_{i_1}^{\eps_1}$ and an interval $I\subseteq K^\eps$ such that 
$w$  restricts to the identity on $I$.
Up to replacing $I$ by a smaller interval, we can assume that for all $k\leq r$, the set $\phi_{i_k}^{\eps_k}\dots \phi_{i_1}^{\eps_1}(I)$
does not contain the projection of a special vertex. Let now $\Tilde I\subseteq \Sigma^\eps$ be a lift of $I$, and let $B'\in\calb'$ be a lift of the band labeled $\phi_1^{\eps_1}$. Then there exists a unique $g_1\in G$ such that $I\subseteq g_1B'$. By applying this fact finitely many times, we see that there is a unique way to lift the word $w$ to the holonomy of a sequence of bands in $\Sigma^\eps$  joining $\Tilde I$ to some other lift $g\Tilde I$. The corresponding leaf segments do not backtrack in $R$,
so $g\neq 1$. 
Since $\Tilde I$ and $g\Tilde I$ have the same image in $T$, this implies that $g$ fixes an arc in $T$, a contradiction 
since $T$ has trivial arc stabilizers \cite[Proposition 4.17]{Hor14-1}. 

If $\Sigma_3$ contained a nondegenerate interval, then by Lemma \ref{measure} we could choose $\epsilon>0$ small enough so that the measure of the set of points of $K^{\eps}$ contained in the domains of at least $3$ partial isometries is strictly greater than the measure of the set of points contained in at most one such domain. This contradicts the fact that generators are independent by \cite[Proposition 6.1]{GLP94}.
\end{proof}

\begin{cor}\label{omega-3-fini}
Assume that no leaf of $\Sigma$ has a terminal point. 
\\ Then the set $\Sigma_3/G$ is finite. 
\end{cor}

\begin{proof}
It is enough to prove that for every vertex $v\in R$, the set $(\Sigma_3\cap\calk_v)/G_v$ is finite. Let $v_0$ be a point in $\calk_v$, which we choose to be equal to the special vertex $v_\Sigma$ in case $G_v$ is infinite.
We associate to every point $x\in \calk_v\cap \Sigma_3$ the collection $\mathcal{B}_x$ of all bands of $\Sigma$ that contain $x$ but do not contain $v_0$. 
Lemma \ref{omega-3} shows that there is always one band in $\mathcal{B}_x$ that does not meet the interval $[v_0,x)$, so $x$ is the projection of $v_0$ to the intersection of all bands in $\calb_x$ (in particular the collection $\mathcal{B}_x$ determines $x$). 
In addition, the set $\mathcal{B}_x$ only takes finitely many values modulo $G_v$ as $x$ varies in $\calk_v\cap \Sigma_3$. This implies that $(\calk_v\cap \Sigma_3)/G_v$ is finite. 
\end{proof}

As explained above, Proposition \ref{surface_3ends} immediately follows from Corollary \ref{omega-3-fini}. The last sentence in Proposition~\ref{surface_3ends} is proved in Lemma~\ref{omega-3}.

\subsection{Finiteness of 3-ended leaves: end of the proof}\label{sec-proof}

\begin{proof}[Proof of Proposition \ref{prop_3ends}]
In view of Lemma \ref{thin}, it is enough to bound the number of orbits of $\Sigma$-leaves with at least 3 ends such that $\call_\Omega$ is contained in $\Omega_{int}$.
By Lemma~\ref{omega-int-fini}, $\Omega_{int}$ has only finitely many orbits of connected components. Let $\Lambda$ be one of these connected components. By Lemma~\ref{coeur_elagage}, we have $\Lambda=\Sigma(T_{\Lambda},R_{\Lambda})$ for some Grushko $G_{\Lambda}$-tree $R_{\Lambda}$, and Lemma~\ref{no-term} and Proposition \ref{surface_3ends} imply that this band complex is of quadratic type. This concludes the proof since a band complex of quadratic type has
finitely many orbits of leaves with at least 3 ends by definition.
\end{proof}

\section{Splitting}\label{sec-splitting}

We now define another inductive process, the splitting induction, that is useful for studying trees of quadratic type: this generalization of the Rauzy--Veech induction was introduced in \cite[Section 4]{CHR11} in the context of free groups. Our main motivation for this is Corollary~\ref{coiffeur} from the next section, which states in particular that if $T\in\ol\calo$ is a mixing tree, then we can build a sequence of Grushko trees $R^{(i)}$ so that the diameters of the base trees of $\Sigma(T,R^{(i)})$ converge to $0$; this will be key when analysing the lamination dual to an arational tree. Throughout the present section, we let $T$ be a tree with dense orbits of quadratic type with respect to some Grushko tree $R$ (see Definition \ref{dfn_surface}), and we let $\Sigma:=\Sigma(T,R)$. We would first like to make the following observation that will be used repeatedly over the section.

\begin{lemma}\label{obs-germs}
Let $T$ be a tree with dense orbits of quadratic type with respect to a Grushko tree $R$, and let $\Sigma:=\Sigma(T,R)$. Let $d$ be a connected component of $\calk_v\setminus\{x\}$ where
$x$ is a point in some base tree $\calk_v$. Then there exists a neighbourhood $U$ of $x$ in $\calk_v$ such that $d\cap U$ meets exactly two bands of $\Sigma$, and $d\cap U$ is contained in these two bands.
\end{lemma}

\begin{proof}
Let $U$ be an open neighbourhood of $x$ in $\calk_v$ such that every band of $\Sigma$ that meets $U$ contains $x$: this exists by Corollary~\ref{finitude-1}. If $U\cap d$ meets three bands $B_1,B_2,B_3$ of $\Sigma$, then each band $B_i$ contains a nondegenerate segment of the form $[x,y_i]_{\calk_v}$ with $y_i\in U\cap d$. The intersection of the three segments $[x,y_i]$ is a nondegenerate segment contained in $U\cap d$ and meeting three bands, contradicting the fact that $\Sigma$ is of quadratic type. Therefore $U\cap d$ meets at most two bands $B_1,B_2$. Since every point in $U\cap d$ is contained in two bands, this implies that $U\cap d$ is contained in both $B_1$ and $B_2$.
\end{proof}

\subsection{Splitting germs}

\begin{figure}[htbp]
  \centering
  \includegraphics{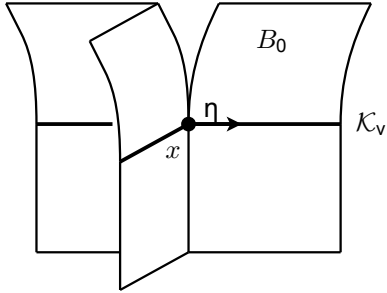}
  \caption{A splitting germ $\eta$.}
  \label{fig_split} 
\end{figure}

Given a vertex $v\in R$, a \emph{splitting germ} in $\calk_v$ (see Figure \ref{fig_split}) is a germ of segment $\eta\subseteq \calk_v$ 
which can be represented by a segment $[x,y)_{\calk_v}$ contained in a band $B_0$
and 
such that $x$ is an endpoint of the base $\ul B_0$ but not of $\calk_v$ 
(here $\ul B_0$ denotes the intersection of $B_0$ with the base tree $\calk_v$). 
We call $x$ the \emph{base point} of the germ $\eta$.

\begin{rk}\label{rk_branch_leaf}
If $\Sigma$ is of quadratic type, then $x$ is a branch point in its leaf (i.e.\ it is contained in at least 3 bands).
Indeed, since $x$ is not terminal in $\calk_v$, there is a direction $\eta'\neq \eta$ in $\calk_v$.
Since $\Sigma$ is of quadratic type, there are two bands $B_1,B_2$ containing $\eta'$, and these bands
are distinct from $B_0$ because $x$ is terminal in $\ul B_0$.
\end{rk}

\subsubsection{Existence of splitting germs} Our goal is now to prove the existence of splitting germs in trees of quadratic type.

\begin{prop}\label{existence-splitting}
Let $T\in\overline{\calo}$ be a tree with dense orbits of quadratic type with respect to $R$.
Then $\Sigma$ contains a splitting germ. 
\end{prop}

The idea of the proof of Proposition~\ref{existence-splitting} is the following. If $\Sigma$ contained no splitting germ, then the set $\cup_{v\in V(R)}\partial\calk_v$ made of points that are terminal in their base tree would be invariant under holonomy of leaves in $\Sigma$. The key point will be to show that if $\eps>0$ is chosen small enough, then the set of terminal points in the approximating trees $\calk_v^{\dg\eps}$ (defined as in Section~\ref{sec-surface-type}) is also invariant under holonomy. Since the trees $\calk_v^{\dg\eps}$ contain only finitely many orbits of terminal points, invariance under the holonomy thus implies that the $p_T$-images of these points have nontrivial stabilizer. As $\eps>0$ can take  uncountably many values, this yields a contradiction.

\begin{figure}[htbp]
  \centering
  \includegraphics{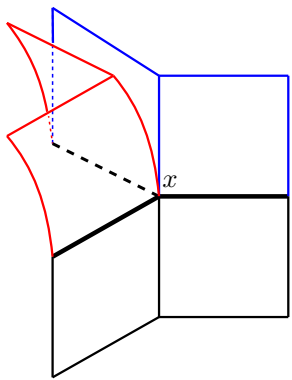}
  \caption{A singular point $x$.}
  \label{fig-singular} 
\end{figure}

Given a vertex $v\in R$, a point $x\in\calk_v$ is \emph{singular} if there exists a band $B$ incident on $\calk_v$, 
and a non-degenerate segment $I=[x,x']_{\calk_v}\subseteq\calk_v$, 
such that $B\cap I=\{x\}$. Notice in particular that base points of splitting germs are singular. Also note that an extremal point $x\in\calk_v$ is singular only if there is a band $B$ with $B\cap\calk_v=\{x\}$. An example of a singular point which is not of any of the above two forms is depicted on Figure~\ref{fig-singular}.

\begin{lemma}\label{singular}
Under the hypotheses of Proposition \ref{existence-splitting}, there are only finitely many orbits of singular points in $\Sigma$.
\end{lemma}

\begin{proof}
Let $v\in R$ be a vertex, and let $v_0\in\calk_v$ be a basepoint, which we choose to be the special point of $\calk_v$ if $G_v$ is infinite. Let $K\subseteq \mathcal{K}_v$ be a connected component of $\calk_v\setminus\{v_0\}$. 

We claim that for every singular point $x\in K$, there exists a band 
$B_x$ incident on $\calk_v$ such that $x$ is the projection of $v_0$ onto $B_x$. The lemma follows from this claim since there are only finitely many orbits of connected components in $\calk_v\setminus\{v_0\}$, and for such a component $K$, there are finitely many bands incident on $K$ that do not contain $v_0$, and clearly $B_x\neq B_y$ if $x\neq y$ are two different singular points in $K$.

\begin{figure}[htbp]
  \centering
  \includegraphics{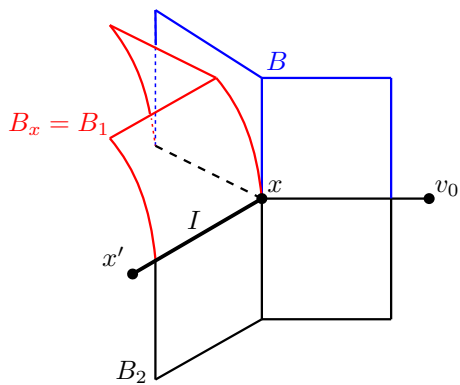}
  \caption{Existence of the band $B_x$ such that $x$ is the projection of $v_0$ onto $B_x$.}
  \label{fig-fs} 
\end{figure}

We now prove our claim, this is illustrated on Figure~\ref{fig-fs}. Since $x$ is singular, there exists a band $B$ incident on $\calk_v$, and a non-degenerate interval $I=[x,x']_{\calk_v}$, 
such that $I\cap B=\{x\}$. If the projection of $v_0$ onto $B$ is $x$, we are done, so we assume otherwise. Then $B\cap [v_0,x]_{\calk_v}$ is nondegenerate, and $I\cap [v_0,x]_{\calk_v}=\{x\}$. Let $U$ be an open neighbourhood of $x$ such that every band that meets $U$ contains $x$ (this exists by Corollary \ref{finitude-1}). 
Since no point in $\Sigma$ is a terminal vertex in its complete $\Sigma$-leaf,  
the intersection $I\cap U$ is covered by two distinct bands $B_1$ and $B_2$ (and these bands contain $x$). 
Moreover, since $I\cap B=\{x\}$, $B_1$ and $B_2$ are distinct from $B$.
If both $B_1$ and $B_2$ met a nondegenerate subsegment of $[v_0,x]_{\calk_v}$, then this subsegment would meet the three bands 
$B_1$, $B_2$ and $B$, contradicting the fact that $\Sigma$ is of quadratic type. Therefore, one of the two bands $B_1$ or $B_2$ satisfies our claim. 
\end{proof}

\begin{proof}[Proof of Proposition \ref{existence-splitting}]
Assume towards a contradiction that there is no splitting germ in $\Sigma$.  
This means that for every vertex $v\in R$, and every base of band $\ul B\subseteq \calk_v$, 
every endpoint of $\ul B$ is an endpoint of $\calk_v$ (this does not imply $\ul B=\calk_v$ as there can be branch points of $\calk_v$
which are not branch points in $\ul B$).
Without loss of generality, we can assume that no base of band of $\Sigma$ is reduced to a point and in particular that extremal points of base
trees $\calk_v$ are not singular.

Let $v\in R$ be a vertex. By Lemma \ref{singular}, there are finitely many $G_v$-orbits of singular points in $\calk_v$, and these points are nonextremal in the bands that contain them because there is no splitting germ in $\Sigma$ by assumption.  Since every point $x$ is contained in finitely many bands modulo $G_x$ (Lemma \ref{lem_finitude_0}),
there exists $\eps_0>0$ such that for every singular point $x$, and every band $B$ containing $x$, we have $x\in\underline{B}^{\dg\eps_0}$
(recall that $x\in\underline{B}^{\dg\eps_0}$ means that $x$ is the midpoint of a segment of length $2\eps_0$ contained in $\ul B$). 
Let $0<\eps<\eps_0$ be  chosen small enough so that the distance between any two singular points in $\Sigma$ is strictly greater than $2\eps$.

We claim that for every point $m\in\calk_v^{\dg\eps}$, and every band $B$ that contains $m$, we have $m\in \ul B^{\dg\eps}$.
Indeed, if $m$ is singular, this follows from the definition of $\eps_0$. Otherwise, the fact that $T$ is of quadratic type implies that there are exactly two bands $B\neq B'$ that contain $m$.
Since $m\in\calk_v^{\dg\eps}$, there exists a segment $[a,b]_{\calk_v}$ 
of length $2\eps$ whose midpoint is $m$.
If $[a,b]_{\calk_v}$ does not contain any singular point, then 
both $B$ and $B'$ cover the interval $[a,b]_{\calk_v}$, so $m\in\underline{B}^{\dg\eps}\cap\underline{B'}^{\dg\eps}$ and we are done. Otherwise, by our choice of $\eps>0$, the interval $[a,b]_{\calk_v}$ contains a unique singular point $x$, and without loss of generality, we can assume that $x\in [a,m]_{\calk_v}$. We have $[x,b]_{\calk_v}\subseteq B\cap B'$, otherwise one of the intersections $B\cap [a,b]_{\calk_v}$ or $B'\cap [a,b]_{\calk_v}$ would have a terminal point $x'\neq x$, and $x'$ would be a second singular point in $[a,b]_{\calk_v}$. By definition of $\eps_0$,  there exist subsegments $J\subseteq \underline{B}$ and $J'\subseteq\underline{B'}$ of length  $2\eps_0$ and whose midpoint is $x$. 
Then $[b,x]_{\calk_v}\cup J\subseteq\underline{B}$, showing that $m\in\underline{B}^{\dg\eps}$, and similarly $m\in\underline{B'}^{\dg\eps}$. This proves our claim.

We now fix $\eps$ small enough so that the  above claim holds for every vertex $v\in R$, and we prove that the family of subsets $\calk_v^{\dg\eps}$, indexed by $v$, is invariant under the holonomy of $\Sigma$.
Indeed, let $x\in \calk_v^{\dg\eps}$ for some vertex $v$ in $R$, and consider a band $B$ containing $x$, joining $\calk_v$ to some $\calk_u$.
 Let $[x,y]_{\call_x}$ be the $\Sigma$-leaf segment of $B$ through $x$. We denote by $\ul B_v$ and $\ul B_u$ the two bases of $B$.
Since the claim holds, $\ul B_v$ contains a segment of length $2\eps$ centered at $x$, hence $y\in\ul B_u^{\dg\eps}$, and $y\in \calk_u^{\dg\eps}$.
This proves the desired invariance under holonomy.

Now let $\partial\calk_u^{\dg\eps}$ be the set of endpoints of $\calk_u^{\dg\eps}$.
Since  $\partial\calk_u^{\dg\eps}/G_u$ is finite for every vertex $u\in R$, so is $(\bigcup_{v\in R} \partial\calk_v^{\dg\eps})/G$.
Since $\calk_u^{\dg\eps}\setminus \partial\calk_u^{\dg\eps}=\bigcup_{\eps'>\eps} \calk_u^{\dg\eps'}$,
the set $\bigcup_{v\in R} \partial\calk_v^{\dg\eps}$ is invariant under holonomy,
i.e.\ for all $x\in \partial\calk_u^{\dg\eps}$, and all vertices $v\in R$ such that $\call_x$ intersects $\calk_v$, we have
$\call_x\cap \calk_v\subseteq \partial\calk_v^{\dg\eps}$.
Since $\call_x$ is infinite (Proposition \ref{prop_infini}),
and  $(\bigcup_{v\in R} \partial\calk_v^{\dg\eps})/G$ is finite, this implies that the stabilizer of $\call_x$ is infinite, and so is the stabilizer of $p_T(x)$.
Since this holds for every small enough $\eps>0$, we get uncountably many points in $T$ with infinite stabilizer.
Since by \cite[Corollary 4.5]{Hor14-1} there are only finitely many orbits of points with non-trivial stabilizer, we get a contradiction.
\end{proof}

\begin{figure}[htbp]
  \centering
  \includegraphics{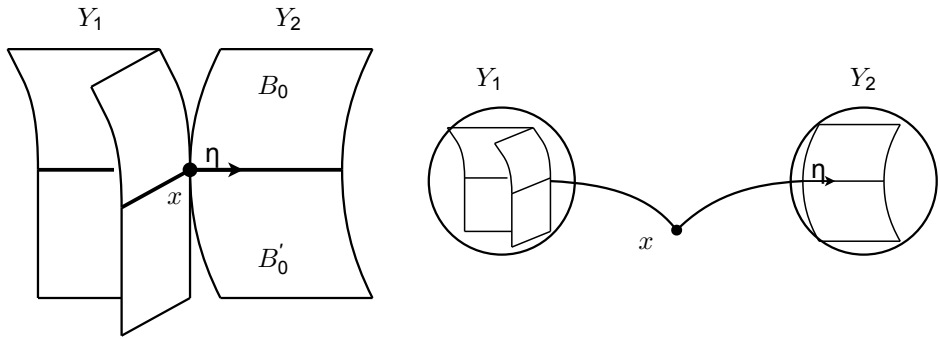}
  \caption{A degenerate splitting germ $\eta$, and the corresponding $(G,\calf)$-free splitting $S$.}
  \label{fig_degenerate}
\end{figure}

\subsubsection{Degenerate splitting germs} Let $v\in R$ be a vertex, and let $\eta$ be a splitting germ, based at a point $x\in\calk_v$. 
We say that the splitting germ $\eta$ is \emph{degenerate} if $x$ is extremal in the two bands containing $\eta$ (see Figure \ref{fig_degenerate} for a degenerate splitting germ; the splitting germ on Figure \ref{fig_split} is non-degenerate).

\begin{lemma}\label{lem_non-degenerate}
If $\Sigma$ contains a degenerate splitting germ, 
then $T$ is compatible with a $(G,\calf)$-free splitting. 
\end{lemma}  

\begin{proof}
This is illustrated in Figure~\ref{fig_degenerate}. Assume that there is a degenerate splitting germ $\eta\subseteq \calk_v$ based at a point $x\in \calk_v$ for some vertex $v\in R$.
Since $x$ is a local cut point of $\Sigma$, and $\Sigma$ is simply connected, $x$ is a global cut point.
Let $(Y_i)_{i\in I}$ be the family made of the closures of the connected components of $\Sigma\setminus G.x$. We define a bipartite simplicial graph $S$ having one vertex for each $Y_i$ and one vertex for each point in $G.x$, where $Y_i$ is joined by an edge to $g.x$ whenever $g.x\in Y_i$. Our goal is to show that some collapse of $S$ is a $(G,\calf)$-free splitting compatible with $T$.

Since each point in $G.x$ is a global cut point, $S$ is a tree.
Some edges of $S$ may have non-trivial stabilizer (this happens only if $x$ is the special vertex of $\calk_v$),
but the edge $e$ joining $x$ to the component $Y_i$ containing $\eta$ has trivial stabilizer because $\eta$ is terminal in $Y_i\cap \calk_v$
and the stabilizer of $e$ fixes $\eta$.
Let $S'$ be the tree obtained from $S$ by collapsing every edge outside $G.e$.
Then $S'$ is a free splitting of $G$, and peripheral subgroups are elliptic so it is a $(G,\calf)$-free splitting.
Since $S'$ has a single orbit of edges and no terminal point, the action of $G$ is minimal.

In order to complete the proof of the lemma, we are thus left checking that $S'$ is compatible with $T$. We first observe that if $i\neq j$, then $p_T(Y_i)\cap p_T(Y_j)$ contains at most one point. 
   
We now assign a subtree $T_u\subseteq T$ for each vertex $u\in S$ as follows:
if $u$ corresponds to the component $Y_i$, we define $T_u=p_T(Y_i)\subseteq T$;
if $u$ corresponds to $g.x$, we define $T_u=\{p_T(g.x)\}$.  Then for all vertices $u\neq u'$ in $S$, the intersection $T_u\cap T_{u'}$ contains at most one point. The following lemma thus completes the proof.
\end{proof} 

\begin{lemma} Let $S$ be a simplicial $(G,\calf)$-tree. For each vertex $u\in S$, let $T_u\subseteq T$ be a subtree, with $T_{gu}=gT_u$ for all $g\in G$. Assume that for all $u\neq u'$, the intersection $T_u\cap T_{u'}$ contains at most one point,
and is non-empty if $u$ is a neighbour of $u'$.
\\
Then $S$ is compatible with $T$.
\end{lemma}

\begin{proof}
Let $\hat T$ be the $\bbR$-tree obtained from the disjoint union of the vertex trees $T_u$
and of the edges of $S$, glued as follows:
if $e=[u,u']$ is an edge of $S$, we let $\{x_e\}=T_u\cap T_{u'}$,
and we glue the endpoints of $e$ to the copies of $x_e$ in $T_u$ and $T_{u'}$, respectively.

Clearly, there is a collapse map from $\hat T$ to $S$. 
Moreover, let $T'$ be the tree obtained from $\hat T$ by collapsing all the edges coming from $S$.
The inclusion maps $T_u\subseteq T$ induce a map $f:T'\ra T$ which is isometric
in restriction to each $T_u$. Since $T_u\cap T_{u'}$ is reduced to a point for each $u\neq u'$,
this implies that $f$ is an isometry. By minimality, $f$ is onto, and $\hat T$ is a common refinement of $S$ and $T$.
\end{proof}

\subsection{The splitting process}\label{sec_process}

From now on, we assume that $T$ is not compatible with any $(G,\calf)$-free splitting and we define a splitting procedure on the band complex $\Sigma$. 

\begin{figure}[htbp]
  \centering
  \includegraphics[width=\linewidth]{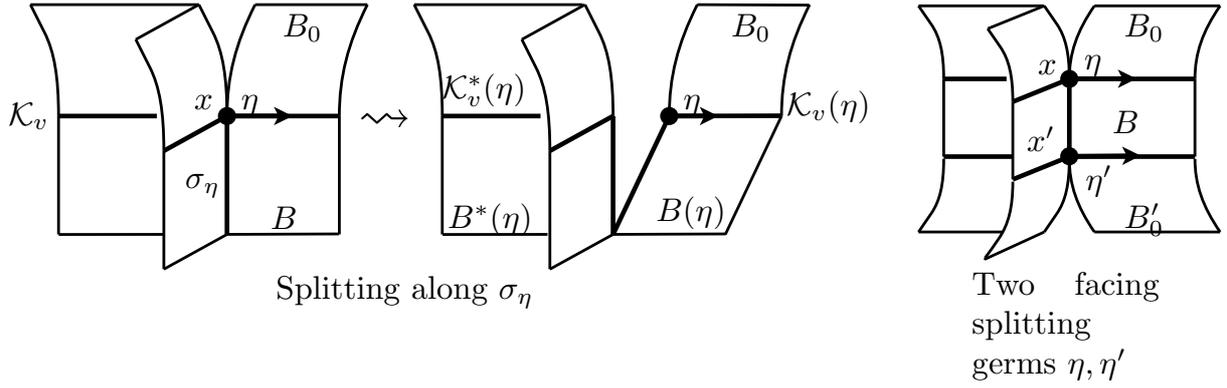}
  \caption{The splitting process.}
  \label{fig_splitting_process}
\end{figure}

Let $\eta$ be a non-degenerate splitting germ based at some point $x\in\calk_v$: recall that this means that one of the two bands $B_0$ containing 
 $\eta$, 
is such that $x$ is an endpoint of 
$B_0\cap \calk_v$, but $x$ is not an endpoint of $B\cap \calk_v$ where $B$ is the other band containing $\eta$.
Note that $B_0$ and $B$ are uniquely defined by these conditions, and we say that $B$ is 
the band \emph{split} by $\eta$.
We define the \emph{splitting leaf segment} $\sigma_\eta$ as $B\cap \call_x$, i.e.\ the leaf segment of $B$ starting from $x$.

Notice that since there are only finitely many orbits of singular points (Lemma \ref{singular}) and finitely many orbits of directions at each of these points, there are only finitely many
orbits of splitting germs.
Since each germ splits a unique band, and since band stabilizers are trivial, this implies that for each band $B$, there are only finitely many germs that split $B$.
Similarly, for each vertex $v$ of $R$, there are only finitely many $G_v$-orbits of splitting germs contained in $\calk_v$.

\subsubsection{A band complex $\Sigma'_\Gamma$ defined by splitting}

Given a $G$-invariant collection $\Gamma$ of splitting germs, we now explain how to define a new band complex $\Sigma'_{\Gamma}$ by cutting the bands of $\Sigma$ split by germs $\eta\in\Gamma$ along the corresponding splitting leaf segments (see Figure \ref{fig_splitting_process}). 
To define this operation, we will need to make the assumption that $\Gamma$ has no pair of facing splitting germs in the following sense: there is no pair of splitting germs $\eta,\eta'\in \Gamma$ 
contained in the two opposite bases of some band, and such that $\eta$ and $\eta'$ have the same image in $T$. We will then prove, using the fact that $T$ is not compatible with any $(G,\calf)$-free splitting, that the collection $\Gamma_{\max}$ made of all splitting germs of $\Sigma$ contains no pair of facing splitting germs (Lemma~\ref{no-facing} below). This will allow us to cut all bands in $\Sigma$ simultaneously, and we will say that the band complex $\Sigma':=\Sigma'_{\Gamma_{\max}}$ is obtained from $\Sigma$ by \emph{applying one step of the splitting process}.
To manage the details of the construction around points with non-trivial stabilizer, and with all splitting germs simultaneously, 
we make this construction more formal, and  proceed in several stages.
 The whole construction is illustrated on Figure~\ref{fig_splitting_partition}.

\begin{figure}[htbp]
  \centering
    \includegraphics[scale=0.95]{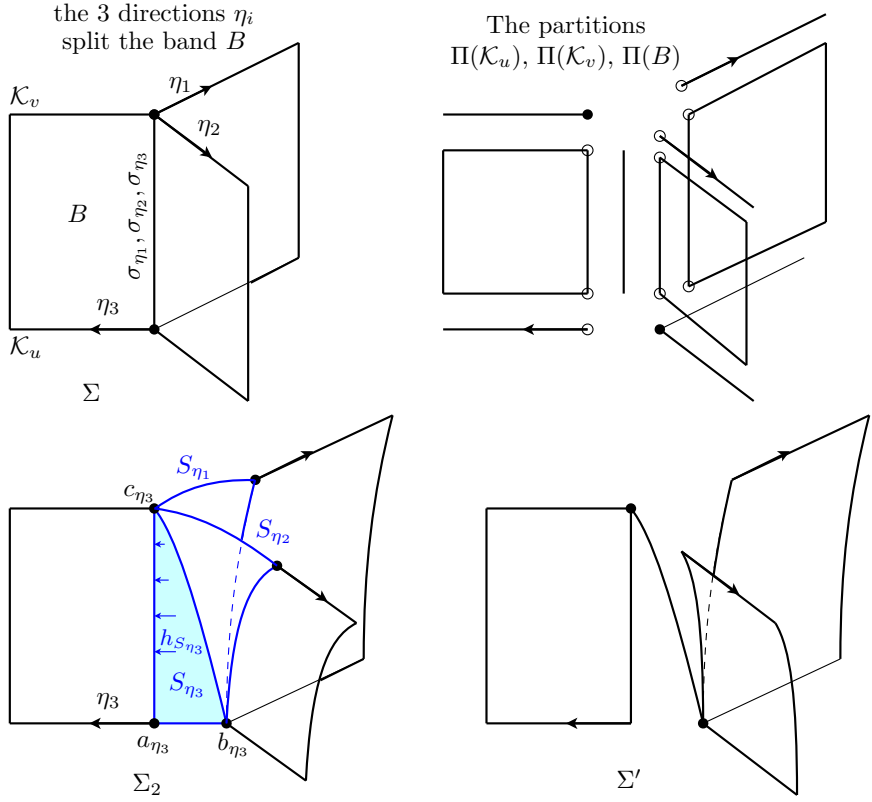}
  \caption{$\Sigma$, the partitions, $\Sigma_2$, and $\Sigma'$}
  \label{fig_splitting_partition}
\end{figure}

From now on, we let $\Gamma$ be a $G$-invariant collection of splitting germs that contains no pair of facing splitting germs. 
Each germ $\eta\in \Gamma$ yields a partition of each base tree $\calk_v$ in the following way.
Let $x_\eta$ be the base point of $\eta$.
If $\eta$ is contained in $\calk_v$, we write $\calk_v=\calk_v(\eta)\dunion \calk_v^*(\eta)$
where $\calk_v(\eta)$ denotes the connected component of $\calk_v\setminus \{x_\eta\}$ containing $\eta$ (this is an open subtree of $\calk_v$),
and $\calk^*_v(\eta)=\calk_v\setminus \calk_v(\eta)$ is its complement (a closed subtree, containing the base point $x_\eta$).
If $\eta$ is not contained in $\calk_v$, we associate the trivial partition.
Similarly, if $\eta$ splits the band $B$, we have a partition $B=B(\eta)\dunion B^*(\eta)$
where $B(\eta)$ is the connected component of $B\setminus \sigma_\eta$ whose closure contains a representative of $\eta$,
and $B^*(\eta)=B\setminus B(\eta)$.
If $\eta$ does not split $B$, then we associate the trivial partition.

Let $\Pi(\calk_v)$  (resp.\ $\Pi(B)$) be the partition of $\calk_v$ (resp. of $B$) induced by all partitions associated to all germs in $\Gamma$.
As an intersection of convex sets, each set of the partition $\Pi(\calk_v)$ is a subtree.
Similarly, each set in $\Pi(B)$ is the product of a subtree of a base of $B$ by an interval. Consider first 
$$\Sigma_1=\left(\coprod_{K\in \Pi(\calk_v), v\in V(R)} \ol K \right)\dunion \left(\coprod_{C\in \Pi(B), \text{B band}} \ol C \right)\dunion
\left(\coprod_{\eta \in \Gamma} S_\eta \right)$$
where $\ol K,\ol C$ denote the closures of $K$ and $C$ in $\Sigma$, 
and where for each splitting germ $\eta\in\Gamma$,   $S_\eta$ is a triangle $(a_\eta,b_\eta,c_\eta)$. The triangle $S_\eta$ comes with an affine map $h_{S_\eta}:S_\eta\to \sigma_\eta$, sending $[a_\eta,b_\eta]$ to $x_\eta$, and $c_\eta$ to the other endpoint of $\sigma_\eta$.

For each $\eta\in \Gamma$, let $B$ be the band split by $\eta$, and let $C_\eta\in \Pi(B)$ be the unique subset whose closure contains the germ $\eta$,
and $C^*_\eta\in \Pi(B)$ the unique subset containing the splitting segment $\sigma_\eta$ (note that $C^*_\eta$ may be reduced to a leaf segment, but not $C_\eta$).
The sets $C_\eta$ and $C^*_\eta$ are disjoint, and $\ol C_\eta\cap\ol C^*_\eta=\sigma_\eta$. The following observation follows from our assumption that $\Gamma$ has no facing germs. 
\\
\\
\noindent \textbf{Observation.} The ordered pair ($\ol C_\eta$,$\ol C^*_\eta$) uniquely determines $\eta$.
\\
Indeed, $\sigma_\eta$ can be recovered as $\sigma_\eta=\ol C_\eta \cap \ol C^*_\eta$;
denote by $\calk_{v_1},\calk_{v_2}$ the two base trees intersecting $\ol C_\eta$,
and by $x_i$ the point of intersection of $\calk_{v_i}$ with $\sigma_\eta$;
since $x_i$ is a terminal point of $\calk_{v_i}\cap \ol C_\eta$ 
there are only two possibilities left for $\eta$, being either based at $x_1$ or $x_2$ (and contained in $\overline{C}_\eta$). 
Since there is no pair of facing splitting germs, this leaves only one possibility for $\eta$.
\\
\\
\indent 
We now define $\Sigma_2$ from $\Sigma_1$ by making the following identifications. For each $\eta\in\Gamma$, we glue the segment $[a_\eta,c_\eta]$ of $S_{\eta}$ to the copy of $\sigma_\eta$ in $\ol C_\eta$, in such a way that $x_\eta$ is identified with $a_\eta$.
Similarly, we glue the  segment $[b_\eta,c_\eta]$ of $S_\eta$ to the copy of $\sigma_\eta$ in $\ol C^*_\eta$, identifying $x_\eta$ with $b_\eta$. The third side $[a_\eta,b_\eta]$ of $S_\eta$ will remain a free face, i.e.\ it will not be glued to anything else.
After this gluing operation, each band $B$ of $\Sigma$ has been replaced by a union of bands $\overline{C}$ and triangles $S_\eta$, and this union is homotopy equivalent to $B$ (via the natural map that is the identity on each band $\overline{C}$ and  restricts to $h_{S_\eta}$ on each triangle $S_\eta$). 
Finally, for each band $B$ of $\Sigma$ incident on $\calk_v$, and each $C\in \Pi(B)$, the intersection $C\cap \calk_v$ is contained in a unique set $K \in\Pi(\calk_v)$
and we identify the copy of $\ol C\cap\calk_v$ in $\ol C$
to the corresponding subset of $\ol K$. We denote by $\Sigma_2$ the obtained complex. Let $h:\Sigma_2\to\Sigma$ be the map that restricts to the inclusion map on each $\overline{K}$ with $K\in\Pi(\calk_v)$ and on each $\overline{C}$ with $C\in\Pi(B)$, and which 
restricts to $h_{S_\eta}$ on each triangle $S_{\eta}$. 

\begin{lemma}\label{sigma-2}
The space $\Sigma_2$ is simply connected.
\end{lemma}

\begin{proof}
We first claim that $h$ has contractible fibers. Indeed, for every $x\in\Sigma$, either 
\begin{itemize}
\item $h^{-1}(x)$ is a point, or  
\item $x$ belongs to the interior of a splitting leaf segment $\sigma$ on a band $B$; then, denoting by $C_x\in\Pi(B)$ the subset of the partition of $B$ that contains $x$, and by $\{x_i\}_{i\in I}$ the other copies of $x$ in $\Sigma_2$ (in the closures of other subsets of the partition $\Pi(B)$), the preimage $h^{-1}(x)$ is a cone with center $x$ over the points $x_i$ (indeed, the above observation implies that for each $i\in I$ there is a unique triangle $S_\eta$ that contains both $x$ and $x_i$), or
\item $x$ is a basepoint of a splitting germ in some $\calk_v$; then, denoting by $K_x\in\Pi(\calk_v)$ the subset of the partition that contains $x$, and by $\{x_i\}$ the other copies of $x$ in $\Sigma_2$ (in the closures of other subsets of the partition $\Pi(\calk_v)$), the preimage $h^{-1}(x)$ is a cone with center $x$ over the points $x_i$. 
\end{itemize}

Consider two finite trees $K_T\subseteq \ol T$, $K_R\subseteq R$,
and let $C:=\Sigma\cap (K_T\times K_R)$.
Let $C_2=h\m(C)$. Then $C$ and $C_2$ are metrizable, finite-dimensional, locally contractible, and $C$ is compact. If $C_2$ is compact, then Lemma~\ref{thm_cell-like} applies and says that $h$ is a homotopy equivalence between
$C_2$ and $C$.
If $C_2$ is not compact, then $C$ contains a special point $v$, and $h\m(v)$ consists of a compact set together with 
infinitely many hanging segments of the form $[a_{\eta_i},b_{\eta_i})$, 
with $b_{\eta_i}$ identified with $v$, and $[a_{\eta_i},b_{\eta_i})$ open in $C_2$ (notice that it is important here to have assumed that $K_T$ is a finite subtree and not just a compact subtree).
We define $C'_2\subseteq C_2$ as the complement of all those hanging segments, a compact set.
Since $h_{|C'_2}$ still has contractible fibers, $h$ induces a homotopy equivalence between $C'_2$ and $C$.
Since $C$ is simply connected by Lemma \ref{contractible}, so are $C'_2$ and $C_2$.

Let now $\gamma:S^1\ra \Sigma_2$ be a loop. Then one can homotope $\gamma$ so that $p_T\circ h\circ\gamma(S^1)$
is a finite subtree of $T$, and $p_R\circ h\circ\gamma(S^1)$
is a finite subtree of $R$. In particular, $h\circ\gamma(S^1)$ is contained in a simply connected compact set $C$ as above,
and $\gamma(S^1)$ is contained in $C_2=h\m(C)$. Since $C_2$ is simply connected, this concludes the proof of Lemma \ref{sigma-2}.
\end{proof}

We then define the deformation retract $\Sigma_3\subseteq \Sigma_2$ by collapsing the free edge $[a_\eta,b_\eta]$ of each triangle $S_\eta$ onto its two other edges
$[a_\eta,c_\eta]\cup [b_\eta,c_\eta]$.
In particular $\Sigma_3$ is  simply connected.  
Notice also that for each leaf $l\subseteq \Sigma$, the preimage $h\m(l)$ is connected, and so is $h\m(l)\cap \Sigma_3$. 

The space $\Sigma_3$ has a structure of a band complex in which the base trees are the images of $\ol K$ for $K\in \Pi(\calk_v)$
and the bands are the images of $\ol C$ for $C\in \Pi(B)$.
It may happen that some base tree $\ol K$ is reduced to a point $\{x\}$. This happens exactly when every germ of segment at $x$ in $\calk_v$ is a splitting germ.
We claim that there is no band joining two base trees reduced to a point. Indeed, let $C=[x,y]$ be such a singleton in $\Sigma_2$, and $B$ the band of $\Sigma$ such that
$C\in \Pi(B)$. Let $\eta$ be a germ in $B$ containing $x$. The germ $\eta'$ at $y$ facing $\eta$ is not a splitting germ, so the base tree containing $y$ in $\Sigma_3$ is not reduced to a point. 

We finally define $\Sigma'_\Gamma\subseteq\Sigma_3$ by removing all base trees $\ol K$ reduced to a point and on which only one singleton is incident, together with the corresponding singletons (recall that $\Gamma$ stands for our initial collection of splitting germs).
Notice from the construction that no point of $\Sigma'_\Gamma$ is terminal in its complete $\Sigma'_\Gamma$-leaf.

\subsubsection{The space $\Sigma'_\Gamma$ is isomorphic to some $\Sigma(T,R'_\Gamma)$.}

We denote by $h_{\Gamma}:\Sigma'_\Gamma\to\Sigma$ the restriction of $h$ to $\Sigma'_\Gamma$. We then let $\pi_T=p_T\circ h_{\Gamma}:\Sigma'_{\Gamma}\ra \ol T$ and  $\pi_R=p_R\circ h_{\Gamma}:\Sigma'_{\Gamma}\ra R$. We define $R'_{\Gamma}:=\Sigma'_{\Gamma}/{\sim}$, where $x\sim y$ if $\pi_R(x)=\pi_R(y)$ and $x,y$ are in the same connected component
of ${\pi_R}\m(\{\pi_R(x)\})$. This has a natural structure of a graph, whose vertices are the connected components of $\Sigma'_\Gamma\cap {\pi_R}^{-1}(V(R))$. We denote by $p_{R'}:\Sigma'_{\Gamma}\ra R'_{\Gamma}$ the quotient map. 
The map $h_\Gamma$ yields a natural $G$-equivariant map $f_\Gamma:R'_\Gamma\to R$, which sends vertex to vertex and edge to edge. By construction, the following diagrams commute:
$$\xymatrix{
  \Sigma'_\Gamma \ar@{->}[r]^{h_{\Gamma}} \ar[d]^{p_{R'}} 
                 &\Sigma \ar[d]^{p_R} \\ 
R'_\Gamma \ar[r]^{f_\Gamma} & R
}\quad
\xymatrix{
  \Sigma'_\Gamma \ar@{->}[r]^{h_\Gamma} \ar[d]^{\pi_{T}} 
  &\Sigma \ar[d]^{p_T} \\ 
T \ar@{=}[r] & T
}
$$

\begin{lemma}\label{splitting-core}
The graph $R'_\Gamma$ is a Grushko tree. In addition, we have $\Sigma'_\Gamma=\Sigma(T,R'_\Gamma)$.
\end{lemma}

\begin{proof}
We first observe that the graph $R'_\Gamma$ is simply connected, hence a tree: this follows from the simple connectedness of $\Sigma'_\Gamma$, by a proof similar to the proof of Lemma \ref{is-Grushko}. 
Since $f_\Gamma$ maps vertex to vertex and edge to edge, edge stabilizers of $R'_\Gamma$ are trivial and point stabilizers are peripheral. Since peripheral groups of $(G,\calf)$ 
fix a point in $\Sigma'_\Gamma$, they also fix a point in $R'_\Gamma$. 
Since the $f_\Gamma$-preimage of any edge in $R$ is a finite union of edges in $R'_\Gamma$, the quotient graph $R'_\Gamma/G$ is compact, so to prove minimality of $R'_\Gamma$, it is enough to check that $R'_\Gamma$ has no terminal vertex.
If a base tree $K$ of $\Sigma'_\Gamma$ is not reduced to a point, consider any germ $\eta$ in $K$, and view it as a germ in some $\calk_v\subseteq \Sigma$. 
Then $\eta$ is contained in two distinct bands of $\Sigma$ which yield two bands in $\Sigma'_\Gamma$ containing $\eta$. 
If on the other hand $K=\{x\}$, then because it has not been removed when defining $\Sigma'_\Gamma$ from $\Sigma_3$, there are at least two singletons in $\Sigma_3$ incident on $K$, which both belong to $\Sigma'_\Gamma$. 
This proves that $R'_\Gamma$ is a Grushko tree. 

We now prove  that 
 $\Sigma'_\Gamma$ coincides with $\Sigma(T, R'_\Gamma)$.
As noted above, for each complete $\Sigma$-leaf $l\subseteq \Sigma$, the preimage $h_\Gamma\m(l)$ is connected
which implies that fibers of $\pi_T$ are connected. Fibers of $p_{R'}$ are connected by construction, so 
the image of $\Sigma'_\Gamma$ in $\ol T\times R'_\Gamma$ under $(\pi_T, p_{R'})$ is a closed connected subset with connected fibers hence contains the core $\calc(\ol T\times R'_\Gamma)$, which is equal to $\Sigma(T,R'_\Gamma)$ by Proposition~\ref{prop_coeur}.

Conversely, let $(x,u)\in \ol T\times R'_\Gamma$ be a point in the image of $\Sigma'_\Gamma$, and let $z\in\Sigma'_\Gamma$ be a preimage of $(x,u)$. Let $\call_z$ be the complete $\Sigma'_\Gamma$-leaf through $z$ in $\Sigma'_\Gamma$. Since no point of $\Sigma'_\Gamma$ is terminal in its complete $\Sigma'_\Gamma$-leaf, it follows that $\call_z$ contains a bi-infinite line $l$ through $z$, and this line isometrically embeds in $R'_\Gamma$; we denote its endpoints by $\alpha,\omega$. Its image in $R$ under $f'_\Gamma$ is $[\alpha,\omega]_R$, and this is also the image of $h_{\Gamma}(l)\subseteq\Sigma$ under $p_R$. It follows that $\calq(\alpha)=\calq(\omega)=x$, 
so $x\in \Omega_u$. Therefore $(x,u)\in\Sigma(T,R'_\Gamma)$. This completes the proof of the lemma.  
\end{proof}

\subsubsection{Splitting all germs at once}

We recall that $\Gamma_{max}$ denotes the collection of all splitting germs in $\Sigma$.

\begin{lemma}\label{no-facing}
The collection $\Gamma_{max}$ contains no pair of facing splitting germs.
\end{lemma}

\begin{proof}
Assume towards a contradiction that there are two facing splitting germs $\eta,\eta'$ in a band $B$ of $\Sigma$. Then $\eta'\notin G.\eta$: indeed, if $\eta'=g\eta$, then $g$ would stabilize $B$ because $B$ is the unique band split by
$\eta$ and is also the unique band split by $\eta'$. Hence we would have $g=1$ and $\eta=\eta'$, a contradiction.
Thus the collection $\Gamma:=G.\eta$ contains no pair of facing germs, so $\Sigma'_{\Gamma}$ is well-defined. In view of Lemma \ref{splitting-core}, we have $\Sigma'_{\Gamma}=\Sigma(T,R_{\Gamma})$. Now it is easy to see that $\eta'$ is a degenerate splitting germ in $\Sigma'_{\Gamma}$. By Lemma \ref{lem_non-degenerate}, this implies that $T$ is compatible with a $(G,\calf)$-free splitting, contrary to our hypothesis.
\end{proof}

\begin{de}
We say that the band complex $\Sigma':=\Sigma'_{\Gamma_{\max}}$
  is obtained from $\Sigma$ by \emph{applying one step of the splitting process}.
\end{de}

Notice that the complex $\Sigma'$ obtained from this construction is again of quadratic type, and it does not contain any degenerate splitting germ because $\Sigma'=\Sigma(T,R')$ for some Grushko $(G,\calf)$-tree $R'$, and $T$ is not compatible with any $(G,\calf)$-free splitting. Therefore, we can iterate the construction to obtain a sequence of band complexes $\Sigma^{(i)}$, where $\Sigma^{(0)}:=\Sigma$, and for each $i\in\mathbb{N}$, the band complex $\Sigma^{(i+1)}$ is obtained from $\Sigma^{(i)}$ by applying one step of the splitting process.
We denote by $h_{\Sigma^{(i)}}:\Sigma^{(i)}\ra \Sigma$  and by $f_{R^{(i)}}:R^{(i)}\ra R$ the corresponding maps.

\begin{figure}[htbp]
  \centering
  \includegraphics{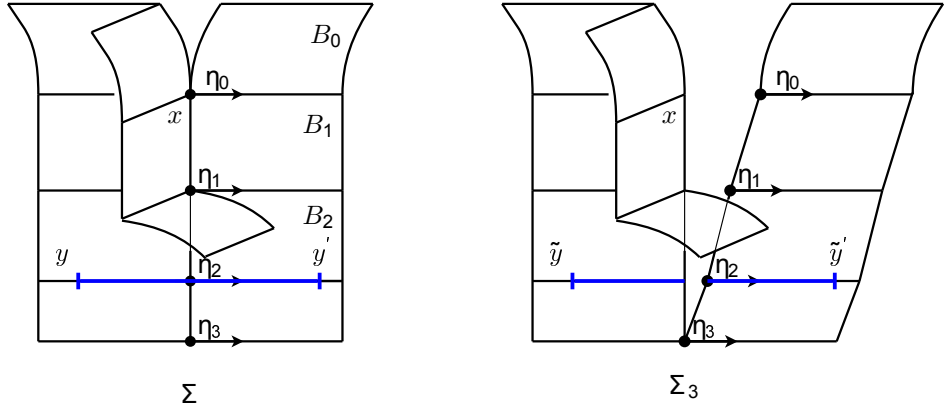}
  \caption{Virtual splitting germs, and a segment $[y,y']$ that gets separated in $\Sigma^{(3)}$.  The direction $\mu$ shows that $\Sigma$ is not clean.}
  \label{fig_virtual}
\end{figure}
 
\subsection{Clean band complexes}

The goal of the present section is to show that, up to replacing $\Sigma$ by $\Sigma^{(i)}$ (i.e.\ splitting for long enough), we can ensure that $\Sigma$ satisfies a few additional properties. These will only be used in Section~\ref{sec-11-surf}. Consider a splitting germ $\eta_0$ in $\Sigma$, $B_1$ the band split by $\eta_0$, and $\eta_1$ the germ facing $\eta_0$ in $B_1$.
For $i\geq 1$, one can then define inductively  $B_{i+1}$ as the band containing $\eta_i$ and distinct from $B_i$, 
and $\eta_{i+1}$ as the germ facing $\eta_i$ in $B_{i+1}$ (see Figure \ref{fig_virtual}).
Then for all $i\geq 1$, $\eta_i$ becomes a splitting germ in $\Sigma^{(i)}$.
We say that $\eta_i$ is a \emph{virtual splitting germ} of level $i$. 
The $\Sigma$-leaf semi-line defined by the concatenation of the leaf segments joining the base points of $\eta_0,\eta_1,\dots,\eta_i\dots$ is called the \emph{splitting semi-line} at $\eta_0$.
The following lemma easily follows.

\begin{lemma}\label{splitting-iterate}
Let $[y,y']_{\calk_v}$ be a segment and $i_0$ be such that $\eta_{i_0}$ is contained in $(y,y')$. 
Then for all $i> i_0$, and for any preimages $\Tilde y,\Tilde y'$ of $y,y'$ in $\Sigma^{(i)}$,
$p_{R^{(i)}}(\Tilde y)\neq p_{R^{(i)}}(\Tilde y')$.\qed
\end{lemma}

\begin{de}[\textbf{\emph{Liftable leaves}}]\label{def-unsplit}
A $\Sigma$-leaf segment $[x,y]_\Sigma$ is \emph{liftable} if for all $i\in\mathbb{N}$, there exists a $\Sigma^{(i)}$-leaf segment $[x_i,y_i]\in\Sigma^{(i)}$ that maps isometrically to $[x,y]_\Sigma$ under the natural map from $\Sigma^{(i)}$ to $\Sigma$. Otherwise $[x,y]_\Sigma$ is \emph{unliftable}.\\
An algebraic leaf $(\alpha,\omega)\in L^2(T)$ is \emph{liftable} (in the splitting process) if $[\alpha,\omega]_\Sigma$ is liftable.
\end{de}

\begin{rk} 
Consider an algebraic leaf $(\alpha,\omega)\in L^2(T)$.
If $[\alpha,\omega]_\Sigma$ lifts to $\Sigma^{(i)}$, then its lift is necessarily $[\alpha,\omega]_{\Sigma^{(i)}}$,
so $(\alpha,\omega)$ is liftable if and only if $[\alpha,\omega]_{\Sigma^{(i)}}$ embeds isometrically under
$h_{\Sigma^{(i)}}:\Sigma^{(i)}\ra \Sigma$.
\end{rk}

\begin{de}[\textbf{\emph{Clean band complexes}}]\label{clean}
Let $T$ be a tree of quadratic type, not compatible with any free splitting, let $R$ be a Grushko tree, and let $\Sigma:=\Sigma(T,R)$. 
\\ We say that $\Sigma$ is \emph{clean} if the following two conditions hold:
    \begin{itemize}
\item For all $i$, $\Sigma^{(i)}$ has the same number of orbits of splitting germs and splitting semi-lines as $\Sigma$. 
In particular, if a leaf segment of $\Sigma$  meets no splitting semi-line, it is liftable.
\item For every splitting semi-line $\sigma$ and every $x\in\sigma$, every transverse direction at $x$ can be pushed to infinity along $\sigma$ in the following sense: for all $y\in\sigma$, there exists a transverse direction $\mu_y$ at $y$ in $\Sigma$, such that $p_T(\mu_y)=p_T(\mu)$.
%
%
\end{itemize}
\end{de}


In Figure~\ref{fig_virtual}, the band complex $\Sigma$ on the left is not clean
since the red direction $\mu$ cannot be pushed to infinity along the splitting semi-line. The goal of the present section is to prove the following fact.

\begin{prop}\label{cleanify}
There exists $i\in\mathbb{N}$ such that $\Sigma^{(i)}$ is clean.
\end{prop}

We will start with the following lemma.

\begin{lemma}\label{tournicota}
Let $T$ be a tree of quadratic type with respect to $R$.
\\ Let $x\in\calk_v$ be a non-special point. Then the valence of $x$ in $\calk_v$ is finite.
\\ If $x$ is a special point, then the number of $G_v$-orbits of directions at $x$ in $\calk_v$ is finite. 
\end{lemma}

\begin{rk}
The conclusion is obvious if $p_T(x)$ has finite valence in $\ol T$, i.e.\ if $p_T(x)$ has finite stabilizer.
\end{rk}

\begin{proof}
Let $x\in\calk_v$, and let $\call$ be the complete $\Sigma$-leaf that contains $x$. We define a \emph{germ} at a point $y\in\call$ to be a germ of direction based at $y$ in the base tree $\calk_u$ containing $y$. If $d$ is such a germ, we define its diameter as the diameter of the connected component of $\calk_u\setminus\{y\}$ containing $d$. Let $\cald$ be the set of all germs at all points in $\call$. We equip $\cald$ with a graph structure, by putting an edge between two germs $d$ and $d'$ whenever they are contained in a common band. Then two germs belong to the same connected component of $\cald$ if and only if they have the same $p_T$-image. Since $\Sigma$ is of quadratic type, using Lemma~\ref{obs-germs},
we see that each connected component of $\cald$ is a line. Since there are finitely many orbits of directions in $T$, there are finitely many $G_\call$-orbits of such lines.

If the lemma is false, then there exists a sequence of germs $d_i$ at $x$ in $\calk_v$ whose diameters decrease to $0$: this follows from compactness of $\calk_v$ if $v$ is not in $V_\infty(R)$, and from compactness of $\calk_v/G_v$ in general. Since there are finitely many $G_\call$-orbits of lines in $\cald$, up to a subsequence, there exists a line $\lambda\subseteq\cald$ and $(g_i)_{i\in\mathbb{N}}\in G_\call^{\mathbb{N}}$ such that $d_i\in g_i.\lambda$ for all $i\in\mathbb{N}$. In other words $g_i^{-1}d_i\in \lambda$, and these germs have the same image in $T$.
Their basepoints $g_i^{-1}x$ are pairwise distinct: indeed, otherwise the two corresponding germs $g_i^{-1}d_i$ and $g_j^{-1}d_j$ would be contained in the same base tree, and therefore be equal since they have the same $p_T$-image; this would contradict the fact that the diameters of these germs decrease to $0$. Up to passing to a subsequence, we can assume that the directions $g_i^{-1}d_i$ go monotonously to one end of the line $\lambda$.

\begin{figure}[htbp]
\centering
\includegraphics{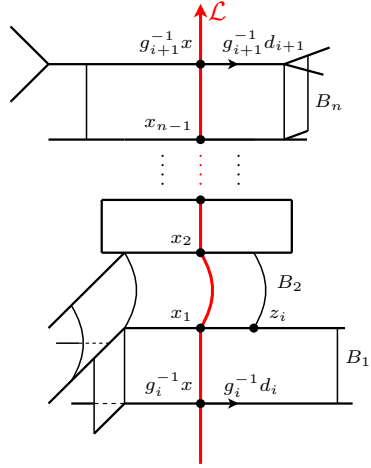} 
\caption{Finding a singular point $z_i$ `between' $g_i^{-1}d_i$ and $g_{i+1}^{-1}d_{i+1}$. Here, $x_{j_i}=x_1$.}
\label{fig-tournicota}
\end{figure}

We are going to find a singular point `between' $g_i^{-1}d_i$ and $g_{i+1}^{-1}d_{i+1}$ for all $i\in\mathbb{N}$; this is illustrated on Figure~\ref{fig-tournicota}. More precisely, let $B_{1},\dots,B_{n}$ be the sequence of bands corresponding to the segment joining $g_i^{-1}d_i$ to $g_{i+1}^{-1}d_{i+1}$ in $\lambda$, and let $x_{0}=g_i\m x, x_1,\dots,x_{n}=g_{i+1}\m x$ be the points of $\call$ in the bases of the bands $B_{j}$. Since $\text{diam}(g_{i+1}^{-1}d_{i+1})<\text{diam}(g_i^{-1}d_i)$, the holonomy of leaves cannot send the set of extremal points of $g_{i+1}^{-1}d_{i+1}$ to the set of extremal points of $g_i^{-1}d_i$; therefore, there exists a singular point $z_i$ and an index $j_i\in [0, n]$ such that $z_i,x_{j_i}$ lie in 
the same base tree of $\Sigma$ and $0<d_\Sigma(z_i,x_{j_i})\le \text{diam}(g_i^{-1}d_i)$. 
 
We claim that the points $y_i:=x_{j_i}$ belong to finitely many $G_\call$-orbits as $i$ varies. Indeed, they all belong to the convex hull $\calh$ in $\call$ of the $G_\call$-orbit of $x$. 
Since $G_\call$ coincides with a point stabilizer in $T$, it has finite Kurosh rank (\cite{Hor14-1})
so $\calh/G_\call$ is compact, which proves our claim. 

Thus, up to passing to a subsequence, there exist $y_0\in\call$ and $h_i\in G_\call$ such that $y_{i}=h_iy_0$. Let $\calk_u$ be the base tree of $\Sigma$ containing $y_0$.
The points $h_i^{-1}z_i$ are singular points in $\calk_u$, they are distinct from $y_0$ and converge to $y_0$. This contradicts that there are only finitely $G_u$-orbits of singular points in $\calk_u$ by Lemma \ref{singular}.   
\end{proof}

We make the following definition.

\begin{de}\label{de-core}
Let $\call$ be a complete $\Sigma$-leaf, and $G_\call$ be its stabilizer.
The \emph{core} $\calc_\call$ of the stabilizer of $\call$ is the convex hull in $\call$ of $\partial(G_\call,\calf_{|\call})$. 
\end{de}

For example, if $G_\call$ is isomorphic to $\mathbb{Z}$ and nonperipheral, then $\calc_\call$ will be a line (while $\call$ may contain branch points). Notice on the other hand that if $G_\call$ is peripheral, then $\calc_\call$ is at most a point.  We take the convention that $\calc_\call$ is empty if the leaf $\call$ has trivial stabilizer. 

\begin{proof}[Proof of Proposition \ref{cleanify}]
We first observe that the number of orbits of splitting germs cannot decrease along the splitting process. In addition, it is bounded: indeed, the number of orbits of \emph{singular leaves}, i.e.\ leaves that contain a singular point, is non-increasing, hence bounded, and the number of orbits of splitting germs is bounded by the number of orbits of directions at points in $T$ that are images of singular leaves. By replacing $\Sigma$ by $\Sigma^{(i)}$ with $i\in\mathbb{N}$ large enough, we can therefore assume that the number of orbits of splitting germs is stabilized. Then every splitting semi-line of $\Sigma^{(1)}$ comes from a splitting semi-line of $\Sigma$. The first assertion in Definition \ref{clean} is then satisfied  (to check the `in particular' statement, notice that if a leaf segment $I$ does not intersect any splitting semi-line, then it lifts to $\Sigma^{(1)}$, and the lift again does not intersect any splitting semi-line). Iterating this argument shows that $I$ lifts to all $\Sigma^{(i)}$.

We claim that every splitting semi-line $\sigma$ has compact intersection with the core $\calc_\call$ of the stabilizer of the complete $\Sigma$-leaf $\call$ that contains it. Otherwise, we can find $x\in\sigma$, and a sequence $(g_i)_{i\in\mathbb{N}}\in G_\call^{\mathbb{N}}$ going to infinity, such that $g_i.x\in\sigma$. Let $\eta_i$ be the virtual splitting germ at $g_i.x$, then $p_T(\eta_i)=p_T(\eta_0)$ for all $i\in\mathbb{N}$. By Lemma \ref{tournicota}, up to passing to a subsequence, we can assume that $\eta_i=g_i\eta_0$, or $\eta_i=g_ia_i\eta_0$ with $a_i\in G_x$ if $x$ is a special point. Since arc stabilizers in $T$ are trivial, we have $g_i=1$ or $g_i=a_i^{-1}$, contradicting the fact that $g_ix$ goes to infinity in $\sigma$. 

We also observe that for all sufficiently large $i\in\mathbb{N}$, no splitting semi-line in $\Sigma^{(i)}$ contains a special point. Indeed, each splitting semi-line contains at most finitely many special points because $\sigma\cap\calc_\call$ is compact. Notice that for all $i\in\mathbb{N}$, the splitting semi-line $\sigma_i\subseteq\Sigma^{(i)}$ corresponding to $\sigma$ projects via  $h_{\Sigma^{(i)}}:\Sigma^{(i)}\to\Sigma$ to the complement in $\sigma$ of its initial segment of length $i$. Therefore, by choosing $i\in\mathbb{N}$ large enough, $\sigma_i$ contains no special point. Since there are finitely many orbits of splitting semi-lines, there exists $i\in\mathbb{N}$ such that no splitting semi-line in $\Sigma^{(i)}$ contains a special point. 

Let $\sigma$ be a splitting semi-line in $\Sigma$. For $x\in\sigma$, we let $\text{Dir}_\Sigma(x)$ be the set of directions $d$ in $T$ such that there exists a transverse direction $\eta$ at $x$ with $p_T(\eta)=d$. Notice that $\text{Dir}_\Sigma(x)$ is finite for every $x\in\Sigma$ which is not special by Lemma \ref{tournicota}. We then let $\text{Dir}_\Sigma(\sigma)$ be the union of $\text{Dir}_\Sigma(x)$ over all $x\in\sigma$. We first claim that $\text{Dir}_\Sigma(\sigma)$ is finite. Indeed, $\sigma$ only meets finitely many singular points, because all points of $\sigma\setminus\calc_\call$ are in distinct orbits. Therefore, there exists a semi-line $\sigma_0=[x_0,\xi]\subseteq\sigma$ such that $\text{Dir}_\Sigma(x)$ does not depend on $x$ for $x\in\sigma_0$. Therefore $\text{Dir}_\Sigma(\sigma)$ is the finite union of all $\text{Dir}_\Sigma(x)$ with $x\in\sigma\setminus\sigma_0$ and of $\text{Dir}_\Sigma(\sigma_0)$. 

In addition, $\text{Dir}_{\Sigma^{(i)}}(\sigma_i)\subseteq\text{Dir}_\Sigma(h_{\Sigma^{(i)}}(\sigma_i))\subseteq\text{Dir}_\Sigma(\sigma)$. Thus, up to changing $\Sigma$ to some $\Sigma^{(i_0)}$ for $i_0$ large enough, for all splitting semi-lines $\sigma$ and all $i$, $\text{Dir}_{\Sigma^{(i)}}(\sigma_i)=\text{Dir}_\Sigma(\sigma)$.  

Let now $d\in\text{Dir}_\Sigma(\sigma)$. Then $d\in\text{Dir}_{\Sigma^{(i)}}(\sigma_i)$. So there is a representative $\mu_i\subseteq\Sigma^{(i)}$ based at a point $x_i\in\sigma_i$, and $h_{\Sigma^{(i)}}(\mu_i)$ shows that $\mu$ can be pushed to infinity along $\sigma$. Therefore, the second assertion in Definition \ref{clean} is also satisfied, showing that $\Sigma^{(i)}$ is clean.
\end{proof}

\section{The case of mixing trees}\label{sec-10}

A tree $T\in\overline{\calo}$ is \emph{mixing} if for all nondegenerate segments $I,J\subseteq T$, there exists a finite collection $g_1,\dots,g_k$ of elements in $G$ such that $J\subseteq g_1I\cup\dots\cup g_kI$.  The following lemma is due to Coulbois--Hilion \cite[Proposition 5.14]{CH14} for free groups.
We consider $R$ a Grushko tree, and we carry all notations $\Sigma,\Omega,\dots$ from Section \ref{sec-pruning}.

\begin{lemma}\label{omega-disc}
Let $T\in\overline{\calo}$ be mixing. 
\\ Then either $\Omega\cap \calk_v$ is totally disconnected for every vertex $v\in R$,
or the pruning process stops, and in particular, $T$ is of quadratic type.
\end{lemma}

\begin{proof}
Assume that there exists a nondegenerate interval $I\subseteq \Omega\cap \calk_v$ for some vertex $v\in R$, and let us prove that the pruning process halts after finitely many steps, which implies that $T$ is of quadratic type. 
Denote by $\Sigma^{(1)}\supset \Sigma^{(2)}\supset\dots$ the subsets of $\Sigma$ given by the pruning process,
and recall (Lemma~\ref{omega-rips}) that $\Omega=\cap_{i\geq 1} \Sigma^{(i)}$.
Define  a $G$-invariant function $\phi:\Sigma\ra \bbR_+\cup\{+\infty\}$ by $$\phi(x):=d_{\call_x}(x,\call_x\cap\Omega).$$ 
We claim that $\phi$ is bounded on $\Sigma$ (in particular $\phi(x)<+\infty$ for all $x\in \Sigma$). 
This will imply that the  pruning process halts after finitely many steps since
 $\sup_{\Sigma^{(i+1)}}\phi=(\sup_{\Sigma^{(i)}}\phi)-1$ as long as $\Sigma^{(i+1)}\neq \Sigma^{(i)}$
and $\sup_{\Sigma^{(i)}}\phi<+\infty$. 

We first claim that for each vertex $u\in R$,  the map $\phi$ is bounded on any interval 
 $J=[a,b]\subseteq \calk_u$, with $p_T(a),p_T(b)$ lying in $T$ (and not in $\overline T\setminus T$).  Indeed, since $T$ is mixing, up to subdividing $J$ into finitely many subintervals,
 we can assume that there exists $g\in G$ such that $p_T(J)\subseteq p_T(g I)$.
 Since fibers of $p_T$ are connected (Proposition~\ref{fibers-connected}), for each point $x\in J$, the distance in $\call_x$ from $x$ to $\call_x\cap g I$ is equal to
$d_R(u,g.v)$. Since $gI\subseteq \Omega$, the restriction of $\phi$ to $J$ is bounded. 

By Lemma \ref{properties}, the set $\caly$ (made of all points in $\Sigma$ that are erased by applying one step of the pruning process) is a finite union of orbits of open finite trees whose endpoints are in $T$ (and not in $\overline T\setminus T$).
It follows that $\phi$ is bounded on $\caly$. This implies that $\phi$ is bounded on $\Sigma$: 
for any point $x\in \Sigma\setminus \Omega$, there is a point $y\in \Term(\call_x)\subseteq \caly$ such that $\phi(y)=d_{\call_x}(y,\call_x\cap\Omega)\geq d_{\call_x}(x,\call_x\cap \Omega)=\phi(x)$.
\end{proof}

In the remainder of this section, we let $T\in\ol\calo$ be a mixing tree that is not compatible with any $(G,\calf)$-free splitting. Let $R$ be a Grushko tree, let $\Sigma:=\Sigma(T,R)$, and for all $i\in\mathbb{N}$, let $\Sigma^{(i)}$ and $R^{(i)}$ be the complex and the tree obtained by first iterating the pruning process until it halts (this may never happen), then applying the splitting process. Denote by $h_{\Sigma^{(i)}}:\Sigma^{(i)}\to\Sigma$ and $p_{R^{(i)}}:\Sigma^{(i)}\to R^{(i)}$ the corresponding maps.

\begin{lemma}\label{lem_separation_splitting}
Let $T\in\overline{\calo}$ be a mixing tree, and assume that $T$ is not compatible with any $(G,\calf)$-free splitting. 
Let $v\in R$ be a vertex, and let $[y,y']\subseteq \calk_v$ be a nondegenerate arc, with $y,y'\in\Omega$. For all $i\in\mathbb{N}$, let $y_i$ (resp. $y'_i$) be a point in $\Sigma^{(i)}$ that maps to $y$ (resp. $y'$) under $h_{\Sigma^{(i)}}$.
\\ Then there exists $i\in\mathbb{N}$ such that $p_{R^{(i)}}(y_i)\neq p_{R^{(i)}}(y'_i)$. 
\end{lemma}

\begin{figure}[htb]
  \centering
 \includegraphics{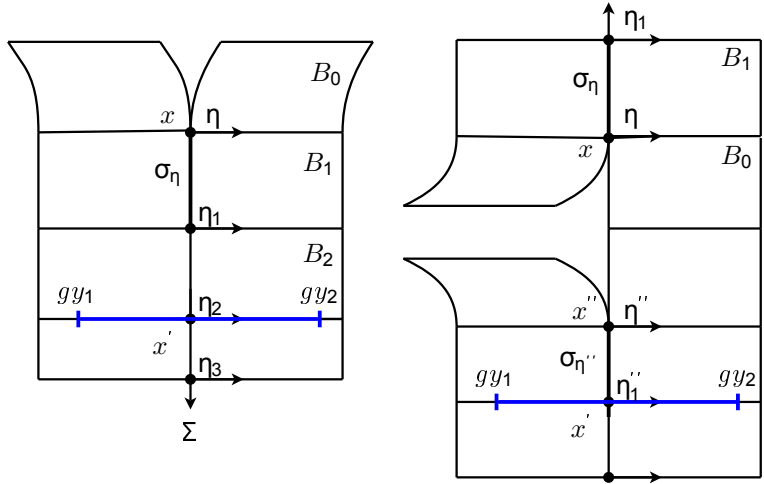}
  \caption{The two cases in the proof of Lemma \ref{lem_separation_splitting}}
  \label{fig_cantor}
\end{figure}

\begin{proof}
If 
the pruning process does not stop, Lemma \ref{omega-disc} shows that $y$ and $y'$ belong  to different connected components of $\Omega$. The conclusion then follows from Lemma~\ref{lem_separation}. 
Therefore, we can assume that the pruning process stops and does not affect the segment $[y,y']$. 
Thus, we reduce to the case where $\Sigma(T,R)$ is of quadratic type.

Let $\eta$ be a splitting germ in $\Sigma$ based at a point $x\in\calk_u$ for some vertex $u\in R$: this exists by Proposition \ref{existence-splitting}. 
Let  $[y_1,y_2]\subseteq [y,y']$ be a non-degenerate subsegment with $y_1, y_2$ distinct from $y$ and $y'$.
Since $T$ is mixing, there exists a segment $I_x\subseteq\calk_u$ representing $\eta$, and $g\in G$ such that  $p_T(I_x)\subseteq p_T(g.[y_1,y_2])$. 
Since fibers of $p_T$ are connected, we deduce that $I_x\times [u,g.v]_R\subseteq\Sigma$. We denote by $x'\in g.[y_1,y_2]$ the unique point in $I_x\times \{g.v\}$ in the same $\Sigma$-leaf as $x$. There are two cases to consider, illustrated on Figure~\ref{fig_cantor}.
\\
\\
\underline{\textbf{Case 1}}: The splitting leaf segment $\sigma_{\eta}$ is contained in $[x,x']_{\call_x}$. 
\\ Since $I_x\times [u,g.v]_R\subseteq\Sigma$, all the germs obtained from $\eta$ by following the holonomy along $[x,x']_{\call_x}$ 
are virtual splitting germs, 
and Lemma \ref{splitting-iterate} concludes.
\\
\\
\underline{\textbf{Case 2}}: The splitting leaf segment $\sigma_\eta$ is not contained in $[x,x']_{\call_x}$. 
\\ Let $x''$ be the point in the $\Sigma$-leaf segment $(x,x']_{\call_x}$ that is closest to $x$ and nonextremal in the base tree $\calk_{u''}$ that contains it (this exists as $x'$ is non-extremal in $\calk_{g.v}$). 
Then the germ $\eta''$ at $x''$ with the same projection as $\eta$ in $T$ 
is a splitting germ in $\Sigma$, and $\sigma_{\eta''}\subseteq [x'',x']_{\call_x}$. The conclusion then follows from the argument from the previous paragraph.
\end{proof}


\begin{cor}\label{coiffeur}
Let $T\in\overline{\calo}$ be a mixing tree that is not compatible with any $(G,\calf)$-free splitting. Then the number $$\max_{v\in V(R^{(i)})}\text{diam}(\calk_{v}^{(i)})$$ converges to $0$ as $i$ goes to $+\infty$.\qed
\end{cor}


\section{Arational trees}\label{sec-arat}

\subsection{Review}

\paragraph*{Definition.} Recall that a \emph{free factor} $H$ of $(G,\calf)$ is a vertex stabilizer of a $(G,\calf)$-free splitting,
and that $H$ is \emph{proper} if $H$ is nonperipheral (in particular nontrivial) and $H\neq G$.
We denote by $T_H$ the $H$-minimal subtree of $T$.

\begin{de}[\emph{\textbf{Arational tree}}]
A tree $T\in\overline{\mathcal{O}}$ is \emph{arational} if $T\in\partial \mathcal{O}$ and for every proper $(G,\mathcal{F})$-free factor $H\subseteq G$, $H\actson T_H$ is a Grushko $(H,\mathcal{F}_{|H})$-tree.
\end{de} 

We will denote by $\mathcal{AT}$ the subspace of $\overline\calo$ made of arational trees. We mention that every arational tree $T$ has dense orbits, and is in fact mixing \cite{Rey12,Hor14-3}. 

\begin{figure}[htb]
  \centering
\includegraphics{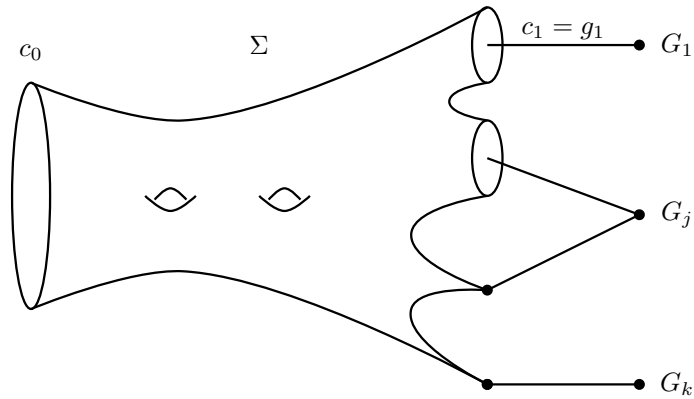}
  \caption{A geometric $(G,\calf)$-graph of groups.}
  \label{fig-arat-surf}
\end{figure}

\paragraph*{Arational surface trees.} We now review the definition of a special class of arational trees (see Figure \ref{fig-arat-surf}). A \emph{geometric $(G,\mathcal{F})$-graph of groups} is a graph of groups $\mathcal{G}$ with fundamental group isomorphic to $G$, obtained in the following way. Let $g\ge 0$, let $\Sigma$ be a compact, connected $2$-orbifold of genus $g$ with conical singularities, with nonempty boundary.  One of the vertex groups of $\mathcal{G}$ is the fundamental group of the orbifold $\Sigma$, and the other vertex groups are the $G_i$'s. To define edges,
consider a set $\calb$ consisting of the whole set of conical singularities of $\calo$, together with a subset of the set of boundary components.
We view the elements of $\calb$ as cyclic subgroups of $\pi_1(\Sigma)$.
To each cyclic group $C\in \calb$ we assign a peripheral subgroup $G_{i_C}$ and an embedding $\phi_C:C\ra G_{i_C}$ (allowed to be onto).
For each $C\in \calb$, we put an edge joining $\Sigma$ to $G_{i_C}$, with edge group $C$, and with monomorphisms given by the inclusion and $\phi_C$.
Boundary components of $\Sigma$ not in $\calb$ have no edge attached. They represent nonperipheral conjugacy classes; we say that they are
\emph{unused}. Note there has to be at least one unused boundary component since otherwise, 
$\pi_1(\calg)$ is freely indecomposable relative to the subgroups $G_i$.

\begin{de}[\emph{\textbf{Arational surface tree}}]\label{dfn_arat-surf}
A tree $T\in\ol\calo$ is an \emph{arational surface tree} if it splits as a graph of actions over a geometric $(G,\mathcal{F})$-graph of groups with a single unused boundary curve, so that the action corresponding to the vertex associated to the orbifold $\Sigma$ is dual to an arational measured foliation on $\Sigma$. 
\end{de}

It was proved in \cite[Section 4.1]{Hor14-3} that arational surface trees are indeed arational. Conversely, every arational tree in $\overline{\calo}$ is either relatively free, or else is an arational surface tree \cite{Rey12,Hor14-3}. Notice that any arational surface tree has (up to conjugacy) a unique nonperipheral elliptic subgroup, which is cyclic and generated by the unused boundary curve of the corresponding orbifold.  

\subsection{Arational trees and alignment-preserving maps} A map $T\to T'$ between two $\mathbb{R}$-trees is \emph{alignment-preserving} if it sends segments to segments (in this case we say that $T'$ is a \emph{collapse} of $T$, or that $T$ is a \emph{refinement} of $T'$). Notice that if $T$ is a tree with dense $G$-orbits, then any collapse of $T$ also has dense $G$-orbits.

\begin{lemma}\label{arational-collapse}
Let $T,\ol T\in\ol\calo$, such that there is an alignment-preserving map $\pi:T\to\ol T$.
\\ Then $T\in\mathcal{AT}$ if and only if $\ol T\in\mathcal{AT}$.
\end{lemma}

\begin{proof}
First assume that $T\notin\mathcal{AT}$. Then there exists a proper $(G,\calf)$-free factor $A$ such that $T_A$ is not a Grushko $(A,\calf_{|A})$-tree. Then $\ol T_A$ (which is a collapse of $T_A$) is not a Grushko $(A,\calf_{|A})$-tree either, so $\ol T\notin \mathcal{AT}$.

Assume now that $T\in\mathcal{AT}$. Assume towards a contradiction that $\ol T\notin\mathcal{AT}$. Then there exists a proper $(G,\calf)$-free factor $A$ such that $A\actson \ol T_A$ is not Grushko. Since $T\in\mathcal{AT}$, the $A$-minimal subtree $T_A$ is a Grushko $(A,\calf_{|A})$-tree, and it maps onto $\ol T_A$. Therefore we can find a proper $(G,\calf)$-free factor $B\subseteq A$ such that $\ol T_B$ is reduced to a point $x\in\ol T$. Since $B$ does not fix any point in $T$, the subtree $\pi\m(x)$ is not reduced to a point (and not equal to $T$). Since $gB\cap B=\emptyset$ for all $g\notin\text{Stab}(x)$, 
the tree $T$ is not mixing, a contradiction. Therefore $\ol T\in\mathcal{AT}$.
\end{proof}

\subsection{Arationality and dual laminations}

Arational trees can also be characterized in terms of their dual laminations in the following way.  

\begin{lemma}\label{arat-lam}
A tree $T\in\partial\mathcal{O}(G,\mathcal{F})$ is arational if and only if no leaf of $L^2(T)$ is carried by a proper $(G,\mathcal{F})$-free factor.
\end{lemma}

\begin{proof}
If $T$ is not arational, there is a proper $(G,\calf)$-free factor $H\subseteq G$ such that 
 $T_H$ is not a Grushko $(H,\calf_H)$-tree. Fix a Grushko $(G,\calf)$-tree $R$. Then the minimal $H$-invariant subtree $R_H$ is a Grushko $(H,\calf_{|H})$-tree. By Lemma \ref{lamination-nonempty}, the dual lamination of $T_H$ is a nonempty subset of $\partial^2 R_H$. Using the fact that the following diagram commutes, we deduce that every leaf in $L^2(T_H)$ yields a leaf of $L^2(T)$ carried by the proper $(G,\calf)$-free factor $H$. 
 
$$\xymatrix{
  L^2(T_H) \ar@{^(->}[r] \ar@{^(->}[d] &\partial^2(H,\calf_{|H})\ar@{^(->}[d] \\
L^2(T)\ar@{^(->}[r] & \partial^2(G,\calf)
}$$ 
 
Conversely, assume that $T$ is arational, and consider a proper $(G,\calf)$-free factor $H\subseteq G$.
Since the minimal $H$-invariant subtree $T_H\subseteq T$ is a Grushko tree, 
the natural map $\partial R_H\ra \partial T_H$ is injective, and agrees with the restriction of $\calq$.
It follows that $\calq_{|\partial R_H}$ is injective, so by  Lemma \ref{fibreQ}, 
$L^2(T)\cap \partial^2 R_H=\es$, i.e.\ no algebraic leaf of $T$ is carried by $H$.
\end{proof}

\subsection{Action of $\calz$-factors on arational trees}

We recall that a \emph{$\mathcal{Z}$-splitting} of $(G,\calf)$ is a minimal, simplicial $(G,\calf)$-tree in which all peripheral subgroups are elliptic, and in which all edge stabilizers are either trivial, or cyclic and nonperipheral. By analogy with free factors, we define a \emph{$\calz$-factor} of $(G,\calf)$ to be a subgroup of $G$ that arises as a vertex stabilizer in some $\calz$-splitting of $(G,\calf)$. Arationality of a tree $T$ tells us that for every proper $(G,\calf)$-free factor $A$, the action $A\actson T_A$ is simplicial, and we want to extend this to $\calz$-factors. 

\begin{prop}\label{arat-Z}
Let $A\subseteq G$ be a proper $\calz$-factor of $(G,\calf)$, and let $T\in\mathcal{AT}$. 
\\ Then $A\actson T_A$ is simplicial.
\end{prop}

Notice however that $T_A$ might fail to be a Grushko $(A,\calf_{|A})$-tree when $T$ is arational surface, because there are $\calz$-factors of $(G,\calf)$ that contain the unused boundary curve of $T$. Our proof of Proposition \ref{arat-Z} relies on the following statement which does not use arationality, 
and whose proof is based on an argument of Reynolds \cite[Lemma 3.10]{Rey11}.

\begin{prop}\label{Reynolds}
Let $T\in\overline{\mathcal{O}}$ be a tree with trivial arc stabilizers. 
Let $H\subseteq G$ be a subgroup and $H\actson S$ 
an action on a simplicial tree.
Consider two distinct vertices $v_a, v_b\in S$, and $A\subseteq H_{v_a}$, $B\subseteq H_{v_b}$ be two subgroups of $H$ fixing $v_a, v_b$ respectively, 
and having finite Kurosh rank as subgroups of $(G,\calf)$.
\\ Assume that both $A\actson T_A$ and $B\actson T_B$ have dense orbits. 
\\ Then for all nondegenerate segments $I\subseteq T_A\cap T_B$, and all edges $e\subseteq [v_a,v_b]$, there exists $g\in (A\cup B)\cap G_e$ such that $gI\cap I$ is nondegenerate.
\end{prop}

Before proving Propositions \ref{Reynolds} and \ref{arat-Z}, we start by mentioning a few consequences of Proposition~\ref{Reynolds}. In the case where edge stabilizers in $S$ are cyclic or peripheral, we deduce the following result.

\begin{cor}\label{cor-Reynolds}
Let $T\in\overline{\mathcal{O}}$ be a tree with trivial arc stabilizers. 
Let $H<G$ be a subgroup and $H\actson S$ 
an action on a simplicial tree whose edge stabilizers are either cyclic or peripheral (possibly trivial).
Consider two distinct vertices $v_a, v_b\in S$, and $A\subseteq H_{v_a},B\subseteq H_{v_b}$ be two subgroups of $H$ fixing $v_a, v_b$, and having finite Kurosh rank in $(G,\calf)$.
\\
If both $A\actson T_A$ and $B\actson T_B$ have dense orbits, then $T_A\cap T_B$ contains at most one point.
\end{cor}

\begin{proof}
Assume by contradiction that there exists a nondegenerate interval $I\subseteq T_A\cap T_B$. Consider an edge $e\subseteq [v_a,v_b]_S$.  Since $G_e$ is cyclic or peripheral, one can change $I$ to a subsegment such that
for all $g\in G_e\setminus \{1\}$, the intersection $gI\cap I$ contains at most one point. This contradicts Proposition~\ref{Reynolds}.
\end{proof}

We mention two more corollaries of Proposition~\ref{Reynolds}. A \emph{transverse family} in a tree  $T\in\overline{\mathcal{O}}$ is a $G$-invariant collection of subtrees of $T$ such that the intersection between any two subtrees in the collection contains at most one point. 

\begin{cor} \label{Reynolds-famille}
Let $T\in\overline{\mathcal{O}}$ be a tree with trivial arc stabilizers, and $G\actson S$ an action
on a simplicial tree whose edge stabilizers are either cyclic or peripheral (possibly trivial). For each vertex $v\in S$ with stabilizer $G_v$, let $T_v\subseteq T$ be the minimal subtree of $G_v$.
\\
Then the collection of all $T_v$ such that $G_v\actson T_v$ has dense orbits is a transverse family of $T$.
\end{cor}

\begin{proof}
This is an immediate consequence of Corollary~\ref{cor-Reynolds}. The fact that $G_v$ has finite Kurosh rank follows for instance 
from \cite[Lemma 1.12]{Gui_actions}.
\end{proof}

\begin{cor} \label{Reynolds-famille2}
Let $S,T$ be as in the above corollary, let $v\in S$ be a vertex, and let $A\subseteq G_v$ be a $(G_v,\calf_{|G_v})$-free factor. Assume that $A\actson T_A$ has dense orbits.
\\
Then the collection of all $G$-translates of $T_A$ is a transverse family of $T$.
\end{cor}

\begin{proof}
If $g\in G\setminus G_v$, then one can apply Proposition \ref{Reynolds} to $B=A^g$ and get that $T_A\cap gT_A$ is at most a point.
If $g\in G_v\setminus A$, apply Proposition \ref{Reynolds} to a free splitting $S'$ of $H=G_v$ in which $A$ is a free factor to get
that $T_A\cap gT_A$ is at most a point. 
\end{proof}

\begin{proof}[Proof of Proposition \ref{Reynolds}]
Assume towards a contradiction that there exists a nondegenerate segment $I\subseteq T_A\cap T_B$ and an edge $e\subseteq [v_a,v_b]$ of $S$ such that for all $g\in (A\cup B)\cap G_e$, the intersection $gI\cap I$ contains at most one point. Since $T_A$ and $T_B$ have dense orbits, one can find finite subtrees $K_A\subseteq T_A$ and $K_B\subseteq T_B$ of volume strictly smaller than $|I|/2$ such that $I$ is covered by finitely many $A$-translates of $K_A$, and $I$ is also covered by finitely many $B$-translates of $K_B$ (this follows for instance from Corollary \ref{bbt}, using the fact that $A$ and $B$ have finite Kurosh rank). Moreover, one can assume that $K_A,K_B$ and $I$ are pairwise disjoint. 
We now build a system of partial isometries on $K_A\cup K_B\cup I$ in the following way. Subdivide $I$ into finitely many non-overlapping subsegments $I_1,\dots, I_k$ such that for all $i\in\{1,\dots,k\}$, there exists $a_i\in A$ with $a_iI_i\subseteq K_A$. For all $i\in\{1,\dots,k\}$, let $\phi_i$ be the partial isometry from $I_i$ to $a_iI_i$ induced by $a_i$. Similarly, we subdivide $I$ into finitely many non-overlapping subsegments $I'_1,\dots,I'_l$ and get partial isometries $\psi_j$ from $I'_j$ to a subset $b'_j I'_j\subseteq K_B$ for some $b'_j\in B$. The volume of $K_A\cup K_B\cup I$ is strictly smaller than $2|I|$, while the total length of the bases of the partial isometries $\phi_i$ and $\psi_j$ equals $2|I|$. Therefore, by \cite[Proposition 6.1]{GLP94}, our system of partial isometries does not have independent generators: this means that one can find a cyclically reduced word $w$ on the partial isometries $\phi_i^{\pm 1}$ and $\psi_j^{\pm 1}$ which fixes a nondegenerate segment of $I$. 
Since the segments $I_i$ do not overlap, any maximal subword of $w$ on the letters  $\{\phi_1^{\pm 1},\dots,\phi_k^{\pm 1}\}$ 
has length $2$ and is of the form
$\phi_i\m \phi_j$ with $j\neq i$, and $a_i\m a_j\in A\setminus\{1\}$  (a similar argument applies for $B$).
This implies that the element $g\in\langle A,B\rangle$ represented by $w$ is of the form $g=u_1v_1u_2\dots u_k v_k$ with $u_i\in A\setminus\{1\}$ and $v_i\in B\setminus\{1\}$ for all $i\in\{1,\dots,k\}$ (except possibly that $u_1=1$ or $v_k=1$), and the $u_i$'s and $v_i$'s all map a non-degenerate segment of $I$ into $I$. 
In particular, we have $u_i\in A\setminus G_e$ and $v_i\in B\setminus G_e$ for all $i\in\{1,\dots,k\}$.
A ping-pong argument in the tree $S$ then shows that $g\neq 1$.
On the other hand, since the partial isometry defined by $w$ fixes a non-degenerate segment in $I$, the element $g$ fixes an arc in $T$,
which contradicts that $T$ has trivial arc stabilizers.
\end{proof}

Before proving Proposition~\ref{arat-Z}, we recall some more terminology. A \emph{transverse covering} in a tree $T$ is a transverse family made of nondegenerate subtrees such that every segment in $T$ can be covered by finitely many subtrees from the family. The \emph{skeleton} of a transverse covering is the bipartite simplicial tree whose vertices are given by the collection of all nondegenerate trees $Y$ in the family,  together with the collection of all points that belong to several distinct trees in the family, the vertex  $v_x$ (corresponding to a point $x$) being joined by an edge $\eps_{x,Y}$ to the vertex $v_Y$ (corresponding to a subtree $Y$) whenever $x\in Y$, see \cite{Gui_limit}. If the $G$-action on $T$ is minimal, then the $G$-action on the skeleton of any transverse covering of $T$ is always minimal \cite[Lemma~4.9]{Gui_limit}. 

\begin{proof}[Proof of Proposition \ref{arat-Z}]
Up to enlarging $A$, we can assume that $A=G_v$ is the full stabilizer of a vertex $v$
in a $\calz$-splitting $S$ of $(G,\calf)$.
Assume by contradiction that $A\actson T_A$ is not simplicial.
By considering the Levitt decomposition of $T_{A}$ as a graph of actions with dense orbits (which has trivial edge stabilizers), we can find an $(A,\mathcal{F}_{A})$-free factor $B\subseteq A$ such that $B\actson T_{B}$ has dense orbits, with $T_{B}$ non-trivial.

By Corollary \ref{Reynolds-famille2}, the family $\{gT_{B}\}_{g\in G}$ is a transverse family,
and hence it is a transverse covering because $T$ is mixing. 
Let $U$ be the skeleton of this transverse covering. Note that no edge $\eps_{x,Y}$ in $U$ has trivial stabilizer since otherwise, minimality $U$ would imply that $B$ is contained in a proper $(G,\calf)$-free factor, contradicting the arationality of $T$. 

Assume first that $T$ is relatively free. 
Then we claim that each vertex $v_x$ is a terminal vertex of $U$.
Since $U$ is minimal, this will imply that $U$ is reduced to a point, i.e.\ $T_B=T$ and $B=G$, a contradiction.
To prove the claim, let $\eps_{x,Y},\eps_{x,Y'}$ be two edges of $U$ incident on $v_x$, and assume without loss of generality that $Y=T_B$.
Let $g\in G$ be such that $gY=Y'$. Then since edge stabilizers of $U$ are non-trivial, $G_x\cap B$ 
and $G_x\cap B^g$ are two non-trivial subgroups of the peripheral group $G_x$. 
In particular, $G_x$ intersects nontrivially $A$ and $A^g$. Since $A=G_v$ is a vertex stabilizer in the $\Z$-splitting $S$, it follows that $g\in A$: 
indeed, the group $G_x$ fixes a vertex $p\in S$,
and $G_x\cap A$ fixes the segment $[p,v]_S$, so $p=v$ since no non-trivial peripheral element fixes an edge of $S$.
Similarly, $G_x\cap A^g$ being non-trivial shows that $p=gv$, so $g$ fixes $v$ i.e.\ $g\in A$.
Since $B$ is an $(A,\calf_{A})$-free factor and $g\in A$, a similar argument now implies that $g\in B$.
Since $B$ preserves $Y=T_B$, we get $Y'=Y$ and  $\eps_{x,Y}=\eps_{x,Y'}$ which proves our claim.

If $T$ is not relatively free, then it is arational surface.
The argument above applies as soon as 
the stabilizer of the vertex $v_x\in U$ is peripheral.
If not, then $G_{v_x}=\grp{b}$ is a cyclic group corresponding to the special boundary component.
Since edge stabilizers of $U$ are non-trivial, every edge $\eps_{x,Y}$ contains a power $b^k$ of $b$.
Since $U$ is minimal, this implies that $b^k$ is an edge stabilizer of a $\calz$-splitting of $G$,
and therefore $b^k$ is contained in a proper $(G,\calf)$-free factor \cite[Lemma~5.11]{Hor14-1}, contradicting that $T$ is arational.
\end{proof}

\section{Reconstructing $L^2(T)$ from one leaf when $T$ is arational}\label{sec-11}

\subsection{Peritransitive closure of a set of algebraic leaves}

\begin{de}\label{dfn_peritransitively_closed}
  An algebraic lamination $L\subseteq\partial^2(G,\calf)$ is said to be
  \emph{transitively closed} (also called \emph{diagonally closed} in
  \cite{CHR11}) if for all $\alpha,\beta,\gamma\in\partial (G,\calf)$,
  \begin{equation}\label{eq1}
    (\alpha,\beta),(\beta,\gamma)\in L\Rightarrow (\alpha,\gamma)\in L\text{ if $\alpha\neq \gamma$.} 
  \end{equation}
  It is \emph{peripherally closed} if for every peripheral group
  $G_v$, all $g,g'\in G_v\setminus\{1\}$, and all
  $\alpha,\alpha'\in\partial (G,\calf)$,
  \begin{equation}
    (\alpha,g\alpha)\in L, (\alpha',g'\alpha')\in L \Rightarrow (\alpha,\alpha')\in L\text{ if }\alpha\neq \alpha'.\label{eq2}
  \end{equation}
\end{de}

\begin{de}[\textbf{Peritransitive closure}]\label{dfn_peritransitive_closure}
  Given a subset $X\subseteq\partial^2(G,\calf)$, we denote by
  $\overline{\mathcal{P}}(X)$ the smallest transitively closed and
  peripherally closed algebraic lamination that contains $X$. We call
  it the \emph{peritransitive closure} of $X$.
\end{de}

\begin{lemma}\label{lem_peritransitively}  
  Let $T\in\ol\calo$ be a tree with dense orbits. Then $L^2(T)$ is peritransitively closed.
\end{lemma}

\begin{proof}
Since $L^2(T)$ corresponds to the fibers of the map $\calq$, it is transitively closed.

To prove that it is peripherally closed, consider $(\alpha,g\alpha)\in L^2(T), (\alpha',g'\alpha')\in L^2(T)$ with $g,g'\in G_v$.
Then $g$ fixes $\calq(\alpha)=\calq(g\alpha)$, and since $g$ fixes no arc in $T$, 
$\calq(\alpha)$ is the unique fixed point $x_v$ of $G_v$ in $T$. The same argument also shows that $\calq(\alpha')=x_v$.
By Proposition \ref{equality-Q}, $(\alpha,\alpha')\in L^2(T)$ if $\alpha\neq \alpha'$.  
\end{proof}

\begin{rk}
Notice that if an algebraic lamination $L$ is peritransitively closed, then 
\begin{eqnarray*}
(\alpha,g\alpha)\in L \text{ for some } g\in G_v\setminus\{1\} \Rightarrow (\alpha,h\alpha)\in L\text{ for all }h\in G_v\setminus\{1\} 
\end{eqnarray*}
as follows from (\ref{eq2}) applied to $\alpha'=h\alpha$ and $g'=hgh^{-1}$. 
If $G_v$ is infinite, this implies the following condition:
\begin{equation}\renewcommand{\theequation}{\addtocounter{equation}{-1}\ref{eq2}'}
(\alpha,g\alpha)\in L \text{ for some } g\in G_v\setminus\{1\} \Rightarrow (\alpha,v)\in L\text{ for all }h\in G_v\setminus\{1\} 
\end{equation}
because $L$ is closed and $(\alpha,g_i\alpha)$ converges to $(\alpha,v)$ for any sequence of distinct elements $g_i\in G_v\setminus\{1\}$. 
Conversely, if $(\alpha,v)\in L$, then $G$-invariance of $L$ implies that $(g\alpha,v)\in L$ for all $g\in G_v$, so if $L$ is transitively closed, we have $(\alpha,g\alpha)\in L$ for all $g\in G_v$. 

Therefore, if all the groups $G_i$ were infinite, we could replace (\ref{eq2}) by (\ref{eq2}') in Definition~\ref{dfn_peritransitively_closed}.
The reason for our more mysterious requirement (\ref{eq2}) is that we also need to deal with the case of finite peripheral groups.
\end{rk}

\subsection{Main result and strategy of the proof}

 The goal of the present section is to prove the following theorem, which allows to reconstruct the dual lamination of an arational tree from a single leaf.

\begin{theo}\label{lamination-arational}
Let $T\in\mathcal{AT}$, and let $l_0\in L^2(T)$ be a leaf that is not carried by any subgroup of $G$ that is elliptic in $T$. Then $L^2(T)=\overline{\mathcal{P}}(l_0)$.
\end{theo}

Note that the assumption on $l_0$ is empty if $T$ is relatively free. On the other hand, if $T$ is not relatively free, then $T$ is arational surface, and the assumption on $l_0$ is saying that $l_0$ should not be a periodic leaf of the form $(c^{-\infty},c^{+\infty})$ where $c\in G$ is represented by a loop in the free boundary curve of the underlying orbifold.\\

Start with a Grushko tree $R$ all whose edges have length $1$. 
Consider the band complexes $\Sigma^{(i)}$ and the corresponding Grushko tree $R^{(i)}$ 
obtained by applying the pruning process on $\Sigma=\Sigma(R,T)$, followed by the splitting process if the pruning process stops. 
Up to changing $R$ to some $R^{(i)}$, we may assume that either the pruning process does not stop,
or that the band complex $\Sigma$ is of quadratic type, and $\Sigma$ is clean for the splitting process (Definition \ref{clean} and Proposition \ref{cleanify}).

Recall from Definition \ref{def-unsplit} that an algebraic leaf $(\alpha,\omega)\in L^2(T)$ is \emph{liftable} (in the splitting process) if
the natural map $\Sigma^{(i)}\ra \Sigma$ sends $[\alpha,\omega]_{\Sigma^{(i)}}$ bijectively to $[\alpha,\omega]_{\Sigma}$; 
this definition also makes sense if the pruning process does not halt, in which case all leaves of $L^2(T)$ are liftable. 
Let $l_0\in L^2(T)$ be as in Proposition~\ref{lamination-arational}. Proposition~\ref{lamination-arational} is a consequence of the following three lemmas (if the pruning process does not halt, Lemmas~\ref{lem_unsplit1} and~\ref{lem_unsplit3} are automatic since every leaf is liftable).

\begin{lemma}\label{lem_unsplit1}
The lamination $\ol\calp(l_0)$ contains a liftable leaf $l_1\in L^2(T)$ which is not carried by any subgroup of $G$ that is elliptic in $T$. 
\end{lemma}


\begin{rk}
It would actually be enough to just prove that $\ol\calp(l_0)$ contains a liftable leaf. Indeed, with a bit more work, one can show that a liftable leaf is never carried by a subgroup of $G$ that is elliptic in $T$. This is automatic if $T$ is relatively free. Otherwise $T$ is arational surface, and in this case one can show that the leaves carried by the free boundary component of the underlying orbifold are not liftable: roughly, this is because the foliation on the orbifold contains half-leaves with one endpoint on the boundary, and these will serve as splitting semilines in the splitting process. However, we will not check the details of this claim, as this is not needed in the present paper.  
\end{rk}


\begin{lemma}\label{lem_unsplit2}
The lamination $\ol\calp(l_1)$ contains all liftable leaves of $L^2(T)\cap\partial_\infty(G,\calf)^2$.
\end{lemma}

\begin{lemma}\label{lem_unsplit3}
Every leaf in $L^2(T)$ is in the peritransitive closure of the set of liftable leaves in $L^2(T)\cap\partial_\infty(G,\calf)^2$.
\end{lemma}

We will start by proving Lemma \ref{lem_unsplit2}, and give the additional arguments needed in the case where $T$ is of quadratic type afterwards. 

\subsection{Proof of Lemma \ref{lem_unsplit2}}

We define the \emph{systole} of a Grushko tree $R$ as the smallest translation length in $R$ of a nonperipheral element of $G$. For all $i\in\mathbb{N}$, denote by $R^{(i)}$ the Grushko tree associated to the band complex $\Sigma^{(i)}$, and recall that there are natural $G$-equivariant maps $f_i:R^{(i+1)}\ra R^{(i)}$ sending vertex to vertex and edge to edge (but $R^{(i)}$ usually has many vertices of valence 2). We set the lengths of the edges of the trees $R^{(i)}$ to $1$ so that the map $f_i$ is $1$-Lipschitz. 

\begin{lemma}\label{systole}
The systole of $R^{(i)}$ converges to $+\infty$ as $i$ goes to $+\infty$. 
\end{lemma}

\begin{proof}
Assume towards a contradiction that there exists $M>0$ such that for all $i\in\mathbb{N}$, 
the systole of $R^{(i)}$ is bounded from above by $M$.
Let $Y_i\subseteq R^{(i)}$ be the (non-empty) set of all 
points in $R^{(i)}$ that are moved by at most $M$ by some non-peripheral element of $G$.
Note that $f_i(Y_{i+1})\subseteq Y_{i}$. Since the number of orbits of edges of $R^{(i)}$ goes to infinity, 
$Y_{i}\neq R^{(i)}$ for $i$ large enough.

We claim that the set $\calh_i$ of stabilizers of the connected components of $Y_i$ is a system of non-peripheral free factors.
Indeed, if $z$ lies in a connected component $ Z$ of $Y_i$, 
and if a non-peripheral element $g\in G$ satisfies $d_{R^{(i)}}(z,gz)\leq M$, then the axis of $g$ is contained in $Z$, showing that the stabilizer of $Z$ is non-peripheral.
By collapsing to a point all the connected components of $Y_i$ in $R^{(i)}$, 
one gets a $(G,\calf)$-free splitting $\ol R^{(i)}$ in which each group in $\calh_i$ is a point stabilizer, which proves our claim.
Since $f_i(Y_{i+1})\subseteq Y_i$, the map $f_i$ induces a $G$-equivariant map $\ol R^{(i+1)}\ra \ol R^{(i)}$, in particular the collection of all point stabilizers of $R^{(i+1)}$ is a free factor system of $(G,\calf)$ which is contained in $\calh_i$ (i.e. every free factor that fixes a point in $R^{(i+1)}$ is contained in a free factor in $\calh_i$). Since there is a bound on the length of a decreasing chain of free factor systems of $(G,\calf)$, 
the free factor system $\calh_i$ is independent of $i$ for $i$ large enough.
We denote by $\calh$ this system of proper free factors.

Fix $H\in \calh$, and consider  $R_H^{(i)}\subseteq R^{(i)}$ the minimal $H$-invariant subtree. 
The number of $H$-orbits of edges in $R_H^{(i)}$ cannot tend to infinity, since otherwise, it would contain for $i$ large enough
a segment of length $>M$ with no branch point of $R^{(i)}$, contradicting that $R_H^{(i)}\subseteq Y_i$. Since $R_H^{(i)}\subseteq f_i(R_H^{(i+1)})$ for all $i\in\mathbb{N}$, this implies that for all $i$ large enough, $f_i$ induces an isometry $R_H^{(i+1)}\ra R_H^{(i)}$.
It follows that there exists a (non-peripheral) element $h\in H$ and $l\leq M$ such that for all $i$ large enough, $||h||_{R^{(i)}}=l$.

From the construction of the graph $R^{(i)}$ from the band complex $\Sigma^{(i)}$, 
this implies that there exists a vertex $v_i$ in $R^{(i)}$, a point $x_i\in \calk_{v_i}\subseteq\Sigma^{(i)}$, 
and a path in $\Sigma^{(i)}$ from $x_i$ to $h x_i$ of the form $\alpha_1\beta_1\dots\alpha_l\beta_l$, 
where each $\alpha_j$ is a leaf segment in a band of $\Sigma^{(i)}$, and each $\beta_j$ is a segment contained in some $\calk_{v_j}\subseteq\Sigma^{(i)}$. 
In particular, for all $i\in\mathbb{N}$, we have $||h||_T\le l D_i$, where $D_i:=\max_{v\in V(R^{(i)})}\text{diam}(\calk_{v})$. Since $T$ is mixing, 
Proposition~\ref{coiffeur} implies that $D_i$ converges to $0$ as $i$ goes to $+\infty$. 
It follows that $||h||_T=0$. Since $h$ is non-peripheral and contained in the proper free factor $H$, this contradicts that $T$ is arational.
\end{proof}

Let $R$ be a Grushko tree, let $F\subseteq R$ be a bounded closed nonempty subtree, such that for all $g\in G$, either $gF=F$, or else $gF\cap F=\emptyset$. These conditions imply that if $gF=F$, then $g$ is peripheral, and in addition the stabilizer $G_F$ of $F$ is either trivial, or equal to $G_v$ for some vertex $v\in R$.

Although not formulated in these terms, our next lemma will show the connectivity of some Whitehead graph of $l_1$ around $F$.
Let $T\in\mathcal{AT}$, and let $l_1\in L^2(T)$ be a liftable leaf that is not carried by any subgroup of $G$ that is elliptic in $T$. Let $E_F$ be the set of edges of $R\setminus F$ incident on $F$.
Let $\sim_F$ be the smallest equivalence relation on $E_F$ such that $e\sim_F e'$ whenever
\begin{equation}
\exists g\in G,\quad e\cup e'\subseteq g.(l_1)_R,\label{eq4}
\end{equation}
or 
\begin{eqnarray} 
\exists g,g'\in G_F\setminus\{1\}, \quad 
        e\sim_F ge \text{ and }
        e'\sim_F g'e'
\label{eq5} \end{eqnarray} 

\begin{rk}\label{rk-equiv}
Given two edges $e,e'\in E_F$, we define a \emph{basic chain} joining $e$ to $e'$ as a sequence of edges $e=e_1,\dots,e_n=e'\in E_F$
such that for each $i\in\{1,\dots,n-1\}$, there exists $g_i\in G$ such that  $e_i\cup e_{i+1}\subseteq g_i.(l_1)_R$. Given $e\in E_F$, a \emph{basic loop at $e$} is a basic chain for which $e_1=e$, and there exists $g\in G_F\setminus\{1\}$ such that $e_n=ge$.
We claim that if $e\sim_F e'$, then either 
\begin{itemize}
\item there is a basic chain of length $\leq |E_F/G_F|$ joining
  $e$ to $e'$, or 
\item there exist two basic loops, at $e$ and $e'$ respectively, both of length $\leq 2|E_F/G_F|$.
\end{itemize}
Indeed, assume first that in any basic chain starting from $e$, there is no pair of edges $e_i\neq e_j$ lying in the same $G_F$-orbit.
Then there is such a chain going from $e$ to $e'$ (one can never apply (\ref{eq5})),
and the first conclusion holds.
Since the symmetric argument holds for basic chains starting from $e'$, we can assume that
there is a basic chain $e=e_1,\dots,e_i,\dots,e_j$ and a basic chain 
$e'=e'_1,\dots,e'_{i'},\dots,e'_{j'}$ and $g,g'\in G_F\setminus \{1\}$ such that $e_j=ge_i$ and  $e'_{j'}=g'e_{i'}$.
Choosing the shortest such chains, one can assume $j,j'\leq |E_F/G_F|+1$.
Now one can produce a basic loop $e=e_1,\dots,e_j=ge_i, g e_{i-1},\dots,g e_1=ge$ of length $j+i-1$
joining $e$ to $ge$. Arguing similarly for $e'$ shows that the second conclusion holds in this case.
\end{rk}

Using arationality of $T$, we will now prove that the above equivalence relation is in fact trivial.

\begin{lemma}\label{Whitehead}
For all $e,e'\in E_F$, we have $e\sim_F e'$. 
\end{lemma}

\begin{proof}
By collapsing $F$ (and its orbit) in $R$, we reduce to the case where $F$ is a point $v$ (and the collapsed tree is still a Grushko tree), 
so from now on we assume that we are in this case. We denote by $E_v$ the set of edges in $R$ incident on $v$. Assume towards a contradiction that the lemma fails.
Let $B\subseteq E_v$ be the set of edges $e\in E_v$ such that there is a basic loop at $e$.
By definition, all edges in $B$ are equivalent, so we can assume $E_v\setminus B\neq \es$.
Let $A_0\subseteq E_v\setminus B$ be an equivalence class, and $A_1=E_v\setminus G_v.A_0$ (notice that $B\subseteq A_1$, but this inclusion might be strict).
It follows from Remark \ref{rk-equiv} and the definition of $B$ that any two edges in $A_0$ are joined by a basic chain, and therefore that no two edges in $A_0$ are in the same $G_v$-orbit (otherwise two such edges would be joined by a basic loop).
Thus, for any $g\in G_v\setminus\{1\}$, $gA_0$ and $A_0$ have no edge in common, whereas $A_1$ is $G_v$-invariant.
By construction, any turn $\{e, e'\}$ at $v$ crossed by a $G$-translate of $l_1$ is either contained in $A_1$ or in $g.A_0$ for
some $g\in G_v$.

\begin{figure}[htbp]
\centering
\includegraphics{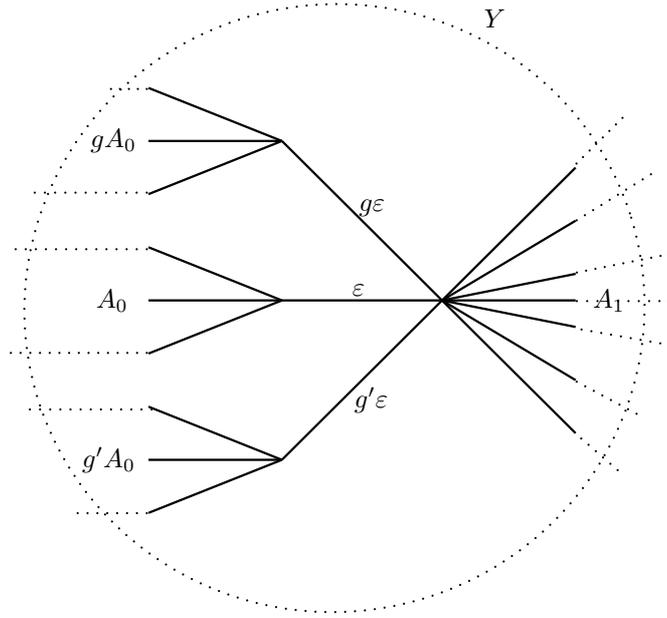}
\caption{Proof of the triviality of the equivalence relation $\sim_F$.}
\label{fig-white}
\end{figure}

Let $Y$ be the tree (represented on Figure~\ref{fig-white}) obtained from $(\dunion_{g\in G_v} g.A_0) \dunion  A_1$ by attaching a new edge $g\eps$ from $v_{A_1}$ to $gv_{A_0}$ for all $g\in G_v$,
where $v_{A_0},v_{A_1}$ denote respectively the copy of $v$ in $A_0$ and $A_1$.
Collapsing the orbit of $\eps$ yields  a natural map from $Y$ to the star of $v$ in $R$. 
There is a natural way to attach back the edges of $R$ to $Y$, which yields a Grushko tree $R'$ that collapses onto $R$. In addition, every $G$-translate of $l_1$ in $R$ lifts to a bi-infinite line in $R'$ which does not cross any edge in the orbit of $\eps$. This contradicts the fact that $l_1$ is not carried by any proper $(G,\calf)$-free factor (Lemma \ref{arat-lam}).
\end{proof}

\begin{proof}[Proof of Lemma \ref{lem_unsplit2}]
Let $(\alpha,\omega)\in L^2(T)\cap (\partial_\infty(G,\calf))^2$ be a liftable leaf. We aim to prove that $(\alpha,\omega)\in\ol\calp(l_1)$. 
Notice however that we do not make any assumption concerning the endpoints of $l_1$ (they could either be both in $\partial_\infty(G,\calf)$, or one of them could belong to $V_\infty(G,\calf)$).
%

We realize $(\alpha,\omega)$ as a $\Sigma$-leaf line $\call$, and let $x\in \Sigma$ be a point contained in $\call$. Since $(\alpha,\omega)$ is liftable, for all $i\in\mathbb{N}$, the line $\call$ naturally lifts to a $\Sigma^{(i)}$-leaf line $\call_i$, and we have $[\alpha,\omega]_{R^{(i)}}=p_{R^{(i)}}(\call_i)$. We let $\tilde x_i$ be the lift of $x$ in $\call_i$, and $x_i:=p_{R^{(i)}}(\tilde x_i)$. Notice that if $f_{R^{(i)}}:R^{(i)}\to R$ is the map obtained by composing the successive maps from $R^{(k+1)}$ to $R^{(k)}$, then $f_{R^{(i)}}(x_i)=x_0$ for all $i\in\mathbb{N}$.

Define a \emph{natural edge} of $R^{(i)}$ as a connected component in $R^{(i)}$ of the complement of the set of branch points together with $G.x_i$.
In particular, each natural edge is an open segment. We denote by $E_i$ the set of natural edges in $R^{(i)}$.
Up to passing to a subsequence, we can therefore assume that $|E_i/G|$ is constant, we denote it by $r$. 
We denote by $\lambda_1^i\le\dots\le\lambda_r^i$ the lengths of the orbits of natural edges. 
Let $r_0\in\{1,\dots,r\}$ be such that the sequence $(\lambda_{r_0}^i)_{i\in\mathbb{N}}$ is bounded, while $\lambda_{r_0+1}^{i}$ converges to $+\infty$ as $i$ goes to $+\infty$ (we also allow for $r_0=0$ in case all lengths are unbounded). Notice that $r_0<r$ because the number of $G$-orbits of edges in $R^{(i)}$ goes to $+\infty$ as $i$ goes to $+\infty$. The $r_0$ first natural edges will be called \emph{short} edges, the other natural edges will be called \emph{long} edges. 
Note that there is no claim that the maps $R^{(i)}\ra R^{(j)}$ preserve the natural edges or the decomposition into short and long edges.

Let $F_i\subseteq R^{(i)}$ be the union of all short edges containing $x_i$ in their closure (this is a connected subtree of $R^{(i)}$).
Clearly, $F_i$ is a closed subtree such that for all $g\in G$, $gF_i\cap F_i=\es$ or $gF_i=F_i$.
We claim that $F_i$ has bounded diameter.
Since short edges have bounded length, it suffices to prove that for $i$ large enough,
any segment $J\subseteq F_i$ contains at most $2r$ natural edges, necessarily short (recall that $r$ is
the number of orbits of natural edges).
If not, $J$ contains two natural edges $e_i$, $g_ie_i$  which are in the same orbit as oriented edges, with the orientation induced by some orientation of $J$.
Thus $g_i$ is hyperbolic with bounded translation length, this is a contradiction for $i$ large enough
since the systole of $R^{(i)}$ goes to infinity (Lemma \ref{systole}).

Since $F_i$ has finite diameter, the stabilizer $G_{F_i}$ of $F_i$ is either trivial or peripheral. 
We claim that $G_{F_i}$ eventually does not depend on $i$. 
Let $i_0$ and $M$ be such that for all $i\geq i_0$, short edges and $F_i$ have diameter at most $M$, 
and every long edge has length greater than $M$.
We first note that for each $i\geq i_0$, there is at most one vertex $w\in R^{(i)}$ with non-trivial stabilizer such that
$d(x_{i},w)\leq M$. Indeed, such a vertex $w$ cannot be separated from $x_{i}$ by a long edge, so lies in $F_{i}$.
But $F_{i}$ contains at most one vertex with non-trivial stabilizer since no hyperbolic element stabilizers $F_{i}$.
This proves the uniqueness of $w$ and that $G_w=G_{F_{i}}$.
Now consider $i\geq i_0$ such that $G_{F_i}$ is non-trivial (if there is no such $i$, the claim is obvious).
Then $G_{F_i}$ fixes some vertex $v_i\in F_i$, and $G_{F_i}=G_{v_i}$.
Since $F_i$ has diameter at most $M$, $d_{R^{(i)}}(x_i,v_i)\leq M$. For $j\leq i$, the image of $v_i$ in $R^{(j)}$ is a vertex $w$ with
stabilizer $G_{F_i}$, and at distance at most $M$ from $x_{j}$. It follows that $G_{F_i}=G_w=G_{F_{j}}$ and the claim follows.

We are now ready to prove that $(\alpha,\omega)\in \ol\calp(l_1)$.
Fix $i\geq i_0$, and let $e_i,e'_i$ be the long edges in $[\alpha,\omega]_{R^{(i)}}$ incident on $F_i$ (these are well-defined because $\alpha,\omega\in\partial_\infty(G,\calf)$ and $x_i\in [\alpha,\omega]_{R^{(i)}}$).
Let $E_{F_i}$ be the set of (long) natural edges incident on $F_i$, and
consider the equivalence relation $\sim_{F_i}$ on $E_{F_i}$ defined above using the leaf $l_1$.
Lemma \ref{Whitehead} implies that $e_i\sim_{F_i} e'_i$. 
By Remark \ref{rk-equiv}, either $e_i$ and $e'_i$ are joined by a basic chain or there are two basic loops at $e_i$ and $e'_i$ respectively.
Up to passing to a subsequence, we can assume that the same case occurs for all $i$.

\begin{figure}[htbp]
  \centering
  \includegraphics{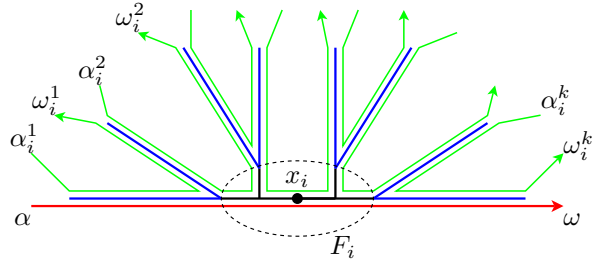}
  \caption{A basic chain of translates of $l_1$.}
  \label{fig_whitehead}
\end{figure}

In the first case, for any given $i\in\mathbb{N}$, let $l_i^1=(\alpha_i^1,\omega_i^1),\dots,l_i^k=(\alpha_i^k,\omega_i^k)$ be a finite sequence of $G$-translates of $l_1$ corresponding to a basic chain provided by Remark~\ref{rk-equiv} (up to passing to a subsequence, we can assume that $k$ is independent of $i$). 
This means that the semi-line $[x_i,\alpha_i^1]_{R^{(i)}}$ contains $e_i$, $[x_i,\omega_i^k]_{R^{(i)}}$ contains $e'_i$,
and for all $1\leq j\leq k-1$, the semi-lines $[x_i,\omega_i^j]_{R^{(i)}}$ and $[x_i,\alpha_i^{j+1}]_{R^{(i)}}$ share a long edge in $E_{F_i}$ (see Figure~\ref{fig_whitehead}). 
Up to passing to a subsequence again, we can assume that all $\alpha_i^j$ and all $\omega_i^j$ converge in $\partial(G,\calf)$ as $i$ goes to $+\infty$, and we denote the limits by $\alpha_{\infty}^j$ and $\omega_{\infty}^j$. 
Since $(l_i^1)_{R^{(i)}}\cap [\alpha,x_i]_{R^{(i)}}$ contain the long edge $e_i$, and since the projection map $p_R^i:R^{(i)}\to R$ 
is isometric in restriction to both $l_i^1$ and $[\alpha,\omega]_{R ^{(i)}}$ (because these two leaves are liftable), we deduce
that $(l_i^1)_R\cap [\alpha,x_0]_{R}$ contains arbitrarily long segments as $i$ goes to $+\infty$. This implies that $\alpha_{\infty}^1=\alpha$. A similar argument shows that $\omega_{\infty}^k=\omega$, and $\omega_{\infty}^j=\alpha_{\infty}^{j+1}$ for all $j\in\{1,\dots,k-1\}$. Up to deleting some of these points, we can assume in addition that $\alpha_{\infty}^{j}\neq\omega_\infty^j$ for all $j\in\{1,\dots,k\}$. Then the leaves $(\alpha_{\infty}^j,\omega_{\infty}^j)$ are obtained as limits of $G$-translates of $l_1$, and therefore they all belong to $\overline{\mathcal{P}}(l_1)$. This implies that $(\alpha,\omega)\in\overline{\mathcal{P}}(l_1)$.

We now assume that for all $i$, there are basic loops at both $e_i$ and $e'_i$.
Let $(\alpha_i^1,\omega_i^1),\dots,(\alpha_i ^k,\omega_i^k)$ and $(\alpha_i'^1,\omega_i'^1),\dots,(\alpha_i'^s,\omega_i'^s)$ 
be the two corresponding sequences of $G$-translates of $l_1$.
In particular, $[x_i,\alpha_i^1]_{R^{(i)}}$ contains $e_i$ and $[x_i,\omega_i^k]_{R^{(i)}}$ contains $g_i e_i$ for some $g_i\in G_v\setminus\{1\}$.
Passing to a subsequence we can take limits $\alpha_{\infty}^j$, $\omega_{\infty}^j$ as above. 
As above, $\alpha_\infty^1=\alpha$, and either $\omega_{\infty}^k=g.\alpha$ if $g_i$ eventually coincides with some $g\in G_v\setminus\{1\}$,
or  $\omega_\infty^k=v\in V_{\infty}(R)$ if $g_i$ takes infinitely many distinct values.
Thus, $\ol\calp(l_1)$ either contains a leaf of the form $(\alpha,g\alpha)$ with $g\in G_v\setminus \{1\}$ or the leaf $(\alpha,v)$
in which case it also contains all the leaves $(g\alpha,v)$ for $g\in G_v$, hence the leaf $(\alpha,g\alpha)$ for any $g\in G_v\setminus\{1\}$.
Arguing similarly with the second basic loop, we get that $\ol\calp(l_1)$ contains a leaf of the form $(\omega,g'\omega)$ for some $g'\in G_v\setminus\{1\}$.
By definition of the peritransitive closure, $(\alpha,\omega)\in \ol\calp(l_1)$.
\end{proof}

\subsection{Additional arguments for trees of quadratic type}\label{sec-11-surf}

We will now assume that $T$ is of quadratic type with respect to a Grushko tree $R$, and that $\Sigma=\Sigma(R,T)$ is clean.
We aim to prove Lemmas \ref{lem_unsplit1} and \ref{lem_unsplit3} (we mention that cleanness will actually only be used in the proof of Lemma~\ref{lem_unsplit3}).

\begin{figure}[htbp]
  \centering
  \includegraphics{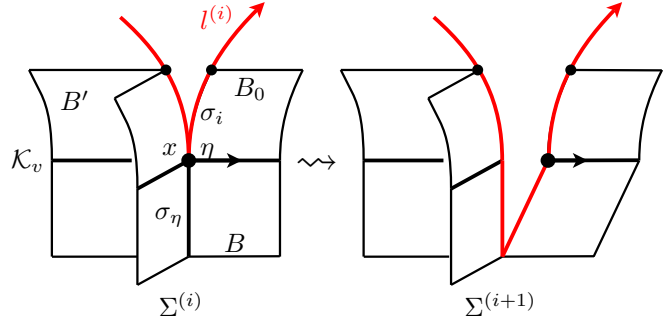}
  \caption{An unliftable leaf.}
  \label{fig_unliftable}
\end{figure}

\begin{proof}[Proof of Lemma \ref{lem_unsplit1}]

Given an algebraic leaf $l\in L^2(T)$, denote by $l^{(i)}$ (resp $l^{(0)}$) the $\Sigma^{(i)}$-leaf line (resp $\Sigma$-leaf line) corresponding to $l$. 
We first prove the following general claim. Given a leaf $l\in L^2(T)$ which is not liftable, $l$ is isolated in the following sense:
    there exists $i\in\bbN$ and a $\Sigma^{(i)}$-leaf segment $\sigma\subseteq l^{(i)}$ such that
    for any sequence $(l_k)_{k\in\bbN}$ of leaves in $L^2(T)$ converging
to $l$ and any $k$ large enough, the $\Sigma^{(i)}$-leaf line $l_k^{(i)}$ contains $\sigma$ 
(in particular $l_k^{(i)}$ is contained in the same complete $\Sigma^{(i)}$-leaf as $l$).
Indeed, since $l$ is not liftable, there exists $i\in\mathbb{N}$ such that (see Figure~\ref{fig_unliftable}):
\begin{itemize}
\item $l$ lifts isometrically to a $\Sigma$-leaf line $l^{(i)}$ in $\Sigma^{(i)}$, and
\item there is a splitting germ $\eta$ based at some point $x\in\Sigma^{(i)}$, such that $l^{(i)}$ goes through the band $B_0$ in which $\eta$ is terminal, followed by some band $B'$ with degenerate intersection with $\eta$.
\end{itemize}
Let $\sigma\subseteq B_0$ be the leaf segment containing $x$. 
For $k$ large enough, $l_k^{(i)}$ also goes through the bands $B_0,B'$ in $\Sigma^{(i)}$. Since $B_0\cap B'=\{x\}$, this implies that $\sigma\subseteq l_k^{(i)}$.  This proves the claim.

Now consider the algebraic leaf $l_0$ of the statement of the lemma, and denote by $\alpha,\omega\in \partial(G,\calf)$ the endpoints of $l_0$. We claim that up to exchanging $\alpha$ and $\omega$, we can assume that $a\alpha\neq \alpha$ for all $a\in G\setminus\{1\}$ (equivalently, $\alpha\in \partial_\infty(G,\calf)$ and that $\alpha$ is not of the form $c^{+\infty}$ for some non-peripheral element $c\in G$). 
Indeed, if $a.\alpha=\alpha$ and $b.\omega=\omega$ for some $a,b\in G\setminus\{1\}$, then $\grp{a,b}$ fixes the point $z=\calq(\alpha)=\calq(\omega)$ in $T$.
If $\grp{a,b}$ is contained in a peripheral group $G_v$, then $\alpha=v=\omega$, a contradiction. If $\grp{a,b}$ is not peripheral, then $T$ is not relatively free
so $T$ is arational surface and $G_z=\grp{c}$ is cyclic, so $(\alpha,\omega)=(c^{\pm\infty},c^{\mp\infty})$ is carried by $G_z$. This contradicts our hypothesis on $l_0$ so the claim follows.

Let now $g_k.e$ be a sequence of oriented edges in $p_{R^{(i)}}(l_0)$ converging to $\alpha$,
and let $l_1\in L^2(T)$ be a limit of a subsequence of $g_k\m l_0$. 
We claim that the leaf $l_1$ is liftable. Indeed, assume towards a contradiction that it is not.
The claim established in the first paragraph of the proof then implies that there exists $i$ such that $g_k\m l_0^{(i)}$ contains $\sigma$ for all $k\in\mathbb{N}$ larger than some $k_0\in\mathbb{N}$.
Thus, $g_k g_{k_0}\m$ is non-peripheral and fixes the point $z=p_T(l_0^{(i)})$ ($=\calq(\alpha)=\calq(\omega)$). 
In particular, $T$ is not relatively free, so as above, $G_z=\grp{c}$, for some $c$ non-peripheral.
Thus, $g_kg_{k_0}\m=c^{j_k}$, and $\alpha$ is a limit of $g_kg_{k_0}\m=a^{j_k}$, a contradiction.
\end{proof}

\begin{figure}[htbp]
\centering
\includegraphics{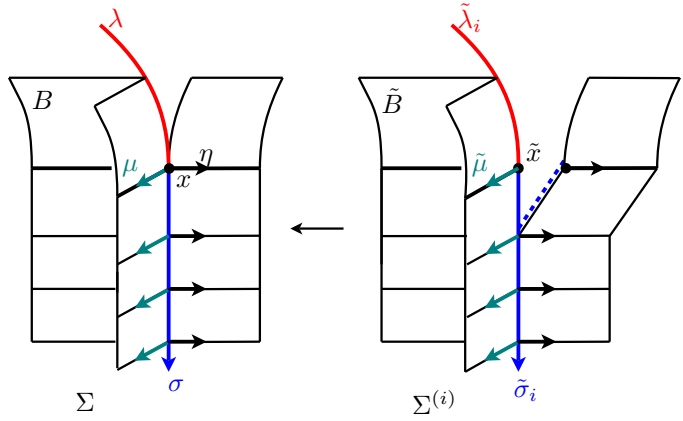}
\caption{Lifting $\sigma\cup\lambda$ to $\Sigma^{(i)}$.}
\label{fig-lift}
\end{figure}

We now turn to the proof of Lemma \ref{lem_unsplit3} saying that any leaf in $L^2(T)$ lies in the peritransitive closure of the set of liftable leaves.
Since $\Sigma$-leaves that do not contain a splitting semi-line are automatically liftable, one only has to take care of
leaves containing a splitting semi-line. Splitting semi-lines themselves lift, and in several ways.
The following lemma allows to lift the concatenation of a splitting semi-line with another liftable segment.

\begin{lemma}\label{blue}
Assume that $\Sigma$ is clean. Let $\lambda$ be a $\Sigma$-leaf segment (either finite, semi-infinite or bi-infinite) that lifts to $\tilde\lambda_i$ in $\Sigma^{(i)}$ for some $i\in \bbN$. Let $\sigma$ be a splitting semi-line in $\Sigma$, whose intersection with $\lambda$ is reduced to one extremity of $\lambda$. 
\\ Then there exists a lift $\tilde \sigma_i$ of $\sigma$ such that $\Tilde\lambda_i\cup \Tilde\sigma_i$ is a lift of $\lambda\cup\sigma$.
\end{lemma}

\begin{proof}
This is illustrated on Figure~\ref{fig-lift}. Let $x$ be the intersection point of $\sigma$ and $\lambda$, and let $\tilde x$ be its lift in $\Tilde\lambda_i$. Let $\tilde B$ be the unique band of $\Sigma^{(i)}$ that contains $\tilde x$ and a nondegenerate leaf segment of $\tilde \lambda_i$. Let $\tilde \mu$ be a direction transverse to $\tilde B$ at $\widetilde x$. Let $\mu$ be its projection in $\Sigma$. Since $\Sigma$ is clean, the direction $\mu$ can be pushed to infinity along $\sigma$. There is a unique lift $\tilde\sigma_i$ of $\sigma$ in $\Sigma^{(i)}$ starting from $\tilde x$, along which $\tilde\mu$ can be pushed to infinity. Then $\tilde \lambda_i\cup\tilde\sigma_i$ is a lift of $\lambda\cup\sigma$.
\end{proof}

\begin{cor}\label{cor-blue}
Assume that $\Sigma$ is clean. Let $\call$ be a complete $\Sigma$-leaf. Then the following holds.
\begin{enumerate}
\item Let $\sigma=[y,\xi]_\call$ and $\sigma'=[y',\xi']_\call$ be two splitting semi-lines contained in $\call$, with $\xi\neq\xi'$, and let $x\in\sigma$ and $x'\in\sigma'$ be such that $[x,x']_\call$ is liftable. Then $[\xi,\xi']_\call$ is liftable.
\item Let $I=[x,\alpha]_\call$ be a liftable leaf semi-line contained in $\call$, and $\sigma=[y,\xi]_\call$ be a splitting semi-line with $x\in \sigma$, with $\alpha\neq\xi$. Then $[\alpha,\xi]_\call$ is liftable.
\end{enumerate}
\end{cor}

\begin{proof}
We will only prove the first claim, the second being a reformulation of Lemma \ref{blue}. Up to changing $[x,x']_\call$ to a subsegment, we can assume that $[x,x']_\call\cap\sigma=\{x\}$ and $[x,x']_\call\cap\sigma'=\{x'\}$. If $x=x'$, we can further change $x$ so that $[x,\xi]_\call\cap [x,\xi']_\call=\{x\}$. Assume first that $x\neq x'$. Lemma~\ref{blue} applied to $\lambda=[x,x']_\call$ and $\sigma$, and to $\lambda$ and $\sigma'$, yields the result. If $x=x'$, we apply Lemma~\ref{blue} to $\lambda=\sigma'$ and $\sigma$.
\end{proof}

We are now in position to complete our proof of Lemma~\ref{lem_unsplit3}.

\begin{proof}[Proof of Lemma \ref{lem_unsplit3}]
Let $(\alpha,\omega)\in L^2(T)$. We aim to prove that $(\alpha,\omega)$ belongs to the peritransitive closure of the set of all liftable leaves
in $L^2(T)\cap\partial_\infty (G,\calf)^2$.
As mentioned above, we can assume that the complete $\Sigma$-leaf $\call$ containing $(\alpha,\omega)$ contains a splitting semi-line.
\\
\\ \textbf{Case 1}: We have $\alpha,\omega\in\partial_\infty (G,\calf)$. 
\\ Let $\call$ be the unique complete $\Sigma$-leaf that admits $\alpha$ and $\omega$ as some of its extremities. Let $G_\call$ be the stabilizer of $\call$, and let $\calc_\call$ be the core of its stabilizer (see Definition \ref{de-core}).
\\
\\
\indent \textit{Case 1.1.} The intersection $[\alpha,\omega]_\call\cap\calc_\call$ is compact. 

We observe that the $\Sigma$-leaf line $[\alpha,\omega]_\call$ is a finite concatenation of liftable leaf segments. Indeed, if some subray $[x_0,\omega]_\call$ does not intersect $\calc_\call$, then either it eventually agrees with a splitting semi-line, or else it eventually misses every splitting semi-line: this is because there are only finitely many orbits of splitting semi-lines in $\Sigma$, and therefore there are only finitely many splitting semi-lines that intersect the connected component of $\call\setminus\calc_\call$ that contains $[x_0,\omega]_\call$.

%

\begin{figure}[htbp]
\centering
\includegraphics[scale=0.9]{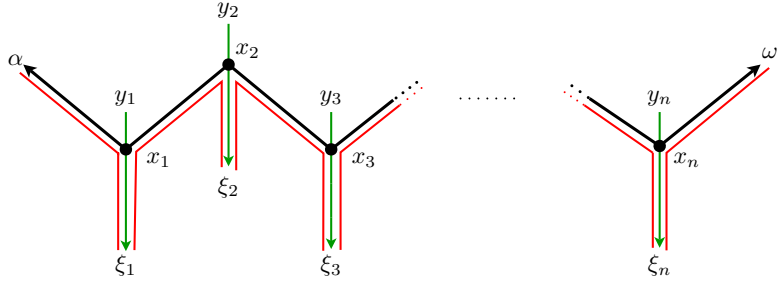} 
\caption{In Case~1.1, the leaf $(\alpha,\omega)$ is in the transitive closure of the set of red leaves, which are all liftable leaves in $L^2(T)\cap\partial_\infty (G,\calf)^2$.}
\label{fig-peritransitive}
\end{figure}

As represented on Figure~\ref{fig-peritransitive}, we write $[\alpha,\omega]_\call$ as a finite concatenation $[\alpha,x_1]_\call\cup[x_1,x_2]_\call\cup\dots\cup [x_n,\omega]_\call$, where each segment in the decomposition is liftable. We can also assume that this decomposition is minimal, i.e.\ none of the leaf segments $[x_i,x_{i+2}]_\call$ is liftable. Since $\Sigma$ is clean, this implies that each $x_i$ lies in a splitting semi-line $\sigma_i=[y_i,\xi_i]_\call$ (represented in green on the figure), with $\xi_i\in\partial_\infty (G,\calf)$. To simplify, assume first that all $\xi_i$ are distinct and distinct from $\alpha$ and $\omega$. Corollary~\ref{cor-blue} shows that $[\alpha,\xi_1]_\call,[\xi_1,\xi_2]_\call,\dots,[\xi_n,\omega]_\call$ are liftable. 
Since $\alpha,\omega$ and all $\xi_i$ belong to $\partial_\infty(G,\calf)$, we conclude that
$(\alpha,\omega)$ is in the transitive closure of liftable leaves in $L^2(T)\cap \partial_\infty(G,\calf)^2$ as required. 
 If there are repetitions, say for example $\alpha\neq\xi_1=\xi_{i_1}\neq\xi_{i_1+1}=\xi_{i_2}\neq\dots\neq\xi_{i_k}\neq\omega$ with $1\le i_1<i_2<\dots<i_k$, then we obtain a chain $[\alpha,\xi_1]_\call,[\xi_{i_1},\xi_{i_2}]_\call,\dots,[\xi_{i_k},\omega]_\call$ of liftable leaf lines. The case where $\alpha=\xi_1$ or $\omega=\xi_n$ is similar and left to the reader.   
\\
\\

\begin{figure}[htbp]
\centering
\includegraphics{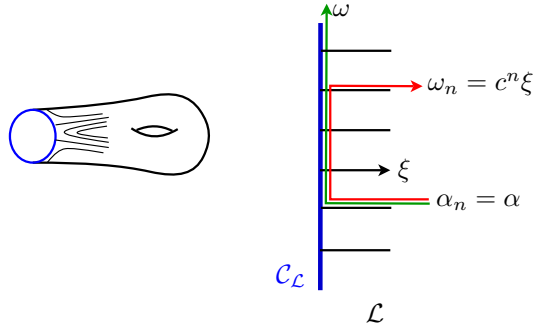}
\caption{In Case 1.2, the green leaf $(\alpha,\omega)$ is in the peritransitive closure of the red leaves $(\alpha_n,\omega_n)$, which are treated by case 1.1.}
\label{fig-lam-surface}
\end{figure}

\indent \textit{Case 1.2.} The intersection $[\alpha,\omega]_\call\cap\calc_\call$ is not compact.  

The argument is illustrated on Figure~\ref{fig-lam-surface}.
In particular, $\calc_\call$ is non-empty and not reduced to a point, so its stabilizer $G_\call$ is non-trivial and not peripheral.
It follows that $T$ is not relatively free hence arational surface, therefore the stabilizer of $\call$ is a non-peripheral infinite cyclic group $\grp{c}$ (and $\grp{c}$ is conjugate to the fundamental group of the free boundary curve  of the underlying orbifold of $T$).
In particular, $\calc_\call$ is a $c$-invariant line.
We can assume that $[\alpha,\omega]_\call$ meets a splitting semi-line $\sigma=[y,\xi]_\call$, otherwise $[\alpha,\omega]_\call$ would be liftable because $\Sigma$ is clean, and there is nothing to prove. 
By Remark \ref{rk_branch_leaf}, $\call$ contains a point of valence at least 3 (namely the origin of $\sigma$), so $\call$ is not a line.
Since $\call$ has no terminal point (Proposition \ref{surface_3ends}), $\call$ has an end $\xi$ not in $\calc_\call$.
We claim that there exist $\alpha_n$ converging to $\alpha$ and $\omega_n$ converging to $\omega$, corresponding to endpoints of $\call$, such that $[\alpha_n,\omega_n]_\call\cap\calc_\call$ is compact. 
Indeed, if $\alpha\notin\partial_\infty \calc_\call$ (as in Figure~\ref{fig-lam-surface}), then $\alpha_n=\alpha$ works, and otherwise we can take $\alpha_n=c^{\pm n}.\xi$. By Case~1.1, for all $n\in\mathbb{N}$, the leaf $[\alpha_n,\omega_n]_\call$ lies in the peritransitive closure of the set of liftable leaves in $L^2(T)\cap\partial_\infty(G,\calf)^2$. Therefore, so does $[\alpha,\omega]_\call$. 
\\
\\
\textbf{Case 2}: Either $\alpha$ or $\omega$, say $\alpha$, belongs to $V_\infty(G,\calf)$. 
\\ Then $\omega\in\partial_\infty(G,\calf)$, for otherwise the subgroup of $G$ generated by the stabilizers of $\alpha$ and $\omega$ would be elliptic in $T$, contradicting arationality. Therefore, for all $g\in G_\alpha\setminus\{1\}$, we have $(\omega,g\omega)\in L^2(T)\cap \partial_\infty (G,\calf)^2$, so it follows from the above argument that $(\omega,g\omega)$ is in the peritransitive closure of the set of liftable leaves in $L^2(T)\cap\partial_\infty (G,\calf)^2$. Since $(\alpha,\omega)$ is the limit of $(g_i\omega,\omega)$ for any infinite sequence $(g_i)_{i\in\mathbb{N}}\in G_\alpha^\mathbb{N}$, we deduce that $(\alpha,\omega)$ also belongs to this peritransitive closure. 
\end{proof}

\section{Unique biduality for arational trees}\label{sec-last}

\subsection{Main result}

We will now establish our unique biduality result for arational trees. Given $T,T'\in\mathcal{AT}$, we write $T\approx T'$ if the trees $\widehat T$ and $\widehat T'$ are homeomorphic when equipped with the observers' topology (where we recall that $\widehat T=\ol T\cup\partial_\infty T$).

\begin{theo}\label{bidual-free}
Let $T,T'\in\overline{\calo}$. Assume that $T\in\mathcal{AT}$, and that $L^2(T)\cap L^2(T')\neq\emptyset$. Additionally, if $T$ is not relatively free, assume that $L^2(T)\cap L^2(T')$ contains a leaf not carried by  $\langle c\rangle$, with $c$ a representative of the unique conjugacy class of nonperipheral elliptic subgroups in $T$. \\ Then $T'\in\mathcal{AT}$, and $T\approx T'$.
\end{theo}


\begin{proof}
Let $l\in L^2(T)\cap L^2(T')$ be a leaf that is not carried by any subgroup of $G$ that is elliptic in $T$. 
By Theorem \ref{lamination-arational}, $L^2(T)$ is equal to the peritransitive closure of $l$. 

We first prove that $T'$ has dense orbits. Assume towards a contradiction that it does not. Then the leaf $l$ is carried by a vertex group $B$ of the Levitt decomposition of $T'$ (Corollary~\ref{carry-gen-l2}). In particular that $B$ is elliptic in a $\calz$-splitting of $(G,\calf)$. By Proposition~\ref{arat-Z}, the $B$-action on its minimal subtree $T_B$ in $T$ is simplicial. Lemma~\ref{l2-carry} applied to $T_B$ then shows that $l$ is carried by a subgroup of $G$ that is elliptic in $T_B$, hence in $T$: this is a contradiction. Therefore $T'$ has dense orbits.

 By Lemma \ref{lem_peritransitively}, $L^2(T')$ is peritransitively closed. This implies that $L^2(T)\subseteq L^2(T')$.

We claim that there exists a surjective $G$-equivariant continuous map $p:\widehat{T}\to\widehat{T'}$ that makes the following diagram commute:

$$\xymatrix{
\partial (G,\calf) \ar@{->>}[r]^{\calq_T} \ar@{->>}[dr]^{\calq_{T'}} & \widehat{T}\ar@{->>}[d]^{p}\\
&\widehat{T'} 
}$$

\noindent where $\calq_T$ and $\calq_{T'}$ are the maps $\calq$ associated to $T$ and $T'$. Indeed, since $T$ and $T'$ have dense orbits, Corollary~\ref{quotient-lamination} shows that the map $\calq_T$ induces a homeomorphism from $\partial (G,\calf)/L^2(T)$ to $\widehat{T}$, and the map $\calq_{T'}$ induces a homeomorphism from $\partial (G,\calf)/L^2(T')$ to $\widehat{T'}$. In addition, since $L^2(T)\subseteq L^2(T')$, we get a map $p:\widehat{T}\to\widehat{T'}$.

We finally check that $p$ is a homeomorphism. Assume towards a contradiction that it is not. Since $T$ is mixing, by \cite[Proposition~3.2]{BR13}, there are uncountably many points $x\in\widehat{T'}$ such that $|p^{-1}(x)|\ge 3$ (notice here that Bestvina and Reynolds are assuming indecomposability of $T$ in the statement of \cite[Proposition 3.2]{BR13}, however their argument only requires $T$ to be mixing). Since $\calq_T$ is surjective, we have $|\calq_{T'}^{-1}(x)|\ge 3$ for all these points $x$. This contradicts Theorem~\ref{Q-preimage}. 

\end{proof}

\begin{cor}\label{feng-luo}
Let $T,T'\in\overline{\calo}$, with $T\in\mathcal{AT}$ relatively free. Let $(T_n)_{n\in\mathbb{N}},(T'_n)_{n\in\mathbb{N}}\in\overline{\calo}^{\mathbb{N}}$ be sequences that converge to $T,T'$ such that for all $n\in\mathbb{N}$, we have $L^2(T_n)\cap L^2(T'_n)\neq\emptyset$.\\
Then $T'\in\mathcal{AT}$, and $T\approx T'$. 
\end{cor}

\begin{proof}
Corollary \ref{feuille-commune} implies that $L^2(T)\cap L^2(T')\neq\emptyset$. As $T$ is assumed to be relatively free, no element in $L^2(T)\cap L^2(T')$ can be carried by a subgroup that is elliptic in $T$. Corollary \ref{feng-luo} then follows from Theorem \ref{bidual-free}. 
\end{proof}

To remove the assumption that $T$ is relatively free, we need a stronger assumption on the dual laminations of $T_n$ and $T'_n$.
Recall that an algebraic leaf is simple if it is a limit of axes of simple elements (i.e.\ contained in a proper $(G,\calf)$-free factor), 
and that $L^2_{simple}(T)$ denotes the set of simple leaves in $L^2(T)$ (see Section \ref{sec_simple}). 

\begin{cor}\label{feng-luo-2}
Let $T,T'\in\overline{\calo}$, with $T\in\mathcal{AT}$. Let $(T_n)_{n\in\mathbb{N}},(T'_n)_{n\in\mathbb{N}}\in\overline{\calo}^{\mathbb{N}}$ be sequences that converge to $T,T'$ such that for all $n\in\mathbb{N}$, we have $L^2_{simple}(T_n)\cap L^2_{simple}(T'_n)\neq\emptyset$.\\
Then $T'\in\mathcal{AT}$, and $T\approx T'$. 
\end{cor}

\begin{proof}
We can assume that $T$ is arational surface since otherwise, Corollary \ref{feng-luo} applies. Let $A=\grp{c}<G$ be a representative of the unique conjugacy class
of non-peripheral elliptic subgroups in $T$.
Corollary~\ref{rat} implies that $L^2_{simple}(T)\cap L^2_{simple}(T')\neq\emptyset$. 
Since $c$ is not a simple element because $T$ is arational, Lemma~\ref{rational} says that a simple leaf 
cannot be carried by $\grp{c}$. Corollary~\ref{feng-luo-2} then follows from Theorem~\ref{bidual-free}. 
\end{proof}

\subsection{The equivalence class of an arational tree}

We will finally use the above analysis to deduce information about the geometry of the equivalence class of an arational tree. We recall that a map between two $\mathbb{R}$-trees is \emph{alignment-preserving} if it sends segments to segments. Two trees $T_1,T_2\in\overline{\calo}$ are \emph{compatible} if there exists a tree $T_0\in\overline{\calo}$ that admits $G$-equivariant alignment-preserving maps onto both $T_1$ and $T_2$. 

\begin{cor}\label{cor-equiv}
Let $T\in\mathcal{AT}$, and let $T'\in\ol\calo$. The following assertions are equivalent.
\renewcommand{\theenumi}{\roman{enumi}}
\begin{enumerate}
\item The compact trees $\hat T$ and $\hat T'$ are homeomorphic when equipped with the observers' topology.
\item The trees $T$ and $T'$ are homeomorphic when equipped with the observers' topology.
\item There exist $G$-equivariant alignment-preserving maps from $T$ to $T'$ and from $T'$ to $T$.
\item The trees $T$ and $T'$ have a common refinement in $\overline{\calo}$.
\item The trees $T$ and $T'$ are both compatible with a common tree in $\overline{\calo}$.
\end{enumerate}
\end{cor}

\begin{proof}
Since $T$ is the complement of endpoints of $\hat T$,  we have $(i)\imp (ii)$.
Since every homeomorphism for the observers' topology is alignment-preserving, we have $(ii)\Rightarrow (iii)$. The implications $(iii)\Rightarrow (iv)$ and $(iv)\Rightarrow (v)$ are obvious. We now prove that $(v)\Rightarrow (i)$. Let $T''$ be a tree which is compatible with both $T$ and $T'$, and let $T_0\in\partial\calo$ be a common refinement of $T$ and $T''$. Then $L^2(T_0)\subseteq L^2(T)\cap L^2(T'')$. Since $T_0$ belongs to $\partial\calo$, it follows from Proposition~\ref{prop:l2simple} that $L^2_{simple}(T_0)\neq\emptyset$. Since the unique conjugacy class of nonperipheral cyclic subgroups that is elliptic in an arational surface tree is always nonsimple, we can apply Theorem~\ref{bidual-free} to deduce that $\hat T$ and $\hat T''$ are homeomorphic for the observers' topology. The same argument also shows that $\hat T''$ and $\hat T'$ are homeomorphic for the observers' topology, and Assertion~$(i)$ follows.
\end{proof}
 
Given a tree $T\in\partial\calo$ with dense orbits, a \emph{length measure} on $T$ is a collection of finite Borel measures $\mu_I$ on all segments $I\subseteq T$ such that for all $J\subseteq I$, we have $\mu_J=(\mu_{I})_{|J}$, and for all $I\subseteq T$ and all $g\in G$, we have $\mu_{gI}=(g_{|I})_{\ast}\mu_I$. The set $M(T)$ of projective classes of non-atomic length measures on $T$ is a finite-dimensional simplex, spanned by the set of ergodic measures on $T$: this was proved in \cite[Corollary 5.4]{Gui00} in the context of free groups, however the proof adapts to our more general setting because $\partial\mathbb{P}\calo$ is known to be finite-dimensional by \cite{Hor14-1}. Any length measure $\mu$ on $T$ determines a tree $T_{\mu}$, obtained by making Hausdorff the pseudo-metric on $T$ given by $d_{\mu}(x,y)=\mu([x,y])$, and there exists a $G$-equivariant alignment-preserving map from $T$ to $T_{\mu}$. Conversely, if $T$ admits an alignment-preserving map onto a tree $T'\in\partial\calo$, then there exists a length measure $\mu$ on $T$ such that $T'=T_{\mu}$. The map that sends a length measure $\mu\in M(T)$ to the projective length function of the tree $T_{\mu}$ is a linear injection \cite[Lemma~5.3]{Gui00}, so the image of $M(T)$ in $\partial\mathbb{P}\calo$ is a simplex $\Sigma(T)$ of the same dimension. In view of Corollary~\ref{cor-equiv}, two trees $T,T'\in\mathbb{P}\mathcal{AT}$ are equivalent if and only if there exist alignment-preserving maps from $T$ to $T'$ and from $T'$ to $T$, so it follows that the equivalence class of $T$ is the finite-dimensional simplex $\Sigma(T)$. We sum up the above discussion in the following statement.

\begin{prop}\label{simplices}
For all $T\in\mathbb{P}\mathcal{AT}$, the set of projective classes of arational $(G,\calf)$-trees that are equivalent to $T$ is a finite-dimensional simplex.\qed
\end{prop}

\bibliographystyle{alpha}
\bibliography{GH1-bib}

\begin{thebibliography}{CHL08b}

\bibitem[BF95]{BF_stable}
M.~Bestvina and M.~Feighn.
\newblock Stable actions of groups on real trees.
\newblock {\em Invent. Math.}, 121(2):287--321, 1995.

\bibitem[BF14]{BF14}
M.~Bestvina and M.~Feighn.
\newblock Hyperbolicity of the complex of free factors.
\newblock {\em Adv. Math.}, 256:104--155, 2014.

\bibitem[BFH97]{BFH97}
M.~Bestvina, M.~Feighn, and M.~Handel.
\newblock Laminations, trees, and irreducible automorphisms of free groups.
\newblock {\em Geom. Funct. Anal.}, 7(2):215--244, 1997.

\bibitem[BFH00]{BFH00}
M.~Bestvina, M.~Feighn, and M.~Handel.
\newblock The {T}its {A}lternative for $\text{Out}({F}_n)$ {I} : {D}ynamics of
  exponentially growing automorphisms.
\newblock {\em Ann. Math.}, 151(2):517--623, 2000.

\bibitem[BFH05]{BFH05}
M.~Bestvina, M.~Feighn, and M.~Handel.
\newblock The {T}its alternative for $\text{Out}({F}_n)$ {II} : {A} {K}olchin
  type theorem.
\newblock {\em Ann. Math.}, 161(1):1--59, 2005.

\bibitem[Bon91]{Bon91}
F.~Bonahon.
\newblock Geodesic currents on negatively curved groups.
\newblock In {\em Arboreal group theory}, volume~19 of {\em Math. Sci. Res.
  Inst. Publ.}, pages 143--168. Springer, New York, 1991.

\bibitem[BR15]{BR13}
M.~Bestvina and P.~Reynolds.
\newblock The boundary of the complex of free factors.
\newblock {\em Duke Math. J.}, 164(11):2213--2251, 2015.

\bibitem[CH14]{CH14}
T.~Coulbois and A.~Hilion.
\newblock Rips {I}nduction: {I}ndex of the dual lamination of an
  $\mathbb{R}$-tree.
\newblock {\em Groups Geom. Dyn.}, 8(1):97--134, 2014.

\bibitem[CHL07]{CHL07}
T.~Coulbois, A.~Hilion, and M.~Lustig.
\newblock Non-unique ergodicity, observers' topology and the dual lamination
  for $\mathbb{R}$-trees.
\newblock {\em Illinois J. Math.}, 51(3):897--911, 2007.

\bibitem[CHL08a]{CHL08-1}
T.~Coulbois, A.~Hilion, and M.~Lustig.
\newblock $\mathbb{R}$-trees and laminations for free groups {I} :algebraic
  laminations.
\newblock {\em J. Lond. Math. Soc.}, 78(3):723--736, 2008.

\bibitem[CHL08b]{CHL08-2}
T.~Coulbois, A.~Hilion, and M.~Lustig.
\newblock $\mathbb{R}$-trees and laminations for free groups {II} :the dual
  lamination of an $\mathbb{R}$-tree.
\newblock {\em J. Lond. Math. Soc.}, 78(3):737--754, 2008.

\bibitem[CHL08c]{CHL08-3}
T.~Coulbois, A.~Hilion, and M.~Lustig.
\newblock $\mathbb{R}$-trees and laminations for free groups {III} :currents
  and dual $\mathbb{R}$-tree metrics.
\newblock {\em J. Lond. Math. Soc.}, 78(3):755--766, 2008.

\bibitem[CHL09]{CHL09}
T.~Coulbois, A.~Hilion, and M.~Lustig.
\newblock $\mathbb{R}$-trees, dual laminations and compact systems of
  isometries.
\newblock {\em Math. Proc. Cambridge Phil. Soc.}, 147(2):345--368, 2009.

\bibitem[CHR15]{CHR11}
T.~Coulbois, A.~Hilion, and P.~Reynolds.
\newblock Indecomposable ${F}_{N}$-trees and minimal laminations.
\newblock {\em Groups Geom. Dyn.}, 9(2):567--597, 2015.

\bibitem[CV86]{CuVo_moduli}
M.~Culler and K.~Vogtmann.
\newblock Moduli of graphs and automorphisms of free groups.
\newblock {\em Invent. Math.}, 84(1):91--119, 1986.

\bibitem[For02]{For_deformation}
M.~Forester.
\newblock Deformation and rigidity of simplicial group actions on trees.
\newblock {\em Geom. Topol.}, 6:219--267 (electronic), 2002.

\bibitem[Gab96]{Gaboriau_bouts}
Damien Gaboriau.
\newblock Dynamique des syst\`emes d'isom\'etries: sur les bouts des orbites.
\newblock {\em Invent. Math.}, 126(2):297--318, 1996.

\bibitem[GHon]{GH15}
V.~Guirardel and C.~Horbez.
\newblock Boundary of the free factor graph and subgroups of the automorphism
  group of a free product.
\newblock in preparation.

\bibitem[GL95]{GL95}
D.~Gaboriau and G.~Levitt.
\newblock The rank of actions on $\mathbb{R}$-trees.
\newblock {\em Ann. scient. Ec. Norm. Sup.}, 28(5):549--570, 1995.

\bibitem[GL07]{GL07}
V.~Guirardel and G.~Levitt.
\newblock The outer space of a free product.
\newblock {\em Proc. London Math. Soc.}, 94(3):695--714, 2007.

\bibitem[GLP94]{GLP94}
D.~Gaboriau, G.~Levitt, and F.~Paulin.
\newblock Pseudogroups of isometries of $\mathbb{R}$ and {R}ips' theorem on
  free actions of $\mathbb{R}$-trees.
\newblock {\em Israel J. Math.}, 87:403--428, 1994.

\bibitem[Gui00]{Gui00}
V.~Guirardel.
\newblock Dynamics of $\text{{O}ut}({F}_n)$ on the boundary of outer space.
\newblock {\em Ann. Scient. Éc. Norm. Sup.}, 33(4):433--465, 2000.

\bibitem[Gui04]{Gui_limit}
V.~Guirardel.
\newblock Limit groups and groups acting freely on $\mathbb{R}^n$-trees.
\newblock {\em Geom. Topol.}, 8(3):1427--1470, 2004.

\bibitem[Gui05]{Gui_coeur}
V.~Guirardel.
\newblock Coeur et nombre d'intersection pour les actions de groupes sur les
  arbres.
\newblock {\em Ann. Scient. Ec. Norm. Sup.}, 38(6):847--888, 2005.

\bibitem[Gui08]{Gui_actions}
V.~Guirardel.
\newblock Actions of finitely generated groups on {$\Bbb R$}-trees.
\newblock {\em Ann. Inst. Fourier (Grenoble)}, 58(1):159--211, 2008.

\bibitem[Gup17]{Gup}
R.~Gupta.
\newblock Relative {C}urrents.
\newblock {\em Conform. Geom. Dyn.}, 21:319--352, 2017.

\bibitem[Ham14]{Ham13}
U.~Hamenstädt.
\newblock The boundary of the free splitting graph and of the free factor
  graph.
\newblock {\em arXiv:1211.1630v5}, 2014.

\bibitem[Hor14]{Hor14-3}
C.~Horbez.
\newblock The {T}its alternative for the automorphism group of a free product.
\newblock {\em arXiv:1408.0546v2}, 2014.

\bibitem[Hor17a]{Hor14-1}
C.~Horbez.
\newblock The boundary of the outer space of a free product.
\newblock {\em Israel J. Math.}, 221(1):179--234, 2017.

\bibitem[Hor17b]{Hor14-2}
C.~Horbez.
\newblock Hyperbolic graphs for free products, and the {G}romov boundary of the
  graph of cyclic splittings.
\newblock {\em J. Topol.}, 9(2):401--450, 2017.

\bibitem[Hu65]{Hu}
S.~Hu.
\newblock {\em Theory of retracts}.
\newblock Wayne State University Press, Detroit, 1965.

\bibitem[Kö27]{Kon27}
D.~König.
\newblock Über eine {S}chlussweise aus dem {E}ndlichen ins {U}nendliche.
\newblock {\em Acta Sci. Math. (Szeged)}, 3(2-3):121--130, 1927.

\bibitem[Kap05]{Kap05}
I.~Kapovich.
\newblock The frequency space of a free group.
\newblock {\em Internat. J. Alg. Comput.}, 15(5-6):939--969, 2005.

\bibitem[Kap06]{Kap06}
I.~Kapovich.
\newblock Currents on free groups.
\newblock In R.~Grigorchuk, M.~Mihalik, M.~Sapir, and Z.~Sunik, editors, {\em
  Topological and {A}symptotic {A}spects of {G}roup {T}heory}, volume 394 of
  {\em AMS Contemp. Math. Series}, pages 149--176, 2006.

\bibitem[KL09]{KL09}
I.~Kapovich and M.~Lustig.
\newblock Geometric intersection number and analogues of the curve complex for
  free groups.
\newblock {\em Geom. Topol.}, 13(3):1805--1833, 2009.

\bibitem[Lac97]{Lac}
R.C. Lacher.
\newblock Cell-like mappings and their generalizations.
\newblock {\em Bull. Amer. Math. Soc.}, 83(4):495--552, 1997.

\bibitem[Lev93]{Levitt_exotic}
Gilbert Levitt.
\newblock La dynamique des pseudogroupes de rotations.
\newblock {\em Invent. Math.}, 113(3):633--670, 1993.

\bibitem[Lev94]{Lev94}
G.~Levitt.
\newblock Graphs of actions on $\mathbb{R}$-trees.
\newblock {\em Comment. Math. Helv.}, 69(1):28--38, 1994.

\bibitem[LL03]{LL03}
G.~Levitt and M.~Lustig.
\newblock Irreducible automorphisms of ${F}_n$ have north-south dynamics on
  compactified outer space.
\newblock {\em J. Inst. Math. Jussieu}, 2(1):59--72, 2003.

\bibitem[LP97]{LP}
G.~Levitt and F.~Paulin.
\newblock Geometric group actions on trees.
\newblock {\em Amer. J. Math.}, 119(1):83--102, 1997.

\bibitem[Pau88]{Pau88}
F.~Paulin.
\newblock Topologie de {G}romov équivariante, structures hyperboliques et
  arbres réels.
\newblock {\em Invent. math.}, 94(1):53--80, 1988.

\bibitem[Rey11]{Rey11}
P.~Reynolds.
\newblock On indecomposable trees in the boundary of outer space.
\newblock {\em Geom. Dedic.}, 153(1):59--71, 2011.

\bibitem[Rey12]{Rey12}
P.~Reynolds.
\newblock Reducing systems for very small trees.
\newblock {\em arXiv:1211.3378v1}, 2012.

\bibitem[vM89]{vM}
J.~van Mill.
\newblock {\em Infinite-{D}imensional {T}opology. {P}rerequisites and
  {I}ntroduction}, volume~43 of {\em North-Holland Mathematical Library}.
\newblock Elsevier Science Publishers, Amsterdam, 1989.

\bibitem[Whi36]{Whi36}
J.H.C. Whitehead.
\newblock On equivalent sets of elements in a free group.
\newblock {\em Ann. Math.}, 37(2):782--800, 1936.

\end{thebibliography}

 \begin{flushleft}
 Vincent Guirardel\\
 Institut de Recherche Math\'ematique de Rennes\\
 Universit\'e de Rennes 1 et CNRS (UMR 6625)\\
 263 avenue du G\'en\'eral Leclerc, CS 74205\\
 F-35042  RENNES C\'edex\\
 \emph{e-mail:}\texttt{vincent.guirardel@univ-rennes1.fr}\\[8mm]
 \end{flushleft}

\begin{flushleft}
Camille Horbez\\
CNRS\\ 
Laboratoire de Math\'ematique d'Orsay\\
Universit\'e Paris-Sud et CNRS (UMR 8628), Universit\'e Paris-Saclay\\ 
F-91405 ORSAY\\
\emph{e-mail:}\texttt{camille.horbez@math.u-psud.fr}
\end{flushleft}

\end{document}